\documentclass[11pt,reqno]{amsart}

\newtheorem{thm}{Theorem}[section]
\newtheorem*{thm*}{Theorem}
\newtheorem{lem}[thm]{Lemma}
\newtheorem*{lem*}{Lemma}

\newtheorem{cor}[thm]{Corollary}
\newtheorem{claim}[thm]{Claim}
\newtheorem{prop}[thm]{Proposition}

\theoremstyle{definition}

\newtheorem{assump}[thm]{Assumption}
\renewcommand{\thecase}{}
\newtheorem*{case*}{Case}

\newtheorem{defn}[thm]{Definition}
\newtheorem*{defn*}{Definition}
\newtheorem{exmp}[thm]{Example}
\newtheorem*{exmp*}{Example}
\newtheorem{hyp}[thm]{Hypothesis}

\newtheorem{notn}[thm]{Notation} 
\newtheorem{prob}[thm]{Problem}
\newtheorem{rmk}[thm]{Remark}
\newtheorem*{rmk*}{Remark}
\renewcommand{\thestep}{}

\theoremstyle{remark}

\makeatletter
\def\alphenumi{
  \def\theenumi{\alph{enumi}}
  \def\p@enumi{\theenumi}
  \def\labelenumi{(\@alph\c@enumi)}}
\makeatother

\makeatletter
\def\thecase{\@arabic\c@case}
\makeatother

\makeatletter
\def\thestep{\@arabic\c@step}
\makeatother

\newcount\hh
\newcount\mm
\mm=\time
\hh=\time
\divide\hh by 60
\divide\mm by 60
\multiply\mm by 60
\mm=-\mm
\advance\mm by \time
\def\hhmm{\number\hh:\ifnum\mm<10{}0\fi\number\mm}

\let\oldmarginpar\marginpar
\renewcommand\marginpar[1]{\-\oldmarginpar[\raggedleft\footnotesize #1]%
{\raggedright\footnotesize #1}}

\renewcommand\emptyset{\varnothing}

\newcommand\HH{\mathbb{H}}

\newcommand\KK{\mathbb{K}}

\newcommand\RR{\mathbb{R}}

\newcommand\fm{{\mathfrak{m}}}

\newcommand\fw{{\mathfrak{w}}}

\newcommand\sB{{\mathscr{B}}}
\newcommand\sC{{\mathscr{C}}}

\newcommand\sO{{\mathscr{O}}}

\newcommand\sU{{\mathscr{U}}}

\newcommand\eps{\varepsilon}

\newcommand\less{\setminus}

\DeclareMathOperator{\esssup}{ess\,sup}

\newcommand\sign{\operatorname{sign}}

\newcommand\supp{\operatorname{supp}}

\numberwithin{equation}{section}

\newcommand{\essinf}{\operatornamewithlimits{ess\ inf}}
\renewcommand{\esssup}{\operatornamewithlimits{ess\ sup}}

\usepackage{amssymb}
\usepackage{hyperref}
\usepackage{mathrsfs}
\usepackage[usenames]{color}
\usepackage{url}

\hypersetup{pdftitle={Existence, Uniqueness and Regularity for Obstacle Problems}}
\hypersetup{pdfauthor={Panagiota Daskalopoulos and Paul M. N. Feehan}}

\begin{document}

\title[Existence, Uniqueness and Regularity for Obstacle Problems]
{Existence, Uniqueness, and Global Regularity for Degenerate Elliptic Obstacle Problems in Mathematical Finance}

\author[P. Daskalopoulos]{Panagiota Daskalopoulos}
\address{Department of Mathematics, Columbia University, New York, NY 10027}
\email[P. Daskalopoulos]{pdaskalo@math.columbia.edu}

\author[P. Feehan]{Paul M. N. Feehan}
\address{Department of Mathematics, Rutgers, The State University of New Jersey, 110 Frelinghuysen Road, Piscataway, NJ 08854}
\email[P. Feehan]{feehan@math.rutgers.edu}

\begin{abstract}
The Heston stochastic volatility process, which is widely used as an asset price model in mathematical finance, is a paradigm for a degenerate diffusion process where the degeneracy in the diffusion coefficient is proportional to the square root of the distance to the boundary of the half-plane. The generator of this process with killing, called the elliptic Heston operator, is a second-order degenerate elliptic partial differential operator whose coefficients have linear growth in the spatial variables and where the degeneracy in the operator symbol is proportional to the distance to the boundary of the half-plane. With the aid of weighted Sobolev spaces, we prove existence, uniqueness, and global regularity of solutions to stationary variational inequalities and obstacle problems for the elliptic Heston operator on unbounded subdomains of the half-plane. In mathematical finance, solutions to obstacle problems for the elliptic Heston operator correspond to value functions for perpetual American-style options on the underlying asset.
\end{abstract}

%



\thanks{PD was partially supported by NSF grant DMS-0905749. PF was partially supported by NSF grant DMS-1059206}

\date{September 6, 2011}

\maketitle
\tableofcontents

\section{Introduction}
\label{subsec:Intro}
We consider questions of existence, uniqueness, and regularity of solutions, $u:\sO\to\RR$, to the obstacle problem
\begin{equation}
\label{eq:IntroObstacleProblem}
\min\{Au-f,u-\psi\} = 0 \quad \hbox{a.e. on }\sO, \quad u = g \quad\hbox{on } \Gamma_1,
\end{equation}
where $\sO\subset\HH$ is a possibly unbounded domain in the open upper half-plane $\HH := \RR^{d-1}\times(0,\infty)$ (where $d\geq 2$), $\Gamma_1 = \partial\sO\cap\HH$ is the portion of the boundary $\partial\sO$ of $\sO$ which lies in $\HH$, $f:\sO\to\RR$ is a source function, the function $g:\sO\cup\Gamma_1\to\RR$ prescribes a Dirichlet boundary condition along $\Gamma_1$ and $\psi:\sO\cup\Gamma_1\to\RR$ is an obstacle function which is compatible with $g$ in the sense that $\psi\leq g$ on $\Gamma_1$, while $A$ is an elliptic differential operator on $\sO$ which is degenerate along the interior, $\Gamma_0$, of $\{y=0\}\cap\partial\sO$ and which we require to non-empty throughout this article. However, no boundary condition is prescribed along $\Gamma_0$. Rather, we shall see that the problem \eqref{eq:IntroObstacleProblem} is well-posed when we seek solutions in suitable function spaces which describe their qualitative behavior near the boundary portion $\Gamma_0$: for example, continuity of derivatives up to $\Gamma_0$ via suitable weighted H\"older spaces (by analogy with \cite{DaskalHamilton1998}) or integrability of derivatives in a neighborhood of $\Gamma_0$ via suitable weighted Sobolev spaces (by analogy with \cite{Koch}). In this article, we set $d=2$ and choose $A$ to be the generator of the two-dimensional Heston stochastic volatility process with killing \cite{Heston1993}, a degenerate diffusion process well known in mathematical finance and a paradigm for a broad class of degenerate Markov processes, driven by $d$-dimensional Brownian motion, and corresponding generators which are degenerate elliptic integro-differential operators:
\begin{equation}
\label{eq:OperatorHestonIntro}
Av := -\frac{y}{2}\left(v_{xx} + 2\rho\sigma v_{xy} + \sigma^2 v_{yy}\right) - (r-q-y/2)v_x - \kappa(\theta-y)v_y + rv, \quad v\in C^\infty(\HH).
\end{equation}
Throughout this article, the coefficients of $A$ are required to obey

\begin{assump}
[Ellipticity condition for the Heston operator coefficients]
\label{assump:HestonCoefficients}
The coefficients defining $A$ in \eqref{eq:OperatorHestonIntro} are constants obeying
\begin{equation}
\label{eq:EllipticHeston}
\sigma \neq 0, -1< \rho < 1,
\end{equation}
and $\kappa>0$, $\theta>0$, $r\geq 0$, and $q \geq 0$.
\end{assump}

\begin{rmk}[A change of variables and the Heston operator coefficients]
\label{rmk:HestonCoefficientb1}
With the aid of simple affine changes of variables on $\RR^2$ which maps $(\HH,\partial\HH)$ onto $(\HH,\partial\HH)$ (Lemma \ref{lem:RescalingHeston}), we can also arrange that the combination of coefficients, $b_1 = r-q- \kappa\theta\rho/\sigma$, is zero and, unless stated otherwise, we shall rely this fact (Assumption \ref{assump:HestonCoefficientb1}) when convenient throughout our article; the constant $b_1$ is one of the coefficients of the derivative, $u_x$, appearing in the bilinear form, $a(\cdot,\cdot)$ (Definition \ref{defn:HestonWithKillingBilinearForm}), associated with the operator $A$.
\end{rmk}

A recent citation search revealed that almost 900 articles\footnote{A Thompson-Reuters Web of Knowledge \cite{ThomsonReuters} citation search performed on June 16, 2011 yielded 883 references.} in scientific journals cite the stochastic volatility model proposed by Steven Heston in \cite{Heston1993} and even this may not include articles on related stochastic volatility models or unpublished technical reports by researchers at industry financial engineering groups. The widespread use of degenerate stochastic processes in financial engineering highlights the need to address a circle of unresolved fundamental questions concerning degenerate Markov processes and related obstacle and boundary value problems. As we describe in \S \ref{subsec:Survey} and \S \ref{subsec:Guide}, important questions regarding existence, uniqueness, and regularity of solutions to problem \eqref{eq:IntroObstacleProblem} or problem \eqref{eq:IntroBoundaryValueProblem} below have not been addressed thus far in the literature on degenerate partial differential operators.

In mathematical finance, a solution $u$ to the elliptic obstacle problem \eqref{eq:IntroObstacleProblem} when $f=0$ can be interpreted as the value function for a \emph{perpetual American-style option} with \emph{payoff} function given by the obstacle function, $\psi$, while a solution $u$ to the corresponding \emph{parabolic} obstacle problem on $\sO\times[0,T]$, with $0<T<\infty$, can be interpreted as the value function for a \emph{finite-maturity} American-style option with payoff function given by a terminal condition function, $h:\sO\to\RR$, which typically coincides on $\sO\times\{T\}$ with the obstacle function, $\psi$. For example, in the case of an American-style put option, one chooses $\psi(x,y) = (E-e^x)^+, (x,y)\in\sO$, where $E>0$ is a positive constant. This class of obstacle problems may be generalized further by considering problems with two obstacles, such as upper and lower obstacle functions, $\psi_1$ and $\psi_2$ \cite{Friedman_1982}, \cite{Petrosyan_Shagholian_Uraltseva}.

To provide a stepping-stone to a solution to the obstacle problem, we shall first need to consider  questions of existence, uniqueness, and regularity of solutions to the elliptic boundary value problem,
\begin{equation}
\label{eq:IntroBoundaryValueProblem}
Au = f \quad \hbox{a.e. on }\sO, \quad u = g \quad\hbox{on } \Gamma_1.
\end{equation}
Like problem \eqref{eq:IntroObstacleProblem}, we will see that \eqref{eq:IntroBoundaryValueProblem} is well-posed without a boundary condition along $\Gamma_0$ when we seek solutions in suitable weighted H\"older or weighted Sobolev spaces. While solutions to \eqref{eq:IntroBoundaryValueProblem} do not have an immediate interpretation in mathematical finance, a solution, $u$, to the corresponding \emph{parabolic} boundary value problem on $\sO\times[0,T]$ can be interpreted as the value function for a \emph{European-style option} with payoff function given by a terminal condition function, $h:\sO\to\RR$. For example, in the case of a European-style put option, one chooses $h(x,y) = (E-e^x)^+, (x,y)\in\sO$.

\subsection{Summary of main results}
\label{subsec:Summary}
We shall state a selection of our main results here and then refer the reader to our guide to this article in \S \ref{subsec:Guide} for more of our results on existence, uniqueness and regularity of solutions to variational equations and inequalities and corresponding obstacle problems. We shall seek solutions to \eqref{eq:IntroObstacleProblem} in the weighted Sobolev space (see Definitions \ref{defn:H1WeightedSobolevSpaces} and \ref{defn:H2WeightedSobolevSpaces})
$$
H^2(\sO,\fw) = \{u \in L^2(\sO,\fw): (1+y)^{1/2}u, (1+y)|Du|,  y|D^2u| \in L^2(\sO,\fw)\},
$$
where the domain $\sO$ is as in Definition \ref{defn:HestonDomain}, $Du=(u_x,u_y)$, $D^2u = (u_{xx}, u_{xy}, u_{yx}, u_{yy})$, all derivatives of $u$ are defined in the sense of distributions, and
$$
\|u\|_{H^2(\sO,\fw)}^2 = \int_\sO\left( y^2|D^2u|^2 + (1+y)^2|Du|^2 + (1+y)u^2\right)\,\fw\,dxdy,
$$
with weight function $\fw:\HH\to(0,\infty)$ given by
$$
\fw(x,y) = y^{\beta-1}e^{-\gamma|x|-\mu y}, \quad (x,y) \in \HH,
$$
where $\beta = 2\kappa\theta/\sigma^2$, $\mu = 2\kappa/\sigma^2$, and $0<\gamma<\gamma_0$, where $\gamma_0$ depends only on the constant coefficients of $A$ in \eqref{eq:OperatorHestonIntro}.

\subsubsection{Existence, uniqueness, and regularity of solutions to the obstacle problem}
We first summarize our main results concerning the obstacle problem \eqref{eq:IntroObstacleProblem}. Because the bilinear form (Definition \ref{defn:HestonWithKillingBilinearForm}) defined by the operator $A$ is non-coercive, the domain $\sO$ is unbounded, the coefficients of $A$ are unbounded, and the Rellich-Kondrachov compact embedding theorem does not always hold for weighted Sobolev spaces or unbounded domains, we shall need to seek solutions when the source function obeys certain pointwise growth properties. Therefore, we introduce the

\begin{defn}[Admissible envelope functions for the obstacle problem]
\label{defn:EnvelopeFunctionsObstacle}
Given $g, \psi\in H^2(\sO,\fw)$, we call $M, m \in H^2(\sO,\fw)$ a pair of \emph{admissible envelope functions for the obstacle problem} \eqref{eq:IntroObstacleProblem} if
$$
m \leq g \leq M \hbox{ on }\Gamma_1, \quad m \leq M \hbox{ on }\sO, \quad Am\leq AM \hbox{ a.e on }\sO, \quad\hbox{and}\quad \psi \leq M \hbox{ on }\sO,
$$
and $M, m$ obey
\begin{gather*}
(1+y)^2M, (1+y)^2m \in L^2(\sO,\fw),
\\
(1+y)^{1/2}M, (1+y)^{1/2}m  \in L^q(\sO,\fw),
\end{gather*}
for some $q>2$.
\end{defn}

\begin{defn}[Admissible source function for the obstacle problem]
\label{defn:SourceFunctionObstacle}
Given $M, m \in H^2(\sO,\fw)$, we call $f \in L^2(\sO,\fw)$ an \emph{admissible source function for the obstacle problem} \eqref{eq:IntroObstacleProblem} if
\begin{gather*}
Am\leq f \leq AM \hbox{ a.e on }\sO,
\\
(1+y)f\in L^2(\sO,\fw).
\end{gather*}
\end{defn}

In order to prove uniqueness of solutions to the non-coercive variational equation (Problem \ref{prob:HestonMixedBVPInhomogeneous}) corresponding to \eqref{eq:IntroBoundaryValueProblem} or inequality (Problem \ref{prob:NonhomogeneousHestonVIProblem}) corresponding to \eqref{eq:IntroObstacleProblem}, we shall need the following auxiliary

\begin{defn}[Barrier function for uniqueness of solutions]
\label{defn:BarrierFunctionObstacle}
Given $M, m, g \in H^2(\sO,\fw)$, we call $\varphi \in H^2(\sO,\fw)$ a \emph{barrier function} for the operator $A$ on the domain $\sO$ if
\begin{equation}
\label{eq:MainTheoremVarphi}
\begin{gathered}
A\varphi\geq Ag \quad\hbox{a.e. on }\sO, \quad A(m+\varphi)> 2Ag \quad\hbox{a.e. on }\sO, \quad\hbox{and}\quad \varphi\geq g \quad\hbox{on }\Gamma_1,
\\
(1+y)\varphi \in L^2(\sO,\fw) \quad\hbox{and}\quad (1+y)^{1/2}\varphi \in L^q(\sO,\fw),
\\
\esssup_{(x,y)\in\sO}\frac{(1+y)(M+\varphi-2g)(x,y)}{A(m+\varphi-2g)(x,y)} < \infty.
\end{gathered}
\end{equation}
\end{defn}

\begin{thm}[Existence and uniqueness of solutions to the obstacle problem]
\label{thm:MainExistenceUniquenessObstacleProblem}
Assume that the constant $r$ in \eqref{eq:OperatorHestonIntro} is strictly positive and that the domain $\sO$ obeys Hypothesis \ref{hyp:DomainCombinedCondition}. Given $g, \psi\in H^2(\sO,\fw)$ which are compatible in the sense that,
$$
\psi\leq g \hbox{ on }\Gamma_1,
$$
let $M, m \in H^2(\sO,\fw)$ be a pair of admissible envelope functions in the sense of Definition \ref{defn:EnvelopeFunctionsObstacle} and, in addition, require that $g, \psi$ obey
$$
(1+y)^{1/2}g \in L^q(\sO,\fw) \quad\hbox{and}\quad (1+y)^{3/2}g, (1+y)\psi \in H^2(\sO,\fw),
$$
for some $q>2$. Let $f \in L^2(\sO,\fw)$ be an admissible source function in the sense of Definition \ref{defn:EnvelopeFunctionsObstacle}. Furthermore, require that there is a barrier function $\varphi \in H^2(\sO,\fw)$ obeying \eqref{eq:MainTheoremVarphi}. Then there exists a unique solution $u\in H^2(\sO,\fw)$ to \eqref{eq:IntroObstacleProblem}, $(1+y)^2u\in L^2(\sO,\fw)$, and $u$ obeys
\begin{gather*}
\max\{m,\psi\} \leq u \leq M \hbox{ on }\sO,
\\
y^\beta(\rho u_x + \sigma u_y) = 0 \quad\hbox{(trace sense)},
\end{gather*}
and there is a positive constant $C$ depending only on the constant coefficients of $A$ and the constants in Hypothesis \ref{hyp:HestonDomainNearGammaOne} prescribing the geometry of the boundary, $\Gamma_1$, such that
\begin{align*}
\|u\|_{H^2(\sO,\fw)} &\leq C\left(\|(1+y)f\|_{L^2(\sO,\fw)} + \|(1+y)^{3/2}g\|_{H^2(\sO,\fw)}  \right.
\\
&\quad + \left. \|(1+y)\psi\|_{H^2(\sO,\fw)} + \|(1+y)^2u\|_{L^2(\sO,\fw)} \right).
\end{align*}
\end{thm}

\begin{rmk}[References to hypotheses in the body of the article]
The hypotheses in Theorem \ref{thm:MainExistenceUniquenessObstacleProblem}, in addition to those on the domain, $\sO$, summarize the conditions \eqref{eq:fBounds} and \eqref{eq:ObstacleFunctionLessThanBoundaryConditionFunction}, together with Hypotheses \ref{hyp:UpperLowerBoundsSolutionsCoercive}, \ref{hyp:AuxBoundUniquenessSolutionsNoncoerciveEquation}, \ref{hyp:NoncoerciveHeston}, \ref{hyp:UpperBoundObstacleFunction}, \ref{hyp:UpperLowerBoundsSolutionsNoncoerciveInequality}, \ref{hyp:AuxBoundUniquenessSolutionsNoncoerciveInequality}, and \ref{hyp:GlobalH2UpperLowerBoundsSolutionsNoncoerciveInequality}, except that here we allow $g$ to be non-zero on $\Gamma_1$. The hypotheses on $g$ arise from the reduction of the inhomogeneous Dirichlet boundary condition, $u=g$ on $\Gamma_1$, to the homogeneous case by subtracting $g$ from $\psi,M,m,\varphi,u$ and subtracting $Ag$ from $f$.
\end{rmk}

\begin{rmk}[Properties of the solution near the boundary portion $\Gamma_0$]
\label{rmk:WeightedNeumannBoundaryProperty}
According to Lemma \ref{lem:ybetaDuZeroTrace}, the ``weighted Neumann property'' for $u$ on the boundary portion $\Gamma_0$ asserted in Theorem \ref{thm:MainExistenceUniquenessObstacleProblem}, that $y^\beta(\rho u_x + \sigma u_y) = 0$ (trace sense), is equivalent to
$$
y^\beta(\rho u_x + \sigma u_y) \to 0 \quad\hbox{in } L^1(\Gamma_0,e^{-\gamma|x|}\,dx) \hbox{ as }y\downarrow 0.
$$
In a sequel to this article, we shall show that under suitable additional regularity hypotheses on the source function, $f$, the solution, $u$, is at least $C^1$ up to $\Gamma_0$ and this trivially implies that the weighted Neumann property for $u$ on $\Gamma_0$ is obeyed.
\end{rmk}

\begin{rmk}[Example of an admissible domain]
A simple example of a domain obeying the conditions of Hypothesis \ref{hyp:HestonDomainNearGammaZero} is $\sO = (x_0,x_1)\times(0,\infty)$, where $-\infty\leq x_0<x_1\leq\infty$, with $\Gamma_0 = (x_0,x_1)\times\{0\}$ and $\Gamma_1 = \{x_0,x_1\}\times(0,\infty)$.
\end{rmk}

\begin{rmk}[Examples of functions $M,m,\varphi,f,g$ obeying the hypotheses]
See Lemmas \ref{lem:PointwiseBoundForSolution} and \ref{lem:VIUniquenessEllipticHestonPostivefuSufficientConditions} for a broad class of non-trivial examples of functions $M,m,\varphi,f,g$ obeying the hypotheses of Theorem \ref{thm:MainExistenceUniquenessObstacleProblem}.
\end{rmk}

\begin{rmk}[Local $H^2$ regularity of solutions]
In applications to mathematical finance, the obstacle function, $\psi$, is typically only Lipschitz (for example, $\psi(x,y) = (E-x)^+$) and only in $H^2(\sU,\fw)$ for some possibly unbounded subdomain $\sU\subsetneqq\sO$. Theorem \ref{thm:H2BoundSolutionHestonVarIneqLipschitz} provides a local version of Theorem \ref{thm:MainExistenceUniquenessObstacleProblem} and shows that $u \in H^2(\sU',\fw)$ for subdomains $\sU'\subset\sU$.
\end{rmk}

\begin{thm}[Regularity of solutions to the obstacle problem in the interior and up to the boundary portion $\Gamma_1$]
\label{thm:MainRegularityObstacleProblem}
Assume the hypotheses of Theorem \ref{thm:MainExistenceUniquenessObstacleProblem} and, for $2<p<\infty$, that
$$
f \in L^p_{\textrm{loc}}(\sO\cup\Gamma_1) \quad\hbox{and}\quad g, \psi \in W^{2,p}_{\textrm{loc}}(\sO\cup\Gamma_1).
$$
Then $u \in W^{2,p}_{\textrm{loc}}(\sO\cup\Gamma_1)$ and, if $\alpha = 1-2/p$, then $u \in C^{1,\alpha}_{\textrm{loc}}(\sO\cup\Gamma_1)$.
\end{thm}

\begin{thm}[Optimal interior regularity of solutions to the obstacle problem]
\label{thm:MainOptimalRegularityObstacleProblem}
Assume the hypotheses of Theorem \ref{thm:MainExistenceUniquenessObstacleProblem} and, for $0<\alpha<1$, that
$$
f \in C^\alpha_{\textrm{loc}}(\sO\cup\Gamma_1) \quad\hbox{and}\quad g, \psi \in C^2_{\textrm{loc}}(\sO\cup\Gamma_1).
$$
Then $u \in C^{1,1}(\sO)$.
\end{thm}

\begin{rmk}[Local $C^{1,1}$ regularity of solutions]
If $\psi$ is only $C^2$ on a relatively open subset of $\sO\cup\Gamma_1$, then Theorem \ref{thm:MainOptimalRegularityObstacleProblem} may be localized as described in Corollary \ref{cor:LocalC11Regularity}.
\end{rmk}

\begin{rmk}[Optimal interior regularity of solutions to the obstacle problem]
It is well-known that the best possible regularity of a solution, $u$, to an elliptic obstacle problem is $u\in C^{1,1}(\sO) = W^{2,\infty}_{\textrm{loc}}(\sO)$, even when the source, boundary data, and obstacle functions and domain boundary are $C^\infty$. A simple, explicit one-dimensional example from mathematical finance which illustrates this phenomenon (albeit with a Lipschitz obstacle function) is provided by the perpetual American-style put option when the underlying asset process is geometric Brownian motion with drift \cite[\S 8.3]{Shreve2}. For the open subset $\sC(u) := \{u>\psi\} \subset \sO$ where $Au=f$, we expect the solution, $u$, to be $C^\infty$ on $\sC(u)$ when $f$ is $C^\infty$ on $\sO$.
\end{rmk}

\begin{rmk}[Optimal regularity of solutions to the obstacle problem up to the boundary portion $\Gamma_1$]
If we strengthened the compatibility condition $\psi\leq g$ on $\Gamma_1$ to $\psi<g$ on $\Gamma_1$, then \cite[Corollary 6.3]{Rodrigues_1987} would yield $u \in C^{1,1}_\textrm{loc}(\sO\cup\Gamma_1)$.
\end{rmk}

\begin{rmk}[Extensions of preceding results in sequels to this article]
See \S \ref{subsec:Sequels} for a survey of our research on extensions and applications of Theorems \ref{thm:MainExistenceUniquenessObstacleProblem}, \ref{thm:MainRegularityObstacleProblem}, and \ref{thm:MainOptimalRegularityObstacleProblem} in sequels to this article.
\end{rmk}

\subsubsection{Existence, uniqueness, and regularity of solutions to the boundary value problem}
Next, we summarize our main results concerning the boundary value problem \eqref{eq:IntroBoundaryValueProblem}. When $\sO=\HH$, it is possible to construct the fundamental solution in terms of confluent hypergeometric functions using the Fourier-Laplace transform (see, for example, \cite{Feehan_integralheston}) or by making use of the affine structure of the coefficients and adapting the method of Heston \cite{Heston1993} (see also \cite{DuffiePanSingleton2000}). However, while explicit formulae for the fundamental solution are important, they alone provide little insight into the boundary behavior of solutions to \eqref{eq:IntroBoundaryValueProblem} or questions of well-posedness for \eqref{eq:IntroBoundaryValueProblem} or the existence of Green's functions when $\sO$ is replaced by even relatively simple domains such as a quadrant, $\RR^+\times\RR^+$, or infinite strip, $(x_0,x_1)\times\RR^+$.

As in the hypotheses of Theorem \ref{thm:MainExistenceUniquenessObstacleProblem}, we require the existence of certain admissible envelope functions, $M,m\in H^2(\sO,\fw)$ compatible with $g\in H^2(\sO,\fw)$, a source function, $f\in L^2(\sO,\fw)$, with admissible growth, and a barrier function, $\varphi\in H^2(\sO,\fw)$. However, the requirements are simpler.

\begin{thm}[Existence and uniqueness of solutions to the boundary value problem]
\label{thm:MainExistenceUniquenessBoundaryValueProblem}
Assume that the constant $r$ in \eqref{eq:OperatorHestonIntro} is strictly positive and that the domain $\sO$ obeys Hypothesis \ref{hyp:DomainCombinedCondition}. Let $f \in L^2(\sO,\fw)$ and $g \in H^2(\sO,\fw)$. Suppose there are functions $M, m \in H^2(\sO,\fw)$ such that
$$
m \leq g \leq M \hbox{ on }\Gamma_1, \quad m \leq M \hbox{ on }\sO, \quad\hbox{and}\quad Am\leq f \leq AM \hbox{ a.e on }\sO,
$$
and $M, m, f$, and $g$ obey
$$
(1+y)M, (1+y)m, (1+y)^{1/2}f\in L^2(\sO,\fw) \quad\hbox{and}\quad (1+y)^{1/2}g\in H^2(\sO,\fw),
$$
then there exists a solution $u\in H^2(\sO,\fw)$ to \eqref{eq:IntroBoundaryValueProblem}, $(1+y)u\in L^2(\sO,\fw)$, and $u$ obeys
\begin{gather*}
m \leq u \leq M \hbox{ on }\sO,
\\
y^\beta(\rho u_x + \sigma u_y) = 0 \quad\hbox{(trace sense)},
\end{gather*}
and there is a positive constant, $C$, depending only on the constant coefficients of the operator, $A$, and the constants in Hypothesis \ref{hyp:HestonDomainNearGammaOne} prescribing the geometry of the boundary, $\Gamma_1$, such that
$$
\|u\|_{H^2(\sO,\fw)} \leq C\left(\|(1+y)^{1/2}f\|_{L^2(\sO,\fw)} + \|(1+y)^{1/2}g\|_{H^2(\sO,\fw)} + \|(1+y)u\|_{L^2(\sO,\fw)}\right).
$$
If there is a function $\varphi \in H^2(\sO,\fw)$ obeying \eqref{eq:MainTheoremVarphi}, then the solution $u$ is unique.
\end{thm}

\begin{rmk}[References to hypotheses in the body of the article]
The hypotheses in Theorem \ref{thm:MainExistenceUniquenessBoundaryValueProblem}, in addition to those on the domain, $\sO$, summarize the conditions \eqref{eq:fBounds}, \eqref{eq:Sqrt1plusyfL2}, and \eqref{eq:OneplusyMmInL2}, together with Hypotheses \ref{hyp:UpperLowerBoundsSolutionsCoercive}, \ref{hyp:NoncoerciveHeston}, \ref{hyp:AuxBoundUniquenessSolutionsNoncoerciveEquation}, and \ref{hyp:AuxBoundUniquenessSolutionsNoncoerciveInequality}, except that here we allow $g$ to be non-zero on $\Gamma_1$.  The hypotheses on $g$ arise from the reduction of the inhomogeneous Dirichlet boundary condition, $u=g$ on $\Gamma_1$, to the homogeneous case by subtracting $g$ from $M,m,\varphi,u$ and subtracting $Ag$ from $f$.
\end{rmk}

Remark \ref{rmk:WeightedNeumannBoundaryProperty} also applies to the solution $u$ provided by Theorem \ref{thm:MainExistenceUniquenessBoundaryValueProblem}.

\begin{thm}[Regularity of solutions to the boundary value problem]
\label{thm:MainRegularityBoundaryValueProblem}
Assume the hypotheses of Theorem \ref{thm:MainExistenceUniquenessBoundaryValueProblem} and also assume $f\in L^q_{\textrm{loc}}(\sO\cup\Gamma_0)$ and $g \in W^{2,q}_{\textrm{loc}}(\sO\cup\Gamma_0)$, for some $q>2+\beta$. Then the solution $u$ to \eqref{eq:IntroBoundaryValueProblem} provided by Theorem \ref{thm:MainExistenceUniquenessBoundaryValueProblem} is in $C^\alpha_{\textrm{loc}}(\sO\cup\Gamma_1)\cap C_{\textrm{loc}}(\bar\sO)$, for $\alpha\in (0,1)$. If in addition, $f\in C^{k,\alpha}_{\textrm{loc}}(\sO\cup\Gamma_1)$ and $g\in C^{k+2,\alpha}_{\textrm{loc}}(\sO\cup\Gamma_1)$ for an integer $k\geq 0$, and the boundary portion $\Gamma_1$ is $C^{k+2,\alpha}$, then the solution $u$ lies in $C^{k+2,\alpha}_{\textrm{loc}}(\sO\cup\Gamma_1)\cap C_{\textrm{loc}}(\bar\sO)$.
\end{thm}

\begin{rmk}[Extensions of preceding results in sequels to this article]
See \S \ref{subsec:Sequels} for a survey of our research on extensions and applications of Theorems \ref{thm:MainExistenceUniquenessBoundaryValueProblem} and \ref{thm:MainRegularityBoundaryValueProblem} in sequels to this article.
\end{rmk}

\subsection{Survey of previous research in degenerate boundary value and obstacle problems}
\label{subsec:Survey}
Questions of existence, uniqueness, and regularity of solutions to ``standard'' obstacle problems (for example, bounded domains with smooth boundary, uniformly elliptic differential operators with smooth coefficients, bounded and smooth functions, and smooth boundary data) are addressed by Bensoussan and Lions \cite{Bensoussan_Lions}, Friedman \cite{Friedman_1982}, Kinderlehrer and Stampacchia \cite{Kinderlehrer_Stampacchia_1980}, Petrosyan, Shagholian, and Uralt'seva \cite{Petrosyan_Shagholian_Uraltseva}, Rodrigues \cite{Rodrigues_1987}, and Troianello \cite{Troianiello}, and elsewhere. Bensoussan and Lions \cite{Bensoussan_Lions} provide a comprehensive treatment of both standard problems and certain extensions which allow, in certain cases and combinations, for unbounded domains and non-coercive operators. However, as we shall see, the features apparent in \eqref{eq:IntroObstacleProblem} present a particular combination of difficulties which, as far as we can tell, is not addressed in the literature. These difficulties include the
\begin{enumerate}
\item Degeneracy of the operator, $A$, along the domain boundary portion $\Gamma_0$,
\item Non-coercivity of the bilinear form, $a(\cdot,\cdot)$, associated with $A$,
\item Unboundedness of the domain, $\sO$,
\item Unboundedness of the coefficients of $A$,
\item Lipschitz regularity of the obstacle function, $\psi$, and
\item Corner points of the domain where the boundary portions $\Gamma_0$ and $\Gamma_1$ meet.
\end{enumerate}
We also allow the source function, $f$, Dirichlet boundary data function, $g$, and obstacle function, $\psi$, to be unbounded. While these choices do present some additional difficulties, these functions are often bounded in typical applications. Also, though the points of $\partial\sO$ where $\Gamma_0$ and $\Gamma_1$ meet are \emph{geometric} corner points, we know from \cite{DaskalHamilton1998} that because $A$ is degenerate along $\Gamma_0$, we may consider, in a certain sense, the boundary portion, $\Gamma_0$, to be an ``interior'' subset of the domain, $\sO$.

\subsubsection{Degenerate partial differential equations}
We next provide a brief survey of some of the literature relevant to problems \eqref{eq:IntroObstacleProblem} and \eqref{eq:IntroBoundaryValueProblem}. Earlier research treating existence, uniqueness, and regularity problems for degenerate elliptic or parabolic partial differential equations includes the articles by Fichera \cite{Fichera_1956}, Glushko and Savchenko \cite{GlushkoSavchenko}, Kohn and Nirenberg \cite{Kohn_Nirenberg_1967}, McKean \cite{McKean_1956}, Murthy and Stampacchia \cite{Murthy_Stampacchia_1968}, Stroock and Varadhan \cite{StroockVaradhan1972}, and monographs such as those of Drabek, Kufner, and Nicolosi \cite{DrabekKufnerNicolosi}, Freidlin \cite{Freidlin}, Friedman \cite{FriedmanSDE}, Levendorski{\u\i} \cite{LevendorskiDegenElliptic}, and Ole{\u\i}nik and Radkevi{\v{c}} \cite{Oleinik_Radkevic}. The accumulation of related research in this field has become vast and we will not attempt to survey it comprehensively here except to note that the hypotheses required by the main theorems in well-known articles such as \cite{Kohn_Nirenberg_1967, Murthy_Stampacchia_1968} either exclude problems such as \eqref{eq:IntroBoundaryValueProblem} because their hypotheses are too restrictive or yield conclusions which are not as strong as Theorem \ref{thm:MainExistenceUniquenessBoundaryValueProblem}.

More recently, the porous medium equation --- a degenerate, quasi-linear parabolic partial differential equation --- has stimulated development of the theory of regularity of solutions to degenerate partial differential equations. The $C^\infty$-regularity of solutions up to the boundary and the $C^\infty$-regularity of the ``free boundaries'' in porous medium problems has been proved by P. Daskalopoulos and her collaborators \cite{DaskalHamilton1998, Daskalopoulos_Rhee_2003} using weighted H\"older spaces defined by a ``cycloidal'' Riemmianian metric on the upper half-plane and independently by Koch \cite{Koch} using a combination of weighted $L^p$ spaces and weighted H\"older spaces defined by the cycloidal metric. (Note, however, that free boundaries in porous medium problems are not necessarily free boundaries in the sense of obstacle problems.) The linearization of the porous medium equation has a structure which is similar to the (parabolic) Heston equation  \cite{DaskalHamilton1998, Koch}, and so research on the porous medium equation is especially relevant to problems \eqref{eq:IntroObstacleProblem} and \eqref{eq:IntroBoundaryValueProblem}. We shall not use weighted H\"older spaces in this article, although we shall in sequels focusing on regularity near the boundary portion, $\Gamma_0$, of solutions to \eqref{eq:IntroObstacleProblem} or \eqref{eq:IntroBoundaryValueProblem}. The weights defining our weighted Sobolev spaces use the same powers of the distance to the boundary portion, $\Gamma_0$, as those employed by Koch, but we also include exponential decay factors which ensure that the upper half-plane, $\HH$, has finite volume with respect to the measure defined by our weight function, $\fw$. Our development of the relevant weighted Sobolev space theory is more comprehensive than that of \cite{Koch} and our application differs from \cite{Koch} in many significant aspects because we allow (i) unbounded domains, $\sO$, (ii) non-coercive operators, $A$, (iii) unbounded lower-order coefficients in $A$, and (iv) non-empty boundary portions, $\Gamma_1$.

While there is considerable body of literature on linear degenerate elliptic or parabolic partial differential equations, much of that concerns operators obeying H\"ormander's hypoellipticity condition, which is not obeyed by the Heston operator, $A$; see, for example, Lunardi \cite{Lunardi_1997} and Priola \cite{Priola_2009}. Research on the existence and uniqueness theory for elliptic or parabolic partial differential equations with unbounded coefficients includes that of Krylov and Priola \cite{Krylov_Priola_2010} and the references contained therein.

Recent investigations using probability methods in the mathematical finance literature and focusing on fundamental questions of existence, uniqueness, and global regularity of solutions to the associated parabolic terminal/boundary value problem are due to Bayraktar and Xing \cite{Bayraktar_Xing_2009}, Constanzino, Nistor and Mazzucato and their collaborators \cite{Cheng_Costanzino_Liechty_Mazzucato_Nistor_2009, Constantinescu_Costanzino_Mazzucato_Nistor_2009}, and Ekstr\"om, Tysk, and Janson \cite{Ekstrom_Tysk_bcsftse, Ekstrom_Tysk_bssvm, Janson_Tysk_2006}.

\subsubsection{Variational inequalities and obstacle problems for degenerate differential operators}
Recent work on obstacle problems, not directly concerned with finance, includes that of Blanchet, Dolbeault, and Monneau \cite{Blanchet_Dolbeault_Monneau_2005}, Caffarelli \cite{Caffarelli_fermi_1998, Caffarelli_jfa_1998}, Caffarelli, Petrosyan, and Shahgholian \cite{Caffarelli_Petrosyan_Shahgholian_2004}, and Caffarelli and Salsa \cite{Caffarelli_Salsa_2005}. While obstacle problems have been considered for certain degenerate elliptic and parabolic problems in \cite{FriedmanSDE}, the cases considered do not fully cover problems of interest in finance such as those defined by the Heston operator.

Recent work, with application to option pricing, on evolutionary variational inequalities and obstacle problems includes that of Bayraktar and Xing \cite{Bayraktar_Xing_2009b}, Chadam and Chen \cite{Chen_Chadam_2006, Chen_Chadam_Jiang_Zheng_2008}, Ekstr\"om and Tysk \cite{Ekstrom_2004, Ekstrom_Tysk_2006}, Laurence and Salsa \cite{Laurence_Salsa_2009}, Nystr\"om \cite{Nystrom_2007, Nystrom_2008}, and Petrosyan and Shahgholian \cite{Petrosyan_Shahgholian_2007}. Previous work on degenerate evolutionary variational inequalities and obstacle problems includes Mastroeni and Matzeu \cite{Mastroeni_Matzeu_1996, Mastroeni_Matzeu_1998} and Touzi \cite{Touzi_1999}. In the case of obstacle problems for hypoelliptic operators see, for example, the work of DiFrancesco, Frentz, Nystr\"om, Pascucci, and Polidoro \cite{DiFrancesco_Pascucci_Polidoro_2008, Frentz_Nystrom_Pascucci_Polidoro_2010, Nystrom_Pascucci_Polidoro_2010}.

\subsubsection{Mild solutions and viscosity solutions}
Some researchers have established existence, uniqueness, and regularity results --- for more flexible notions of solutions --- to degenerate variational inequality or obstacle problems. For example, Barbu and Marinelli \cite{Barbu_Marinelli} have established such results for \emph{mild} solutions to a class of such problems, while Lee \cite{Lee_2001} and Savin \cite{Savin_2004} have obtained results on the existence, uniqueness, and regularity of \emph{viscosity solutions} \cite{Crandall_Ishii_Lions_1992} to obstacle problems involving the Monge-Amp\`ere operator. However, the primary focus of our research --- in this article and its sequels --- is on the existence and uniqueness of solutions to variational equations and inequalities, interpreted as \emph{weak} solutions to boundary value and obstacle problems, and the regularity theory required to prove existence and uniqueness of \emph{classical} solutions to boundary value and obstacle problems defined by degenerate elliptic-parabolic operators.

\subsection{Further research and applications to probability and mathematical finance}
\label{subsec:Sequels}
We briefly summarize work in progress or near completion on extensions of results of this article to elliptic and parabolic obstacle problems defined by generators of degenerate diffusion processes and motivated by option valuation problems in mathematical finance.

\subsubsection{Regularity of solutions to elliptic obstacle problems}
In a sequel to this article, we augment Theorem \ref{thm:MainOptimalRegularityObstacleProblem} by proving that the solution, $u$, is continuous up to the boundary $\partial\sO$ and then that $u \in C^{1,1}_{\textrm{loc}}(\bar\sO)$ by adapting the weighted H\"older space methods of \cite{DaskalHamilton1998}. We achieve this regularity result even when the obstacle function, $\psi$, is only Lipschitz by adapting arguments of Laurence and Salsa \cite{Laurence_Salsa_2009}. In addition, we augment Theorem \ref{thm:MainRegularityBoundaryValueProblem} by proving that the solution, $u$, is in $C^{k,\alpha}_{\textrm{loc}}(\bar\sO)$, again by adapting the weighted H\"older space methods of \cite{DaskalHamilton1998}.

\subsubsection{Existence, uniqueness, and regularity of solutions to parabolic obstacle problems}
In a sequel to this article, we also consider evolutionary variational inequalities and obstacle problems for the parabolic Heston operator, $-\partial_t + A$, and prove analogues of Theorems \ref{thm:MainExistenceUniquenessObstacleProblem} and \ref{thm:MainOptimalRegularityObstacleProblem}, together with analogues of Theorems \ref{thm:MainExistenceUniquenessBoundaryValueProblem} and \ref{thm:MainRegularityBoundaryValueProblem} in the case of evolutionary variational equations and terminal/boundary value problems.

\subsubsection{Geometry and regularity of the free boundary}
The free boundary, $F(u) := \sO\cap\partial\sC(u)$, in an obstacle problem is the boundary of the open subset $\sC(u)$ of the domain $\sO$ where the solution is strictly greater than the obstacle function, $u>\psi$, and equality holds in the partial differential inequality. Motivated by the beautiful results of Laurence and Salsa \cite{Laurence_Salsa_2009} in the case of the non-degenerate, multi-dimensional geometric Brownian motion with drift, we use a combination of probabilistic and analytical methods to determine the geometry and regularity of the free boundary defined by the obstacle problem for the parabolic Heston operator.

\subsection{Extensions to degenerate operators in higher dimensions}
\label{subsec:Extensions}
The Heston stochastic volatility process and its associated generator serve as paradigms for degenerate Markov processes and their degenerate elliptic generators which appear widely in mathematical finance.

\subsubsection{Degenerate diffusion processes and partial differential operators}
\label{subsubsec:DegenerateDiffusion}
Generalizations of the Heston process to higher-dimensional, degenerate diffusion processes may be accommodated by extending the framework developed in this article and we shall describe extensions in a sequel. First, the two-dimensional Heston process has natural $d$-dimensional analogues \cite{DaFonseca_Grasselli_Tebaldi_2008} defined, for example, by coupling non-degenerate $(d-1)$-diffusion processes with degenerate one-dimensional processes \cite{Cherny_Engelbert_2005, Mandl, Zettl}. Elliptic differential operators arising in this way have time-independent, affine coefficients but, as one can see from standard theory \cite{GT, Krylov_LecturesHolder, Krylov_LecturesSobolev, Lieberman} and previous work of Daskalopoulos and her collaborators \cite{DaskalHamilton1998, Daskalopoulos_Rhee_2003} on the porous medium equation, we would not expect significant new difficulties to arise when extending the methods and results of this article to the case of higher dimensions and variable coefficients, depending on both spatial variables and time and possessing suitable regularity and growth properties.

\subsubsection{Degenerate Markov processes and partial-integro differential operators}
\label{subsubsec:Markov}
The Heston process also has natural extensions to $d$-dimensional degenerate affine jump-diffusion processes with Markov generators which are degenerate elliptic partial-integro differential operators. A well-known example of such a two-dimensional process is due to Bates \cite{Bates1996} and the definition of this process has been extended to higher dimensions by Duffie, Pan, and Singleton \cite{DuffiePanSingleton2000}. Stationary jump diffusion processes of this kind and their partial-integro differential operator generators naturally lie within the framework of Feller processes and Feller generators \cite{Jacob_v1, Jacob_v2, Jacob_v3}, where the non-local nature of the partial-integro differential operators provides new challenges when considering obstacle problems; see \cite{Caffarelli_Figalli_2011ppt} for recent research by Caffarelli and Figalli in this direction.

\subsection{Mathematical highlights and guide to the article}
\label{subsec:Guide}
For the convenience of the reader, we provide a brief outline of the article. We begin in \S \ref{sec:EnergyEstimates} by defining the weighted Sobolev spaces and H\"older spaces we shall need throughout this article, discuss the equivalence between weak and strong solutions, and derive the key energy estimates for the bilinear form defined by the Heston operator (Propositions \ref{prop:EnergyGardingHeston} and \ref{prop:FullEnergyHeston}). In \S \ref{sec:VariationalEquation}, we establish existence and uniqueness of solutions to the variational equation for the elliptic Heston operator (Theorem \ref{thm:ExistenceUniquenessEllipticHeston_Improved}). In \S \ref{sec:StatVarInequality}, we adapt the methods of Bensoussan and Lions \cite[Chapter 3, \S 1]{Bensoussan_Lions} to prove existence and uniqueness of solutions to the non-linear penalized equation (Theorem \ref{thm:ExistenceUniquenessEllipticHestonPenalizedProblem}), a coercive variational inequality (Theorem \ref{thm:VIExistenceUniquenessEllipticCoerciveHeston}), and finally the non-coercive variational inequality (Theorem \ref{thm:VIExistenceUniquenessEllipticHeston_Improved}) corresponding to the obstacle problem \eqref{eq:IntroObstacleProblem}. In \S \ref{sec:H2RegularityEquality}, we prove $H^2(\sO,\fw)$ regularity and a priori $H^2(\sO,\fw)$ estimates for solutions to the variational equation corresponding to \eqref{eq:IntroBoundaryValueProblem} (Theorem \ref{thm:GlobalRegularityEllipticHestonSpecial}) together with H\"older continuity of solutions on $\bar\sO$ (Theorem \ref{thm:HolderContinuityHestonStatVarEquality}). Theorem \ref{thm:MainExistenceUniquenessBoundaryValueProblem} is proved at the end of \S \ref{subsec:H2RegularityEquality} while Theorem \ref{thm:MainRegularityBoundaryValueProblem} is proved at the end of \S \ref{subsec:HolderRegularityEquality}. Finally, in \S \ref{sec:RegularityInequality} we prove $H^2(\sO,\fw)$ regularity for solutions to the variational inequality corresponding to the obstacle problem \eqref{eq:IntroObstacleProblem} (Theorem \ref{thm:VIRegularityEllipticHestonNoncoercive}). Theorem \ref{thm:MainExistenceUniquenessObstacleProblem} is proved at the end of \S \ref{subsec:H2RegularityInequality} while Theorems \ref{thm:MainRegularityObstacleProblem} and \ref{thm:MainOptimalRegularityObstacleProblem} are proved at the end of \S \ref{subsec:LocalW2pC1AlphaRegularityInequality}. Because the obstacle function is often not in $H^2(\sO,\fw)$, we extend Theorem \ref{thm:VIRegularityEllipticHestonNoncoercive} to the case where the obstacle function is only in $H^2(\sU,\fw)$ for some open subset $\sU\subseteqq\sO$ (Theorem \ref{thm:H2BoundSolutionHestonVarIneqLipschitz}). With the aid of additional hypotheses on the source and obstacle functions, we obtain $W^{2,p}$, $C^{1,\alpha}$, and $C^{1,1}$ regularity of the solution (Theorem \ref{thm:LocalW2pRegularityHestonVI} and Corollaries \ref{cor:LocalC1alphaRegularity} and \ref{cor:LocalC11Regularity}).

Appendix \ref{sec:WeightedSobolevSpaces} contains the proofs of our results for our weighted Sobolev spaces and which underpin the methods of this article in an essential way; Appendix \ref{sec:LaxMilgram} describes a few simple consequences of the Lax-Milgram theorem which we use repeatedly; Appendix \ref{sec:CIR} summarizes an example illustrating subtleties in the boundary behavior of solutions to the elliptic Cox-Ingersoll-Ross equation, and thus the Heston equation, near the boundary portion, $\Gamma_0$.


\subsection{Notation and conventions} In the definition and naming of function spaces, including spaces of continuous functions, H\"older spaces, or Sobolev spaces, we follow Adams \cite{Adams} and alert the reader to occasional differences in definitions between \cite{Adams} and standard references such as Gilbarg and Trudinger \cite{GT} or Krylov \cite{Krylov_LecturesHolder, Krylov_LecturesSobolev}. These differences matter when the domain, $\sO$, is unbounded, as we allow throughout this article. For the signs attached to coefficients in our differential operators, we follow the conventions of Bensoussan and Lions \cite{Bensoussan_Lions} and Evans \cite{Evans}, keeping in mind our interest in applications to probability, and noting that their sign conventions are often opposite to those of \cite{GT}. We denote $\RR_+ := (0,\infty)$, $\bar\RR_+ := [0,\infty)$, $\HH := \RR\times\RR_+$, and $\bar\HH := \RR\times\bar\RR_+$. For $x\in\RR$, we set $x^+ = \max\{x,0\}$, $x^- = -\min\{x,0\}$, so $x=x^+-x^-$ and $|x| = x^+ + x^-$, a convention which differs from that of \cite[\S 7.4]{GT}. When we label a condition an \emph{Assumption}, then it is considered to be universal and in effect throughout this article and so not referenced explicitly in theorem and similar statements; when we label a condition a \emph{Hypothesis}, then it is only considered to be in effect when explicitly referenced.


\subsection{Acknowledgments}
P. Feehan thanks Peter Carr for introducing him in 2004 to stochastic volatility models and the Heston model in mathematical finance and encouraging his initial research on the Heston partial differential equation. He is grateful to Bruno Dupire, Pat Hagan, Peter Laurence, Victor Nistor, and Sergei Levendorski{\u\i} for helpful conversations and references. He especially thanks his Ph.D. student, Camelia Pop, for many useful discussions on degenerate partial differential equations and diffusion processes.

\section{Energy estimates}
\label{sec:EnergyEstimates}
We begin in \S \ref{subsec:RescalingHeston} by describing our assumptions on the Heston operator coefficients together with a simple affine change of coordinates which preserves the structure of the Heston operator but ensures that certain combinations of coefficients can be assumed to be zero without loss of generality (Lemma \ref{lem:RescalingHeston}), simplifying the derivation of certain estimates throughout this article. In \S \ref{subsec:FunctionSpaces}, we describe the weighted Sobolev spaces (Definitions \ref{defn:H1WeightedSobolevSpaces} and \ref{defn:H2WeightedSobolevSpaces}) we shall need in order to prove existence of solutions to variational equations and inequalities defined by the Heston operator, together with higher regularity properties. In \S \ref{subsec:HestonBilinearForm}, we define the bilinear form associated with the Heston operator (Definition \ref{defn:HestonWithKillingBilinearForm}), establish an integration by parts formula (Lemma \ref{lem:HestonIntegrationByParts}), and discuss when solutions to variational equations may be interpreted as strong solutions to a boundary value problem for the Heston partial differential operator (Lemma \ref{lem:HestonWeightedNeumannBVPHomogeneous}). In \S \ref{subsec:HestonEnergyEstimates}, we derive the key bilinear form estimates we will need, namely, a G\r{a}rding inequality (Proposition \ref{prop:EnergyGardingHeston}) and a continuity estimate (Proposition \ref{prop:FullEnergyHeston}). We conclude in \S \ref{subsec:BilinearFormEnergyIdentityEstimate} by deriving certain additional bilinear form identities and commutator estimates which we will need throughout this article.

\subsection{Heston operator coefficients and a change of variables}
\label{subsec:RescalingHeston}
It will be convenient to define the constants
\begin{gather}
\label{eq:DefnBetaMu}
\beta := \frac{2\kappa\theta}{\sigma^2} \quad\hbox{and}\quad \mu := \frac{2\kappa}{\sigma^2},
\\
\label{eq:DefinitionA1B1}
a_1 := \frac{\kappa\rho}{\sigma}-\frac{1}{2} \quad\hbox{and}\quad b_1 := r-q-\frac{\kappa\theta\rho}{\sigma}.
\end{gather}
The interpretation of $\beta, \mu$ is discussed in \cite{Feehan_cirsl}, while the role of $a_1, b_1$ is explained in Definition \ref{defn:HestonWithKillingBilinearForm}.

\begin{rmk}[Interpretation of the coefficients]
\label{rmk:HestonCoefficients}
The conditions on the coefficients ensure $\beta > 0$ and so the Heston stochastic process and solutions to the partial differential equations and obstacle problems will have tractable behavior. In mathematical finance, the constants  $q, r, \kappa, \sigma, \theta$ have the interpretation described in \cite{Heston1993}.

The conditions \eqref{eq:EllipticHeston} ensure that $y^{-1}A$ is uniformly elliptic on $\HH$. Indeed,
\begin{equation}
\label{eq:HestonModulusEllipticity}
\frac{y}{2}(\xi_1^2 + 2\rho\sigma\xi_1\xi_2 + \sigma^2\xi_2^2) \geq \nu_0 y(\xi_1^2 + \xi_2^2), \quad\forall (\xi_1,\xi_2) \in \RR^2,
\end{equation}
where
\begin{equation}
\label{eq:DefnNuZero}
\nu_0 := \min\{1,(1-\rho^2)\sigma^2\},
\end{equation}
and $\nu_0 > 0$ by Assumption \ref{assump:HestonCoefficients}.
\end{rmk}

It will prove useful to examine the effect of a certain affine change in the independent variables $(x,y)$ in equations or inequalities involving the Heston operator $A$ in \eqref{eq:OperatorHestonIntro}.

\begin{lem}[Affine changes of coordinates]
\label{lem:RescalingHeston}
Let the differential operator $A$ be given by \eqref{eq:OperatorHestonIntro}. Then, using compositions of changes of independent variables of the form
$$
(x,y) \mapsto (x, ay) \quad\hbox{and}\quad (x,y) \mapsto (x+my, y), \quad (x,y) \in \HH,
$$
where $m\geq 0, a>0$, and rescalings of the dependent variable,
$$
u\mapsto bu,
$$
where $b>0$, the equation $Au=f$ on $\sO$ can be transformed to one of the form $\tilde A\tilde u = \tilde f$ on $\tilde \sO$, where $\tilde A$ has the same structure as $A$ in \eqref{eq:OperatorHestonIntro} and its coefficients obey Assumption \ref{assump:HestonCoefficients} but the analogous combination of coefficients, $\tilde b_1$ in \eqref{eq:DefinitionA1B1}, is zero.
\end{lem}

\begin{proof}
We first consider the special case
\begin{equation}
\label{eq:HestonCoefficientsSpecialCase}
\rho/(r-q)\sigma > 0.
\end{equation}
Letting $(x,z)=(x,ay)$, for $a>0$, and writing $u(x,y) =: v(x,z)$, we have $u_y = av_z$, $u_{xy} = av_{xz}$, $u_{yy} = a^2v_{zz}$, and $u_{xx} = v_{xx}$. Then $Au = f$ becomes
$$
Au = -\frac{a^{-1}z}{2}\left(v_{xx} + 2\rho\sigma av_{xz} + \sigma^2 a^2v_{zz}\right) - (r-q-a^{-1}z/2)v_x - \kappa(\theta-a^{-1}z)av_z + rv = f,
$$
and thus
$$
-\frac{z}{2}\left(v_{xx} + 2\rho\sigma av_{xz} + \sigma^2 a^2v_{zz}\right) - (ra-qa-z/2)v_x - \kappa a(\theta a - z)v_z + ra v = af.
$$
Setting $\tilde f(x,z) := af(x,y)$ and
\begin{equation}
\label{eq:RescaledHestonParameters}
\tilde\sigma := \sigma a, \quad\tilde\kappa := \kappa a, \quad\tilde\theta := \theta a, \quad\tilde\rho := \rho, \quad\tilde r := ra,
\quad\tilde q := qa,
\end{equation}
the Heston equation takes the equivalent form
\begin{equation}
\label{eq:RescaledHestony}
\tilde Av := -\frac{z}{2}\left(v_{xx} + 2\tilde\rho\tilde\sigma v_{xz} + \tilde\sigma^2 v_{zz}\right) - (\tilde r - \tilde q - z/2)v_x
- \tilde\kappa(\tilde\theta - z)v_z + \tilde r v = \tilde f.
\end{equation}
Now we examine the effect on the combination of coefficients, $b_1$. From \eqref{eq:DefinitionA1B1}, we have
$$
\tilde b_1
=
\tilde r - \tilde q - \frac{\tilde \kappa\tilde \theta\tilde \rho}{\tilde\sigma}
=
a(r-q) - \frac{\kappa\theta\rho}{a\sigma}.
$$
Hence, $\tilde b_1 = 0$ when $\rho/(r-q)\sigma > 0$ and
$$
a^2 = \frac{\kappa\theta\rho}{(r-q)\sigma},
$$
so it suffices to choose $a := \sqrt{\kappa\theta\rho/(r-q)\sigma}$. This completes the proof of the special case.

For the general case, when \eqref{eq:HestonCoefficientsSpecialCase} may not hold, we first apply a translation $(z,y)=(x+my,y)$, where $m>0$, write $u(x,y) =: v(z,y)$, so that $u_x = v_z$, $u_y = mv_z + v_y$, $u_{xx} = v_{zz}$, $u_{xy} = mv_{zz} + v_{zy}$, and $u_{yy} = m(v_z)_y + (v_y)_y = m(mv_{zz} + v_{zy}) + mv_{yz} + v_{yy} = m^2v_{zz} + 2mv_{zy} + v_{yy}$. Thus,
\begin{align*}
Au &= -\frac{y}{2}\left(v_{zz} + 2\rho\sigma (mv_{zz} + v_{zy}) + \sigma^2 (m^2v_{zz} + 2mv_{zy} + v_{yy})\right)
\\
&\quad - (r-q-y/2)v_z - \kappa(\theta-y)(mv_z + v_y) + rv,
\end{align*}
and hence
\begin{equation}
\label{eq:TranslatedHestonx}
\begin{aligned}
Au &= -\frac{y}{2}\left((1+2\rho\sigma m + \sigma^2m^2)v_{zz} + 2(\rho\sigma + m\sigma^2)v_{zy} + \sigma^2 v_{yy}\right)
\\
&\quad - (r-q + \kappa\theta m - (\kappa m + 1/2)y)v_z - \kappa(\theta-y)v_y + rv.
\end{aligned}
\end{equation}
Therefore, setting $\xi := 1+2\rho\sigma m + \sigma^2m^2$ and noting that $\xi>0, \forall m>0$, since $\sigma\neq 0$ and $\rho\in(-1,1)$, we obtain
$$
\bar Av := -\frac{y}{2}\left(v_{zz} + 2\bar\rho\bar\sigma v_{zy} + \bar\sigma^2 v_{yy}\right)
- (\bar r - \bar q - by/2)v_z - \bar\kappa(\theta-y)v_y + \bar rv = \bar f,
$$
where
\begin{gather*}
\bar\rho\bar\sigma := \frac{m\sigma^2}{\xi}, \quad \bar\sigma := \frac{|\sigma|}{\sqrt{\xi}},
\quad \bar\rho := \frac{\rho\sigma + m\sigma^2}{|\sigma|\sqrt{\xi}}, \quad \bar r := \frac{r}{\xi},
\\
\bar q := \frac{q - \kappa\theta m}{\xi}, \quad b := \frac{2\kappa m + 1}{\xi}, \quad\bar\kappa := \frac{\kappa}{\xi},
\quad \bar f := \frac{f}{\xi}.
\end{gather*}
(Note that
$$
\bar\beta := \frac{2\bar\kappa\theta}{\bar\sigma^2} =  \frac{2\kappa\theta}{\sigma^2} = \beta,
$$
so $\bar\beta = \beta$, as expected.) Observe that $b > 0$ and $\bar\sigma > 0$ and, for large enough $m>0$, we have $\bar\rho > 0$. In addition,
$$
\bar\rho^2 = \frac{(\rho\sigma + m\sigma^2)^2}{\sigma^2\xi} = \frac{\rho^2 + 2\rho\sigma m + m^2\sigma^2}{1+2\rho\sigma m + \sigma^2m^2} < 1,
$$
since $\rho\in (-1,1)$, and thus $0<\bar\rho<1$. Moreover, for large enough $m>0$, we have $\bar r-\bar q>0$, so $\bar\rho/(\bar r-\bar q)\sigma > 0$. By rescaling the dependent variable, $\bar v := bv$, we obtain
$$
-\frac{y}{2}\left(\bar v_{xx} + 2\bar\rho\bar\sigma \bar v_{xy} + \bar\sigma^2 \bar v_{yy}\right) - (b^{-1}(\bar r-\bar q)-y/2)\bar v_x - \bar\kappa(\bar\theta-y)\bar v_y + \bar r \bar v = b^{-1}\bar f.
$$
Defining $\hat f := b^{-1}\bar f$ and $\hat q$ by $b^{-1}(\bar r-\tilde q) =: \bar r - \hat q$ and noting that $\bar r - \tilde q >0$ implies $\tilde r - \hat q >0$, we see that we are back in the situation of the special case \eqref{eq:HestonCoefficientsSpecialCase}, with $\bar\rho/(\tilde r-\hat q) > 0$, and so that rescaling argument applies. This completes the proof.
\end{proof}

\begin{rmk}[Invariance of $\beta$ under coordinate changes]
\label{rmk:RescalingHestonAndBeta}
The proof of Lemma \ref{lem:RescalingHeston} shows that the coordinate changes considered have no effect on $\beta$.
\end{rmk}

\begin{rmk}[Effect of coordinate changes on the Sobolev weight, shape of the domain, and Dirichlet data and obstacle functions]
\label{rmk:RescalingHestonAndSobolevSpaces}
Note that the weight $\fw$ in \eqref{eq:HestonWeight} is \emph{not} invariant under the coordinate changes described in the hypotheses of Lemma \ref{lem:RescalingHeston}; see Definitions \ref{defn:H1WeightedSobolevSpaces} and \ref{defn:H2WeightedSobolevSpaces}. Similarly, the shape of the domain $\sO\subseteq\HH$ is only invariant under
changes of coordinates of the form $(x,y) \mapsto (x, y+my)$ when $\sO=\HH$; while this change does matter away from a neighborhood of $\Gamma_0$, it does matter near $\Gamma_0$ since we will later assume (see Hypothesis \ref{hyp:HestonDomainNearGammaZero}). Similarly, the coordinate changes described in Lemma \ref{lem:RescalingHeston} also mean that the Dirichlet data function, $g$, in \eqref{eq:IntroObstacleProblem} or \eqref{eq:IntroBoundaryValueProblem} will replaced by a function, $\tilde g$ (unless $g=0$ on $\sO\cup\Gamma_1$, in which case the homogeneous Dirichlet boundary condition remains unchanged), while the obstacle function, $\psi$, in \eqref{eq:IntroObstacleProblem} will be replaced by an obstacle function, $\tilde\psi$.
\end{rmk}

With Remarks \ref{rmk:RescalingHestonAndSobolevSpaces} and \ref{rmk:WhyNeedHypothesisOnGammaZero} (below) in mind, we therefore assume throughout the article that the reduction in Lemma \ref{lem:RescalingHeston} has already been applied in order to satisfy

\begin{assump}
[Condition on the Heston operator coefficients]
\label{assump:HestonCoefficientb1}
The coefficients defining $A$ in \eqref{eq:OperatorHestonIntro} have the property that $b_1=0$ in \eqref{eq:DefinitionA1B1}.
\end{assump}

\subsection{Function spaces}
\label{subsec:FunctionSpaces}
As we noted in \S \ref{subsec:Intro}, we shall assume that the spatial domain has the following structure throughout this article:

\begin{defn}[Spatial domain for the Heston partial differential equation]
\label{defn:HestonDomain}
Let $\sO \subset \HH$ be a possibly unbounded domain with boundary $\partial\sO$, let $\Gamma_1 := \HH\cap\partial\sO$, let $\Gamma_0$ denote the interior of $\{y=0\}\cap\partial\sO$, and require that $\Gamma_0$ is non-empty.
\end{defn}

We write $\partial\sO = \Gamma_0\cup\bar\Gamma_1 = \bar\Gamma_0\cup\Gamma_1$ and note that the boundary portions $\Gamma_0$ and $\Gamma_1$ are relatively open in $\partial\sO$. If $\Gamma_0$ were empty then standard methods \cite{Bensoussan_Lions, Friedman_1982, GT, Kinderlehrer_Stampacchia_1980} would apply to all of the problems considered in this article.

\begin{hyp}[Hypothesis on the domain near $\Gamma_0$]
\label{hyp:HestonDomainNearGammaZero}
For $\sO$ as in Definition \ref{defn:HestonDomain}, there is a positive constant, $\delta_0$, such that for all $0<\delta\leq\delta_0$,
\begin{align*}
\sO^0_\delta := \sO\cap(\RR\times(0,\delta)) &= \Gamma_0 \times (0,\delta),
\\
\Gamma_1\cap(\RR\times(0,\delta)) &= \partial\Gamma_0\times (0,\delta),
\end{align*}
where $\Gamma_0\subseteqq\RR$ is a finite union of open intervals.
\end{hyp}

\begin{rmk}[Need for the hypothesis on the domain near $\Gamma_0$]
\label{rmk:WhyNeedHypothesisOnGammaZero}
If our article had allowed for elliptic operators with variable coefficients, $a^{ij}, b^i, c$, with suitable regularity and growth properties, then we could replace Hypothesis \ref{hyp:HestonDomainNearGammaZero} with the more geometric requirement that $\bar\Gamma_1\pitchfork\{y=0\}$ ($C^k$-transverse intersection, $k\geq 1$) by making use of $C^k$-diffeomorphisms of $\bar\HH$ to ``straighten'' the boundary, $\Gamma_1$, near where it meets $\Gamma_0$.
\end{rmk}

\begin{hyp}[Hypothesis on the domain near $\Gamma_1$]
\label{hyp:HestonDomainNearGammaOne}
For a domain, $\sO$, as in Definition \ref{defn:HestonDomain}, and constant, $\delta_0$, as in Hypothesis \ref{hyp:HestonDomainNearGammaZero}, integer $k\geq 1$, and $\alpha\in [0,1)$, require that the boundary portion, $\Gamma_1$, has the \emph{uniform $C^k$-regularity property} \cite[\S 4.6]{Adams} (respectively, uniform $C^{k,\alpha}$-regularity property, when $\alpha\in(0,1)$): there exists a locally finite open cover, $\{U_j\}$, of $\Gamma_1\cap (\RR\times(\delta_0/2,\infty))$ with $U_j \subset \RR\times(\delta_0/4,\infty), \forall j$ and a corresponding sequence $\{\Phi_j\}$ of $C^k$-smooth (respectively, $C^{k,\alpha}$-smooth) one-to-one transformations (see \cite[\S 3.34]{Adams}) with $\Phi_j$ taking $U_j$ onto $B(1)$, where $B(R) := \{(x,y)\in\RR^2: x^2+y^2<R^2\}$, such that
\begin{enumerate}
\item There is a constant, $\delta_1>0$, such that $\cup_{j=1}^\infty U_j' \supset \sO^1_{\delta_1}\cap(\RR\times(\delta_0/2,\infty))$, where $U_j' = \Psi_j(B(1/2))$, $\Psi = \Phi^{-1}$, and \cite[\S 4.5]{Adams}
$$
\sO^1_\delta := \{P\in \sO: \hbox{dist}(P,\Gamma_1) < \delta\}, \quad \delta > 0.
$$
\item There is a finite constant, $R_1\geq 1$, such that every collection of $R_1+1$ of sets $U_j$ has empty intersection;
\item For each $j$, $\Phi_j(U_j\cap\sO) = \{(x,y)\in B(1): x>0\}$;
\item If $(\phi_{j,1},\phi_{j,2})$ and $(\psi_{j,1},\psi_{j,2})$ denote the components of $\Phi_j$ and $\Psi_j$, respectively, then there exists a finite constant, $M_1$, such that, for all multi-indices $\alpha$, $|\alpha|\leq 2$, for $i=1,2$, and every $j$, we have
\begin{align*}
|D^\alpha\phi_{j,i}(z)| &\leq M_1, \quad z\in U_j,
\\
|D^\alpha\psi_{j,i}(w)| &\leq M_1, \quad w\in B(1).
\end{align*}
\end{enumerate}
\end{hyp}

\begin{rmk}[Application of Hypotheses \ref{hyp:HestonDomainNearGammaZero} and \ref{hyp:HestonDomainNearGammaOne}]
We will need Hypotheses \ref{hyp:HestonDomainNearGammaZero} and \ref{hyp:HestonDomainNearGammaOne} when we derive certain global estimates and regularity properties of a solution to a boundary value or obstacle problem near $\partial\sO$. When $\Gamma_1\subset\HH$ is a bounded $C^k$-curve, then it has the uniform $C^k$-regularity property \cite{Adams}.
\end{rmk}

We shall also need

\begin{hyp}[Extension operator property of the domain]
\label{hyp:GammeOneExtensionProperty}
For a domain, $\sO$, as in Definition \ref{defn:HestonDomain} and an integer $k\geq 1$, require that there is a \emph{simple $k$-extension operator from $\sO$ to $\HH$} in the sense of Definition \ref{defn:ExtensionOperator} (compare \cite[\S 4.24]{Adams}).
\end{hyp}

\begin{rmk}[Application of Hypothesis \ref{hyp:GammeOneExtensionProperty}]
Hypothesis \ref{hyp:GammeOneExtensionProperty} is required when we consider traces of functions on $\Gamma_1$.
\end{rmk}

We augment the standard definitions of spaces of continuous (and smooth) functions in \cite[\S 1.25 \& \S 1.26]{Adams}, \cite[\S 5.1]{Evans} with

\begin{defn}[Spaces of continuous functions]
\label{defn:ContinuousFunctions}
Let $\sU\subseteqq\RR^d$ be a domain with boundary $\partial\sU$ and closure $\bar\sU=\sU\cup\partial\sU$.
\begin{enumerate}
\item Let $T\subseteqq\partial \sU$ be relatively open. For any integer $\ell\geq 0$, then $C^\ell_{\textrm{loc}}(\sU\cup T)$ denotes the vector space of functions $u$ on $\sU$ with partial derivatives, $D^\alpha u$, for $0\leq |\alpha|\leq \ell$, which are continuous on $\sU$ and have continuous extensions to $\sU\cup T$. (Compare \cite[\S 4.4]{GT}.) When $T=\partial\sU$ (respectively, $T=\emptyset$), we abbreviate $C^\ell_{\textrm{loc}}(\sU\cup\partial\sU)$ by $C^\ell_{\textrm{loc}}(\bar\sU)$ (respectively, $C^\ell_{\textrm{loc}}(\sU\cup\emptyset)$ by $C^\ell(\sU)$). When $\ell=0$, we abbreviate $C^0_{\textrm{loc}}(\sU\cup T)$ by $C_{\textrm{loc}}(\sU\cup T)$.
\item Denote $C^\infty_{\textrm{loc}}(\sU\cup T) := \cap_{\ell\geq 0}C^\ell_{\textrm{loc}}(\sU\cup T)$.
\item Let $C^\infty_0(\sU\cup T)$ denote the subspace of $C^\infty$ functions with compact support in $\sU\cup T$. When $T=\partial\sU$ (respectively, $T=\emptyset$), we abbreviate $C^\infty_0(\sU\cup\partial\sU)$ by $C^\infty_0(\bar\sU)$ (respectively, $C^\infty_0(\sU\cup\emptyset)$ by $C^\infty_0(\sU)$).
\item As in \cite[\S 1.26]{Adams}, let $C^\ell(\bar\sU)$ denote the Banach space of functions $u$ on $\sU$ with partial derivatives, $D^\alpha u$, for $0\leq |\alpha|\leq \ell$, which are \emph{bounded} and \emph{uniformly} continuous on $\sU$.
\item As in \cite[\S 3.10]{Kufner}, denote $C^\infty(\bar\sU) := \cap_{\ell\geq 0}C^\ell(\bar\sU)$.
\end{enumerate}
\end{defn}

\begin{rmk}
Because we consider \emph{unbounded} domains, it is important to note the following:
\begin{enumerate}
\item Compare the definition of $C^\ell(\bar\sU)$ and related vector spaces in \cite[p. 10, \S 4.1, \& p. 73]{GT}, where it is only assumed that the derivatives $D^\alpha u$ are continuous on $U$, with continuous extensions to $\bar\sU$. We emphasize the distinction here because in \cite{GT} the authors typically assume that $\sU$ is \emph{bounded} whereas we wish to include the case where $\sU$ is \emph{unbounded}. (In other words, the definition of $C^\ell(\bar\sU)$ in \cite[p. 10]{GT} coincides with our definition of $C^\ell_{\textrm{loc}}(\bar\sU)$.)
\item We could have equivalently defined $C^\ell_{\textrm{loc}}(\bar\sU)$ as the vector space of functions $u$ on $\sU$ with partial derivatives, $D^\alpha u$, for $0\leq |\alpha|\leq \ell$, which are bounded and uniformly continuous on \emph{bounded subsets} of $\sU$.
\item When $\sU$ is bounded, then $C^\ell_{\textrm{loc}}(\bar\sU) = C^\ell(\bar\sU)$.
\end{enumerate}
\end{rmk}

By analogy with the definitions of the standard Sobolev spaces $W^{1,2}(\sO)$, $W^{1,2}_0(\sO)$ in \cite[\S 3.1]{Adams} and weighted Sobolev spaces \cite[\S 1, \S 3.4, \& \S 3.8]{Kufner} we introduce the

\begin{defn}[First-order weighted Sobolev spaces]
\label{defn:H1WeightedSobolevSpaces}
Let $\sO\subseteqq\HH$ be a domain. We choose a positive weight function,
\begin{equation}
\label{eq:HestonWeight}
\fw(x,y) := y^{\beta-1}e^{-\gamma|x|-\mu y}, \quad (x,y) \in \HH,
\end{equation}
for a suitable\footnote{See Proposition \ref{prop:EnergyGardingHeston} for constraints on the choice of $\gamma$.} positive constant, $\gamma$. Let $L^2(\sO,\fw)$ be the space of all measurable functions $u:\sO\to\RR$ for which
$$
\|u\|_{L^2(\sO,\fw)}^2 := \int_\sO u^2\fw\,dxdy < \infty,
$$
and denote $H^0(\sO,\fw) := L^2(\sO,\fw)$.
\begin{enumerate}
\item If $Du := (u_x,u_y)$ and $u_x,u_y$ are defined in the sense of distributions \cite[\S 1.57]{Adams}, we set
\begin{align*}
H^1(\sO,\fw) := \{u \in L^2(\sO,\fw): (1+y)^{1/2}u \hbox{ and } y^{1/2}|Du| \in L^2(\sO,\fw)\},
\end{align*}
and
\begin{equation}
\label{eq:H1NormHeston}
\|u\|_{H^1(\sO,\fw)}^2 := \int_\sO \left(y|Du|^2 + (1+y)u^2\right)\fw\,dxdy.
\end{equation}
\item Let $T\subseteq\partial\sO$ be relatively open and let $H^1_0(\sO\cup T,\fw)$ be the closure in $H^1(\sO,\fw)$ of $C^\infty_0(\sO\cup T)$.
\end{enumerate}
\end{defn}

\begin{rmk}[Comments on first-order weighted Sobolev spaces]
Note that:
\begin{enumerate}
\item We shall most often appeal to the case when $T=\Gamma_i, i=0,1$. Compare \cite[p. 7]{Garroni_Menaldi_2002} or \cite[pp. 215--216]{GT}. When $T=\emptyset$, we denote
$$
H^1_0(\sO\cup T,\fw) = H^1_0(\sO,\fw),
$$
that is, the closure of $C^\infty_0(\sO)$ in $H^1(\sO,\fw)$, while if $T=\partial\sO$, then $H^1_0(\sO\cup T,\fw)$ is the closure of $C^\infty_0(\bar\sO)$ in $H^1(\sO,\fw)$ and Corollary \ref{cor:KufnerPowerWeight} yields
$$
H^1_0(\sO\cup T,\fw) = H^1(\sO,\fw).
$$
\item For brevity and when the context is clear, we shall often denote
$$
H := L^2(\sO,\fw) \quad\hbox{and}\quad V = H^1(\sO,\fw), H^1_0(\sO\cup\Gamma_0,\fw), \hbox{ or } H^1_0(\sO,\fw),
$$
and
$$
|u|_H := \|u\|_{L^2(\sO,\fw)} \quad\hbox{and}\quad \|u\|_V := \|u\|_{H^1(\sO,\fw)}.
$$
\item In the present article, we shall not require $W^{k,p}(\sO,\fw)$ or its variants when $p\not= 2$, and so, for brevity, we denote $W^{k,2}(\sO,\fw)$ by $H^k(\sO,\fw)$, $k=0,1,2$, and similarly for its variants.
\item By a straightforward modification of the proof of \cite[Theorem 3.2]{Adams}, one can show that the spaces $H^k(\sO,\fw)$, $k=0,1$, and $H^1_0(\sO\cup T,\fw)$ are Banach spaces; compare \cite[Proposition 2.1.2]{Turesson_2000}.
\item The spaces $H^k(\sO,\fw)$, $k=0,1$, and $H^1_0(\sO\cup T,\fw)$ are Hilbert spaces with the inner products,
\begin{gather*}
(u,v)_{L^2(\sO,\fw)} := \int_\sO uv \fw\,dxdy
\\
(u,v)_{H^1(\sO,\fw)} := \int_\sO \left(y\langle Du,Dv \rangle + (1+y)uv\right)\fw\,dxdy.
\end{gather*}
\item We let $H^{-1}(\sO,\fw)$ denote the dual space to $H^1_0(\sO\cup\Gamma_0,\fw)$; compare \cite[\S 3.5]{Adams}, \cite[\S 5.9.1]{Evans}.
\end{enumerate}
\end{rmk}

\begin{rmk}[Alternative choices of Sobolev weight]
\label{rmk:HestonWeightSmoothed}
We could alternatively have chosen
$$
\fw(x,y) = y^{\beta-1}e^{-\gamma\sqrt{1+x^2}-\mu y}, \quad (x,y) \in \HH,
$$
but the simpler choice \eqref{eq:HestonWeight} will be adequate. \qed
\end{rmk}

\begin{rmk}[Finite volume of the spatial domain]
\label{rmk:FiniteVolumeDomain}
The choice of weight, $\fw$, in \eqref{eq:HestonWeight} ensures that $\sO\subseteqq\HH$ has finite measure, $\hbox{Vol}(\sO,\fw) := \int_\sO 1\,\fw\,dxdy < \infty$, when $0<\beta<\infty$, $\mu>0$, and $\gamma>0$. This point is important in compactness arguments; see \S \ref{subsec:WeightedSobolevContinuousCompactEmbeddings} for an explanation.
\end{rmk}

\begin{rmk}[Doubling and $A_p$ properties]
The weight $\fw$ on $\HH$ is neither $A_p$, $1\leq p<\infty$, in the sense of \cite[Definition 1.2.2]{Turesson_2000} nor doubling in the sense of \cite[Definition 1.2.6]{Turesson_2000}. However, when $\HH$ is equipped with the cycloidal metric of \cite[\S I.1]{DaskalHamilton1998} or \cite[\S 4.3]{Koch}, then the weight $y^\beta$ is $A_p$, $1\leq p<\infty$ by \cite[Corollary 4.3.4]{Koch}.
\end{rmk}

We will also need analogues of the standard second-order Sobolev spaces, $W^{2,2}(\sO)$ \cite[\S 7.5]{GT}:

\begin{defn}[Second-order weighted Sobolev spaces]
\label{defn:H2WeightedSobolevSpaces}
Let $\sO\subseteqq\HH$ be a domain and let
$$
H^2(\sO,\fw) := \{u \in L^2(\sO,\fw): (1+y)^{1/2}u, (1+y)|Du|,  y|D^2u| \in L^2(\sO,\fw)\},
$$
where $D^2u := (u_{xx}, u_{xy}, u_{yx}, u_{yy})$, all derivatives of $u$ are defined in the sense of distributions \cite[\S 1.57]{Adams}, and
\begin{equation}
\label{eq:H2NormHeston}
\|u\|_{H^2(\sO,\fw)}^2 := \int_\sO\left( y^2|D^2u|^2 + (1+y)^2|Du|^2 + (1+y)u^2\right)\,\fw\,dxdy.
\end{equation}
\end{defn}

\begin{rmk}[Comments on second-order weighted Sobolev spaces]
Note that:
\begin{enumerate}
\item We let $H^k_{\textrm{loc}}(\sO,\fw)$, $k=0,1,2$, denote the space of functions $u$ for which $u\in H^k(\sO',\fw)$ for all $\sO'\Subset\sO$.
\item The space $H^2(\sO,\fw)$ is a Banach space (again by modification of the proof of \cite[Theorem 3.2]{Adams}) and a Hilbert space with the inner product,
$$
(u,v)_{H^2(\sO,\fw)} := \int_\sO \left(y^2\langle D^2u,D^2v \rangle + (1+y)^2\langle Du,Dv \rangle + (1+y)uv\right)\fw\,dxdy.
$$
\end{enumerate}
\end{rmk}

Our definition of $H^2(\sO,\fw)$ is motivated, in part, by our requirements that
$$
H^2(\sO,\fw) \subset H^1(\sO,\fw)
$$
and that
$$
A: H^2(\sO,\fw) \to L^2(\sO,\fw), \quad u\mapsto Au,
$$
be a bounded operator.

\subsection{Bilinear form associated with the Heston operator}
\label{subsec:HestonBilinearForm}
We introduce the

\begin{defn}[Heston bilinear form]
\label{defn:HestonWithKillingBilinearForm}
If $u, v \in H^1(\sO,\fw)$ and $a_1,b_1$ are as in \eqref{eq:DefinitionA1B1}, then we call
\begin{equation}
\label{eq:HestonWithKillingBilinearForm}
\begin{aligned}
a(u,v) &:= \frac{1}{2}\int_\sO\left(u_xv_x + \rho\sigma u_yv_x
+ \rho\sigma u_xv_y + \sigma^2u_yv_y\right)y\,\fw\,dxdy
\\
&\quad - \frac{\gamma}{2}\int_\sO\left(u_x + \rho\sigma u_y\right)v \sign(x)y\,\fw\,dxdy
\\
&\quad - \int_\sO(a_1y + b_1)u_xv\,\fw\,dxdy + \int_\sO ruv\,\fw\,dxdy,
\end{aligned}
\end{equation}
the \emph{bilinear form associated with the Heston operator}, $A$, in \eqref{eq:OperatorHestonIntro}.\qed
\end{defn}

\begin{lem}[Integration by parts for the Heston operator]
\label{lem:HestonIntegrationByParts}
Require that the domain $\sO$ obeys Hypotheses \ref{hyp:HestonDomainNearGammaZero} and \ref{hyp:GammeOneExtensionProperty} with $k=1$. Suppose $u\in H^2(\sO,\fw)$ and $v\in H^1(\sO,\fw)$. Then $Au\in L^2(\sO,\fw)$ and
\begin{equation}
\label{eq:HestonIntegrationByPartsFormula}
(Au,v)_{L^2(\sO,\fw)} = a(u,v) - \frac{1}{2}\int_{\Gamma_1}\left(n_0(u_x + \rho\sigma u_y) + n_1(\rho\sigma u_x + \sigma^2 u_y)\right)vy\fw\,dS,
\end{equation}
where $\mathbf{n} := (n_0,n_1)$ is the outward-pointing unit normal vector field along $\Gamma_1$, $dS$ is the curve measure on $\Gamma_1$ induced by Lebesgue measure on $\RR^2$, and the integrand on $\Gamma_1$ is defined in the trace sense.
\end{lem}

\begin{rmk}
\label{rmk:HestonIntegrationByParts}
Equation \eqref{eq:HestonIntegrationByPartsFormula} does not necessarily hold if the hypothesis $u\in H^2(\sO,\fw)$ is relaxed to $u\in H^2_{\textrm{loc}}(\sO,\fw)\cap H^1(\sO,\fw)$ and $Au\in L^2(\sO,\fw)$. Example \ref{exmp:CIR} and \cite[\S 13.4.21 \& \S 13.5.8]{AbramStegun} show that there are functions $u\in H^2_{\textrm{loc}}(\sO,\fw)\cap H^1(\sO,\fw)$ with $Au = 0$ on $\sO$ but $y^\beta u_y = \Gamma(\beta)/\Gamma(1+r/\kappa)\not= 0$ along $\Gamma_0$ and so the $\Gamma_0$-boundary integral in \eqref{eq:Gamma0BdryTermIsZero} is non-zero for such a function, $u$.
\end{rmk}

\begin{proof}[Proof of Lemma \ref{lem:HestonIntegrationByParts}]
We begin by reducing the problem to the case where $u \in C^2(\bar\sO)$ and $v\in C^1(\bar\sO)$. By Corollary \ref{cor:KufnerPowerWeightBoundedDerivatives}, the space $C^2(\bar\sO)$
is dense in $H^2(\sO,\fw)$ and $C^1(\bar\sO)$ is dense $H^1(\sO,\fw)$, so we may choose $\{u_l\}_{l\geq 0}\subset C^2(\bar\sO)$, a sequence converging in $H^2(\sO,\fw)$ to $u\in H^2(\sO,\fw)$, and choose $\{v_m\}_{m\geq 0}\subset C^1(\bar\sO)$, a sequence converging strongly in $H^1(\sO,\fw)$ to $v\in H^1(\sO,\fw)$. Our hypotheses on $\sO$ imply that Lemma \ref{lem:Traces} is applicable. Then Lemma \ref{lem:Traces} ensures that the sequence of traces, $\{u_l|_{\Gamma_1}\}_{l\geq 0}$, converges in $H^1(\Gamma_1,\fw)$ to $u|_{\Gamma_1} \in H^1(\Gamma_1,\fw)$ and that the sequence of traces, $\{y^{1/2}v_m|_{\Gamma_1}\}_{m\geq 0}$, converges in $L^2(\Gamma_1,\fw)$ to $y^{1/2}v|_{\Gamma_1} \in L^2(\Gamma_1,\fw)$. Consequently, by taking limits as $l,m\to\infty$ of the integral over $\Gamma_1$ in \eqref{eq:HestonIntegrationByPartsFormula} (with $u_l,v_m$ in place of $u,v$), we obtain
\begin{equation}
\label{eq:TraceDefnGamma1Integrand}
\begin{aligned}
{}&\lim_{l,m\to\infty}\int_{\Gamma_1}\left(n_0(u_{l,x} + \rho\sigma u_{l,y}) + n_1(\rho\sigma u_{l,x} + \sigma^2 u_{l,y})\right)v_my\fw\,dS
\\
&=
\int_{\Gamma_1}\left(n_0(u_x + \rho\sigma u_y) + n_1(\rho\sigma u_x + \sigma^2 u_y)\right)vy\fw\,dS.
\end{aligned}
\end{equation}
Thus, it suffices to prove the identity \eqref{eq:HestonIntegrationByPartsFormula} when $u \in C^2(\bar\sO)$ and $v\in C^1(\bar\sO)$.

From its definition in \eqref{eq:OperatorHestonIntro}, we observe that the expression $Au$ in $\sO$ may be written,
\begin{align*}
Au &= -\frac{1}{2}y^{1-\beta}\left(\left(y^\beta u_x\right)_x + \rho\sigma\left(y^\beta u_x\right)_y
+ \rho\sigma\left(y^\beta u_y\right)_x + \sigma^2\left(y^\beta u_y\right)_y\right)
\\
&\quad + \frac{\rho\sigma}{2}\beta u_x + \frac{\sigma^2}{2} \beta u_y - \left(r-q-\frac{y}{2}\right)u_x - \kappa(\theta-y)u_y + ru \quad\hbox{on }\sO.
\end{align*}
Thus, using $\beta := 2\kappa\theta/\sigma^2$ by \eqref{eq:DefnBetaMu} and recalling that $b_1 := r-q- \kappa\theta\rho/\sigma$ by \eqref{eq:DefinitionA1B1}, the preceding expression simplifies to
\begin{equation}
\label{eq:CleanedUpBasicHestonDivergenceFormula}
\begin{aligned}
Au &= -\frac{1}{2}y^{1-\beta}\left(\left(y^\beta u_x + \rho\sigma y^\beta u_y\right)_x + \left(\rho\sigma y^\beta u_x + y^\beta\sigma^2 u_y\right)_y\right)
\\
&\quad - b_1u_x + \frac{y}{2}u_x + \kappa yu_y + ru \quad\hbox{on }\sO.
\end{aligned}
\end{equation}
Multiplying both sides of \eqref{eq:CleanedUpBasicHestonDivergenceFormula} by $v\fw$, where $\fw = y^{\beta-1}e^{-\mu y -\gamma|x|}$ by \eqref{eq:HestonWeight}, gives
\begin{align*}
\int_\sO (Au)v\fw\,dxdy &= \frac{1}{2}\int_\sO \left(\left(y^\beta u_x + \rho\sigma y^\beta u_y\right)_x
+ \left(\rho\sigma y^\beta u_x + y^\beta\sigma^2 u_y\right)_y\right)ve^{-\mu y -\gamma|x|}\,dxdy
\\
&\quad + \int_\sO \left(-b_1u_x + \frac{y}{2}u_x + \kappa yu_y + ru\right)v\fw\,dxdy.
\end{align*}
Integrating by parts, using $(e^{-\gamma|x|})_x = -\gamma\sign(x)e^{-\gamma|x|}$ and $(e^{-\mu y})_y = -\mu e^{-\mu y}$, gives
\begin{equation}
\label{eq:HestonIntegrationByPartsFormulaRaw}
\begin{aligned}
{}& \int_\sO (Au)v\fw\,dxdy
\\
&= \frac{1}{2}\int_\sO y\left(u_xv_x + \rho\sigma u_xv_y + \rho\sigma u_yv_x + \sigma^2 u_yv_y\right)\fw\,dxdy
\\
&\quad -\frac{1}{2}\int_\sO y\left\{\gamma(u_x+\rho\sigma u_y)\sign(x) + \mu(\rho\sigma u_x+\sigma^2u_y\right\}v\fw\,dxdy
\\
&\quad + \int_\sO \left(-b_1u_x + \frac{y}{2}u_x + \kappa yu_y + ru\right)v\fw\,dxdy
\\
&\quad - \frac{1}{2}\int_{\Gamma_1} \left(n_0\left(y^\beta u_x + \rho\sigma y^\beta u_y\right) + n_1\left(\rho\sigma y^\beta u_x + y^\beta\sigma^2 u_y\right)\right)
ve^{-\mu y -\gamma|x|}\,dS
\\
&\quad - \frac{1}{2}\int_{\Gamma_0} n_1\left(\rho\sigma y^\beta u_x + y^\beta\sigma^2 u_y\right)e^{-\mu y -\gamma|x|}\,dx.
\end{aligned}
\end{equation}
Using $\mu := 2\kappa/\sigma^2$ by \eqref{eq:DefnBetaMu}, recalling that $a_1 := \kappa\rho/\sigma-1/2$ from \eqref{eq:DefinitionA1B1}, and gathering terms, the preceding expression becomes
\begin{align*}
\int_\sO (Au)v\fw\,dxdy
&= a(u,v) - \frac{1}{2}\int_{\Gamma_1} \left(n_0\left(u_x + \rho\sigma u_y\right) + n_1\left(\rho\sigma u_x + \sigma^2 u_y\right)\right)v y\fw\,dS
\\
&\quad - \frac{1}{2}\int_{\Gamma_0} n_1\left(\rho\sigma u_x + \sigma^2 u_y\right)v y\fw\,dx,
\end{align*}
where $a(u,v)$ is defined by \eqref{eq:HestonWithKillingBilinearForm}. But
$$
\int_{\Gamma_0} n_1\left(\rho\sigma u_x + \sigma^2 u_y\right)v y\fw\,dx = -\int_{\Gamma_0}\left(\rho\sigma u_x + \sigma^2 u_y\right)v y^\beta e^{-\gamma|x|-\mu y}\,dx,
$$
with $u_x, u_y, v \in C(\bar\sO)$, $\beta > 0$, and $n_1=-1$ along $\Gamma_0$, so
\begin{equation}
\label{eq:Gamma0BdryTermIsZero}
\int_{\Gamma_0} n_1\left(\rho\sigma u_x + \sigma^2 u_y\right)v y\fw\,dx
=
\int_{\Gamma_0} n_1\left(\rho\sigma u_x + \sigma^2 u_y\right)v y^\beta e^{-\gamma|x|-\mu y}\,dx = 0,
\end{equation}
and this yields \eqref{eq:HestonIntegrationByPartsFormula} for $u\in C^2(\bar\sO)$ and $v\in C^1(\bar\sO)$. This completes the proof.
\end{proof}

\begin{prob}[Classical solution to a homogeneous boundary value problem]
\label{prob:HestonMixedBVPHomogeneousClassical}
Given a function $f\in C^\alpha(\sO)$, for some $0<\alpha<1$, we call a function $u\in C^{2,\alpha}(\sO)\cap C_{\textrm{loc}}(\sO\cup\Gamma_1)$ a \emph{classical solution} to a boundary value problem for the Heston operator with homogeneous Dirichlet condition along $\Gamma_1$ if
\begin{align}
\label{eq:IntroHestonMixedProblemHomogeneous}
Au &= f \quad \hbox{on }\sO,
\\
\label{eq:IntroHestonMixedProblemHomogeneousBC}
u &= 0 \quad\hbox{on } \Gamma_1,
\\
\label{eq:WeightedNeumannHomogeneousBCLimit}
\lim_{y\downarrow 0}y^\beta(\rho u_x + \sigma u_y) &= 0 \quad\hbox{on } \Gamma_0.
\end{align}
\end{prob}

\begin{rmk}[Well-posedness of Problem \ref{prob:HestonMixedBVPHomogeneousClassical} and nature of the boundary condition along $\Gamma_0$]
\label{rmk:WellPosedClassicalBVP}
We shall see that additional hypotheses on $f$ are required to ensure that Problem \ref{prob:HestonMixedBVPHomogeneousClassical} is well-posed. (For example, \cite[Theorem 6.13]{GT} adds the hypothesis that $f \in L^\infty(\sO)$, though we will not require such a strong assumption in this article.) Note that if $u\in C^1_{\textrm{loc}}(\sO\cup\Gamma_0)$, then \eqref{eq:WeightedNeumannHomogeneousBCLimit} is obeyed automatically; Example \ref{exmp:CIR} explains the need for condition \eqref{eq:WeightedNeumannHomogeneousBCLimit}.
\end{rmk}

\begin{prob}[Strong solution to a homogeneous boundary value problem]
\label{prob:HestonStrongMixedBVPHomogeneous}
Given a function $f\in L^2(\sO,\fw)$, we call a function $u\in H^2(\sO,\fw)$ a \emph{strong solution} to a boundary value problem for the Heston operator with \emph{homogeneous} Dirichlet boundary condition on $\Gamma_1$ if $u$ obeys \eqref{eq:IntroHestonMixedProblemHomogeneous} (a.e. on $\sO$) and \eqref{eq:IntroHestonMixedProblemHomogeneousBC}.
\end{prob}

Lemma \ref{lem:HestonWeightedNeumannBVPHomogeneous} motivates the following definition of a solution to a variational equation for the Heston operator, by analogy with \cite[pp. 215--216]{GT}:

\begin{prob}[Solution to a homogeneous variational equation]
\label{prob:HestonWeakMixedBVPHomogeneous}
Given a function $f\in L^2(\sO,\fw)$, we call a function $u\in H^1_0(\sO\cup\Gamma_0)$ a \emph{solution to the variational equation} for the Heston operator with \emph{homogeneous} Dirichlet boundary condition on $\Gamma_1$ if
\begin{equation}
\label{eq:IntroHestonWeakMixedProblemHomogeneous}
a(u,v) = (f,v)_H, \quad \forall v \in H^1_0(\sO\cup\Gamma_0).
\end{equation}
\end{prob}

Lemma \ref{lem:HestonWeightedNeumannBVPHomogeneous} below explains why we may view solutions to Problem \ref{prob:HestonWeakMixedBVPHomogeneous} as ``weak solutions'' to Problem \ref{prob:HestonMixedBVPHomogeneousClassical} or \ref{prob:HestonStrongMixedBVPHomogeneous}

\begin{lem}[Equivalence of variational and strong solutions]
\label{lem:HestonWeightedNeumannBVPHomogeneous}
Require that the domain $\sO$ obeys Hypotheses \ref{hyp:HestonDomainNearGammaZero} and \ref{hyp:GammeOneExtensionProperty} with $k=1$. Let $f\in L^2(\sO,\fw)$ and suppose $u \in H^2(\sO,\fw)$.
\begin{enumerate}
\item If $u\in H^1_0(\sO\cup\Gamma_0,\fw)$ and $u$ solves Problem \ref{prob:HestonWeakMixedBVPHomogeneous}, then $u$ solves Problem \ref{prob:HestonStrongMixedBVPHomogeneous}.

\item If $u$ obeys \eqref{eq:IntroHestonMixedProblemHomogeneous} (a.e. on $\sO$) and \eqref{eq:IntroHestonMixedProblemHomogeneousBC}, then $u\in H^1_0(\sO\cup\Gamma_0,\fw)$ and solves Problem \ref{prob:HestonWeakMixedBVPHomogeneous}.
\end{enumerate}
\end{lem}

\begin{proof} Lemma \ref{lem:H2SobolevEmbedding} implies that $u \in C^\alpha_{\textrm{loc}}(\sO\cup\Gamma_1)$ when $u\in H^2(\sO,\fw)$, as our hypotheses on $\sO$ ensure that $\sO$ obeys a uniform interior cone condition.

\emph{(1) Assume $u\in H^1_0(\sO\cup\Gamma_0,\fw)$ solves Problem \ref{prob:HestonWeakMixedBVPHomogeneous}}. From \eqref{eq:HestonIntegrationByPartsFormula} and \eqref{eq:IntroHestonWeakMixedProblemHomogeneous} we obtain, for all $v \in H^1_0(\sO\cup\Gamma_0,\fw)$,
\begin{equation}
\label{eq:IntroHestonWeightedNeumannBVPHomogeneousStrong}
(Au,v)_{L^2(\sO,\fw)} = (f,v)_{L^2(\sO,\fw)},
\end{equation}
since we $v=0$ on $\Gamma_1$ (trace sense) by Lemma \ref{lem:EvansGamma1TraceZero}. In particular, $(Au,v)_{L^2(\sO,\fw)} = (f,v)_{L^2(\sO,\fw)}, \forall v\in H^1_0(\sO,\fw)$, since $H^1_0(\sO,\fw)\subseteqq H^1_0(\sO\cup\Gamma_0,\fw)$, and thus $Au = f$ a.e. on $\sO$. The fact that $u=0$ on $\Gamma_1$ follows from Lemmas \ref{lem:Traces} and \ref{lem:EvansGamma1TraceZero}.

\emph{(2) Assume $u \in H^2(\sO,\fw)$ obeys \eqref{eq:IntroHestonMixedProblemHomogeneous} (a.e. on $\sO$) and \eqref{eq:IntroHestonMixedProblemHomogeneousBC}}. Since $u=0$ on $\Gamma_1$, then $u\in H^1_0(\sO\cup\Gamma_0,\fw)$ by Lemma \ref{lem:EvansGamma1TraceZero}. Thus, Lemma \ref{lem:HestonIntegrationByParts} implies that \eqref{eq:IntroHestonWeakMixedProblemHomogeneous} holds, since the $\Gamma_1$-boundary integral in \eqref{eq:HestonIntegrationByPartsFormula} is zero when $v=0$ on $\Gamma_1$ (trace sense).
\end{proof}

\begin{lem}[Weighted-Neumann boundary property of functions in $H^2(\sO,\fw)$]
\label{lem:HestonWeightedNeumannBoundaryProperty}
Require that the domain $\sO$ obeys Hypothesis \ref{hyp:HestonDomainNearGammaZero}. If $u \in H^2(\sO,\fw)$, then $u$ obeys
\begin{equation}
\label{eq:WeightedNeumannHomogeneousBCProb}
y^\beta(\rho u_x + \sigma u_y) = 0 \quad\hbox{on } \Gamma_0 \hbox{ (trace sense).}
\end{equation}
\end{lem}

\begin{proof} Lemma \ref{lem:TracesGammaZero} implies that $y^\beta Du|_{\Gamma_0} \in L^2(\Gamma_0,e^{-\gamma|x|}\,dx;\RR^2)$ because $u \in H^2(\sO,\fw)$. The derivation of \eqref{eq:Gamma0BdryTermIsZero} in the proof of Lemma \ref{lem:HestonIntegrationByParts} shows that
$$
\int_{\Gamma_0} \left(\rho u_x + \sigma u_y\right)v y\fw\,dx = \int_{\Gamma_0} y^\beta\left(\rho u_x + \sigma u_y\right)v e^{-\gamma|x|}\,dx = 0,
$$
for all $v\in C^1(\bar\sO)$ when $u\in C^2(\bar\sO)$, and hence when $u \in H^2(\sO,\fw)$ by the approximation argument used in the proof of Lemma \ref{lem:HestonIntegrationByParts}. Therefore, $y^\beta(\rho u_x + \sigma u_y) = 0$ on $\Gamma_0$ (trace sense), which gives \eqref{eq:WeightedNeumannHomogeneousBCProb} (trace sense).
\end{proof}

\begin{rmk}[Nature of the boundary property \eqref{eq:WeightedNeumannHomogeneousBCProb}]
Note that if $u\in C^1_{\textrm{loc}}(\sO\cup\Gamma_0)$, then $u$ automatically has the property \eqref{eq:WeightedNeumannHomogeneousBCProb}; see also Lemma \ref{lem:ybetaDuZeroTrace} for another interpretation of \eqref{eq:WeightedNeumannHomogeneousBCProb}.
\end{rmk}

\begin{rmk}[Homogeneous boundary value problem for the Heston operator with Dirichlet condition along $\Gamma_0$]
A version of Problem \ref{prob:HestonWeakMixedBVPHomogeneous}, with an additional homogeneous Dirichlet boundary condition, $u=0$ on $\Gamma_0$, when $0<\beta<1$, is easily seen to be well-posed by methods which are almost identical to those employed in this article. However, solutions to this Dirichlet problem, when $0<\beta<1$, are not assured to be any more than $C^0$ up to the boundary, $\Gamma_0$, as Example \ref{exmp:CIR} illustrates.
\end{rmk}

We can also pose the corresponding \emph{inhomogeneous} boundary value problems for the Heston operator:

\begin{prob}[Classical solution to an inhomogeneous boundary value problem]
\label{prob:HestonMixedBVPInhomogeneousClassical}
Given functions $f\in C^\alpha(\sO)$, for some $0<\alpha<1$, and $g\in C^{2,\alpha}(\sO)\cap C_{\textrm{loc}}(\sO\cup\Gamma_1)$, we call a function $u\in C^{2,\alpha}(\sO)\cap C_{\textrm{loc}}(\sO\cup\Gamma_1)$ a \emph{classical solution} to a boundary value problem for the Heston operator with \emph{inhomogeneous} Dirichlet condition along $\Gamma_1$ if $u$ obeys \eqref{eq:IntroHestonMixedProblemHomogeneous}, \eqref{eq:WeightedNeumannHomogeneousBCLimit}, and
\begin{equation}
\label{eq:IntroHestonMixedProblemInhomogeneousBC}
u = g \quad\hbox{on } \Gamma_1.
\end{equation}
\end{prob}

\begin{prob}[Strong solution to an inhomogeneous boundary value problem]
\label{prob:HestonStrongMixedBVPInhomogeneous}
Given functions $f\in L^2(\sO,\fw)$ and $g\in H^2(\sO,\fw)$, we call a function $u\in H^2(\sO,\fw)$ a \emph{strong solution} to a boundary value problem for the Heston operator with \emph{inhomogeneous} Dirichlet condition along $\Gamma_1$ if $u$ obeys \eqref{eq:IntroHestonMixedProblemHomogeneous} (a.e. on $\sO$) and \eqref{eq:IntroHestonMixedProblemInhomogeneousBC}.
\end{prob}

\begin{prob}[Solution to an inhomogeneous variational equation]
\label{prob:HestonMixedBVPInhomogeneous}
Given functions $f\in L^2(\sO,\fw)$ and $g\in H^1(\sO,\fw)$, we call a function $u\in H^1(\sO,\fw)$ a \emph{solution to a variational equation} for the Heston operator with \emph{inhomogeneous} Dirichlet condition along $\Gamma_1$ if $u-g \in H^1_0(\sO\cup\Gamma_0)$, and
\begin{equation}
\label{eq:IntroHestonWeakMixedProblem}
a(u,v) = (f,v)_H, \quad \forall v \in H^1_0(\sO\cup\Gamma_0).
\end{equation}
\end{prob}

Again, a suitable version of Lemma \ref{lem:HestonWeightedNeumannBVPHomogeneous} explains why we may view solutions to Problem \ref{prob:HestonMixedBVPInhomogeneous} as ``weak solutions'' to Problem \ref{prob:HestonMixedBVPInhomogeneousClassical} or \ref{prob:HestonStrongMixedBVPInhomogeneous}. Unless stated otherwise, we restrict our attention to the study of the \emph{homogeneous} cases, Problems \ref{prob:HestonMixedBVPHomogeneousClassical}, \ref{prob:HestonStrongMixedBVPHomogeneous}, and \ref{prob:HestonWeakMixedBVPHomogeneous}.

\subsection{Energy estimates for the Heston operator}
\label{subsec:HestonEnergyEstimates}
We first derive a G\r{a}rding estimate for the bilinear form \eqref{eq:HestonWithKillingBilinearForm}.

\begin{prop}[Diagonal continuity and G\r{a}rding estimates]
\label{prop:EnergyGardingHeston}
Require that the domain $\sO$ obeys Hypotheses \ref{hyp:HestonDomainNearGammaZero} and \ref{hyp:GammeOneExtensionProperty} with $k=1$. Relax the requirement that Assumption \ref{assump:HestonCoefficientb1} is in effect. Then there are positive constants $C_1, C_2, \gamma_0$, depending only the constant coefficients of $A$, such that for all
\begin{equation}
\label{eq:gammapositivesmallenough}
0<\gamma\leq \gamma_0,
\end{equation}
and all $u\in H^1_0(\sO\cup\Gamma_0,\fw)$, the bilinear form $a(\cdot,\cdot)$ in \eqref{eq:HestonWithKillingBilinearForm} obeys
\begin{align}
\label{eq:HestonBilinearUpperDiagonal}
|a(u,u)| &\leq C_1\|u\|_V^2,
\\
\label{eq:HestonBilinearFormGarding}
a(u,u) &\geq \frac{1}{2}C_2\|u\|_V^2 - C_2\|(1+y)^{1/2}u\|_H^2,
\end{align}
where $C_2 := \min\{\sigma^2(1-\rho^2)/2, (1-\rho^2)/2\}$, $C_3 := \frac{1}{2}\max\{|a_1|,|b_1|\}$ (with $a_1,b_1$ are as in \eqref{eq:DefinitionA1B1}), $\gamma_0 := C_2/2C_3$, while $C_4 := \max\{\sigma^2(1-\rho^2)/2, (1-\rho^2)/2\}$ and $C_1:=\max\{C_4+\gamma C_3, r+\gamma C_3\}$.

When Assumption \ref{assump:HestonCoefficientb1} is in effect, so $b_1=0$, then Hypotheses \ref{hyp:HestonDomainNearGammaZero} and \ref{hyp:GammeOneExtensionProperty} are not required and \eqref{eq:HestonBilinearUpperDiagonal} and \eqref{eq:HestonBilinearFormGarding} hold for all $u\in H^1(\sO,\fw)$ and $\gamma \geq 0$.
\end{prop}

\begin{rmk}[Application of affine changes of coordinates]
\label{rmk:AffineChangeCoordinatesGarding}
Given Assumption \ref{assump:HestonCoefficientb1}, we can require that $b_1=0$ when applying Proposition \ref{prop:EnergyGardingHeston} and as needed elsewhere in this article.
\end{rmk}

\begin{proof}[Proof of Proposition \ref{prop:EnergyGardingHeston}]
To obtain \eqref{eq:HestonBilinearFormGarding}, observe that
$$
\rho^2 + 2\rho\sigma + \rho^2\sigma^2 = (\rho + \rho\sigma)^2 \geq 0,
$$
and so
$$
2\rho\sigma \geq -\rho^2 - \rho^2\sigma^2.
$$
Thus,
\begin{align*}
a(u,u) &= \frac{1}{2}\int_\sO\left(u_x^2 + 2\rho\sigma u_yu_x + \sigma^2u_y^2\right)y\,\fw\,dxdy +  \int_\sO ru^2\,\fw\,dxdy
\\
&\quad -\gamma \frac{1}{2}\int_\sO\left(u_x + \rho\sigma u_y\right)u \sign(x)y\,\fw\,dxdy
- \int_\sO(a_1y + b_1)u_xu\,\fw\,dxdy
\\
&\geq \frac{1}{2}\int_\sO\left((1-\rho^2)u_x^2 + \sigma^2(1-\rho^2)u_y^2\right)y\,\fw\,dxdy  + \int_\sO ru^2\,\fw\,dxdy
\\
&\quad -\gamma \frac{1}{2}\int_\sO\left(u_x + \rho\sigma u_y\right)u \sign(x)y\,\fw\,dxdy
- \frac{1}{2}\int_\sO(a_1y + b_1)(u^2)_x\,\fw\,dxdy.
\end{align*}
Because Hypothesis \ref{hyp:GammeOneExtensionProperty} holds when $k=1$, we may suppose that $u\in H^1_0(\sO\cup\Gamma_0,\fw)$ is the restriction of a function $\tilde u \in H^1(\HH,\fw)$ with $\tilde u = u$ on $\sO$ and $\tilde u = 0$ on $\HH\less\sO$ by a straightforward analogue of Theorem \ref{thm:BoundaryTrace} for our weighted Sobolev spaces; for simplicity, we will denote this extension again by $u$. Integrating by parts with respect to $x$ via Lemma \ref{lem:HestonIntegrationByParts} and using the fact that $u=0$ on $\Gamma_1$ in the trace sense (since $u\in H^1_0(\sO\cup\Gamma_0,\fw)$),
\begin{align*}
\frac{1}{2}\int_\sO(a_1y + b_1)(u^2)_x\,\fw\,dxdy &= \frac{1}{2}\int_{\Gamma_1} (a_1y + b_1)u^2\,\fw\,dy
- \frac{1}{2}\int_\sO(a_1y + b_1)u^2\,\fw_x\,dxdy
\\
&= -\frac{\gamma}{2}\int_\sO(a_1y + b_1)u^2\sign(x)\,\fw\,dxdy
\\
&\leq \gamma C_3'\int_\sO(1+y)u^2\,\fw\,dxdy,
\end{align*}
where $C_3' := \frac{1}{2}\max\{|a_1|,|b_1|\}$. Moreover,
\begin{align*}
{}&\gamma \frac{1}{2}\int_\sO\left(u_x + \rho\sigma u_y\right)u \sign(x)y\,\fw\,dxdy
\\
&\quad\leq \gamma C_3''\left(\int_\sO(|u_x| + |u_y|)|u| y\,\fw\,dxdy\right)
\\
&\quad\leq \gamma C_3''\left[\left(\int_\sO u_x^2y\,\fw\,dxdy\right)^{1/2}+\left(\int_\sO u_y^2y\,\fw\,dxdy\right)^{1/2}\right]
\left(\int_\sO u^2 y\,\fw\,dxdy\right)^{1/2}
\\
&\quad\leq \gamma \frac{C_3''}{2}\left[\int_\sO (u_x^2+u^2) y\,\fw\,dxdy + \int_\sO (u_y^2+u^2) y\,\fw\,dxdy\right]
\\
&\quad= \gamma \frac{C_3''}{2}\int_\sO\left(u_x^2 + u_y^2\right)y\,\fw\,dxdy + \gamma C_3''\int_\sO u^2y\,\fw\,dxdy,
\end{align*}
where $C_3'':=\max\{1/2, |\rho|\sigma/2\}$. Combining the preceding three inequalities and setting $C_2 := \min\{\sigma^2(1-\rho^2)/2, (1-\rho^2)/2\}$, and $C_3=C_3'+C_3''$ gives
\begin{align*}
a(u,u) &\geq C_2\int_\sO\left(u_x^2 + u_y^2\right)y\,\fw\,dxdy +  r\int_\sO u^2\,\fw\,dxdy
\\
&\quad - \gamma C_3\int_\sO\left(u_x^2 + u_y^2\right)y\,\fw\,dxdy - \gamma C_3\int_\sO (1+y)u^2\,\fw\,dxdy
\\
&= C_2\int_\sO\left(u_x^2 + u_y^2\right)y\,\fw\,dxdy +  r\int_\sO u^2\,\fw\,dxdy +  C_2\int_\sO (1+y)u^2\,\fw\,dxdy
\\
&\quad - \gamma C_3\int_\sO\left(u_x^2 + u_y^2\right)y\,\fw\,dxdy - \gamma C_3\int_\sO (1+y)u^2\,\fw\,dxdy - C_2\int_\sO (1+y)u^2\,\fw\,dxdy,
\end{align*}
and thus, using $r\geq 0$,
\begin{equation}
\label{eq:HestonBilinearFormPreGarding}
\begin{aligned}
a(u,u) &\geq C_2\int_\sO\left(u_x^2 + u_y^2\right)y\,\fw\,dxdy + C_2\int_\sO (1+y)u^2\,\fw\,dxdy
\\
&\quad - \gamma C_3\int_\sO\left(u_x^2 + u_y^2\right)y\,\fw\,dxdy - \gamma C_3\int_\sO (1+y)u^2\,\fw\,dxdy
\\
&\quad - C_2\int_\sO (1+y)u^2\,\fw\,dxdy.
\end{aligned}
\end{equation}
Choosing $\gamma_0 := C_2/2C_3$ and $0<\gamma\leq\gamma_0$ in \eqref{eq:HestonBilinearFormPreGarding} yields the lower bound \eqref{eq:HestonBilinearFormGarding} for $a(u,u)$.

Virtually the same argument, with $C_4 := \max\{\sigma^2(1-\rho^2)/2, (1-\rho^2)/2\}$, yields
\begin{align*}
|a(u,u)| &\leq C_4\int_\sO\left(u_x^2 + u_y^2\right)y\,\fw\,dxdy +  r\int_\sO u^2\,\fw\,dxdy
\\
&\quad + \gamma C_3\int_\sO\left(u_x^2 + u_y^2\right)y\,\fw\,dxdy + \gamma C_3\int_\sO (1+y)u^2\,\fw\,dxdy,
\end{align*}
and the upper bound for $a(u,u)$ follows with $C_1:=\max\{C_4+\gamma C_3, r+\gamma C_3\}$.

When $b_1=0$, the condition $u\in H^1_0(\sO\cup\Gamma_0,\fw)$ can be relaxed to $u\in H^1(\sO,\fw)$ since integration by parts with respect to $x$ is no longer required to estimate the term
$$
\frac{1}{2}\int_\sO b_1(u^2)_x\,\fw\,dxdy,
$$
and the estimates hold for any $\gamma\geq 0$. This completes the proof.
\end{proof}

\begin{rmk}[Refinement when $r>0$]
\label{rmk:EnergyGardingHeston}
The lower bound \eqref{eq:HestonBilinearFormGarding} for $a(u,u)$ can be sharpened slightly when $r>0$ to
$$
a(u,u) \geq \frac{1}{2}C_2'\|u\|_V^2 - C_2'\|y^{1/2}u\|_H^2,
\quad \forall u \in V,
$$
where $C_2'=\min\{C_2,r\}$, but this refinement seems to bring little benefit in practice.
\end{rmk}

\begin{assump}[Choice of the constant $\gamma$ in the definition of the Sobolev weight]
\label{assmp:gammmachoice}
For the remainder of this article, we choose $\gamma=\gamma_0$ in \eqref{eq:HestonWeight}, where $\gamma_0$ is defined in the statement of Proposition \ref{prop:EnergyGardingHeston} in terms of the constant coefficients of the operator $A$ in \eqref{eq:OperatorHestonIntro}.
\end{assump}

\begin{prop}[Continuity estimate]
\label{prop:FullEnergyHeston}
Relax the requirements that Assumptions \ref{assump:HestonCoefficientb1} or \ref{assmp:gammmachoice} are in effect. Then, for all $\gamma\geq 0$ and $u, v \in H^1(\sO,\fw)$,
\begin{equation}
\label{eq:StrongHestonBilinearFormContinuity}
|a(u,v)| \leq C_5\|u\|_{H^1(\sO,\fw)}\left(\|v\|_V + \|y^{-1/2}v\|_{L^2(\sO,\fw)}\right),
\end{equation}
where $C_5>0$ depends at most on the coefficients $r,q,\kappa,\theta,\rho,\sigma$, and $\gamma$.

When Assumption \ref{assump:HestonCoefficientb1} is in effect, so $b_1=0$, then, for all $\gamma\geq 0$ and $u, v \in H^1(\sO,\fw)$,
\begin{equation}
\label{eq:StrongerHestonBilinearFormContinuity}
|a(u,v)| \leq C_5\|u\|_{H^1(\sO,\fw)}\|v\|_{H^1(\sO,\fw)}.
\end{equation}
\end{prop}

\begin{rmk}[Application of affine changes of coordinates]
\label{rmk:AffineChangeCoordinatesContinuity}
With the aid of Lemma \ref{lem:RescalingHeston}, we may assume without loss of generality that $b_1=0$ when applying Proposition \ref{prop:FullEnergyHeston} and as needed throughout the remainder of this article.
\end{rmk}

\begin{rmk}[Alternative to affine changes of coordinates]
\label{rmk:FullEnergyHeston}
When $\beta<1$ and $v\in H_0^1(\sO,\fw)$ or even $H_0^1(\sO\cup\Gamma_1,\fw)$, so $u=0$ along $\Gamma_0$ in the trace sense, then $\|y^{-1/2}v\|_H < \infty$ is finite by Theorem \ref{thm:HardyInequality}; when $\beta\geq 1$, then $H^1(\sO,\fw) = H_0^1(\sO\cup\Gamma_0,\fw)$ by Lemma \ref{lem:ImprovedH1ApproximationLemma} and so Theorem \ref{thm:HardyInequality} applies.
\end{rmk}

\begin{proof}[Proof of Proposition \ref{prop:FullEnergyHeston}]
To obtain the upper bound \eqref{eq:StrongHestonBilinearFormContinuity} for $|a(u,v)|$, write
$$
a(u,v) = a^{2,0}(u,v) + a^1(u,v),
$$
where
\begin{align*}
a^1(u,v) &:= -\frac{\gamma}{2}\int_\sO\left(u_x + \rho\sigma u_y\right)v \sign(x)y\,\fw\,dxdy
\\
&\quad - \int_\sO(a_1y + b_1)u_xv\,\fw\,dxdy,
\end{align*}
and separately consider $a^1(u,v)$ and $a^{2,0}(u,v)$. First, observe that
\begin{align*}
|(a_1yu_x,v)_{L^2(\sO,\fw)}| &\leq |a_1|\|y^{1/2}u_x\|_{L^2(\sO,\fw)}\|y^{1/2}v\|_{L^2(\sO,\fw)}
\\
&\leq C_6'\|u\|_V\|v\|_V,
\end{align*}
where $C_6' := |a_1|$. Second, note that
\begin{align*}
{}&\frac{\gamma}{2}\left|\int_\sO\left(u_x + \rho\sigma u_y\right)v \sign(x)y\,\fw\,dxdy\right|
\\
&\qquad \leq
C_6''\left(\|y^{1/2}u_x\|_{L^2(\sO,\fw)} + \|y^{1/2}u_y\|_{L^2(\sO,\fw)}\right)\|y^{1/2}v\|_{L^2(\sO,\fw)}
\\
&\qquad \leq C_6''\|u\|_V\|v\|_V,
\end{align*}
where $C_6'' := \max\{\gamma/2,\gamma\rho\sigma/2\}$. Third, we have
\begin{align*}
|(b_1u_x,v)_{L^2(\sO,\fw)}|
&\leq |b_1|\|y^{1/2}u_x\|_{L^2(\sO,\fw)}\|y^{-1/2}v\|_{L^2(\sO,\fw)}
\\
&\leq C_6'''\|u\|_V\|y^{-1/2}v\|_{L^2(\sO,\fw)},
\end{align*}
where $C_6''' := |b_1|$. Combining the preceding three estimates yields the estimate \eqref{eq:StrongHestonBilinearFormContinuity} for the term $a^1(u,v)$, with constant $C_6 := C_6'+C_6''+C_6'''$.

For the term
$$
a^{0,2}(u,v) := \frac{1}{2}\int_\sO\left(u_xv_x + \rho\sigma u_yv_x
+ \rho\sigma u_xv_y + \sigma^2 u_yv_y\right)y\,\fw\,dxdy
+ \int_\sO ruv\,\fw\,dxdy,
$$
observe that
\begin{align*}
|a^{0,2}(u,v)| &\leq \frac{1}{2}\left(y^{1/2}u_x,y^{1/2}v_x\right)_{L^2(\sO,\fw)} + \frac{\rho\sigma}{2}\left(y^{1/2}u_y,y^{1/2}v_x\right)_{L^2(\sO,\fw)}
\\
&\quad + \frac{\rho\sigma}{2}\left(y^{1/2}u_x,y^{1/2}v_y\right)_{L^2(\sO,\fw)}
+ \frac{\sigma^2}{2}\left(y^{1/2}u_y,y^{1/2}v_y\right)_{L^2(\sO,\fw)} + r(u,v)_{L^2(\sO,\fw)}
\\
&\leq C_7\|u\|_V\|v\|_V,
\end{align*}
where $C_7 := \frac{1}{2}+\rho\sigma+\frac{\sigma^2}{2}+r$. This yields the estimate \eqref{eq:StrongHestonBilinearFormContinuity} for the term $a^{2,0}(u,v)$, with constant $C_7$. Combining these observations gives the desired estimate \eqref{eq:StrongHestonBilinearFormContinuity} for $a(u,v)$ with constant $C_5 := C_6+C_7$.

When Assumption \ref{assump:HestonCoefficientb1} is in effect, so $b_1=0$, the estimate \eqref{eq:StrongerHestonBilinearFormContinuity} follows immediately from the proof of \eqref{eq:StrongHestonBilinearFormContinuity} in the case $b\neq 0$ since we do not need to estimate the term $(b_1u_x,v)_{L^2(\sO,\fw)}$.
\end{proof}

\subsection{Bilinear form energy identity and estimate}
\label{subsec:BilinearFormEnergyIdentityEstimate}
We shall employ the useful identities and estimates described here at several points in this article.

\begin{lem}[Bilinear form energy identity]
\label{lem:EnergyIdentityIntegrationByParts}
Let $u \in H^1(\sO,\fw)$ and $\varphi \in C^2_0(\RR^2)$. Then
\begin{equation}
\label{eq:EnergyIdentityIntegrationByParts}
\begin{aligned}
{}&a(\varphi u, \varphi u) - a(u, \varphi^2 u)
\\
&\quad = \frac{1}{2}\sigma^2(u, y\varphi_x^2 u)_{L^2(\sO,\fw)} + \sigma^2(u, y\varphi_x\varphi_y u)_{L^2(\sO,\fw)}
+ \frac{1}{2}\sigma^2(u, y\varphi_y^2 u)_{L^2(\sO,\fw)}
\\
&\qquad - \frac{\gamma}{2}\rho\sigma(u, y\varphi(\varphi_x + \varphi_y)\sign(x)u)_{L^2(\sO,\fw)}.
\end{aligned}
\end{equation}
\end{lem}

\begin{proof}
Let $I_1, I_2, I_3$, and $I_4$ denote the four integral terms in the expression \eqref{eq:HestonWithKillingBilinearForm} for the bilinear form, $a(u,v)$. First, we compute
\begin{align*}
((\varphi u)_x, y(\varphi  u)_x)_{L^2(\sO,\fw)} -  (u_x, y(\varphi^2 u)_x)_{L^2(\sO,\fw)}
&= (u, y\varphi_x^2 u)_{L^2(\sO,\fw)},
\\
((\varphi u)_y, y(\varphi  u)_x)_{L^2(\sO,\fw)} - (u_y, y(\varphi^2 u)_x)_{L^2(\sO,\fw)}
&=  (u, y\varphi_x\varphi_y u)_{L^2(\sO,\fw)} + (u, y\varphi\varphi_y u_x)_{L^2(\sO,\fw)}
\\
&\quad - (u_y, y\varphi\varphi_x u)_{L^2(\sO,\fw)},
\\
((\varphi u)_x, y(\varphi u)_y)_{L^2(\sO,\fw)} - (u_x, y(\varphi^2u)_y)_{L^2(\sO,\fw)}
&=  (u, y\varphi_y\varphi_x u)_{L^2(\sO,\fw)} + (u, y\varphi\varphi_x u_y)_{L^2(\sO,\fw)}
\\
&\quad - (u_x, y\varphi\varphi_y u)_{L^2(\sO,\fw)},
\\
((\varphi u)_y, y(\varphi u)_y)_{L^2(\sO,\fw)} - (u_y, y(\varphi^2u)_y)_{L^2(\sO,\fw)}
&=  (u, y\varphi_y^2 u)_{L^2(\sO,\fw)},
\end{align*}
and so
$$
I_1(\varphi u, \varphi u) - I_1(u, \varphi^2u) = \frac{1}{2}\sigma^2(u, y\varphi_x^2 u)_{L^2(\sO,\fw)} + \sigma^2(u, y\varphi_x\varphi_y u)_{L^2(\sO,\fw)}
+ \frac{1}{2}\sigma^2(u, y\varphi_y^2 u)_{L^2(\sO,\fw)}.
$$
Second, we obtain
\begin{align*}
I_2(\varphi u, \varphi u) &= - \frac{\gamma}{2}\int_\sO\left((\varphi u)_x + \rho\sigma (\varphi u)_y\right)\varphi u \sign(x)y\,\fw\,dxdy
\\
&= - \frac{\gamma}{2}\int_\sO\left(u_x + \rho\sigma u_y\right)\varphi^2 u \sign(x)y\,\fw\,dxdy
\\
&\quad - \frac{\gamma}{2}\int_\sO \rho\sigma(\varphi_x + \varphi_y) u \varphi u \sign(x)y\,\fw\,dxdy
\\
&= I_2(u, \varphi^2 u) - \frac{\gamma}{2}\int_\sO \rho\sigma\varphi(\varphi_x + \varphi_y) u^2\sign(x)y\,\fw\,dxdy.
\end{align*}
Third, we see by inspection that
$$
I_3(\varphi u, \varphi u) = I_3(u, \varphi^2 u) \quad\hbox{and}\quad I_4(\varphi u, \varphi u) = I_4(u, \varphi^2 u).
$$
Combining the identities for $I_i(\varphi u, \varphi u)$, $i=1,2,3,4$, yields \eqref{eq:EnergyIdentityIntegrationByParts}.
\end{proof}

\begin{rmk}[Significance of the identity \eqref{eq:EnergyIdentityIntegrationByParts}]
The important feature of the identity \eqref{eq:EnergyIdentityIntegrationByParts} is that the right-hand side contains no derivatives of $u$.
\end{rmk}

The identity \eqref{eq:EnergyIdentityIntegrationByParts} may also be derived using an expression for the commutator, $[A,\varphi]$, and Lemma \ref{lem:HestonIntegrationByParts}, although this method is less direct. Suppose $\varphi \in C^2_0(\RR^2)$ and $u\in H^2_{\textrm{loc}}(\sO)$. From \eqref{eq:OperatorHestonIntro}, we obtain
\begin{align*}
[A,\varphi]u &= -\frac{y}{2}\left(\varphi_{xx}u + 2\varphi_x u_x + 2\rho\sigma\left(\varphi_{xy}u + \varphi_x u_y + \varphi_y u_x\right)
+ \sigma^2\left(\varphi_{yy}u + 2\varphi_y u_y \right)\right)
\\
&\quad - (r-q-y/2)\varphi_x u - \kappa(\theta-y)\varphi_y u,
\end{align*}
and thus
\begin{equation}
\label{eq:ACommutator}
\begin{aligned}
{}[A,\varphi]u &= -y\left((\varphi_x + \rho\sigma\varphi_y)u_x + (\rho\sigma\varphi_x + \sigma^2\varphi_y)u_y\right)
\\
&\quad - \frac{y}{2}\left(\varphi_{xx} + 2\rho\sigma\varphi_{xy} + \sigma^2\varphi_{yy}\right)u
\\
&\quad - (r-q-y/2)\varphi_x u - \kappa(\theta-y)\varphi_y u.
\end{aligned}
\end{equation}
Since $[A,\varphi]$ is a first-order partial differential operator, the identity \eqref{eq:ACommutator} is valid when $u\in H^1_{\textrm{loc}}(\sO)$.

\begin{rmk}[Commutator identity for the coercive Heston operator]
The identity \eqref{eq:ACommutator} remains otherwise unchanged when $A$ is replaced by $A_\lambda$ using \eqref{eq:CoerciveHestonOperator}.
\end{rmk}

With the preceding observations in hand, we obtain

\begin{cor}[Bilinear form energy and commutator identities]
\label{cor:CommutatorInnerProduct}
Let $u, v \in H^1_0(\sO\cup\Gamma_0,\fw)$ and $\varphi \in C^2_0(\RR^2)$. Then
\begin{equation}
\label{eq:BilinearFormSquaredFunction}
a(\varphi u, \varphi v) = a(u,\varphi^2 v) + ([A,\varphi]u,\varphi v)_{L^2(\sO,\fw)},
\end{equation}
and, when $u=v$,
\begin{equation}
\begin{aligned}
\label{eq:CommutatorInnerProduct}
([A,\varphi]u,\varphi u)_{L^2(\sO,\fw)}
&= \frac{1}{2}\sigma^2(u, y\varphi_x^2 u)_{L^2(\sO,\fw)} + \sigma^2(u, y\varphi_x\varphi_y u)_{L^2(\sO,\fw)}
\\
&\quad + \frac{1}{2}\sigma^2(u, y\varphi_y^2 u)_{L^2(\sO,\fw)} - \frac{\gamma}{2}\rho\sigma(u, y\varphi(\varphi_x + \varphi_y)\sign(x)u)_{L^2(\sO,\fw)}.
\end{aligned}
\end{equation}
\end{cor}

\begin{proof}
We temporarily require, in addition, that $u \in C^\infty_0(\sO\cup\Gamma_0)$ and recall that $C^\infty_0(\sO\cup\Gamma_0)$ is dense in $H^1_0(\sO\cup\Gamma_0,\fw)$ by Definition \ref{defn:H1WeightedSobolevSpaces}. Then
\begin{align*}
a(\varphi u, \varphi v) &= (A(\varphi u),\varphi v)_{L^2(\sO,\fw)} \quad\hbox{(Lemma \ref{lem:HestonIntegrationByParts})}
\\
&= (\varphi Au,\varphi v)_{L^2(\sO,\fw)} + ([A,\varphi]u,\varphi v)_{L^2(\sO,\fw)}
\\
&= (Au,\varphi^2 v)_{L^2(\sO,\fw)} + ([A,\varphi]u,\varphi v)_{L^2(\sO,\fw)}
\\
&= a(u,\varphi^2 v) + ([A,\varphi]u,\varphi v)_{L^2(\sO,\fw)} \quad\hbox{(Lemma \ref{lem:HestonIntegrationByParts})}.
\end{align*}
Since the left-hand and right-hand terms in the preceding identity are well-defined for any $u,v \in H^1_0(\sO\cup\Gamma_0,\fw)$, we obtain \eqref{eq:BilinearFormSquaredFunction} by choosing a sequence $\{u_n\}_{n\geq 1} \subset C^\infty_0(\sO\cup\Gamma_0)$ which converges strongly in $H^1(\sO,\fw)$ to $u$ and taking limits as $n\to \infty$.

We obtain the identity \eqref{eq:CommutatorInnerProduct} by comparing \eqref{eq:EnergyIdentityIntegrationByParts} and \eqref{eq:BilinearFormSquaredFunction}.
\end{proof}

\begin{rmk}[Coercive bilinear form and operator inner product identities]
The identity \eqref{eq:BilinearFormSquaredFunction} remains otherwise unchanged when $a(\cdot,\cdot)$ is replaced by $a_\lambda(\cdot,\cdot)$ using \eqref{eq:BilinearFormCoerciveHeston}; the identity \eqref{eq:CommutatorInnerProduct} remains otherwise unchanged when $A$ is replaced by $A_\lambda$ using \eqref{eq:CoerciveHestonOperator}.
\end{rmk}

\begin{cor}[Commutator inner product estimate]
\label{cor:CommutatorInnerProductEstimate}
Let $u \in H^1_0(\sO\cup\Gamma_0,\fw)$ and $\varphi \in C^2_0(\RR^2)$. Then there is a constant $C$, depending only on the constant coefficients of $A$, such that
\begin{equation}
\label{eq:CommutatorInnerProductEstimate}
|([A,\varphi]u,\varphi u)_{L^2(\sO,\fw)}| \leq C\|y^{1/2}(|D\varphi| + |D\varphi|^{1/2})u\|_{L^2(\sO,\fw)}^2.
\end{equation}
\end{cor}

\begin{proof}
The estimate follows immediately from \eqref{eq:CommutatorInnerProduct}.
\end{proof}

\section{Existence and uniqueness of solutions to the variational equation}
\label{sec:VariationalEquation}
In this section we establish existence and uniqueness of solutions to the variational equation for the elliptic Heston operator, Problem \ref{prob:HestonWeakMixedBVPHomogeneous}. In \S \ref{subsec:CoerciveVarInequality} we prove existence and uniqueness for the case of a Heston operator which is modified so that its associated bilinear form is coercive (Theorem \ref{thm:ExistenceUniquenessEllipticCoerciveHeston}) and in \S \ref{subsec:NoncoerciveVarInequality} we extend that result to the full non-coercive case (Theorem \ref{thm:ExistenceUniquenessEllipticHeston_Improved}).

\subsection{Existence and uniqueness of solutions to the coercive variational equation}
\label{subsec:CoerciveVarInequality}
The inequality \eqref{eq:HestonBilinearFormGarding} illustrates that the bilinear form \eqref{eq:HestonWithKillingBilinearForm} is not necessarily coercive in the sense of Theorem \ref{thm:LaxMilgram} but it motivates the following

\begin{defn}[Coercive Heston operator and associated bilinear form]
\label{defn:CoerciveHestonBilinearForm}
Let $A$ be as in \eqref{eq:OperatorHestonIntro} and let $a:V\times V\to\RR$ be given by \eqref{eq:HestonWithKillingBilinearForm}, where $V=H^1_0(\sO\cup\Gamma_0,\fw)$. Define the differential operator $A_\lambda$ by
\begin{equation}
\label{eq:CoerciveHestonOperator}
A_\lambda := A + \lambda(1+y),
\end{equation}
and define the bilinear form
$$
a_\lambda: V\times V \to \RR, \quad (u,v) \mapsto a_\lambda(u,v),
$$
by
\begin{equation}
\label{eq:BilinearFormCoerciveHeston}
a_\lambda(u,v) := a(u,v) + \lambda((1+y)u,v)_{L^2(\sO,\fw)}, \quad\forall u,v \in V.
\end{equation}
\end{defn}

The following lemma explains when the bilinear form \eqref{eq:BilinearFormCoerciveHeston} is coercive:

\begin{lem}[Energy estimates for the coercive bilinear form]
\label{lem:LargeEnoughLambdaToEnsureCoercivity}
There is a positive constant, $\lambda_0$, depending only on the constant coefficients of $A$, such that for all $\lambda\geq \lambda_0$, the bilinear form \eqref{eq:BilinearFormCoerciveHeston} is \emph{continuous} and \emph{coercive} in the sense that,
\begin{align}
\label{eq:ContinuousCoerciveHeston}
|a_\lambda(u,v)| &\leq C\|u\|_V\|v\|_V, \quad \forall u,v \in V,
\\
\label{eq:CoerciveHeston}
a_\lambda(v,v) &\geq \nu_1\|v\|_V^2, \quad \forall v \in V,
\end{align}
where $C$, $\lambda$, and $\nu_1$ are positive constants depending only on the constant coefficients of $A$.
\end{lem}

\begin{proof}
The bilinear form $a_\lambda:V\times V\to\RR$ is continuous for any $\lambda\in\RR$ since
\begin{align*}
|a_\lambda(u,v)| &\leq |a(u,v)| + \lambda\left|\left((1+y)^{1/2}u,(1+y)^{1/2}v\right)_H\right|
\\
&\leq C_5\|u\|_V\|v\|_V + \lambda\left|(1+y)^{1/2}u\right|_H\left|(1+y)^{1/2}v\right|_H
\quad\hbox{(by \eqref{eq:StrongerHestonBilinearFormContinuity})}
\\
&\leq (C_5+\lambda)\|u\|_V\|v\|_V, \quad\hbox{(by Definition \ref{defn:H1WeightedSobolevSpaces}),}
\end{align*}
for all $u,v \in V$, yielding \eqref{eq:ContinuousCoerciveHeston} with $C = C_5+\lambda$. Moreover,
\begin{align*}
a_\lambda(v,v)| &= a(v,v) + \lambda|(1+y)^{1/2}v|_H^2
\\
&\geq \frac{1}{2}C_2\|v\|_V^2 - C_2|(1+y)^{1/2}v|_H^2 + \lambda|(1+y)^{1/2}v|_H^2 \quad\hbox{(by \eqref{eq:HestonBilinearFormGarding})}
\\
&\geq \nu_1\|v\|_V^2,
\end{align*}
where we choose
\begin{equation}
\label{eq:DefnNu1Lambda0}
\nu_1 := \frac{1}{2}C_2 \quad\hbox{and}\quad\lambda \geq \lambda_0 := C_2,
\end{equation}
and note that $C_2$ only depends on the constant coefficients of $A$ and so the same is true for $\nu_1$ and $\lambda_0$.
\end{proof}

The following assumption will be in effect for the remainder of this article.

\begin{assump}[Coercive Heston bilinear form]
\label{assump:CoerciveHeston}
In Definition \ref{defn:CoerciveHestonBilinearForm} we choose $\lambda$ to be the constant $\lambda_0$ given by Lemma \ref{lem:LargeEnoughLambdaToEnsureCoercivity} so that inequality \eqref{eq:CoerciveHeston} holds.
\end{assump}

We then have the following analogue of \cite[Theorem 2.5.1]{Bensoussan_Lions}.

\begin{thm}[Existence and uniqueness of solutions to the coercive variational equation]
\label{thm:ExistenceUniquenessEllipticCoerciveHeston}
For all $f\in H$, there exists a unique $u\in V$ which solves
\begin{equation}
\label{eq:VariationalEqualityCoercive}
a_\lambda(u,v) = (f,v)_H, \quad\forall v\in V.
\end{equation}
\end{thm}

\begin{proof}
Existence and uniqueness follows from our energy estimates (Lemma \ref{lem:LargeEnoughLambdaToEnsureCoercivity}) for $a_\lambda(u,v)$ and the Lax-Milgram Theorem \ref{thm:LaxMilgram}.
\end{proof}

\begin{cor}[A priori estimate for solutions to the variational equation]
\label{cor:ExistenceUniquenessEllipticCoerciveHestonApriori}
Let $f\in H$. If $u\in V$ is a solution to \eqref{eq:VariationalEqualityCoercive}, then
\begin{equation}
\label{eq:VariationalEqualityCoerciveBoundfH}
\|u\|_V \leq \nu_1^{-1}|f|_H,
\end{equation}
where $\nu_1$ is the constant in \eqref{eq:CoerciveHeston}.
\end{cor}

\begin{proof}
The inequality \eqref{eq:VariationalEqualityCoerciveBoundfH} follows from \eqref{eq:LaxMilgramAPrioriEstimatefH}.
\end{proof}

The following comparison principle is an analogue of the weak maximum principle \cite[Theorems 3.3 \& 8.1]{GT}:

\begin{cor}[A priori comparison principle for solutions to the coercive variational equation]
\label{cor:ExistenceUniquenessEllipticCoerciveHeston}
Let $f_1,f_2\in H$. If $u_1,u_2\in V$ are solutions to
\eqref{eq:VariationalEqualityCoercive}, with $f$ replaced by $f_1,f_2$,
respectively, then $f_2\geq f_1\implies u_2\geq u_1$ a.e. on $\sO$.
\end{cor}

\begin{proof}
Suppose $f_2\geq f_1$. Since $a_\lambda(u_i,v) = (f_i,v)_H, \forall v\in V$, for $i=1,2$ by \eqref{eq:VariationalEqualityCoercive}, we have
$$
a_\lambda(u_2-u_1, v) + (f_2-f_1, v)_H = 0, \quad\forall v\in V.
$$
Taking $v=(u_2-u_1)^-$ in the preceding equation and noting that $v\in V$ by Lemma \ref{lem:SobolevSpaceClosedUnderMaxPart} and $a(v^+,v^-)=0$ for all $v\in V$, we must have
$$
a_\lambda((u_2-u_1)^-,(u_2-u_1)^-) + (f,(u_2-u_1)^-)_H = 0,
$$
so that, by \eqref{eq:CoerciveHeston} and the fact that $f_2-f_1\geq 0$ a.e. on
$\sO$,
$$
\nu_1\|(u_2-u_1)^-\|_V^2 \leq a_\lambda((u_2-u_1)^-,(u_2-u_1)^-) = -(f_2-f_1,(u_2-u_1)^-)_H \leq 0.
$$
Thus, $(u_2-u_1)^-=0$ a.e. on $\sO$ and hence $u_2-u_1\geq 0$ a.e. on $\sO$.
\end{proof}

\begin{rmk}[Non-negative solutions]
\label{rmk:ExistenceUniquenessEllipticCoerciveHeston}
We can take $f=f_2$, $u_2=u$ and $f_1=0$, $u_1=0$ in Corollary
\ref{cor:ExistenceUniquenessEllipticCoerciveHeston} to give $f\geq 0\implies
u\geq 0$ a.e. on $\sO$.
\end{rmk}

We may refine Corollary \ref{cor:ExistenceUniquenessEllipticCoerciveHeston} with the aid of

\begin{hyp}[Conditions on envelope functions]
\label{hyp:UpperLowerBoundsSolutionsCoercive}
There are $M, m \in H^2(\sO,\fw)$ such that
\begin{gather}
\label{eq:SourceFunctionTraceBounds}
m \leq 0 \leq M \quad\hbox{ on }\Gamma_1,
\\
\label{eq:mMIneqOnDomain}
m \leq M \quad\hbox{on }\sO,
\\
\label{eq:SourceFunctionBounds}
Am \leq AM \quad\hbox{a.e. on }\sO.
\end{gather}
\end{hyp}

Since $A_\lambda = A + \lambda(1+y)$ by \eqref{eq:CoerciveHestonOperator} and $m \leq M$ on $\sO$ by \eqref{eq:mMIneqOnDomain}, then \eqref{eq:SourceFunctionBounds} implies that
\begin{equation}
\label{eq:AlambdamMIneqOnDomain}
A_\lambda m \leq A_\lambda M \quad\hbox{a.e. on }\sO.
\end{equation}
We then obtain:

\begin{lem}[Refined a priori comparison principle for the coercive variational equation]
\label{lem:ComparisionPrincipleCoerciveHestonVarEquality}
Let $M, m \in H^2(\sO,\fw)$ obey \eqref{eq:SourceFunctionTraceBounds}, \eqref{eq:mMIneqOnDomain}, and \eqref{eq:SourceFunctionBounds}. Suppose $f\in L^2(\sO,\fw)$ and that $f$ obeys
\begin{equation}
\label{eq:fBoundsCoercive}
A_\lambda m \leq f \leq A_\lambda M \quad\hbox{a.e. on }\sO.
\end{equation}
If $u\in H^1_0(\sO\cup\Gamma_0)$ is a solution to \eqref{eq:VariationalEqualityCoercive}, then $u$ obeys
$$
m \leq u \leq M \quad\hbox{a.e on }\sO.
$$
\end{lem}

\begin{proof}
Take $1/\eps = c$ in the definition \eqref{eq:PenalizationOperator} of $\beta_\eps$ and, setting $c=0$ and thus $\beta_\eps=0$, the conclusion in Lemma \ref{lem:ComparisionPrincipleCoerciveHestonVarEquality} follows from the proof of Lemma \ref{lem:ComparisionPrinciplePenalizedHeston}.
\end{proof}

We need to demonstrate that the hypotheses for the comparison results in Lemma \ref{lem:ComparisionPrincipleCoerciveHestonVarEquality} and elsewhere in this article are not vacuous and that, for suitable functions $f$, there exist functions $M, m \in H^2(\sO,\fw)$ obeying Hypothesis \ref{hyp:UpperLowerBoundsSolutionsCoercive} and such that \eqref{eq:fBoundsCoercive} holds.

\begin{lem}[Upper and lower pointwise envelopes for source functions]
\label{lem:PointwiseBoundForSolution}
Suppose $N, n \in C^\infty(\HH)$ obey
\begin{equation}
\label{eq:RawPointwiseBoundForSourceFunction}
n \leq N \quad\hbox{a.e. on } \sO,
\end{equation}
where
\begin{equation}
\label{eq:UpperLowerSourceFunctionBounds}
\begin{aligned}
n(x,y) &:= c_0 + c_2y + c_3(1+y)e^{\ell x} + c_4(1+y)e^{ky},
\\
N(x,y) &:= C_0 + C_2y + C_3(1+y)e^{Lx} + C_4(1+y)e^{Ky}, \quad (x,y) \in \HH,
\end{aligned}
\end{equation}
for constants $c_i, C_i \in \RR, i=0,\ldots,4$
and positive constants $k, K, \ell, L$ obeying
\begin{gather}
\label{eq:kleqKellleqL}
k \leq K, \quad \ell \leq L,
\\
\label{eq:cileqCi}
c_i \leq C_i, \quad i=0,\ldots,4,
\\
\label{eq:ExponentialIntegrabilityConstraints}
2k < \mu, \quad 2K < \mu, \quad 2\ell < \gamma, \quad 2L < \gamma.
\end{gather}
In addition, require that the $c_i, k, \ell$ obey
\begin{enumerate}
\item If $c_0\neq 0$, then $r>0$;
\item If $c_2\neq 0$, then $\min\{\kappa,r\}>0$;
\item If $c_3\neq 0$, then $r > \ell(r-q)^+$ and $0<\ell<1$;
\item If $c_4\neq 0$, then $0<k<\min\{2\kappa, r/\kappa\theta\}$;
\end{enumerate}
and similarly for the $C_i, K, L$. Choose constants $d_i\in \RR, i=0,\ldots,4$, depending only on $c_i, k, \ell$ and the constant coefficients of $A$, as in \eqref{eq:d0d1d2ratios}, \eqref{eq:d3c3ratio}, and \eqref{eq:d4c4ratio}; choose constants $D_i\in \RR, i=0,\ldots,4$, depending only on $C_i, K, L$ and the constant coefficients of $A$, as in \eqref{eq:D0D1D2ratios}, \eqref{eq:D3C3ratio}, and \eqref{eq:D4C4ratio}; and require that
\begin{equation}
\label{eq:dileqDi}
d_i \leq D_i, \quad i=0,\ldots,4.
\end{equation}
If we define
\begin{equation}
\label{eq:UpperLowerSolutionBounds}
\begin{aligned}
m(x,y) &:= d_0 + d_2y + d_3e^{\ell x} + d_4e^{k y}, \quad (x,y)\in\HH,
\\
M(x,y) &:= D_0 + D_2y + D_3e^{Lx} + D_4e^{Ky}, \quad (x,y)\in\HH,
\end{aligned}
\end{equation}
then $M, m \in H^2(\HH,\fw)$ and $M, m$ obey
\begin{gather}
\label{eq:mleqMonUpperHalfSpace}
m \leq M \quad\hbox{on } \HH,
\\
\label{eq:NleqAMonUpperHalfSpace}
Am \leq n \quad\hbox{and}\quad N \leq AM \quad\hbox{on } \HH.
\end{gather}
If $c_i\leq 0 \leq C_i, i=0,\ldots,4$, then $n\leq 0\leq N$ on $\HH$ and \eqref{eq:mleqMonUpperHalfSpace} may be strengthened to
\begin{equation}
\label{eq:mleqZeroleqMonUpperHalfSpace}
m \leq 0 \leq M \quad\hbox{on } \HH.
\end{equation}
\end{lem}

\begin{rmk}
\label{rmk:ExponentialinxNotabsx}
The bounds in Lemma \ref{lem:PointwiseBoundForSolution} are expressed in terms of $e^{\ell x}, e^{Lx}$ and \emph{not} $e^{\ell |x|}, e^{L|x|}$.
\end{rmk}

Lemma \ref{lem:PointwiseBoundForSolution} immediately yields

\begin{cor}[Upper and lower pointwise envelopes for source functions]
\label{cor:PointwiseBoundForSolution}
Suppose $N, n \in H^2(\HH,\fw)$ are defined as in \eqref{eq:UpperLowerSourceFunctionBounds}. If a function $f\in L^2(\sO,\fw)$ obeys
\begin{equation}
\label{eq:PointwiseBoundForSourceFunction}
n \leq f \leq N \quad\hbox{a.e. on } \sO,
\end{equation}
and $M, m \in H^2(\HH,\fw)$ are defined as in \eqref{eq:UpperLowerSolutionBounds}, then $f$ obeys
\begin{equation}
\label{eq:MatchingPointwiseBoundForSourceFunction}
Am \leq f \leq AM \quad\hbox{a.e. on } \sO.
\end{equation}
\end{cor}

\begin{proof}
The inequalities \eqref{eq:MatchingPointwiseBoundForSourceFunction} follow from \eqref{eq:RawPointwiseBoundForSourceFunction} and \eqref{eq:NleqAMonUpperHalfSpace}.
\end{proof}

\begin{exmp}[Affine upper and lower bounds for $f$ and $u$]
Suppose that there exist non-negative constants $k_i,m_i, i=0,1$ such that
$$
-rk_0 - k_1(r+\kappa)y \leq f(x,y) \leq rm_0 + m_1(r+\kappa)y \quad\hbox{a.e. } (x,y)\in \sO,
$$
Then Lemma \ref{lem:PointwiseBoundForSolution} (see \eqref{eq:d0d1d2ratios} and \eqref{eq:D0D1D2ratios}) and Theorem \ref{thm:ExistenceUniquenessEllipticHeston_Improved} imply that
$$
-(k_0+k_1\kappa\theta/r) - k_1y \leq u(x,y) \leq (m_0 + m_1\kappa\theta/r) + m_1y \quad\hbox{a.e. } (x,y) \in \sO.
$$
This observation is often useful in applications. \qed
\end{exmp}

\begin{proof}[Proof of Lemma \ref{lem:PointwiseBoundForSolution}]
From the Definition \ref{defn:H2WeightedSobolevSpaces} of $H^2(\sO,\fw)$ and the definition \eqref{eq:HestonWeight} of $\fw$, we see that $d_0+d_2y \in H^2(\sO,\fw)$ while $e^{Lx} \in H^2(\sO,\fw)$ when $2L < \gamma$, and, recalling from \eqref{eq:DefnBetaMu} that $\mu = 2\kappa/\sigma^2$, we have $e^{Ky} \in H^2(\sO,\fw)$ when $2K < 2\kappa/\sigma^2$; similarly, for the terms comprising $m$. Hence $M, m\in H^2(\sO,\fw)$.

From \eqref{eq:OperatorHestonIntro}, we recall that
$$
Au = -\frac{y}{2}\left(u_{xx} + 2\rho\sigma u_{xy} + \sigma^2 u_{yy}\right) - (r-q-y/2)u_x - \kappa(\theta-y)u_y + ru, \quad u \in C^\infty(\HH).
$$
If $m(x,y) = d_0+d_2y$, then
$$
Am(x,y) = - \kappa(\theta-y)d_2 + r(d_0 + d_2y),
$$
and so
$$
Am(x,y) = (rd_0 - d_2\kappa\theta) + (\kappa+r)d_2y.
$$
Setting $Am(x,y) = c_0 + c_2y$, we obtain
\begin{align*}
rd_0 - d_2\kappa\theta &= c_0,
\\
(\kappa+r)d_2 &= c_2,
\end{align*}
and thus,
\begin{equation}
\label{eq:d0d1d2ratios}
\begin{aligned}
d_2 &:= \frac{c_2}{\kappa+r},
\\
d_0 &:= \frac{1}{r}\left(c_0 + \frac{c_2\kappa\theta}{\kappa+r}\right), \quad\hbox{if } r>0.
\end{aligned}
\end{equation}
Similarly, if $M(x,y) = D_0 + D_2y$ and setting $AM(x,y) = C_0 + C_2y$, we obtain, if $r > 0$,
\begin{equation}
\label{eq:D0D1D2ratios}
\begin{aligned}
D_2 &:= \frac{C_2}{\kappa+r},
\\
D_0 &:= \frac{1}{r}\left(C_0 + \frac{C_2\kappa\theta}{\kappa+r}\right), \quad\hbox{if } r>0.
\end{aligned}
\end{equation}
This completes the derivation of the affine bounds.

If $M(x,y) = D_3e^{Lx}, x\in\RR$ then
\begin{align*}
AM(x,y) &= -\frac{y}{2}D_3L^2e^{Lx} - D_3(r-q-y/2)Le^{Lx} + D_3re^{Lx}
\\
&= D_3(r - L(r-q))e^{Lx} + \frac{y}{2}D_3L(1-L)e^{Lx}.
\end{align*}
If $r > L(r-q)$ and $0<L<1$, then $AM(x,y) \geq C_3(1+y)e^{Lx}, \forall(x,y)\in\HH$, provided
$$
D_3 \geq \frac{C_3}{r - L(r-q)} \hbox{ and } D_3 \geq \frac{2C_3}{L(1-L)}
$$
and thus we may choose
\begin{equation}
\label{eq:D3C3ratio}
D_3 := \max\left\{\frac{C_3}{r - L(r-q)}, \frac{2C_3}{L(1-L)}\right\}.
\end{equation}
This yields the upper bound in $e^{Lx}$. Similarly, $Am(x,y) \leq c_3(1+y)e^{\ell x}, \forall(x,y)\in\HH$, provided
\begin{equation}
\label{eq:d3c3ratio}
d_3 := \min\left\{\frac{c_3}{r - \ell(r-q)}, \frac{2c_3}{\ell(1-\ell)}\right\},
\end{equation}
where $r > \ell(r-q)$ and $0<\ell<1$.

If $M(x,y) = D_4e^{Ky}$ then
\begin{align*}
AM &= -\frac{y}{2}D_4K^2e^{Ky} - D_4\kappa(\theta-y)Ke^{Ky} + D_4re^{Ky}
\\
&= D_4(r - \kappa\theta K)e^{Ky} + yD_4K(\kappa-K/2)e^{Ky}.
\end{align*}
If $0<K<\min\{2\kappa, r/\kappa\theta\}$, then $AM(x,y) \geq C_4(1+y)e^{Ky}, \forall(x,y)\in\HH$, provided
$$
D_4 \geq \frac{C_4}{r - \kappa\theta K} \hbox{ and } D_4 \geq \frac{C_4}{K(\kappa-K/2)},
$$
and thus we may choose
\begin{equation}
\label{eq:D4C4ratio}
D_4 := \max\left\{\frac{C_4}{r - \kappa\theta K}, \frac{C_4}{K(\kappa-K/2)}\right\}.
\end{equation}
This yields the upper bound in $e^{Ky}$. Similarly, $Am(x,y) \leq c_4(1+y)e^{ky}, \forall(x,y)\in\HH$, provided
\begin{equation}
\label{eq:d4c4ratio}
d_4 := \min\left\{\frac{c_4}{r - \kappa\theta k}, \frac{c_4}{k(\kappa-k/2)}\right\},
\end{equation}
where $0<k<\min\{2\kappa, r/\kappa\theta\}$. This completes the derivation of the exponential bounds.

By adding the preceding inequalities, we see that $Am \leq n$ and $N \leq AM$ on $\HH$. The conditions \eqref{eq:kleqKellleqL} and \eqref{eq:cileqCi} ensure that $M, m$ obey \eqref{eq:mleqMonUpperHalfSpace}. If in addition, $c_i\leq 0 \leq C_i, i=0,\ldots,4$, then $n\leq 0\leq N$ on $\HH$, $d_i\leq 0 \leq D_i, i=0,\ldots,4$, and $M, m$ obey \eqref{eq:mleqZeroleqMonUpperHalfSpace}. This completes the proof of the lemma.
\end{proof}

\subsection{Existence and uniqueness of solutions to the non-coercive variational equation}
\label{subsec:NoncoerciveVarInequality}
Because the bilinear form \eqref{eq:HestonWithKillingBilinearForm} is not necessarily coercive, we shall require that the associated operator $A$ obeys the weaker ``non-coercive'' condition (compare \cite[Equation (3.1.6)]{Bensoussan_Lions}) in order to establish existence and uniqueness of solutions to non-coercive variational equalities and inequalities.

\begin{hyp}[Non-coercive condition]
\label{hyp:NoncoerciveHeston}
The coefficient $r$ in the definition \eqref{eq:OperatorHestonIntro} of $A$ obeys
\begin{equation}
\label{eq:NoncoerciveHeston}
r > 0.
\end{equation}
\end{hyp}

We need additional conditions to ensure uniqueness of a solution to the non-coercive variational equation.

\begin{hyp}[Auxiliary condition for uniqueness]
\label{hyp:AuxBoundUniquenessSolutionsNoncoerciveEquation}
There exists a $\varphi \in H^2(\sO,\fw)$ obeying
\begin{gather}
\label{eq:ShiftFunctionBound}
A\varphi \geq 0 \quad \hbox{ a.e. on }\sO,
\\
\label{eq:PositiveShiftedSourceFunction}
A(m+\varphi) > 0 \quad\hbox{ a.e. on }\sO,
\\
\label{eq:varphiIneqOnDomain}
\varphi \geq 0 \quad\hbox{on }\sO,
\\
\label{eq:OneplusyVarphiInL2}
(1+y)\varphi \in L^2(\sO,\fw),
\\
\label{eq:Sqrt1plusyAvarphiL2}
(1+y)^{1/2}A\varphi \in L^2(\sO,\fw),
\\
\label{eq:SupSourceSolutionRatioBoundRaw}
\esssup_{(x,y)\in\sO}\frac{(1+y)(M+\varphi)(x,y)}{A(m+\varphi)(x,y)} < \infty,
\end{gather}
where the functions $M, m \in H^2(\sO,\fw)$ are as in Hypothesis \ref{hyp:UpperLowerBoundsSolutionsCoercive}
\end{hyp}

Observe that if $\varphi$ is as in in Hypothesis \ref{hyp:AuxBoundUniquenessSolutionsNoncoerciveEquation}, then $\varphi \in C_{\textrm{loc}}(\sO\cup\Gamma_1)$ by Lemma \ref{lem:H2SobolevEmbedding} and thus \eqref{eq:varphiIneqOnDomain} yields
\begin{equation}
\label{eq:ShiftFunctionTraceBound}
\varphi \geq 0 \quad\hbox{on }\Gamma_1.
\end{equation}

By analogy with \cite[Theorem 2.5.2]{Bensoussan_Lions}, we have

\begin{thm}[Existence and uniqueness for solutions to the non-coercive variational equation]
\label{thm:ExistenceUniquenessEllipticHeston_Improved}
Assume Hypothesis \ref{hyp:NoncoerciveHeston} holds. Suppose there are functions $M, m \in H^2(\sO,\fw)$ obeying \eqref{eq:SourceFunctionTraceBounds}, \eqref{eq:mMIneqOnDomain}, and \eqref{eq:SourceFunctionBounds}. If $f\in H$ obeys
\begin{equation}
\label{eq:fBounds}
Am \leq f \leq AM \quad\hbox{a.e. on }\sO,
\end{equation}
then there exists a solution $u\in V$ to Problem \ref{prob:HestonWeakMixedBVPHomogeneous} and $u$ obeys
\begin{equation}
\label{eq:uBounds}
m \leq u \leq M \quad\hbox{a.e. on }\sO.
\end{equation}
Moreover, if there is a function $\varphi \in H^2(\sO,\fw)$ obeying Hypothesis \ref{hyp:AuxBoundUniquenessSolutionsNoncoerciveEquation} and
the domain, $\sO$, obeys Hypothesis \ref{hyp:DomainCombinedCondition}, then the solution $u$ is unique.
\end{thm}

\begin{rmk}[Sufficient conditions for existence and uniqueness in Theorem \ref{thm:ExistenceUniquenessEllipticHeston_Improved}]
\label{rmk:ExistenceUniquenessEllipticHestonSufficientConditions}
Lemma \ref{lem:PointwiseBoundForSolution} may be used to provide non-trivial examples of $M, m \in H^2(\sO,\fw)$ such that \eqref{eq:SourceFunctionTraceBounds}, \eqref{eq:mMIneqOnDomain}, and \eqref{eq:SourceFunctionBounds} hold and non-trivial examples of $f \in L^2(\sO,\fw)$ such that \eqref{eq:fBounds} holds; see the proof of Lemma \ref{lem:VIUniquenessEllipticHestonPostivefuSufficientConditions} for details. Lemma \ref{lem:VIUniquenessEllipticHestonPostivefuSufficientConditions} may be used to provide non-trivial examples of $\varphi \in H^2(\sO,\fw)$ obeying Hypothesis \ref{hyp:AuxBoundUniquenessSolutionsNoncoerciveEquation}.
\end{rmk}

\begin{rmk}[Point-wise bounds obeyed by the unique solution]
The pointwise bounds \eqref{eq:uBounds} for the solution are a posteriori bounds because they are consequence of the proof of existence for Theorem \ref{thm:ExistenceUniquenessEllipticHeston_Improved} rather than a priori bounds satisfied by any solution to Problem \ref{prob:HestonWeakMixedBVPHomogeneous}.
\end{rmk}

\begin{cor}[A posteriori comparison principle for solutions to the variational equation]
\label{cor:ExistenceUniquenessEllipticHeston}
Assume the hypotheses of Theorem
\ref{thm:ExistenceUniquenessEllipticHeston_Improved}. If $f_1,f_2\in H$ obey
\eqref{eq:fBounds} and $u_1,u_2\in V$ are the unique solutions to Problem
\ref{prob:HestonWeakMixedBVPHomogeneous} with $f$ replaced by $f_1,f_2$,
respectively, then $f_2\geq f_1\implies u_2\geq u_1$ a.e. on $\sO$.
\end{cor}

\begin{proof}
We see that $f := f_2-f_1 \in H$ obeys
$$
0 \leq f \leq A(M-m) \quad\hbox{a.e. on }\sO,
$$
and $u := u_2-u_1$ solves $a(u,v)=(f,v)_H$, $\forall v\in V$. Since $M-m \geq 0
\hbox{ on }\Gamma_1$, Theorem
\ref{thm:ExistenceUniquenessEllipticHeston_Improved} (with $m$ replaced by $0$
and $M$ replaced by $M-m$) implies that $u$ is the unique solution to
$a(u,v)=(f,v)_H$, $\forall v\in V$, and thus obeys
$$
0 \leq u \leq M-m \quad\hbox{a.e. on }\sO,
$$
and hence $u_2\geq u_1$ a.e. on $\sO$.
\end{proof}

\begin{lem}[A priori estimate for solutions to the variational equation]
\label{lem:ExistenceUniquenessEllipticHestonAprioriEstimate}
If $f\in H$ and $u\in V$ is a solution to Problem \ref{prob:HestonWeakMixedBVPHomogeneous}, then
\begin{equation}
\label{eq:VariationalEqualityHestonBoundH}
\|u\|_V \leq C\left(|f|_H + |(1+y)u|_H\right),
\end{equation}
where $C=\nu_1^{-1}+\lambda_0$ and $\nu_1, \lambda$ are the constants in Lemma \ref{lem:LargeEnoughLambdaToEnsureCoercivity}.
\end{lem}

\begin{proof}
Since $u$ solves \eqref{eq:IntroHestonWeakMixedProblemHomogeneous}, then
$$
a_{\lambda}(u,v) = a(u,v) + \lambda((1+y)u,v)_H = (f+\lambda(1+y)u,v)_H, \quad\forall v\in V,
$$
and \eqref{eq:VariationalEqualityHestonBoundH} follows from \eqref{eq:VariationalEqualityCoerciveBoundfH}.
\end{proof}

\begin{proof}[Proof of existence in Theorem \ref{thm:ExistenceUniquenessEllipticHeston_Improved}]
We adapt the proof of existence of \cite[Theorem 2.5.2]{Bensoussan_Lions}. We shall construct a solution $u$ as the limit, in a suitable sense, of a sequence $\{u_n\}_{n\geq 0}$. Set $u_0=0$ and use Theorem \ref{thm:ExistenceUniquenessEllipticCoerciveHeston} to define a sequence $\{u_n\}_{n\geq 0}$ by
\begin{equation}
\label{eq:IncreasingSolutionSequenceConstruction}
a(u_n,v) + \lambda((1+y)u_n,v)_H = (f+\lambda(1+y)u_{n-1},v)_H, \quad \forall v\in V, n\geq 1.
\end{equation}
Setting $u_0=0$ and $n=1$ in \eqref{eq:IncreasingSolutionSequenceConstruction} implies that $u_1$ obeys
$$
a(u_1,v) + \lambda((1+y)u_1,v)_H = a_\lambda(u_1,v) = (f,v)_H, \quad \forall v\in V.
$$
Theorem \ref{thm:ExistenceUniquenessEllipticCoerciveHeston} and Lemma \ref{lem:ComparisionPrincipleCoerciveHestonVarEquality} imply that there exists a solution $u_1\in V$ obeying
$$
m \leq u_1 \leq M \quad\hbox{a.e. on }\sO.
$$
We shall now show that
\begin{equation}
\label{eq:IncreasingSolutionSequence}
m \leq u_1 \leq \cdots \leq u_n \leq \cdots \leq  M \quad\hbox{a.e. on } \sO.
\end{equation}
We suppose $u_{n-1}\geq u_{n-2}$ a.e on $\sO$ and show that $u_n\geq u_{n-1}$ a.e. on $\sO$. By taking the difference of the equations \eqref{eq:IncreasingSolutionSequenceConstruction} defining $u_n$ and $u_{n-1}$, we obtain
$$
a_\lambda(u_n-u_{n-1},v) = \lambda((1+y)(u_{n-1}-u_{n-2}),v)_H, \quad n\geq 2.
$$
We then take $v=(u_n-u_{n-1})^-$ in the preceding identity to give
\begin{align*}
{}&-a_\lambda((u_n-u_{n-1})^-,(u_n-u_{n-1})^-)
\\
&\quad = \lambda((1+y)(u_{n-1}-u_{n-2}),(u_n-u_{n-1})^-)_H
\\
&\quad\geq 0,
\end{align*}
so that
$$
\nu_1\|(u_n-u_{n-1})^-\|_V^2 \leq a_\lambda((u_n-u_{n-1})^-,(u_n-u_{n-1})^-) \leq 0.
$$
Hence, $(u_n-u_{n-1})^-=0$ a.e. on $\sO$ and thus $u_n\geq u_{n-1}$ a.e. on $\sO$. Therefore, the sequence $\{u_n\}_{n\geq 0}$ is increasing, as asserted in \eqref{eq:IncreasingSolutionSequence}.

Next we show that
\begin{equation}
\label{eq:UpperBoundSolutionSequence}
u_n \leq M \quad\hbox{a.e. on } \sO, \quad \forall n\geq 0.
\end{equation}
Since $M \in H^2(\sO,\fw)$, we have by Lemma \eqref{lem:HestonIntegrationByParts},
\begin{equation}
\label{eq:BilinearFormEvaluateOnLineary}
a(M,v) = (AM,v)_H.
\end{equation}
Since $M \geq 0$ on $\Gamma_1$ by the hypothesis \eqref{eq:SourceFunctionTraceBounds} and $u_n\in V = H^1(\sO\cup\Gamma_0$, so $u_n=0$ on $\Gamma_1$ (trace sense), then $(u_n-M)^+ = 0$ on $\Gamma_1$ (trace sense) and $(u_n-M)^+ \in V$ by the method of proof of Lemma \ref{lem:SobolevSpaceClosedUnderMaxPart}. We take $v=(u_n-M)^+$ in the  equation \eqref{eq:IncreasingSolutionSequenceConstruction} defining $u_n$ to give
\begin{align*}
{}&a(u_n-M, (u_n-M)^+) + a(M,(u_n-M)^+)
\\
{}&\quad + \lambda((1+y)(u_n-u_{n-1}), (u_n-M)^+)_H
\\
&= (f, (u_n-M)^+)_H,
\end{align*}
or
\begin{align*}
{}& a_\lambda((u_n-M)^+, (u_n-M)^+)_H + a(M,(u_n-M)^+)
\\
&\quad - \lambda((1+y)(u_n-M), (u_n-M)^+)_H + \lambda((1+y)(u_n-u_{n-1}), (u_n-M)^+)_H
\\
&=  (f, (u_n-M)^+)_H.
\end{align*}
Simplifying the preceding identity, gives
\begin{align*}
{}&a_\lambda((u_n-M)^+, (u_n-M)^+) + (AM -f ,(u_n-M)^+)_H
\\
{}&\quad + \lambda((1+y)(M-u_{n-1}), (u_n-M)^+)_H = 0.
\end{align*}
But $AM - f\geq 0$ a.e. on $\sO$ by \eqref{eq:fBounds} while $M-u_{n-1}\geq 0$ a.e on $\sO$ by the induction hypothesis, so
$$
a_\lambda((u_n-M)^+, (u_n-M)^+) \leq 0,
$$
and therefore $(u_n-M)^+ = 0$ a.e. on $\sO$, since
$$
\nu_1\|(u_n-M)^+\|_V^2 \leq a_\lambda((u_n-M)^+, (u_n-M)^+)
$$
by \eqref{eq:CoerciveHeston}. This proves the upper bound in \eqref{eq:IncreasingSolutionSequence}.

We deduce from \eqref{eq:IncreasingSolutionSequence} that there is a Borel measurable function $u:\sO\to\RR$ defined by
$$
u := \lim_{n\to\infty}u_n \quad\hbox{a.e. on }\sO,
$$
and
$$
u_n \nearrow u \hbox{ monotonically a.e.  on } \sO.
$$
In particular, $u$ obeys \eqref{eq:uBounds}.

The inequalities \eqref{eq:CoerciveHeston} and \eqref{eq:IncreasingSolutionSequenceConstruction} (with $v=u_n$) give
$$
\nu_1\|u_n\|_V^2 \leq \lambda((1+y)u_{n-1},u_n)_H + (f,u_n)_H, \quad\forall n\geq 0.
$$
Therefore,
\begin{equation}
\label{eq:CoerciveUnBoundPrelim}
\begin{aligned}
\nu_1\|u_n\|_V^2 &\leq \lambda\|(1+y)^{1/2}u_{n-1}\|_{L^2(\sO,\fw)}\|(1+y)^{1/2}u_n\|_{L^2(\sO,\fw)}
\\
&\quad + \|f\|_{L^2(\sO,\fw)}\|u_n\|_{L^2(\sO,\fw)}, \quad\forall n\geq 0.
\end{aligned}
\end{equation}
Since $M, m\in H^2(\sO,\fw)$ by hypothesis, we have $\|(1+y)^{1/2}\max\{|M|,|m|\}\|_{L^2(\sO,\fw)}< \infty$ and \eqref{eq:IncreasingSolutionSequence} gives
$$
\|(1+y)^{1/2}u_n\|_{L^2(\sO,\fw)} \leq \|(1+y)^{1/2}\max\{|M|,|m|\}\|_{L^2(\sO,\fw)} < \infty, \quad\forall n\geq 0.
$$
Combining the preceding estimate, the $L^2(\sO,\fw)$ bounds for $u_n$ implied by \eqref{eq:IncreasingSolutionSequence}, and \eqref{eq:CoerciveUnBoundPrelim} yields
\begin{equation}
\label{eq:SolutionSequenceVBounded}
\|u_n\|_V \leq C\left(\|(1+y)^{1/2}\max\{|M|,|m|\}\|_{L^2(\sO,\fw)} + \|f\|_{L^2(\sO,\fw)}\right) < \infty, \quad\forall n\geq 0.
\end{equation}
for some positive constant $C$ depending only on the constant coefficients of $A$. Given \eqref{eq:SolutionSequenceVBounded}, we may pass to a subsequence and obtain
$$
u_n \rightharpoonup u \hbox{ weakly in } V.
$$
We can therefore take limits in \eqref{eq:IncreasingSolutionSequenceConstruction} to conclude that $u$ is a solution to \eqref{eq:IntroHestonWeakMixedProblemHomogeneous}.
\end{proof}

Before we turn to the proof of uniqueness in Theorem \ref{thm:ExistenceUniquenessEllipticHeston_Improved}, we require some auxiliary lemmas.

\begin{lem}[Existence of an auxiliary function for the proof of uniqueness in Theorem \ref{thm:ExistenceUniquenessEllipticHeston_Improved}]
\label{lem:EllipticHestonUniquenessAuxiliaryFunction}
Assume Hypothesis \ref{hyp:NoncoerciveHeston} holds and require that the domain, $\sO$, obeys Hypothesis \ref{hyp:DomainCombinedCondition}. Let $\varphi$ be as in Hypothesis \ref{hyp:AuxBoundUniquenessSolutionsNoncoerciveEquation}. Then there exists a function $u_\varphi \in H_0^1(\sO\cup\Gamma_0)$ which solves
\begin{equation}
\label{eq:Defnuvarphi}
a(u_\varphi,v) = (A\varphi,v)_H, \quad\forall v\in H_0^1(\sO\cup\Gamma_0).                                                                                                                                                                                                                        \end{equation}
Moreover, $u \in H^2(\sO,\fw)$ and $u$ obeys
\begin{align}
\label{eq:Auvarphi}
Au_\varphi &= A\varphi \quad\hbox{a.e. on }\sO,
\\
\label{eq:uDirichletGammaOne}
u_\varphi &= 0 \quad\hbox{on } \Gamma_1,
\end{align}
and
\begin{equation}
\label{eq:uvarphiBounds}
0\leq u_\varphi \leq \varphi \quad\hbox{on }\sO.
\end{equation}
\end{lem}

\begin{proof}
Existence of a function $u_\varphi \in H_0^1(\sO\cup\Gamma_0)$ solving \eqref{eq:Defnuvarphi} is provided by Theorem \ref{thm:ExistenceUniquenessEllipticHeston_Improved} (existence), while the facts that $u \in H^2(\sO,\fw)$ and solves \eqref{eq:Auvarphi} follow from Theorem \ref{thm:ExistenceUniquenessH2RegularEllipticHeston} (existence, with $M,m,f$ replaced by $\varphi,0,A\varphi$). Because of \eqref{eq:ShiftFunctionBound} and \eqref{eq:ShiftFunctionTraceBound}, we obtain \eqref{eq:uvarphiBounds} from \eqref{eq:uBounds} in Theorem \ref{thm:ExistenceUniquenessEllipticHeston_Improved} (existence) with $M, m$ replaced by $\varphi, 0$.
\end{proof}

\begin{lem}[Reduction to the case of existence when the source function is positive and the solution non-negative]
\label{lem:EllipticHestonReductionNonnegativeSourceFunction}
Assume the hypotheses of Theorem \ref{thm:ExistenceUniquenessEllipticHeston_Improved} for existence and uniqueness and let $u_\varphi \in H_0^1(\sO\cup\Gamma_0)\cap H^2(\sO,\fw)$ be as in Lemma \ref{lem:EllipticHestonUniquenessAuxiliaryFunction}. Define
\begin{equation}
\label{eq:DefineTildeMm}
\tilde m : = m+u_\varphi \quad\hbox{and}\quad \tilde M := M+u_\varphi,
\end{equation}
where $m,M$ are as in the hypotheses of Theorem \ref{thm:ExistenceUniquenessEllipticHeston_Improved}. Then $\tilde m, \tilde M \in H^2(\sO,\fw)$ obey
\begin{align}
\label{eq:tildeSourceFunctionTraceBounds}
\tilde m &\leq 0 \leq \tilde M \quad \hbox{on }\Gamma_1,
\\
\label{eq:PositiveTildeSourceFunction}
A\tilde m &> 0 \quad\hbox{ a.e. on }\sO,
\\
\label{eq:tildeSourceFunctionBounds}
A\tilde m &\leq A\tilde M \quad\hbox{ a.e. on }\sO.
\end{align}
Let $f$ be as in the hypotheses of Theorem \ref{thm:ExistenceUniquenessEllipticHeston_Improved}. Then
\begin{equation}
\label{eq:definetildef}
\tilde f := f+A\varphi = f+Au_\varphi
\end{equation}
obeys
\begin{equation}
\label{eq:tildefBounds}
A\tilde m \leq \tilde f \leq A\tilde M \quad\hbox{a.e. on }\sO.
\end{equation}
Then existence in Theorem \ref{thm:ExistenceUniquenessEllipticHeston_Improved} of a solution, $u$, to Problem \ref{prob:HestonWeakMixedBVPHomogeneous} defined by $f$ and obeying the bounds \eqref{eq:uBounds} is equivalent to existence of a solution, $\tilde u$, to Problem \ref{prob:HestonWeakMixedBVPHomogeneous} defined by $\tilde f$ and obeying
\begin{equation}
\label{eq:tildeuBounds}
\tilde m \leq \tilde u \leq \tilde M \quad \hbox{ a.e. on }\sO.
\end{equation}
\end{lem}

\begin{proof}
We first verify the assertions in the preamble to the statement of equivalence of existence of solutions. We see that \eqref{eq:tildeSourceFunctionTraceBounds} holds because of \eqref{eq:SourceFunctionTraceBounds}, \eqref{eq:DefineTildeMm}, and the fact that $u_\varphi=0$ on $\Gamma_1$. Moreover, \eqref{eq:tildeSourceFunctionBounds} follows from \eqref{eq:DefineTildeMm} and \eqref{eq:SourceFunctionBounds}, while \eqref{eq:PositiveTildeSourceFunction} follows from \eqref{eq:Auvarphi} and \eqref{eq:PositiveShiftedSourceFunction}. The inequalities \eqref{eq:tildefBounds} for $\tilde f$ are immediate from \eqref{eq:fBounds} and \eqref{eq:definetildef}.

\emph{Existence of $\tilde u$ implies existence of $u$.} By assumption, there exists a function $\tilde u \in H_0^1(\sO\cup\Gamma_0)$ obeying
\begin{equation}
\label{eq:TildeProblemHestonHomgeneous}
a(\tilde u,v) = (\tilde f,v)_H, \quad \forall v\in H_0^1(\sO\cup\Gamma_0),
\end{equation}
and  \eqref{eq:tildeuBounds}. By \eqref{eq:tildeuBounds}, we have
\begin{equation}
\label{eq:ExplicitTildeuBounds}
m+u_\varphi\leq \tilde u \leq M+u_\varphi \quad\hbox{a.e. on } \sO.
\end{equation}
Therefore, setting
\begin{equation}
\label{eq:DefnUintermsUvarphi}
u = \tilde u - u_\varphi \in H_0^1(\sO\cup\Gamma_0),
\end{equation}
we see that $u$ obeys \eqref{eq:uBounds}. Next,
\begin{align*}
a(u, v) &= a(\tilde u-u_\varphi, v)
\\
&= a(\tilde u, v) - a(u_\varphi, v)
\\
&= (\tilde f, v)_H - (A\varphi, v)_H \quad\hbox{(by \eqref{eq:TildeProblemHestonHomgeneous} and \eqref{eq:Defnuvarphi})}
\\
&= (f, v)_H, \quad \forall v\in V  \quad\hbox{(by \eqref{eq:definetildef}).}
\end{align*}
Hence, $u$ obeys \eqref{eq:IntroHestonWeakMixedProblemHomogeneous} and is a solution to Problem \ref{prob:HestonWeakMixedBVPHomogeneous} defined by the source function, $f$.

\emph{Existence of $u$ implies existence of $\tilde u$.} By assumption, there is a solution $u\in H_0^1(\sO\cup\Gamma_0)$ to Problem \ref{prob:HestonWeakMixedBVPHomogeneous} defined by the source function, $f$, and $u$ obeys \eqref{eq:uBounds}. Use \eqref{eq:DefnUintermsUvarphi} to define $\tilde u = u + u_\varphi$, so \eqref{eq:uBounds} implies that $\tilde u$ obeys \eqref{eq:ExplicitTildeuBounds} and thus $\tilde u$ obeys \eqref{eq:tildeuBounds}. Then,
\begin{align*}
a(\tilde u, v) &= a(u + u_\varphi, v)
\\
&= a(u, v) + a(u_\varphi, v)
\\
&= (f, v)_H + (A\varphi, v)_H \quad\hbox{(by \eqref{eq:IntroHestonWeakMixedProblemHomogeneous} and \eqref{eq:Defnuvarphi})}
\\
&= (\tilde f, v)_H, \quad \forall \tilde v \in H_0^1(\sO\cup\Gamma_0) \quad\hbox{(by \eqref{eq:definetildef}).}
\end{align*}
Hence, $\tilde u$ obeys \eqref{eq:TildeProblemHestonHomgeneous} and is a solution to Problem \ref{prob:HestonWeakMixedBVPHomogeneous} defined by the source function, $\tilde f$.
\end{proof}

\begin{lem}[Non-negative solutions]
\label{lem:utildeNonnegative}
Assume the hypotheses of Theorem \ref{thm:ExistenceUniquenessEllipticHeston_Improved} for existence and uniqueness. Let $\tilde u \in H_0^1(\sO\cup\Gamma_0)$ be a solution to Problem \ref{prob:HestonWeakMixedBVPHomogeneous} defined by $\tilde f$ as in \eqref{eq:definetildef}. Then $\tilde u$ obeys
\begin{equation}
\label{eq:utildeNonnegative}
\tilde u \geq 0 \quad\hbox{a.e. on }\sO.
\end{equation}
\end{lem}

\begin{proof}
Observe that \eqref{eq:PositiveTildeSourceFunction} and \eqref{eq:tildefBounds} imply that $\tilde f$ obeys $0 < \tilde f \leq A\tilde M$ a.e. on $\sO$ and hence, replacing $\tilde m$ by zero in \eqref{eq:tildeuBounds}, we obtain \eqref{eq:utildeNonnegative}.
\end{proof}

\begin{lem}[Reduction to the case of uniqueness when the source function is positive and the solution non-negative]
\label{lem:EllipticHestonReductionNonnegativeSourceFunctionUniqueness}
Assume the hypotheses of Theorem \ref{thm:ExistenceUniquenessEllipticHeston_Improved} for existence and uniqueness and let $\tilde f$ be as in \eqref{eq:definetildef}. Then uniqueness of a solution, $u$, to Problem \ref{prob:HestonWeakMixedBVPHomogeneous} defined by $f$ is equivalent to uniqueness of a solution, $\tilde u$, to Problem \ref{prob:HestonWeakMixedBVPHomogeneous} defined by $\tilde f$.
\end{lem}

\begin{proof}
Let $u_\varphi$ be as in Lemma \ref{lem:EllipticHestonUniquenessAuxiliaryFunction}. Lemma \ref{lem:EllipticHestonReductionNonnegativeSourceFunction} implies that $u_i\in H_0^1(\sO\cup\Gamma_0), i=1,2$ are two solutions to Problem \ref{prob:HestonWeakMixedBVPHomogeneous} defined by $f$ if and only if $\tilde u_i := u_i + u_\varphi\in H_0^1(\sO\cup\Gamma_0), i=1,2$ are two solutions to Problem \ref{prob:HestonWeakMixedBVPHomogeneous} defined by $\tilde f$. Therefore, $u_1 = u_2$ if and only if $\tilde u_1 = \tilde u_2$ and this yields the conclusion.
\end{proof}

The technical lemma below shows that the hypotheses of Theorem \ref{thm:ExistenceUniquenessEllipticHeston_Improved} on $f, M, m, \varphi$ are not vacuous and, furthermore, may be strengthened so that in addition \eqref{eq:SupSourceSolutionRatioBoundRaw} holds.

\begin{lem}[Auxiliary bound for uniqueness]
\label{lem:VIUniquenessEllipticHestonPostivefuSufficientConditions}
There exist non-trivial functions $f \in L^2(\sO,\fw)$ and $M,m,\varphi \in H^2(\sO,\fw)$ which obey  \eqref{eq:SourceFunctionTraceBounds}, \eqref{eq:mMIneqOnDomain}, \eqref{eq:SourceFunctionBounds}, \eqref{eq:fBounds}, and the conditions of Hypothesis \ref{hyp:AuxBoundUniquenessSolutionsNoncoerciveEquation}.
\end{lem}

\begin{proof}
We may suppose that $f\in L^2(\sO,\fw)$ obeys bounds $n \leq f \leq N$ a.e. on $\sO$, where $n, N \in L^2(\HH,\fw)$ are defined by \eqref{eq:UpperLowerSourceFunctionBounds} with constants $c_i \leq 0, C_i \geq 0$, for all $i$.

Choose $M, m$ as prescribed in Lemma \ref{lem:PointwiseBoundForSolution} to ensure that $Am \leq n, AM \geq N$ on $\HH$ and so $M,m,f$ obey the bounds \eqref{eq:SourceFunctionBounds} and \eqref{eq:fBounds}. From their definitions in Lemma \ref{lem:PointwiseBoundForSolution}, we have $D_i \geq 0, d_i \leq 0$, for all $i$, and so the bounds \eqref{eq:SourceFunctionTraceBounds} and \eqref{eq:mMIneqOnDomain} for $M, m$ hold because, a fortiori, $m \leq 0 \leq M$ on $\HH$.

Now choose $\varphi$ with non-negative coefficients, $D_i'\geq 0$, using the formula \eqref{eq:UpperLowerSolutionBounds} with $D_i$ replaced by $D_i'$ such that $A\varphi \geq 0$ and $\varphi \geq 0$ on $\HH$ and thus \eqref{eq:ShiftFunctionBound} and \eqref{eq:varphiIneqOnDomain} hold. When the coefficients, $D_i'$, are chosen as described in the statement of Lemma \ref{lem:PointwiseBoundForSolution}, then one sees by inspection that \eqref{eq:OneplusyVarphiInL2} and \eqref{eq:Sqrt1plusyAvarphiL2} hold. Moreover, $\varphi \in H^2(\HH,\fw)$ and, in addition, the coefficients $D_i'$ in the definition of $\varphi$ may be chosen large enough that
\begin{equation}
\label{eq:ChoiceOfFunctionp}
(m+\varphi)(x,y) \geq p_3e^{Lx} + p_4e^{Ky}, \quad (x,y)\in\HH,
\end{equation}
where $p_3>0, p_4>0$ are constants. The proof of Lemma \ref{lem:PointwiseBoundForSolution} shows that there is a constant $\zeta>0$, depending only on $p_3,p_4$ and the constant coefficients of $A$ such that
$$
A(m+\varphi)(x,y) \geq \zeta(1+y)(e^{Lx} + e^{Ky}) > 0, \quad (x,y) \in \HH,
$$
and thus \eqref{eq:PositiveShiftedSourceFunction} holds since, a fortiori, $A(m+\varphi) > 0$ on $\HH$. Furthermore, \eqref{eq:ChoiceOfFunctionp} and our definitions of $M, \varphi$ via \eqref{eq:UpperLowerSolutionBounds} ensure that there is a constant $C>0$ such that
$$
(M+\varphi)(x,y) \leq C(1 + e^{Lx} + e^{Ky}), \quad (x,y) \in \HH.
$$
Hence,
$$
\frac{(1+y)(M+\varphi)(x,y)}{A(m+\varphi)(x,y)} \leq \frac{2C}{\zeta}, \quad (x,y) \in \HH,
$$
and this yields \eqref{eq:SupSourceSolutionRatioBoundRaw}. This completes the proof of Lemma \ref{lem:VIUniquenessEllipticHestonPostivefuSufficientConditions}.
\end{proof}

\begin{rmk}[Additional growth conditions on the envelope functions] 
\label{rmk:Sqrt1plusyMmVarphiInLq}
Moreover, the $M, m, \varphi \in H^2(\HH,\fw)$ constructed in the proof of Lemma \ref{lem:VIUniquenessEllipticHestonPostivefuSufficientConditions} also obey $(1+y)^{1/2}M, (1+y)^{1/2}m, (1+y)^{1/2}\varphi \in L^q(\HH,\fw)$ for $q>2$ as required by \eqref{eq:Sqrt1plusyMmInLq} and \eqref{eq:Sqrt1plusyVarphiInLq}.
\end{rmk}

\begin{proof}[Proof of uniqueness in Theorem \ref{thm:ExistenceUniquenessEllipticHeston_Improved}] We adapt the proof of uniqueness in \cite[Theorem 2.5.2]{Bensoussan_Lions}. We assume the reduction embodied in Lemma \ref{lem:EllipticHestonReductionNonnegativeSourceFunction} but, to simplify notation, we shall omit the ``tildes'' and write $f,u$ for $\tilde f, \tilde u$. For clarity of exposition, we denote equations and inequalities involving Borel measurable functions as holding everywhere when they hold almost everywhere.

Suppose $u_1, u_2$ are two solutions to \eqref{eq:IntroHestonWeakMixedProblemHomogeneous}, assumed non-negative by \eqref{eq:utildeNonnegative}. If $u_1\neq u_2$ on $\sO$ we may suppose that
\begin{equation}
\label{eq:u1notlessthanu2}
u_1 \not\leq u_2 \quad\hbox{on }\sO.
\end{equation}
Otherwise, we could interchange the roles of $u_1$ and $u_2$ in the remainder of the proof of uniqueness. We introduce $\alpha_0 \in L^2(\sO,\fw)$, where $0\leq \alpha_0\leq 1$ on $\sO$, such that
\begin{equation}
\label{eq:VIalphaSoln1LessSoln2}
\alpha_0 u_1\leq u_2 \quad\hbox{on }\sO,
\end{equation}
by setting $\alpha_0(x,y) := 1$ if $u_1(x,y) \leq u_2(x,y)$ and $\alpha_0(x,y) := u_2(x,y)/u_1(x,y)$ if $u_1(x,y) \geq u_2(x,y)$, for each $(x,y)\in\sO$, so that
\begin{equation}
\label{eq:DefnAlpha0}
\alpha_0(x,y) := \min\left\{1, \frac{u_2(x,y)}{u_1(x,y)}\right\}, \quad (x,y) \in \sO,
\end{equation}
where we define $0/0 := 0$ and $1/0 := +\infty$. Note that $\alpha_0 u_1  = \min\{u_1,u_2\}$ and $\min\{u_1,u_2\} \in H^1_0(\sO\cup\Gamma_0,\fw)$ by Lemma \ref{lem:SobolevSpaceClosedUnderMaxPart}. Since $u_1\neq u_2$ on $\sO$, then
$$
0\leq \bar\alpha_0 < 1,
$$
where
\begin{equation}
\label{eq:DefnBarAlpha0}
\bar\alpha_0 := \essinf_{(x,y) \in \sO}\alpha_0(x,y).
\end{equation}
Otherwise, if $\bar\alpha_0 = 1$, we would have $\alpha_0  = 1$ on $\sO$ and $u_1 = \alpha_0 u_1 \leq u_2$ on $\sO$, contradicting our assumption \eqref{eq:u1notlessthanu2}. We introduce a \emph{constant} $\beta_0$ such that
\begin{equation}
\label{eq:barAlpha0lessthanBeta0}
\bar\alpha_0<\beta_0<1,
\end{equation}
and
\begin{equation}
\label{eq:VIBetaLambdafuInequality}
\begin{aligned}
f_1 &:= \beta_0\left(f + \lambda(1+y)u_1\right)
\\
&\leq f + \lambda (1+y)u_2
\\
&=: f_2 \quad\hbox{on }\sO.
\end{aligned}
\end{equation}
Indeed, such a $\beta_0$ exists since, because of \eqref{eq:VIalphaSoln1LessSoln2}, the inequality \eqref{eq:VIBetaLambdafuInequality} will hold if
$$
\beta_0f + \lambda(1+y)\beta_0 u_1 \leq f + \lambda(1+y)\bar\alpha_0 u_1,
$$
that is, if
$$
\lambda(1+y)(\beta_0-\bar\alpha_0)u_1 \leq (1-\beta_0)f,
$$
and hence if
\begin{equation}
\label{eq:EnoughToMakeSupu1LessThanOrEqualToInff}
\esssup_{(x,y)\in\sO}\frac{(1+y)u_1(x,y)}{f(x,y)} \leq \frac{(1-\beta_0)}{\lambda(\beta_0-\bar\alpha_0)}.
\end{equation}
Now $A(m+\varphi) = A(m+u_\varphi) = A\tilde m$ on $\sO$ by \eqref{eq:Auvarphi} and $\tilde M = M + u_\varphi \leq M+\varphi$ on $\sO$ by \eqref{eq:uvarphiBounds}. Therefore
$$
\frac{(1+y)\tilde M(x,y)}{A\tilde m(x,y)} \leq \frac{(1+y)(M+\varphi)(x,y)}{A(m+\varphi)(x,y)}, \quad (x,y) \in \sO,
$$
and so \eqref{eq:SupSourceSolutionRatioBoundRaw} yields
\begin{equation}
\label{eq:SupSourceSolutionRatioBound}
\esssup_{(x,y)\in\sO}\frac{(1+y)\tilde M(x,y)}{A\tilde m(x,y)} < \infty.
\end{equation}
By the upper bound for $u_1$ in \eqref{eq:tildeuBounds} (with $\tilde u$ replaced by $u_1$) and the lower bound for $f$ in \eqref{eq:tildefBounds} (with $\tilde f$ replaced by $f$), we have
\begin{align*}
\esssup_{(x,y)\in\sO}\frac{(1+y)u_1(x,y)}{f(x,y)} &\leq \esssup_{(x,y)\in\sO}\frac{(1+y)\tilde M(x,y)}{A\tilde m(x,y)}
\\
&< \infty \quad\hbox{(by \eqref{eq:SupSourceSolutionRatioBound})}.
\end{align*}
Therefore, we can find $\beta_0$ (sufficiently close to $\bar\alpha_0$) obeying \eqref{eq:barAlpha0lessthanBeta0} such that the inequalities \eqref{eq:EnoughToMakeSupu1LessThanOrEqualToInff} and thus \eqref{eq:VIBetaLambdafuInequality} hold.

We note that $\beta_0 u_1$ satisfies the variational equation
\begin{align*}
a_\lambda(\beta_0 u_1, \beta_0 v) &= a(\beta_0 u_1, \beta_0 v) + \lambda(\beta_0 (1+y)u_1, \beta_0 v)_H
\quad\hbox{(by \eqref{eq:BilinearFormCoerciveHeston})}
\\
&= (\beta_0 f, \beta_0 v)_H + \lambda(\beta_0 (1+y)u_1, \beta_0 v)_H \quad\hbox{(by \eqref{eq:IntroHestonWeakMixedProblemHomogeneous})}
\\
&= (\beta_0 f + \lambda\beta_0 (1+y)u_1, \beta_0 v)_H
\\
&= (f_1,\beta_0 v)_H
\quad\hbox{(by definition \eqref{eq:VIBetaLambdafuInequality} of $f_1$)},
\end{align*}
for all $\beta_0 v \in V$, and so
$$
a_\lambda(\beta_0 u_1, v) = (f_1, v)_H, \quad \forall v \in V.
$$
Moreover, $u_2$ satisfies the variational equation
\begin{align*}
a_\lambda(u_2,v) &= a(u_2,v) + \lambda((1+y)u_2,v)_H
\\
&= (f,v)_H + \lambda((1+y)u_2,v)_H
\quad\hbox{(by \eqref{eq:IntroHestonWeakMixedProblemHomogeneous})}
\\
&= (f+\lambda (1+y)u_2,v)_H
\\
&= (f_2,v)_H
\quad\hbox{(by definition \eqref{eq:VIBetaLambdafuInequality} of $f_2$)}, \quad \forall v \in V.
\end{align*}
We are now within the setting of Corollary \ref{cor:ExistenceUniquenessEllipticCoerciveHeston} since \eqref{eq:CoerciveHeston} holds and
$$
f_1 \leq f_2 \quad\hbox{on $\sO$ by \eqref{eq:VIBetaLambdafuInequality}}.
$$
Therefore, Corollary \ref{cor:ExistenceUniquenessEllipticCoerciveHeston} implies that
\begin{equation}
\label{eq:betazerou1lequ2}
\beta_0 u_1 \leq u_2 \quad\hbox{on } \sO.
\end{equation}
But $\beta_0>\bar\alpha_0$ by \eqref{eq:barAlpha0lessthanBeta0} and so we obtain a contradiction to the choice of $\alpha_0$ in \eqref{eq:VIalphaSoln1LessSoln2}. To see this, observe that for any $0<\eps<\beta_0-\bar\alpha_0$, the definition \eqref{eq:DefnBarAlpha0} of $\bar\alpha_0$ implies that there exists a Borel measurable subset $\sU\subset\sO$ with positive measure such that $\alpha_0 \leq \bar\alpha_0 + \eps$ on $\sU$ and so, because $\bar\alpha_0 + \eps < \beta_0$, we obtain
\begin{equation}
\label{eq:alpha0lessthanbeta0}
\alpha_0 < \beta_0 \quad\hbox{on }\sU.
\end{equation}
But $\beta_0 < 1$ by \eqref{eq:barAlpha0lessthanBeta0} and therefore \eqref{eq:alpha0lessthanbeta0} yields $\alpha_0 < 1$ on $\sU$ and so the definition of $\alpha_0$ in \eqref{eq:DefnAlpha0} implies that $u_1>u_2$ on $\sU$ and $u_2 = \alpha_0u_1$ on $\sU$. Therefore,
\begin{align*}
u_2 &= \alpha_0u_1
\\
&< \beta_0 u_1 \quad\hbox{on } \sU \quad\hbox{(by \eqref{eq:alpha0lessthanbeta0})}
\\
&\leq u_2 \quad\hbox{on } \sO \quad\hbox{(by \eqref{eq:betazerou1lequ2})},
\end{align*}
a contradiction. Thus, we must have $u_1=u_2$ on $\sO$. This concludes the proof of uniqueness in Theorem \ref{thm:ExistenceUniquenessEllipticHeston_Improved}.
\end{proof}

\begin{rmk}[Comments on the proof of uniqueness in Theorem \ref{thm:ExistenceUniquenessEllipticHeston_Improved}]
Because the comparison principle bounds \eqref{eq:uBounds} are a posteriori rather than a priori, we cannot appeal to \eqref{eq:uBounds} to provide a simple proof of uniqueness: the proof of existence in Theorem \ref{thm:ExistenceUniquenessEllipticHeston_Improved} merely shows that some, but not necessarily all, solutions obey \eqref{eq:uBounds}.
\end{rmk}

\begin{rmk}[Unique solutions and the Rellich-Kondrachov theorem]
\label{rmk:NoRellichAndZeroEigenvalue}
Because the continuous embedding $H^1(\sO,\fw)\to L^2(\sO,\fw)$ is not known to be compact (see \S \ref{subsec:WeightedSobolevContinuousCompactEmbeddings}), unlike the case of usual the Rellich-Kondrachov compact embedding theorem for standard Sobolev spaces and bounded domains, it is not known whether an analogue of \cite[Theorem 6.2.6]{Evans} holds for our weighted Sobolev spaces. Therefore, we avoid arguments in this article which might rely on a Rellich-Kondrachov compact embedding theorem for weighted Sobolev spaces.
\end{rmk}

\section{Existence and uniqueness of solutions to the variational inequality}
\label{sec:StatVarInequality}
We shall adapt the framework of \cite[\S 3]{Bensoussan_Lions} to the case of the degenerate Heston operator and describe the modifications required for the proofs of existence and uniqueness for a variational inequality. We begin in \S \ref{subsec:StationaryFrameworkHeston} with a formulation of the obstacle and variational inequality problems and provide conditions for when they are equivalent (Lemma \ref{lem:VIStrongFormHeston}). In \S \ref{subsec:PenalizedEquation}, we prove existence and uniqueness of solutions to a non-linear penalized equation (Theorem \ref{thm:ExistenceUniquenessEllipticHestonPenalizedProblem}). In \S \ref{subsec:CoerciveStatVarInequality}, we show that solutions to the penalized equation and their a priori estimates provide stepping-stones to existence of solutions of a coercive variational inequality (Theorem \ref{thm:VIExistenceUniquenessEllipticCoerciveHeston}). Finally, in \S \ref{subsec:NoncoerciveStatVarInequality}, we show that the existence of solutions to a coercive variational inequality leads in turn to existence of solutions to a non-coercive variational inequality (Theorem \ref{thm:VIExistenceUniquenessEllipticHeston_Improved}).

\subsection{Formulation of the variational inequality and obstacle problem}
\label{subsec:StationaryFrameworkHeston}
We begin with

\begin{prob}[Classical solution to an obstacle problem with inhomogeneous Dirichlet boundary condition]
\label{prob:HestonObstacleInhomogeneousClassical}
Given functions $f\in C^\alpha(\sO)$, for some $0<\alpha<1$, $g\in C^{2,\alpha}(\sO)\cap C_{\textrm{loc}}(\sO\cup\Gamma_1)$, and $\psi\in C_{\textrm{loc}}(\sO\cup\Gamma_1)$ with
\begin{equation}
\label{eq:ObstacleFunctionLessThanBoundaryConditionFunction}
\psi \leq g \quad\hbox{on }\Gamma_1,
\end{equation}
we call $u\in C^{1,1}(\sO)\cap C_{\textrm{loc}}(\sO\cup\Gamma_1)$ a \emph{classical solution} to an obstacle problem for the elliptic Heston operator with inhomogeneous Dirichlet condition along $\Gamma_1$ if
\begin{align}
\label{eq:ObstacleProblemHestonClassical}
\min\{Au-f,u-\psi\} &= 0 \quad \hbox{on }\sO,
\\
\label{eq:ObstacleHestonInhomogeneousBC}
u &= g \quad\hbox{on } \Gamma_1,
\\
\label{eq:ObstacleWeightedNeumannHomogeneousBCLimit}
\lim_{y\downarrow 0}y^\beta(\rho u_x + \sigma u_y) &= 0 \quad\hbox{on } \Gamma_0.
\end{align}
\end{prob}

See Remark \ref{rmk:WellPosedClassicalBVP} for a discussion of the hypotheses on $f$ and the boundary condition \eqref{eq:ObstacleWeightedNeumannHomogeneousBCLimit}.

\begin{prob}[Strong solution to an obstacle problem with inhomogeneous Dirichlet boundary condition]
\label{prob:HestonObstacleInhomogeneousStrong}
Given functions $f\in L^2(\sO,\fw)$, $g\in H^2(\sO,\fw)$, and $\psi\in H^2(\sO,\fw)$ obeying \eqref{eq:ObstacleFunctionLessThanBoundaryConditionFunction}, we call $u\in H^2(\sO,\fw)$ a \emph{strong solution} to an an obstacle problem for the elliptic Heston operator with inhomogeneous Dirichlet boundary condition along $\Gamma_1$ if $u$ obeys \eqref{eq:ObstacleProblemHestonClassical} (a.e. on $\sO$) and \eqref{eq:ObstacleHestonInhomogeneousBC}.
\end{prob}

We state our variational inequality problem for the Heston operator in the case of inhomogeneous Dirichlet boundary conditions:

\begin{prob}[Variational inequality with inhomogeneous Dirichlet boundary condition]
\label{prob:NonhomogeneousHestonVIProblem}
Given functions $f\in L^2(\sO,\fw)$, $g\in H^1(\sO,\fw)$, and $\psi\in H^1(\sO,\fw)$ obeying \eqref{eq:ObstacleFunctionLessThanBoundaryConditionFunction} in the sense that
$$
(\psi-g)^+ \in H^1_0(\sO\cup \Gamma_0,\fw),
$$
we call $u\in H^1(\sO,\fw)$ a solution to the variational inequality for the Heston operator with inhomogeneous Dirichlet boundary condition along $\Gamma_1$ if
\begin{equation}
\label{eq:VIProblemHeston}
\begin{gathered}
a(u,v-u) \geq (f,v-u)_{L^2(\sO,\fw)},
\\
u\geq\psi \hbox{ a.e. on }\sO \hbox{ and } u-g \in H^1_0(\sO\cup\Gamma_0,\fw),
\\
\quad \forall v\in H^1(\sO,\fw) \hbox{ with } v\geq\psi \hbox{ a.e. on }\sO  \hbox{ and } v-g \in H^1_0(\sO\cup\Gamma_0,\fw).
\end{gathered}
\end{equation}
\end{prob}

\begin{rmk}[Reduction to a variational inequality with homogeneous Dirichlet boundary condition]
\label{rmk:ReductionInhomogeneousVariationalInequality}
When $\Gamma_1\neq \emptyset$, we can reduce to the case of a \emph{homogeneous} Dirichlet boundary condition along $\Gamma_1$, without loss of generality, by noting that $u\in H^1(\sO,\fw)$ is a solution to Problem \ref{prob:NonhomogeneousHestonVIProblem} if and only if $\tilde u := u-g \in H^1_0(\sO\cup\Gamma_0,\fw)$ is a solution to Problem \ref{prob:NonhomogeneousHestonVIProblem} with test functions $\tilde v := v-g \in H^1_0(\sO\cup\Gamma_0,\fw)$, source function $\tilde f := f-g \in L^2(\sO,\fw)$, and obstacle function $\tilde\psi := \psi-g \in H^1(\sO,\fw)$ obeying $\tilde\psi \leq 0$ on $\Gamma_1$ in the sense that $\tilde\psi^+ \in H^1_0(\sO\cup\Gamma_0,\fw)$.
\end{rmk}

Therefore, given Remark \ref{rmk:ReductionInhomogeneousVariationalInequality}, for the remainder of the article, we may assume without loss of generality reductions to variational inequalities and obstacle problems with homogeneous Dirichlet boundary condition on $Gamma_1$.

\begin{prob}[Variational inequality with homogeneous Dirichlet boundary condition]
\label{prob:HomogeneousHestonVIProblem}
Given functions $f\in L^2(\sO,\fw)$ and $\psi\in H^1(\sO,\fw)$ such that
\begin{equation}
\label{eq:ObstacleFunctionLessThanZero}
\psi \leq 0 \hbox{ on }\Gamma_1,
\end{equation}
in the sense that
$$
\psi^+ \in H^1_0(\sO\cup \Gamma_0,\fw),
$$
we call $u\in H^1_0(\sO\cup\Gamma_0,\fw)$ a solution to the variational inequality for the Heston operator with homogeneous Dirichlet boundary condition along $\Gamma_1$ if
\begin{equation}
\label{eq:VIProblemHestonHomgeneous}
\begin{gathered}
a(u,v-u) \geq (f,v-u)_{L^2(\sO,\fw)}
\\
u\geq\psi \hbox{ a.e. on }\sO,
\\
\quad \forall v\in H^1_0(\sO\cup\Gamma_0,\fw) \hbox{ with } v\geq\psi \hbox{ a.e. on }\sO.
\end{gathered}
\end{equation}
\end{prob}

\begin{rmk}[Specialization to a variational equation]
The obstacle condition $u\geq \psi$ on $\sO$ becomes vacuous when $\psi=-\infty$ and, for this reason, we see that it is sufficient to require $\psi^+ \in H^1(\sO,\fw)$ (or even $L^2(\sO,\fw)$) rather than $\psi \in H^1(\sO,\fw)$ as in Problem \ref{prob:HomogeneousHestonVIProblem}.
\end{rmk}

\begin{notn}[Function spaces]
For brevity, we shall often denote
$$
H := L^2(\sO,\fw), \quad V := H^1_0(\sO\cup\Gamma_0,\fw), \quad\hbox{and}\quad \KK := \{v\in V| v\geq \psi \hbox{ a.e. on } \sO\}.
$$
\end{notn}

\begin{rmk}[Convexity of $\KK$]
\label{rmk:ConvexityObstacleTestFunctionSpace}
Note that $\KK\subset V$ is a \emph{convex} subset, since if $u,v \in \KK$, then $zu + (1-z)v \in V$ for all $z\in\RR$  while, for all $z\in [0,1]$, we have $zu + (1-z)v \geq z\psi + (1-z)\psi = \psi$.
\end{rmk}

\begin{rmk}[Choice of source function]
Rather than require $f\in H$ in Problem \ref{prob:HomogeneousHestonVIProblem}, it is enough to assume $f\in V'$ when considering questions of existence and uniqueness.
\end{rmk}

\begin{rmk}[Lower rather than upper obstacle functions]
The directions of our obstacle function inequalities are the opposite of those in \cite{Bensoussan_Lions}, so that instead of \cite[Equations (3.1.11) \& (3.1.14)]{Bensoussan_Lions} we define $\KK$ using $v\geq \psi$ and require $\psi\leq 0$ on $\Gamma_1$.
\end{rmk}

\begin{exmp}[Examples of obstacle functions]
\label{exmp:ObstacleFunction}
For the problem of determining the price of a perpetual American-style put option, we would choose
\begin{equation}
\label{eq:ObstacleHeston}
\psi(x,y) := (E-e^x)^+, \quad (x,y) \in \sO,
\end{equation}
for a constant $E>0$ (the strike) and $f = 0$ on $\sO$. Note that this choice of function $\psi$ is \emph{Lipschitz} and lies in $H^1(\sO,\fw)$ but not in $H^2(\sO,\fw)$. In the case of the corresponding \emph{evolutionary} variational inequality, the terminal condition would also be given by
\begin{equation}
\label{eq:TerminalConditionHeston}
h(x,y) := (E-e^x)^+, \quad (x,y) \in \sO.
\end{equation}
See \cite{Haug} for additional examples of obstacle functions arising in mathematical finance. \qed
\end{exmp}

Given a function $\psi \in H^1(\sO,\fw)$ as in Problem \ref{prob:HomogeneousHestonVIProblem}, we have $\psi^+ \in H^1_0(\sO\cup\Gamma_0,\fw)$ and $\psi^+ = \psi + \psi^- \geq \psi$ a.e. on $\sO$ and therefore, by analogy with \cite[Equation (3.1.13)]{Bensoussan_Lions}, the following universal assumption is automatically satisfied by choosing $v_0 = \psi^+$.

\begin{assump}[Non-empty convex subset $\KK \subset V$]
\label{assump:NonemptyConvexSetHeston}
The subset $\KK \subset V$ is non-empty and contains some element $v_0 \in\KK$.
\end{assump}

By analogy with \cite[Equation (3.1.20)]{Bensoussan_Lions}, we have

\begin{lem}[Equivalence of variational and strong solutions]
\label{lem:VIStrongFormHeston}
Let $f\in L^2(\sO,\fw)$, $\psi\in H^2(\sO,\fw)$ be functions such that \eqref{eq:ObstacleFunctionLessThanZero} holds, and suppose $u\in H^2(\sO,\fw)$. Then the following hold:
\begin{enumerate}
\item If $u \in H^1_0(\sO\cup\Gamma_0,\fw)$ obeys \eqref{eq:VIProblemHestonHomgeneous}, then $u$ obeys \eqref{eq:ObstacleProblemHestonClassical} (a.e. on $\sO$) and \eqref{eq:ObstacleHestonInhomogeneousBC} (with $g = 0$).
\item If $u$ obeys \eqref{eq:ObstacleProblemHestonClassical} (a.e. on $\sO$) and \eqref{eq:ObstacleHestonInhomogeneousBC} (with $g = 0$), then $u \in H^1_0(\sO\cup\Gamma_0,\fw)$ and $u$ obeys \eqref{eq:VIProblemHestonHomgeneous}.
\end{enumerate}
\end{lem}

\begin{proof}
Consider (1) and suppose $u \in H^1_0(\sO\cup\Gamma_0,\fw)$ obeys \eqref{eq:VIProblemHestonHomgeneous}. We wish to show first that
$$
Au-f \geq 0, \quad u-\psi \geq 0, \quad (Au-f)(u-\psi) = 0 \hbox{ a.e. on } \sO.
$$
To see this, observe that \eqref{eq:VIProblemHestonHomgeneous} and integration by parts (Lemma \ref{lem:HestonIntegrationByParts}) yields
\begin{equation}
\label{eq:VIProblemHestonStrong}
(Au-f,v-u)_H \geq 0, \quad \forall v \in \KK.
\end{equation}
Take $v=u+\varphi$, $\varphi\in C^\infty_0(\sO)$, with $\varphi\geq 0$. Therefore, we have $v\in \KK$, given that $u\in\KK$, and so
$$
(Au-f,\varphi)_H \geq 0, \quad \forall \varphi \in C^\infty_0(\sO),
$$
and hence
$$
Au-f \geq 0 \quad \hbox{a.e. on }\sO.
$$
Since $\psi\in L^2(\sO,\fw)$ by hypothesis, we may choose $v=\psi$ in \eqref{eq:VIProblemHestonStrong} to give
$$
(Au-f,\psi-u)_H \geq 0,
$$
Since $Au-f \geq 0$ and $\psi-u\leq 0$ a.e. on $\sO$, we also have
$$
(Au-f,\psi-u)_H \leq 0,
$$
and hence $(Au-f,\psi-u)_H = 0$, which gives
$$
\int_\sO(Au-f)(u-\psi)\fw\,dxdy = 0.
$$
Because $(Au-f)(u-\psi) \geq 0$ a.e. on $\sO$, we obtain $(Au-f)(\psi-u) = 0$ a.e. on $\sO$. In addition, $u=0$ on $\Gamma_1$ by Lemma \ref{lem:EvansGamma1TraceZero}. Thus, $u$ obeys \eqref{eq:ObstacleProblemHestonClassical} (a.e. on $\sO$) and \eqref{eq:ObstacleHestonInhomogeneousBC} (with $g = 0$).

Consider (2) and suppose $u$ obeys \eqref{eq:ObstacleProblemHestonClassical} (a.e. on $\sO$) and \eqref{eq:ObstacleHestonInhomogeneousBC} (with $g = 0$). We then obtain $u \in H^1_0(\sO\cup\Gamma_0,\fw)$ by Lemma \ref{lem:EvansGamma1TraceZero}. Suppose $v-u = \varphi \in C^\infty_0(\sO\cup\Gamma_0)$, with $v\geq \psi$ a.e. on $\sO$. Then $\{\varphi < 0\} = \{v < u\} \subset \{\psi < u\}$ and therefore we must have $Au-f = 0$ a.e. on $\{\varphi < 0\} \subset \sO$ by \eqref{eq:ObstacleProblemHestonClassical}. Thus,
\begin{align*}
0 &\leq \int_{\sO\cap\{\varphi \geq 0\}} (Au-f)\varphi\fw\,dxdy \quad\hbox{(by \eqref{eq:ObstacleProblemHestonClassical})}
\\
&= \int_{\sO\cap\{\varphi \geq 0\}} (Au-f)\varphi\fw\,dxdy + \int_{\sO\cap\{\varphi < 0\}} (Au-f)\varphi\fw\,dxdy
\\
&= \int_\sO (Au-f)\varphi\fw\,dxdy
\\
&= \int_\sO Au\varphi\fw\,dxdy - \int_\sO f\varphi\fw\,dxdy
\\
&= a(u,\varphi) - (f,\varphi)_H \quad\hbox{(by Lemma \ref{lem:HestonIntegrationByParts})},
\end{align*}
and thus \eqref{eq:VIProblemHestonHomgeneous} holds for all $v = u+\varphi$ with $v\geq\psi$ a.e on $\sO$ and $\varphi \in C^\infty_0(\sO\cup\Gamma_0)$. Since $C^\infty_0(\sO\cup\Gamma_0,\fw)$ is dense in $H^1_0(\sO\cup\Gamma_0,\fw)$, then \eqref{eq:VIProblemHestonHomgeneous} holds for all $v \in H^1_0(\sO\cup\Gamma_0,\fw)$ with $v\geq\psi$.
\end{proof}

\subsection{Existence and uniqueness of solutions to the penalized equation}
\label{subsec:PenalizedEquation}
Existence of solutions to the coercive variational inequality (see Theorem \ref{thm:VIExistenceUniquenessEllipticCoerciveHeston}) is proved by first establishing existence, given $\eps>0$, of $u_\eps \in V$ which solves the \emph{penalized problem} associated with Problem \ref{prob:HomogeneousHestonVIProblem}. We define the \emph{penalization operator} by
\begin{equation}
\label{eq:PenalizationOperator}
\beta_\eps(w) := -\frac{1}{\eps}(\psi-w)^+, \quad w\in H^1_0(\sO\cup\Gamma_0).
\end{equation}
Recall that $a_\lambda$ is the coercive bilinear form \eqref{eq:BilinearFormCoerciveHeston}.

\begin{prob}[Penalized equation for the coercive Heston bilinear form]
\label{prob:HestonPenalizedEquation}
Given a function $f\in L^2(\sO,\fw)$, we call a function $u\in H^1_0(\sO\cup\Gamma_0)$ a solution to the penalized equation for the coercive Heston bilinear form if
\begin{equation}
\label{eq:PenalizedProblem}
a_\lambda(u_\eps,v) + (\beta_\eps(u_\eps),v)_H = (f,v)_H, \quad \forall v\in H^1_0(\sO\cup\Gamma_0).
\end{equation}
\end{prob}

We may also write $\beta_\eps(w) = -\frac{1}{\eps}(w - \psi)^-$. Note that our definition of $\beta_\eps$ uses a sign opposite to that of \cite[Equation (1.24)]{Bensoussan_Lions} since we seek functions $u$ such that $u\geq\psi$ a.e. on $\sO$ (and not $u\leq\psi$ on $\sO$) and thus need to ``penalize" functions $u$ such that $u<\psi$ on subsets of $\sO$ with positive measure.

\begin{lem}[Monotonicity of the penalization operator]
\label{lem:Monotone}
The penalization operator, $\beta_\eps$ in \eqref{eq:PenalizationOperator}, is \emph{monotone} in the sense that
\begin{equation}
\label{eq:MonotonePenalizationOperator}
(\beta_\eps(u)-\beta_\eps(\tilde u), u-\tilde u)_H \geq 0, \quad \forall u, \tilde u \in V.
\end{equation}
\end{lem}

\begin{proof}
Write
\begin{align*}
\eps\beta_\eps(u)-\eps\beta_\eps(\tilde u) &= -(\psi-u)^+ + (\psi-\tilde u)^+,
\\
u-\tilde u &= -(\psi-u)^+ + (\psi-u)^- + (\psi-\tilde u)^+ - (\psi-\tilde u)^-,
\end{align*}
and observe that
\begin{align*}
(\eps\beta_\eps(u)-\eps\beta_\eps(\tilde u), u-\tilde u)_H &= \|(\psi-u)^+\|_H^2 + \|(\psi-\tilde u)^+\|_H^2 - 2((\psi-u)^+, (\psi-\tilde u)^+)_H
\\
&\quad + ((\psi-\tilde u)^+,(\psi-u)^-)_H + ((\psi-u)^+,(\psi-\tilde u)^-)_H
\\
&\geq  ((\psi-\tilde u)^+,(\psi-u)^-)_H + ((\psi-u)^+,(\psi-\tilde u)^-)_H
\\
&\qquad \quad\hbox{(using $2ab \leq a^2+b^2, a,b, \in\RR$)}
\\
&\geq 0,
\end{align*}
as desired, using the fact in the last inequality that each of the terms $(\psi-\tilde u)^+$, $(\psi-u)^-)_H$, $((\psi-u)^+$, $(\psi-\tilde u)^-$ is non-negative.
\end{proof}

We have the following analogue of the a priori estimate \cite[Equation (3.1.44)]{Bensoussan_Lions} for solutions to the penalized equation.

\begin{lem}[A priori estimate for a solution to the penalized equation]
\label{lem:PenalizedEquationAPrioriEstimate}
If $u_\eps\in V$ is a solution to Problem \ref{prob:HestonPenalizedEquation}, then
\begin{equation}
\label{eq:PenalizedEquationAPrioriEstimate}
\|u_\eps\|_V \leq C\left(\|f\|_H + \|\psi^+\|_V\right), \quad \forall\eps > 0,
\end{equation}
where $C$ depends only on the constant coefficients of $A$.
\end{lem}

\begin{proof}
We have $\psi^+ \in H^1_0(\sO\cup\Gamma_0,\fw)$ by \eqref{eq:ObstacleFunctionLessThanZero} and so we may choose $v=u_\eps-\psi^+ \in H^1_0(\sO\cup\Gamma_0,\fw)$ in \eqref{eq:PenalizedProblem}. Noting that $\beta_\eps(\psi^+)=0$ since $\psi^+ = \psi + \psi^- \geq \psi$ a.e. on $\sO$, we obtain
\begin{equation}
\label{eq:PenalizedProblem_IncEqn}
a_\lambda(u_\eps,u_\eps-\psi^+) + (\beta_\eps(u_\eps)-\beta_\eps(\psi^+),u_\eps-\psi^+)_H = (f, u_\eps-\psi^+)_H.
\end{equation}
Since $\beta_\eps$ is monotone, applying \eqref{eq:MonotonePenalizationOperator} to the preceding identity yields
$$
a_\lambda(u_\eps,u_\eps-\psi^+) \leq (f, u_\eps-\psi^+)_H,
$$
and so
$$
a_\lambda(u_\eps-\psi^+,u_\eps-\psi^+) \leq (f, u_\eps-\psi^+)_H - a_\lambda(\psi^+,u_\eps-\psi^+).
$$
Combining the preceding inequality with the G\r{a}rding inequality \eqref{eq:CoerciveHeston} and the continuity estimate \eqref{eq:ContinuousCoerciveHeston} for $a_\lambda(u,v)$ gives
$$
\nu_1\|u_\eps-\psi^+\|_V^2 \leq C\left(\|f\|_H + \|\psi^+\|_V\right)\|u_\eps-\psi^+\|_V,
$$
and therefore we obtain \eqref{eq:PenalizedEquationAPrioriEstimate}.
\end{proof}

We also have the following analogue of the a priori estimate for the penalization term \cite[Equation (3.1.36)]{Bensoussan_Lions}.

\begin{lem}[A priori estimate for the penalization term]
\label{lem:PenalizationTermAPrioriEstimate}
If $u_\eps\in V$ is a solution to Problem \ref{prob:HestonPenalizedEquation}, then
\begin{equation}
\label{eq:PenalizationTermAPrioriEstimate}
\|(\psi-u_\eps)^+\|_H \leq C\sqrt{\eps}\left(\|f\|_H + \|\psi^+\|_V\right), \quad \forall\eps > 0,
\end{equation}
where $C$ depends only on the constant coefficients of $A$.
\end{lem}

\begin{proof}
As in the proof of Lemma \ref{lem:PenalizedEquationAPrioriEstimate}, choose $v = u_\eps - \psi^+$ in \eqref{eq:PenalizedProblem} to give
\begin{align*}
\left|(\beta_\eps(u_\eps),u_\eps-\psi^+)_H\right| &= \left|-a_\lambda(u_\eps,u_\eps-\psi^+) + (f, u_\eps-\psi^+)_H\right|
\\
&\leq C\|u_\eps\|_V\|u_\eps-\psi^+\|_V + \|f\|_H\|u_\eps-\psi^+\|_H \quad\hbox{(by \eqref{eq:ContinuousCoerciveHeston}),}
\end{align*}
and thus
\begin{equation}
\label{eq:PenalizationTermAPrioriEstimate_prefinal}
\left|(\beta_\eps(u_\eps),u_\eps-\psi^+)_H\right| \leq C\left(\|f\|_H + \|\psi^+\|_V\right)^2 \quad\hbox{(by \eqref{eq:PenalizedEquationAPrioriEstimate})},
\end{equation}
where $C>0$ depends only on the constant coefficients of $A$ and is independent of $\eps$. But
\begin{align*}
(\beta_\eps(u_\eps),\psi-u_\eps)_H &= (\beta_\eps(u_\eps),\psi^+-u_\eps)_H + (\beta_\eps(u_\eps),\psi-\psi^+)_H
\\
&\geq (\beta_\eps(u_\eps),\psi^+-u_\eps)_H \quad\hbox{(by \eqref{eq:PenalizationOperator} and fact that $\psi-\psi^+ = -\psi^- \leq 0$)}
\\
&\geq -C\left(\|f\|_H + \|\psi^+\|_V\right)^2 \quad\hbox{(by \eqref{eq:PenalizationTermAPrioriEstimate_prefinal})}.
\end{align*}
Hence, because $\beta_\eps(u_\eps) = -\eps^{-1}(\psi-u_\eps)^+$ by \eqref{eq:PenalizationOperator}, we obtain
\begin{align*}
-\eps^{-1}((\psi-u_\eps)^+,(\psi-u_\eps)^+)_H &= -\eps^{-1}((\psi-u_\eps)^+,\psi-u_\eps)_H
\\
&\geq -C\left(\|f\|_H + \|\psi^+\|_V\right)^2,
\end{align*}
and this yields \eqref{eq:PenalizationTermAPrioriEstimate}.
\end{proof}

By analogy with \cite[Theorem 3.1.2]{Bensoussan_Lions} we have

\begin{thm}[Existence and uniqueness of solutions to the penalized equation]
\label{thm:ExistenceUniquenessEllipticHestonPenalizedProblem}
There exists a unique solution to Problem \ref{prob:HestonPenalizedEquation}.
\end{thm}

\begin{proof}[Proof of uniqueness in Theorem \ref{thm:ExistenceUniquenessEllipticHestonPenalizedProblem}]
The proof of uniqueness is almost identical to that of the proof of \cite[Theorem 3.1.2]{Bensoussan_Lions}. Let $u,\tilde u$ be two solutions of \eqref{eq:PenalizedProblem}. Then, substituting $u,\tilde u$ in \eqref{eq:PenalizedProblem}, subtracting the resulting equations, and choosing $v=u-\tilde u$ yields
$$
a_\lambda(u-\tilde u,u-\tilde u) + (\beta_\eps(u)-\beta_\eps(\tilde u), u-\tilde u)_H = 0.
$$
The operator $\beta_\eps$ is monotone by Lemma \ref{lem:Monotone}. Hence,
$$
a_\lambda(u-\tilde u,u-\tilde u) \leq 0,
$$
and \eqref{eq:CoerciveHeston} ensures that $\|u-\tilde u\|_V=0$, so $u = \tilde u$ a.e. on $\sO$.
\end{proof}

\begin{lem}[Fixed point lemma]
\label{lem:Brouwer}
\cite[Lemma 1.4.3]{Lions_1969}
Let $m\geq 1$ and let $F:\RR^m\to\RR^m$ be a continuous map such that for a suitable $\varrho>0$ one has
$$
(F(\xi),\xi) \geq 0, \quad \forall \xi\in\RR^m, |\xi|=\varrho.
$$
Then there exists $\xi_0$, $|\xi_0|\leq\varrho$, such that $F(\xi_0)=0$.
\end{lem}

\begin{proof}[Proof of existence in Theorem \ref{thm:ExistenceUniquenessEllipticHestonPenalizedProblem}]
The proof of existence is similar to that of the proof of \cite[Theorem 3.1.2]{Bensoussan_Lions}. We introduce a family of $V_m\subset V$, $m=1,2,3,\ldots$, of $m$-dimensional subspaces such that
\begin{itemize}
\item For each $v\in V$, there is a $v_m\in V_m$, for each $m\geq 1$, such that
\begin{equation}
\label{eq:StrongConvergenceFiniteDimSubspaceSequence}
\|v-v_m\|_V\to 0, \quad m\to\infty.
\end{equation}
\item There exists $v_0\in V_m\cap \KK$, for all $m\geq 1$.
\end{itemize}
We may choose $v_0 = \psi^+ \in \KK$ and let $v_0',v_1',\ldots,v_k',\ldots$ be an orthonormal basis for the Hilbert space $V$, where $v_0' := v_0/\|v_0\|_V$ if $v_0 \neq 0$ and if $v_0 = 0$, then choose any $v_0' \in V$ with $\|v_0'\|_V = 1$. Let $V_m := \hbox{span}\{v_0',\ldots,v_{m-1}'\}$, $m\geq 1$.

We now consider the finite-dimensional problem, which is to find $u_m\in V_m$ such that
\begin{equation}
\label{eq:PenalizedProblemFiniteDim}
a_\lambda(u_m,v) + (\beta_\eps(u_m),v)_H = (f,v)_H, \quad \forall v\in V_m.
\end{equation}
Such a $u_m$ exists by Lemma \ref{lem:Brouwer} and by \eqref{eq:PenalizedEquationAPrioriEstimate}, we have
\begin{equation}
\label{eq:BoundForSequence}
\|u_m\|_V \leq C\left(\|f\|_H + \|\psi^+\|_V\right),
\end{equation}
where $C>0$ depends only on the constant coefficients of $A$ and is independent of $m$ and $\eps$. We can therefore extract a weakly convergent subsequence, again denoted $u_m$, such that
\begin{equation}
\label{eq:WeakConvergence_u_m}
u_m \rightharpoonup u_\eps\quad\hbox{weakly in }V, \quad m\to\infty.
\end{equation}
Moreover, \eqref{eq:PenalizationTermAPrioriEstimate} yields
\begin{equation}\label{eq:SequenceObstacleDiffBound}
\|(\psi-u_m)^+\|_H \leq C\sqrt{\eps}\left(\|f\|_H + \|\psi^+\|_V\right),
\end{equation}
where $C>0$ depends only on the constant coefficients of $A$ and is independent of $m$ and $\eps$.  From \eqref{eq:SequenceObstacleDiffBound}, by passing to a weakly convergent subsequence, we can assume that as $m\to\infty$,
\begin{equation}
\label{eq:WeakConvergence_u_mDiffPsi}
(\psi-u_m)^+ \rightharpoonup \chi\quad\hbox{weakly in }H,
\end{equation}
for some $\chi \in H$.

For an arbitrary $v\in V$, we may choose $\{v_m\}_{m\geq 0}\in V$ satisfying \eqref{eq:StrongConvergenceFiniteDimSubspaceSequence} and replace $v$ in \eqref{eq:PenalizedProblemFiniteDim} by $v_m$ to give
$$
a_\lambda(u_m,v_m) + (\beta_\eps(u_m),v_m)_H = (f,v_m)_H, \quad \forall m \geq 0.
$$
As $m\to\infty$, we have $v_m$ converges strongly to $v\in V$ by \eqref{eq:StrongConvergenceFiniteDimSubspaceSequence} and $u_m$ converges weakly to $u_\eps\in V$ by \eqref{eq:WeakConvergence_u_m}, and $\beta_\eps(u_m)$ converges weakly to $\eps^{-1}\chi\in H$ by \eqref{eq:WeakConvergence_u_mDiffPsi}. Therefore, by Lemma \ref{lem:BilinearFormWeakLimit}, we can take limits in the preceding identity as $m\to\infty$ and obtain
\begin{equation}
\label{eq:PenalizedProblemLimit}
a_\lambda(u_\eps,v) + (\eps^{-1}\chi,v)_H = (f,v)_H, \quad \forall v\in V.
\end{equation}
From \eqref{eq:PenalizedProblemLimit} we see that $u_\eps$ will be the desired solution to \eqref{eq:PenalizedProblem} provided we can show
\begin{equation}
\label{eq:LimitBetaSequenceIsBetaLimit}
\eps^{-1}\chi = \beta_\eps(u_\eps),
\end{equation}
that is, provided we can show
$$
\chi = (\psi - u_\eps)^+.
$$
Because the embedding $V\to H$ is not necessarily compact, we adapt the alternative monotonicity proof of \cite[Theorem 1.1.2]{Bensoussan_Lions} to prove \eqref{eq:LimitBetaSequenceIsBetaLimit}. Define
$$
X_m := a_\lambda(u_m-v_m,u_m-v_m) + (\beta_\eps(u_m)-\beta_\eps(v_m),u_m-v_m)_H,
$$
where $v_m$ obeys \eqref{eq:StrongConvergenceFiniteDimSubspaceSequence}. From \eqref{eq:CoerciveHeston} and \eqref{eq:MonotonePenalizationOperator} we have
$$
X_m \geq 0, \quad m\geq 1.
$$
Moreover, using \eqref{eq:PenalizedProblemFiniteDim}, we have
$$
X_m =  (f,u_m-v_m)_H - a_\lambda(v_m,u_m-v_m) - (\beta_\eps(v_m),u_m-v_m)_H,
$$
from which we deduce that
\begin{equation}
\label{eq:PenalizedExistIntermedEq1}
(f,u_\eps-v)_H - a_\lambda(v,u_\eps-v) - (\beta_\eps(v),u_\eps-v)_H \geq 0,
\end{equation}
by taking limits as $m\to\infty$ and applying Lemma \ref{lem:BilinearFormWeakLimit}. But replacing $v\in V$ in \eqref{eq:PenalizedProblemLimit} with $v-u_\eps\in V$ gives
\begin{equation}
\label{eq:PenalizedExistIntermedEq2}
a_\lambda(u_\eps,v-u_\eps) + (\eps^{-1}\chi,v-u_\eps)_H -(f,u_\eps-v)_H = 0.
\end{equation}
By adding \eqref{eq:PenalizedExistIntermedEq1} and \eqref{eq:PenalizedExistIntermedEq2}, we deduce that
$$
a_\lambda(u_\eps-v,u_\eps-v) + (\eps^{-1}\chi-\beta_\eps(v),u_\eps-v)_H \geq 0, \quad \forall v\in V.
$$
Taking $v = u_\eps - \delta\varphi$ with arbitrary $\delta>0$ and $\varphi\in V$ we then deduce that
$$
\delta^2 a_\lambda(\varphi,\varphi) + \delta(\eps^{-1}\chi-\beta_\eps(u_\eps - \delta\varphi),\varphi)_H \geq 0.
$$
Dividing by $\delta$ and letting $\delta\to 0$, we therefore obtain
$$
(\eps^{-1}\chi-\beta_\eps(u_\eps),\varphi)_H \geq 0, \quad\forall \varphi \in V,
$$
so that
$$
\eps^{-1}\chi = \beta_\eps(u_\eps).
$$
This proves \eqref{eq:LimitBetaSequenceIsBetaLimit} and completes the proof of existence.
\end{proof}

Before proceeding to the statement and proof of another useful a priori estimate, we will need

\begin{lem}[Strong monotonicity of the penalization operator]
\label{lem:StronglyMonotone}
The penalization operator, $\beta_\eps$ in \eqref{eq:PenalizationOperator}, is \emph{strongly monotone} in the sense that
\begin{equation}
\label{eq:StronglyMonotonePenalizationOperator}
(\beta_\eps(u)-\beta_\eps(\tilde u), \varphi^2(u-\tilde u))_H \geq 0, \quad \forall u, \tilde u \in V,
\end{equation}
if $\varphi \in C^\infty_0(\RR^2)$.
\end{lem}

\begin{proof}
We adapt the proof of Lemma \ref{lem:Monotone} and observe that
\begin{align*}
(\eps\beta_\eps(u)-\eps\beta_\eps(\tilde u),  \varphi^2(u-\tilde u))_H
&= \|\varphi(\psi-u)^+\|_H^2 + \|\varphi(\psi-\tilde u)^+\|_H^2 - 2(\varphi(\psi-u)^+, \varphi(\psi-\tilde u)^+)_H
\\
&\quad + (\varphi(\psi-\tilde u)^+,\varphi(\psi-u)^-)_H + (\varphi(\psi-u)^+, \varphi(\psi-\tilde u)^-)_H
\\
&\geq 0,
\end{align*}
as desired.
\end{proof}

For the proof of the next lemma and at several later points in this article, we will need to use a cutoff function with certain properties, so we fix a choice below.

\begin{defn}[Cutoff function]
\label{defn:RadialCutoffFunction}
Let $\eta \in C^\infty(\RR)$ be a cutoff function such that $0\leq\eta\leq 1$, $\eta = 1$ on $(-\infty,1)$, $\eta = 0$ on $(2,\infty)$, while $|\eta'|\leq 2$ and $|\eta''|\leq 4$ on $\RR$. For $R\geq 1$, let $\zeta_R := \eta(\hbox{dist}(\cdot, O)/R) \in C^\infty_0(\RR^2)$ be the corresponding cutoff function such that $0\leq\zeta_R\leq 1$, $\zeta_R = 1$ on $B(R)$ and $\zeta_R = 0$ on $\RR^2\less B(2R)$, where $B(R) := \{(x,y)\in\RR^2: x^2+y^2 < R^2\}$, and, for all $R\geq 1$,
\begin{align}
\label{eq:AuxiliarySpecialWeightedH1EstimateCutoffFunction}
|D\zeta_R| &\leq 10 \quad\hbox{on $\RR^2$},
\\
\label{eq:CutoffFunctionWithUniformC2Bound}
|D^2\zeta_R| &\leq 100 \quad\hbox{on $\RR^2$}.
\end{align}
\end{defn}

A straightforward calculation yields

\begin{lem}
\label{lem:RadialCutoffFunction}
For $R\geq 1$ and $\zeta_R$ as in Definition \ref{defn:RadialCutoffFunction}, $D\zeta_R$ and $\supp D^2\zeta_R$ have support in $B(2R)\less \bar B(R)$, and
\begin{align}
\label{eq:RadialCutoffFunctionFirstDerivative}
|D\zeta_R| &\leq 10R^{-1} \quad\hbox{on $\RR^2$},
\\
\label{eq:RadialCutoffFunctionSecondDerivative}
|D^2\zeta_R| &\leq 100R^{-2} \quad\hbox{on $\RR^2$}.
\end{align}
\end{lem}

\begin{lem}[A priori estimate for a solution to the penalized equation]
\label{lem:ExistenceUniquenessEllipticHestonPenalizedProblemPowery}
Assume the hypotheses of Theorem \ref{thm:ExistenceUniquenessEllipticHestonPenalizedProblem}. If $s\geq 1/2$, $y^sf \in L^2(\sO,\fw)$, $y^{2s-1/2}\psi \in H^1(\sO\cup\Gamma_0,\fw)$, $u_\eps\in H^1(\sO\cup\Gamma_0,\fw)$ is a solution to Problem \ref{prob:HestonPenalizedEquation}, and $y^su_\eps \in L^2(\sO,\fw)$, then $y^su_\eps \in H^1(\sO,\fw)$, and
\begin{equation}
\label{eq:PenalizedEquationAPosterioriEstimatePoweryCorrected}
\|y^su_\eps\|_{H^1(\sO,\fw)} \leq C\left(\|y^sf\|_{L^2(\sO,\fw)} + \|(1+y^s)u_\eps\|_{L^2(\sO,\fw)} + \|(1+y^{2s-1/2})\psi^+\|_{H^1(\sO,\fw)} \right),
\end{equation}
where the positive constant $C$ depends only on $s, \gamma$, and the constant coefficients of $A$.
\end{lem}

\begin{proof}
We proceed as in the proof of Lemma \ref{lem:PenalizedEquationAPrioriEstimate}, but now choose $\varphi \in C^\infty_0(\RR^2)$ and $v = \varphi^2(u_\eps-\psi^+)$ to give
\begin{equation}
\label{eq:PenalizedProblem_StrongIncEqn}
a_\lambda(u_\eps,\varphi^2(u_\eps-\psi^+)) + (\beta_\eps(u_\eps)-\beta_\eps(\psi^+), \varphi^2(u_\eps-\psi^+))_H = (f, \varphi^2(u_\eps-\psi^+))_H.
\end{equation}
Since $\beta_\eps$ is strongly monotone by Lemma \ref{lem:StronglyMonotone}, applying \eqref{eq:StronglyMonotonePenalizationOperator} to the preceding identity yields
$$
a_\lambda(u_\eps, \varphi^2(u_\eps-\psi^+)) \leq (f, \varphi^2(u_\eps-\psi^+))_H,
$$
and so
$$
a_\lambda(u_\eps-\psi^+, \varphi^2(u_\eps-\psi^+)) \leq (\varphi f, \varphi(u_\eps-\psi^+))_H - a_\lambda(\psi^+,\varphi^2(u_\eps-\psi^+)).
$$
Applying the commutator energy identity \eqref{eq:BilinearFormSquaredFunction} to the preceding inequality yields
\begin{align*}
{}&a_\lambda(\varphi(u_\eps-\psi^+), \varphi(u_\eps-\psi^+))
\\
&= a_\lambda(u_\eps-\psi^+, \varphi^2(u_\eps-\psi^+)) + ([A,\varphi](u_\eps-\psi^+), \varphi(u_\eps-\psi^+))_H
\\
&\leq (\varphi f, \varphi(u_\eps-\psi^+))_H - a_\lambda(\psi^+,\varphi^2(u_\eps-\psi^+)) + ([A,\varphi](u_\eps-\psi^+), \varphi(u_\eps-\psi^+))_H
\\
&= (\varphi f, \varphi(u_\eps-\psi^+))_H - a_\lambda(\varphi\psi^+,\varphi(u_\eps-\psi^+)) + ([A,\varphi]\psi^+, \varphi(u_\eps-\psi^+))_H
\\
&\quad + ([A,\varphi](u_\eps-\psi^+), \varphi(u_\eps-\psi^+))_H.
\end{align*}
Combining the preceding inequality with the G\r{a}rding inequality \eqref{eq:CoerciveHeston} and the continuity estimate \eqref{eq:ContinuousCoerciveHeston} for $a_\lambda(\cdot,\cdot)$ gives
\begin{align*}
\|\varphi(u_\eps-\psi^+)\|_V^2 &\leq C\left(|\varphi f|_H|\varphi(u_\eps-\psi^+)|_H + \|\varphi \psi^+\|_V\|\varphi(u_\eps-\psi^+)\|_V\right.
\\
&\quad + \left. |\varphi[A,\varphi]\psi^+|_H|(u_\eps-\psi^+)|_H + |([A,\varphi](u_\eps-\psi^+), \varphi(u_\eps-\psi^+))_H|\right),
\end{align*}
where the constant $C$ depends only on the constant coefficients of $A$. Applying the commutator identity \eqref{eq:ACommutator} and the commutator estimate \eqref{eq:CommutatorInnerProductEstimate} in the preceding inequality yields
\begin{align*}
\|\varphi(u_\eps-\psi^+)\|_V^2 &\leq C\left(|\varphi f|_H|\varphi(u_\eps-\psi^+)|_H + \|\varphi \psi^+\|_V\|\varphi(u_\eps-\psi^+)\|_V \right)
\\
&\quad + C\left(|y\varphi|D\varphi||D\psi^+||_H + |y\varphi|D^2\varphi|\psi^+|_H + |(1+y)\varphi|D\varphi|\psi^+|_H\right)|(u_\eps-\psi^+)|_H
\\
&\quad + C|y^{1/2}(|D\varphi| + |D\varphi|^{1/2})(u_\eps-\psi^+)|_H^2,
\end{align*}
where the constant $C$ depends only on $\gamma$ and the constant coefficients of $A$. Let $\zeta_R \in C^\infty_0(\RR^2)$ be the cutoff function in Definition \ref{defn:RadialCutoffFunction} and choose $\varphi = \zeta_Ry^s$, so
\begin{align*}
|D\varphi| &\leq y^s|D\zeta_R| + s\zeta_Ry^{s-1},
\\
|D^2\varphi| &\leq y^s|D^2\zeta_R| + 2s|D\zeta_R| y^{s-1} + s|s-1|\zeta_Ry^{s-2},
\end{align*}
noting that $D\zeta_R$ and $D^2\zeta_R$ are supported in $\bar B(2R)\less B(R)$. Substituting these pointwise inequalities into the preceding estimate for $\|\varphi(u_\eps-\psi^+)\|_V$, using
$$
\|\varphi \psi^+\|_V^2 = |y^{1/2}D(\varphi \psi^+)|_H^2 + |(1+y)^{1/2}\varphi \psi^+|_H^2 \quad\hbox{(by Definition \ref{defn:H1WeightedSobolevSpaces})},
$$
and $D(\varphi \psi^+) = \varphi D\psi^+ + (D\varphi) \psi^+$, and
$$
|D(\varphi \psi^+)| \leq \zeta_Ry^sD\psi^+ + (y^s|D\zeta_R| + s\zeta_Ry^{s-1})\psi^+,
$$
gives
\begin{align*}
\|\varphi(u_\eps-\psi^+)\|_V^2 &\leq C|y^s f|_H|y^s(u_\eps-\psi^+)|_H
\\
&\quad + C\left(|y^{s-1/2}\psi^+|_H + |y^{s+1/2}D\psi^+|_H + |(1+y)^{1/2}y^s\psi^+|_H \right)\|\varphi(u_\eps-\psi^+)\|_V
\\
&\quad + C\left(|(y^{2s+1}|D\zeta_R| + y^{2s})|D\psi^+||_H + |(y^{2s+1}|D^2\zeta_R| + y^{2s}|D\zeta_R| + y^{2s-1})\psi^+|_H \right.
\\
&\qquad + \left. |(1+y)(y^{2s}|D\zeta_R| + y^{2s-1})\psi^+|_H\right)|(u_\eps-\psi^+)|_H
\\
&\quad + C|y^{1/2}(|(y^s|D\zeta_R| + y^{s-1})| + |(y^s|D\zeta_R| + y^{s-1})|^{1/2})(u_\eps-\psi^+)|_H^2,
\end{align*}
where the positive constant $C$ depends only on $s, \gamma$, and the constant coefficients of $A$. By using rearrangement and taking square roots, we obtain
\begin{align*}
\|\varphi(u_\eps-\psi^+)\|_V &\leq C\left(|y^s f|_H + |y^s(u_\eps-\psi^+)|_H \right.
\\
&\quad + |y^{s-1/2}\psi^+|_H + |y^{s+1/2}D\psi^+|_H + |(1+y)^{1/2}y^s\psi^+|_H
\\
&\quad + |(y^{2s+1}|D\zeta_R| + y^{2s})|D\psi^+||_H + |(y^{2s+1}|D^2\zeta_R| + y^{2s}|D\zeta_R| + y^{2s-1})\psi^+|_H
\\
&\quad + |(1+y)(y^{2s}|D\zeta_R| + y^{2s-1})\psi^+|_H + |(u_\eps-\psi^+)|_H
\\
&\quad + \left. |y^{1/2}(|(y^s|D\zeta_R| + y^{s-1})| + |(y^s|D\zeta_R| + y^{s-1})|^{1/2})(u_\eps-\psi^+)|_H \right).
\end{align*}
Applying Lemma \ref{lem:RadialCutoffFunction} to estimate $y|D\zeta_R| \leq 10$ and $y^2|D^2\zeta_R| \leq 100$ yields
\begin{align*}
\|\zeta_Ry^s(u_\eps-\psi^+)\|_V &\leq C\left(|y^s f|_H + |y^s(u_\eps-\psi^+)|_H \right.
\\
&\quad + |y^{s-1/2}\psi^+|_H + |y^{s+1/2}D\psi^+|_H + |(1+y)^{1/2}y^s\psi^+|_H
\\
&\quad + |y^{2s}D\psi^+|_H + |y^{2s-1}\psi^+|_H + |(1+y)y^{2s-1}\psi^+|_H
\\
&\quad + |(u_\eps-\psi^+)|_H + \left. |y^{1/2}(y^{s-1} + y^{(s-1)/2})(u_\eps-\psi^+)|_H \right).
\end{align*}
But
\begin{align*}
\|\zeta_Ry^s(u_\eps-\psi^+)\|_V^2 &= |y^{1/2}D(\zeta_Ry^s(u_\eps-\psi^+))|_H^2 + |(1+y)^{1/2}\zeta_Ry^s(u_\eps-\psi^+)|_H^2,
\\
D(\zeta_Ry^s(u_\eps-\psi^+)) &= \zeta_RD(y^s(u_\eps-\psi^+)) + (D\zeta_R)y^s(u_\eps-\psi^+),
\end{align*}
and so
\begin{align*}
{}&|y^{1/2}\zeta_RD(y^s(u_\eps-\psi^+))|_H + |(1+y)^{1/2}\zeta_Ry^s(u_\eps-\psi^+)|_H
\\
&\leq |y^{1/2}D(\zeta_Ry^s(u_\eps-\psi^+))|_H + |y^{s+1/2}|D\zeta_R|(u_\eps-\psi^+)|_H + |(1+y)^{1/2}\zeta_Ry^s(u_\eps-\psi^+)|_H
\\
&\leq |y^{1/2}D(\zeta_Ry^s(u_\eps-\psi^+))|_H + |(1+y)^{1/2}\zeta_Ry^s(u_\eps-\psi^+)|_H + 10|y^{s-1/2}(u_\eps-\psi^+)|_H
\\
&\quad\hbox{(by \eqref{eq:RadialCutoffFunctionFirstDerivative})}
\\
&\leq \|\zeta_Ry^s(u_\eps-\psi^+)\|_V + 10|y^{s-1/2}(u_\eps-\psi^+)|_H.
\end{align*}
Combining the preceding inequality with the preceding estimate for $\|\zeta_Ry^s(u_\eps-\psi^+)\|_V$, taking the limit as $R\to\infty$, and applying the dominated convergence theorem and the Definition \ref{defn:H1WeightedSobolevSpaces} of $\|\cdot\|_V$, yields
\begin{align*}
\|y^s(u_\eps-\psi^+)\|_V &\leq C\left(|y^s f|_H + |y^s(u_\eps-\psi^+)|_H \right.
\\
&\quad + |y^{s-1/2}\psi^+|_H + |y^{s+1/2}D\psi^+|_H + |(1+y)^{1/2}y^s\psi^+|_H
\\
&\quad + |y^{2s}D\psi^+|_H + |y^{2s-1}\psi^+|_H + |(1+y)y^{2s-1}\psi^+|_H
\\
&\quad + |(u_\eps-\psi^+)|_H + \left. |y^{1/2}(y^{s-1} + y^{(s-1)/2})(u_\eps-\psi^+)|_H \right).
\end{align*}
Finally, since
\begin{align*}
{}&|y^{s-1/2}\psi^+|_H + |y^{s+1/2}D\psi^+|_H + |(1+y)^{1/2}y^s\psi^+|_H
\\
&\quad + |y^{2s}D\psi^+|_H + |y^{2s-1}\psi^+|_H + |(1+y)y^{2s-1}\psi^+|_H
\\
&\leq C\left(|y^{1/2}D((1 + y^{2s-1/2})\psi^+)|_H + |(1+y)^{1/2}(1 + y^{2s-1/2})\psi^+|_H\right)
\\
&\leq C\|(1+y^{2s-1/2})\psi^+\|_V,
\end{align*}
when $s\geq 1/2$ and where $C$ depends only on $s$, we obtain \eqref{eq:PenalizedEquationAPosterioriEstimatePoweryCorrected} from the preceding estimate for $\|y^s(u_\eps-\psi^+)\|_V$.
\end{proof}

To obtain a comparison principle for solutions to the Heston penalized equation, we require

\begin{hyp}[Condition on the upper envelope and obstacle functions]
\label{hyp:UpperBoundObstacleFunction}
Suppose $\psi$ is as in Problem \ref{prob:HomogeneousHestonVIProblem}. Require that there is a function $M\in H^2(\sO,\fw)$ which obeys
\begin{equation}
\label{eq:MgeqPsi}
\psi \leq M \quad\hbox{a.e. on }\sO.
\end{equation}
\end{hyp}

We then obtain

\begin{lem}[A priori comparison principle for a solution to the penalized equation]
\label{lem:ComparisionPrinciplePenalizedHeston}
Suppose there are $M, m \in H^2(\sO,\fw)$ obeying  \eqref{eq:SourceFunctionTraceBounds}, \eqref{eq:mMIneqOnDomain}, \eqref{eq:SourceFunctionBounds}, and \eqref{eq:MgeqPsi}. Let $f\in L^2(\sO,\fw)$ and require that $f$ obeys \eqref{eq:fBoundsCoercive}. If $u_\eps\in H^1_0(\sO\cup\Gamma_0,\fw)$ is a solution to Problem \ref{prob:HestonPenalizedEquation}, then $M, m$, and $u_\eps$ obey
\begin{equation}
\label{eq:uepsBounds}
m \leq u_\eps \leq M \quad\hbox{a.e. on }\sO.
\end{equation}
\end{lem}

\begin{proof}
We first show that $u_\eps\leq M$ on $\sO$. We have $(u_\eps - M)^+ \in H^1(\sO,\fw)$ by Lemma \ref{lem:SobolevSpaceClosedUnderMaxPart}. Since $M\geq 0$ on $\Gamma_1$ by \eqref{eq:SourceFunctionTraceBounds} and $u_\eps = 0$ on $\Gamma_1$ (trace sense) by Lemma \ref{lem:PartialEvansTraceZero}, we have $u_\eps-M\leq 0$ on $\Gamma_1$ (trace sense) and thus $(u_\eps-M)^+=0$ on $\Gamma_1$ (trace sense). Therefore $(u_\eps-M)^+ \in H^1_0(\sO\cup\Gamma_0,\fw)$ and we can substitute $v := (u_\eps - M)^+$ in \eqref{eq:PenalizedProblem} to give
$$
a_\lambda(u_\eps,(u_\eps-M)^+) + (\beta_\eps(u_\eps),(u_\eps-M)^+)_H = (f,(u_\eps-M)^+)_H.
$$
If $u_\eps(P)> M(P)$, so $(u_\eps(P)-M(P))^+ > 0$ for some $P\in\sO$, then $u_\eps(P)>\psi(P)$ since $M\geq \psi$ a.e. on $\sO$ by \eqref{eq:MgeqPsi}, and thus $(\psi(P)-u_\eps(P))^+=0$; on the other hand, if $u_\eps(P)\leq M(P)$ for some $P\in\sO$, then $(u_\eps(P)-M(P))^+ = 0$. Therefore, \eqref{eq:PenalizationOperator} gives
$$
(\beta_\eps(u_\eps),(u_\eps-M)^+)_H = -\frac{1}{\eps}((\psi-u_\eps)^+, (u_\eps-M)^+)_H = 0.
$$
Next, integration by parts (Lemma \ref{lem:HestonIntegrationByParts}) yields
$$
a_\lambda(M,(u_\eps-M)^+) = (A_\lambda M,(u_\eps-M)^+)_H.
$$
By subtracting the preceding two equations we obtain
$$
a_\lambda(u_\eps-M,(u_\eps-M)^+) + (\beta_\eps(u_\eps),(u_\eps-M)^+)_H = (f-A_\lambda M,(u_\eps-M)^+)_H,
$$
and thus
$$
a_\lambda((u_\eps-M)^+,(u_\eps-M)^+) = (f-A_\lambda M,(u_\eps-M)^+)_H.
$$
Hence,
\begin{align*}
\nu_1\|(u_\eps-M)^+\|_V^2 &\leq a_\lambda((u_\eps-M)^+,(u_\eps-M)^+) \quad\hbox{(by \eqref{eq:CoerciveHeston})}
\\
&= (f-A_\lambda M,(u_\eps-M)^+)_H
\\
&\leq 0 \quad\hbox{(by \eqref{eq:fBoundsCoercive})}.
\end{align*}
Therefore, $(u_\eps-M)^+ = 0$ a.e. on $\sO$ and so $u_\eps \leq M$ a.e. on $\sO$. (Note that since $M, m$ obey \eqref{eq:mMIneqOnDomain} and \eqref{eq:SourceFunctionBounds}, then $M, m$ obey \eqref{eq:AlambdamMIneqOnDomain}.)

Next we show that $u_\eps\geq m$ on $\sO$. We have $(u_\eps-m)^- \in H^1(\sO,\fw)$ by Lemma \ref{lem:SobolevSpaceClosedUnderMaxPart}. Since $m\leq 0$ on $\Gamma_1$ by \eqref{eq:SourceFunctionTraceBounds} and $u_\eps = 0$ on $\Gamma_1$ (trace sense) by Lemma \ref{lem:PartialEvansTraceZero}, we have $u_\eps-m\geq 0$ on $\Gamma_1$ (trace sense) and thus $(u_\eps-m)^-=0$ on $\Gamma_1$ (trace sense). Therefore $(u_\eps-m)^- \in H^1_0(\sO\cup\Gamma_0,\fw)$ and we can substitute $v := (u_\eps-m)^-$ in \eqref{eq:PenalizedProblem} to give
$$
a_\lambda(u_\eps,(u_\eps-m)^-) + (\beta_\eps(u_\eps),(u_\eps-m)^-)_H = (f,(u_\eps-m)^-)_H,
$$
while integration by parts (Lemma \ref{lem:HestonIntegrationByParts}) yields
$$
a_\lambda(m,(u_\eps-m)^-) = (A_\lambda m,(u_\eps-m)^-)_H.
$$
Subtracting,
$$
a_\lambda(u_\eps-m,(u_\eps-m)^-) + (\beta_\eps(u_\eps),(u_\eps-m)^-)_H = (f-A_\lambda m,(u_\eps-m)^-)_H.
$$
Thus,
$$
-a_\lambda((u_\eps-m)^-,(u_\eps-m)^-) + (\beta_\eps(u_\eps),(u_\eps-m)^-)_H = (f-A_\lambda m,(u_\eps-m)^-)_H.
$$
Hence, \eqref{eq:PenalizationOperator} gives
$$
a_\lambda((u_\eps-m)^-,(u_\eps-m)^-) + \frac{1}{\eps}((\psi-u_\eps)^+, (u_\eps-m)^-)_H = (A_\lambda m - f,(u_\eps-m)^-)_H.
$$
By \eqref{eq:CoerciveHeston} and the facts that $(\psi-u_\eps)^+\geq 0$, $(u_\eps-m)^-\geq 0$, and $A_\lambda m - f \leq 0$ a.e. on $\sO$ by \eqref{eq:fBoundsCoercive},
\begin{align*}
\nu_1\|(u_\eps-m)^-\|_V &\leq a_\lambda((u_\eps-m)^-,(u_\eps-m)^-)
\\
&\leq (A_\lambda m - f,(u_\eps-m)^-)_H
\\
&\leq 0.
\end{align*}
Therefore, $(u_\eps-m)^- = 0$ a.e. on $\sO$ and so $u_\eps \geq m$ a.e. on $\sO$.
\end{proof}

\subsection{Existence and uniqueness of solutions to the coercive variational inequality}
\label{subsec:CoerciveStatVarInequality}
We will need to consider a coercive version of Problem \ref{prob:HomogeneousHestonVIProblem}

\begin{prob}[Coercive variational inequality with homogeneous Dirichlet boundary condition]
\label{prob:HomogeneousHestonVIProblemCoercive}
We call $u\in H^1_0(\sO\cup\Gamma_0,\fw)$ a solution to the coercive variational inequality for the Heston operator with homogeneous Dirichlet boundary condition along $\Gamma_1$ if $u$ is a solution to Problem \ref{prob:HomogeneousHestonVIProblem} when \eqref{eq:VIProblemHestonHomgeneous} is replaced by
\begin{equation}
\label{eq:VIProblemHestonHomogeneousCoercive}
\begin{gathered}
a_\lambda(u,v-u) \geq (f,v-u)_{L^2(\sO,\fw)} \quad\hbox{with } u\geq\psi \hbox{ a.e. on }\sO,
\\
\quad \forall v\in H^1_0(\sO\cup\Gamma_0,\fw) \hbox{ with } v\geq\psi \hbox{ a.e. on }\sO.
\end{gathered}
\end{equation}
\end{prob}

\begin{lem}[A priori estimate for solutions to the coercive variational inequality]
\label{lem:VIExistenceUniquenessEllipticCoerciveHeston}
If $u\in H^1(\sO\cup\Gamma_0,\fw)$ is the solution to Problem \ref{prob:HomogeneousHestonVIProblemCoercive}, then
\begin{equation}
\label{eq:EllipticCoerciveHestonAPosterioriEstimate}
\|u\|_{H^1(\sO,\fw)} \leq C\left(\|f\|_{L^2(\sO,\fw)} + \|\psi^+\|_{H^1(\sO,\fw)}\right),
\end{equation}
where $C$ depends only on the constant coefficients of $A$.
\end{lem}

\begin{proof}
Substituting $v=\psi^+$ in \eqref{eq:VIProblemHestonHomogeneousCoercive} and simplifying yields
$$
-a_\lambda(u,u) \geq (f,\psi)_{L^2(\sO,\fw)} - (f,u)_{L^2(\sO,\fw)} - a_\lambda(u,\psi^+).
$$
Applying \eqref{eq:ContinuousCoerciveHeston} and \eqref{eq:CoerciveHeston}, we obtain
$$
\nu_1\|u\|_{H^1(\sO,\fw)} \leq \|f\|_{L^2(\sO,\fw)}\|\psi\|_{L^2(\sO,\fw)} + \|f\|_{L^2(\sO,\fw)}\|u\|_{L^2(\sO,\fw)} + C\|u\|_{H^1(\sO,\fw)}\|\psi^+\|_{H^1(\sO,\fw)},
$$
and therefore, using rearrangement and taking square roots, we obtain \eqref{eq:EllipticCoerciveHestonAPosterioriEstimate}.
\end{proof}

Recall that $a_\lambda(\cdot,\cdot)$ is defined by \eqref{eq:BilinearFormCoerciveHeston}. By analogy with \cite[Theorem 3.1.1]{Bensoussan_Lions} we have:

\begin{thm}[Existence and uniqueness for the coercive variational inequality]
\label{thm:VIExistenceUniquenessEllipticCoerciveHeston}
There exists a unique solution $u\in\KK$ to Problem \ref{prob:HomogeneousHestonVIProblemCoercive}.
\end{thm}

\begin{proof}[Proof of uniqueness in Theorem \ref{thm:VIExistenceUniquenessEllipticCoerciveHeston}]
The proof of uniqueness is almost identical to that of the proof of \cite[Theorem 3.1.1]{Bensoussan_Lions}. Indeed, suppose $u_1,u_2$ are possible solutions. We take $v=u_2$ (respectively, $v=u_1$) in the inequality \eqref{eq:VIProblemHestonHomogeneousCoercive} relating to $u_1$ (respectively, $u_2$), to give
\begin{align*}
a_\lambda(u_1,u_2-u_1) &\geq (f,u_2-u_1)_H,
\\
a_\lambda(u_2,u_1-u_2) &\geq (f,u_1-u_2)_H,
\end{align*}
and by adding, we obtain
$$
-a_\lambda(u_1-u_2,u_1-u_2) \geq 0,
$$
and hence
$$
\nu_1\|u_1-u_2\|_V^2 \leq a_\lambda(u_1-u_2,u_1-u_2) \leq 0,
$$
and so $u_1  = u_2$ a.e. on $\sO$.
\end{proof}

\begin{proof}[Proof of existence in Theorem \ref{thm:VIExistenceUniquenessEllipticCoerciveHeston}]
Given $\eps>0$, let $u_\eps$ be the unique solution to \eqref{eq:PenalizedProblem} produced by Theorem \ref{thm:ExistenceUniquenessEllipticHestonPenalizedProblem}. By \eqref{eq:PenalizedEquationAPrioriEstimate}, the sequence $\{u_\eps\}_{\eps\in(0,1]} \subset H^1_0(\sO\cup\Gamma_0,\fw)$ is uniformly bounded in $H^1(\sO,\fw)$. Therefore, we can extract a subsequence, also denoted by $\{u_\eps\}_{\eps\in(0,1]}$, such that
\begin{equation}
\label{eq:uepsWeaklyH1Convergent}
u_\eps \rightharpoonup u \quad\hbox{weakly in $V$ as } \eps\to 0,
\end{equation}
for some $u \in V$. We deduce from \eqref{eq:PenalizationTermAPrioriEstimate} that
\begin{equation}
\label{eq:PenalizationTermAPrioriEstimate_ProofCoerciveVIExistence}
(\psi-u_\eps)^+ \to 0 \quad\hbox{strongly in $H$ as } \eps\to 0.
\end{equation}
The proof of \eqref{eq:LimitBetaSequenceIsBetaLimit}, replacing $u_m\rightharpoonup u_\eps$ as $m\to\infty$ by $u_\eps\rightharpoonup u$ as $\eps\to 0$, yields
$$
(\psi - u_\eps)^+ \rightharpoonup (\psi - u)^+ \quad\hbox{weakly in $H$ as } \eps\to 0.
$$
Therefore, $(\psi - u)^+=0$ by \eqref{eq:PenalizationTermAPrioriEstimate_ProofCoerciveVIExistence}. Hence,
$$
u \in \KK.
$$
Equation \eqref{eq:PenalizedProblem}, for $v\in H^1_0(\sO\cup\Gamma_0,\fw)$ and $v=u_\eps$, yields
\begin{align*}
a_\lambda(u_\eps,v) + (\beta_\eps(u_\eps),v)_H = (f,v),
\\
a_\lambda(u_\eps,u_\eps) + (\beta_\eps(u_\eps),u_\eps)_H = (f,u_\eps).
\end{align*}
When $v\in \KK$ we have $\beta_\eps(v)=0$ and so, subtracting the second equation from the first, we obtain
$$
a_\lambda(u_\eps,v-u_\eps) - (f,v-u_\eps)_H = (\beta_\eps(v)-\beta_\eps(u_\eps),v-u_\eps)_H \geq 0,
$$
where the inequality follows from \eqref{eq:MonotonePenalizationOperator}. Therefore,
$$
a_\lambda(u_\eps,v) - (f,v-u_\eps)_H \geq a_\lambda(u_\eps,u_\eps),
$$
and hence, taking the limit as $\eps\to 0$ and applying \eqref{eq:uepsWeaklyH1Convergent} and Lemma \ref{lem:BilinearFormWeakLimit}, we find that
$$
a_\lambda(u,v) - (f,v-u)_H \geq \liminf_{\eps\to 0}a_\lambda(u_\eps,u_\eps) \geq a_\lambda(u,u).
$$
Therefore,
$$
a_\lambda(u,v-u) \geq (f,v-u)_H, \quad\forall v\in \KK,
$$
and $u$ is a solution to the variational inequality \eqref{eq:VIProblemHestonHomogeneousCoercive}, as desired.
\end{proof}

\begin{cor}[A posteriori estimate for a solution to the coercive variational inequality]
\label{cor:VIExistenceUniquenessEllipticCoerciveHestonPowery}
Suppose $s\geq 1/2$ and that there are functions $M, m \in H^2(\sO,\fw)$ obeying \eqref{eq:SourceFunctionTraceBounds}, \eqref{eq:mMIneqOnDomain},\eqref{eq:SourceFunctionBounds}, \eqref{eq:MgeqPsi}, and
\begin{equation}
\label{eq:1pluspowerymMLq}
(1+y^s)M, (1+y^s)m \in L^q(\sO,\fw) \quad\hbox{for some } q>2.
\end{equation}
Require that $f \in L^2(\sO,\fw)$ obeys \eqref{eq:fBoundsCoercive} and
\begin{equation}
\label{eq:poweryfL2}
y^sf \in L^2(\sO,\fw),
\end{equation}
while $\psi \in H^1(\sO,\fw)$ obeys
\begin{equation}
\label{eq:1pluspoweryfpsiH1}
(1+y^{2s-1/2})\psi^+ \in H^1(\sO,\fw).
\end{equation}
If $u\in H^1(\sO\cup\Gamma_0,\fw)$ is the unique solution to Problem \ref{prob:HomogeneousHestonVIProblemCoercive}, then $y^su \in H^1(\sO,\fw)$, and
\begin{equation}
\label{eq:EllipticCoerciveHestonAPosterioriEstimatePoweryCorrected}
\|y^su\|_{H^1(\sO,\fw)} \leq C\left(\|y^sf\|_{L^2(\sO,\fw)} + \|(1+y^s)u\|_{L^2(\sO,\fw)} + \|(1+y^{2s-1/2})\psi^+\|_{H^1(\sO,\fw)} \right),
\end{equation}
where $C$ depends only on $s$ and the constant coefficients of $A$.
\end{cor}

\begin{proof}
Let $\{u_\eps\}_{\eps\in(0,1]} \subset H^1_0(\sO\cup\Gamma_0,\fw)$ be the sequence defined in the proof of existence in Theorem \ref{thm:VIExistenceUniquenessEllipticCoerciveHeston}, with $u_\eps \rightharpoonup u$ weakly in $H^1(\sO,\fw)$ as $\eps\to 0$ by \eqref{eq:uepsWeaklyH1Convergent}.
Lemma \ref{lem:ComparisionPrinciplePenalizedHeston} implies that $M, m$, and the $u_\eps$ obey the pointwise bound \eqref{eq:uepsBounds}. Therefore, by \eqref{eq:uepsBounds} and \eqref{eq:1pluspowerymMLq}, we have
\begin{equation}
\label{eq:1pluspoweryuBoundbyMm}
(1+y^s)|u_\eps| \leq (1+y^s)(|M|+|m|) \quad\hbox{a.e. on }\sO.
\end{equation}
Hence, Corollary \ref{cor:WeakUpperLowerBoundsIsStrongLp} implies that, after passing to a subsequence, we may suppose
\begin{equation}
\label{eq:1pluspoweryuepsL2StrongLimit}
(1+y^s)u_\eps \to (1+y^s)u \quad\hbox{strongly in } L^2(\sO,\fw) \hbox{ as }\eps\to 0.
\end{equation}
Taking limits as $\eps\to 0$ in the inequality \eqref{eq:PenalizedEquationAPosterioriEstimatePoweryCorrected} yields
$$
\liminf_{\eps\to 0}\|y^su_\eps\|_{H^1(\sO,\fw)} \leq C\left(\|y^sf\|_{L^2(\sO,\fw)} + \|(1+y^s)u\|_{L^2(\sO,\fw)} + \|(1+y^{2s-1/2})\psi^+\|_{H^1(\sO,\fw)} \right).
$$
Moreover, \eqref{eq:PenalizedEquationAPosterioriEstimatePoweryCorrected} and \eqref{eq:1pluspoweryuBoundbyMm} imply that the sequence $\{y^su_\eps\}_{\eps\in(0,1]}$ is uniformly bounded in $H^1(\sO,\fw)$ and so, after passing to a subsequence again, we may assume that $y^su_\eps\rightharpoonup v$ weakly in $H^1(\sO,\fw)$ as $\eps\to 0$, for some $v \in H^1_0(\sO\cup\Gamma_),\fw)$. Since $y^su_\eps \to y^su$ strongly in $L^2(\sO,\fw)$ as $\eps\to 0$ by \eqref{eq:1pluspoweryuepsL2StrongLimit},  we must have $v  = y^su$ a.e. on $\sO$. Because $\|y^su\|_{H^1(\sO,\fw)} \leq \liminf_{\eps\to 0}\|y^su_\eps\|_{H^1(\sO,\fw)}$, we obtain the inequality \eqref{eq:EllipticCoerciveHestonAPosterioriEstimatePoweryCorrected} from the preceding estimate for $\liminf_{\eps\to 0}\|y^su_\eps\|_{H^1(\sO,\fw)}$.
\end{proof}

\begin{rmk}[A posteriori estimate]
The estimate in \ref{cor:VIExistenceUniquenessEllipticCoerciveHestonPowery} is only a posteriori because we obtain it by taking limits of the solutions to the penalized equation and rely on the fact that the solution to the coercive variational inequality is unique.
\end{rmk}

By analogy with \cite[Theorem 3.1.4]{Bensoussan_Lions}, which assumes that $\sO$ is bounded and $A$ is uniformly elliptic with bounded coefficients, we have:

\begin{thm}[A priori comparison principle for the coercive variational inequality]
\label{thm:VIExistenceUniquenessEllipticCoerciveHestonIncreasingF}
Assume the hypotheses of Theorem \ref{thm:VIExistenceUniquenessEllipticCoerciveHeston}. If
\begin{align}
\label{eq:f2geqf1}
f_2 \geq f_1 \quad \hbox{a.e. on }\sO,
\\
\label{eq:psi2geqpsi1}
\psi_2 \geq \psi_1\quad \hbox{a.e. on }\sO,
\end{align}
and $u_i$ solve Problem \ref{prob:HomogeneousHestonVIProblemCoercive} with $f,\psi$ replaced by $f_i,\psi_i$, $i=1,2$, then $u_2\geq u_1$ a.e. on
\end{thm}

\begin{proof}
In the variational inequality \eqref{eq:VIProblemHestonHomogeneousCoercive} for $u_1$,
$$
a_\lambda(u_1,v-u_1) \geq (f_1,v-u_1)_H,
$$
we can take $v-u_1 = -(u_2-u_1)^-$, provided $v\geq\psi_1$, to give
\begin{equation}
\label{eq:AprioriIneqOne}
-a_\lambda(u_1,(u_2-u_1)^-) \geq -(f_1,(u_2-u_1)^-)_H.
\end{equation}
Recall that we denote $w=w^+-w^-$, where $w^+ := \max\{w,0\}$ and $w^- := -\min\{w,0\}$. When $u_1(P)\leq u_2(P)$, then $(u_2(P)-u_1(P))^-=0$ and $v(P)=u_1(P)\geq \psi_1(P)$, while if $u_1(P)\geq u_2(P)$, then
\begin{align*}
v(P) &= u_1(P) - (u_2(P)-u_1(P))^-
\\
&= u_1(P) + u_2(P)-u_1(P)
\\
&= u_2(P) \geq \psi_2(P)
\\
&\geq \psi_1(P) \quad\hbox{(by \eqref{eq:psi2geqpsi1}},
\end{align*}
and thus $v \geq \psi_1$ a.e. on $\sO$, as needed for \eqref{eq:AprioriIneqOne}.

In the variational inequality \eqref{eq:VIProblemHestonHomogeneousCoercive} for $u_2$,
$$
a_\lambda(u_2,v-u_2) \geq (f_2,v-u_2)_H,
$$
take $v-u_2 = (u_2-u_1)^-\geq 0$ to give $v\geq u_2 \geq \psi_2$ and thus
\begin{equation}
\label{eq:AprioriIneqTwo}
a_\lambda(u_2,(u_2-u_1)^-) \geq (f_2,(u_2-u_1)^-)_H.
\end{equation}
Adding inequalities \eqref{eq:AprioriIneqOne} and \eqref{eq:AprioriIneqTwo} gives
$$
a_\lambda(u_2-u_1,(u_2-u_1)^-) \geq (f_2-f_1,(u_2-u_1)^-)_H \geq 0 \quad\hbox{(by \eqref{eq:f2geqf1})}.
$$
Consequently, using $a(v^+,v^-)=0$ for all $v\in V$ and applying \eqref{eq:CoerciveHeston} yields
$$
\nu_1\|(u_2-u_1)^-\|_V^2 \leq a_\lambda((u_2-u_1)^-,(u_2-u_1)^-) \leq 0,
$$
so that $(u_2-u_1)^-=0$ and thus $u_2 \geq u_1$ a.e. on $\sO$, as desired.
\end{proof}

To obtain a comparison principle for solutions to the coercive Heston variational inequality, we require

\begin{hyp}[Conditions on envelope functions]
\label{hyp:UpperLowerBoundsSolutionsCoerciveInequality}
There are $M, m \in H^2(\sO,\fw)$ obeying \eqref{eq:SourceFunctionTraceBounds}, \eqref{eq:mMIneqOnDomain}, \eqref{eq:SourceFunctionBounds}, \eqref{eq:MgeqPsi}, and
\begin{equation}
\label{eq:mMInLq}
M, m \in L^q(\sO,\fw) \quad\hbox{for some } q>2.
\end{equation}
\end{hyp}

We then have

\begin{lem}[A posteriori comparison principle for the variational inequality]
\label{lem:ComparisionPrincipleCoerciveHestonVI}
Suppose there are functions $M, m \in H^2(\sO,\fw)$ obeying \eqref{eq:SourceFunctionTraceBounds}, \eqref{eq:mMIneqOnDomain}, \eqref{eq:SourceFunctionBounds}, \eqref{eq:MgeqPsi}, and \eqref{eq:mMInLq}. Let $f\in L^2(\sO,\fw)$ and require that $f$ obeys \eqref{eq:fBoundsCoercive}. If $u\in H^1_0(\sO\cup\Gamma_0,\fw)$ is the unique solution to Problem \ref{prob:HomogeneousHestonVIProblemCoercive}, then $M, m$, and $u$ obey
\begin{equation}
\label{eq:uBoundedObstacleCoercive}
\max\{m,\psi\} \leq u \leq M \quad\hbox{a.e. on }\sO.
\end{equation}
\end{lem}

\begin{proof}
Let $\{u_\eps\}_{\eps\in(0,1]} \subset H^1_0(\sO\cup\Gamma_0,\fw)$ be the sequence defined in the proof of existence in Theorem \ref{thm:VIExistenceUniquenessEllipticCoerciveHeston}, with $u_\eps \rightharpoonup u$ weakly in $H^1(\sO,\fw)$ by \eqref{eq:uepsWeaklyH1Convergent}. Lemma \ref{lem:ComparisionPrinciplePenalizedHeston} implies that $M, m$, and the $u_\eps$ obey the pointwise bound \eqref{eq:uepsBounds}. Therefore, by \eqref{eq:uepsBounds} and \eqref{eq:mMInLq}, Corollary \ref{cor:WeakUpperLowerBoundsIsStrongLp} implies that, after passing to a subsequence, we may suppose $u_\eps \to u$ strongly in $L^2(\sO,\fw)$ and thus, again after passing to a subsequence, pointwise a.e. on $\sO$ by Corollary \ref{cor:Billingsley}. Therefore the conclusion follows by taking pointwise limits in \eqref{eq:uepsBounds}.
\end{proof}

\subsection{Existence and uniqueness of solutions to the non-coercive variational inequality}
\label{subsec:NoncoerciveStatVarInequality}
We can now proceed to the proof of the general, ``non-coercive'' case with the aid of

\begin{hyp}[Conditions on envelope functions]
\label{hyp:UpperLowerBoundsSolutionsNoncoerciveInequality}
There are $M, m \in H^2(\sO,\fw)$ obeying
\begin{align}
\label{eq:Sqrt1plusyMmInLq}
(1+y)^{1/2}M, (1+y)^{1/2}m &\in L^q(\sO,\fw) \quad\hbox{for some } q>2,
\\
\label{eq:OneplusyMmInL2}
(1+y)M, (1+y)m &\in L^2(\sO,\fw).
\end{align}
\end{hyp}

\begin{hyp}[Auxiliary condition for uniqueness of solutions to the non-coercive variational inequality]
\label{hyp:AuxBoundUniquenessSolutionsNoncoerciveInequality}
There is a $\varphi \in H^2(\sO,\fw)$ obeying
\begin{equation}
\label{eq:Sqrt1plusyVarphiInLq}
(1+y)^{1/2}\varphi \in L^q(\sO,\fw) \quad\hbox{for some } q>2.
\end{equation}
\end{hyp}

By analogy with \cite[Theorem 3.1.5]{Bensoussan_Lions}, which assumes that $\sO$ is bounded and $A$ is uniformly elliptic with bounded coefficients, we have

\begin{thm}[Existence and uniqueness of a solution to the non-coercive variational inequality]
\label{thm:VIExistenceUniquenessEllipticHeston_Improved}
Assume \eqref{eq:NoncoerciveHeston} holds. Suppose there are $M, m\in H^2(\sO,\fw)$ obeying \eqref{eq:SourceFunctionTraceBounds}, \eqref{eq:mMIneqOnDomain}, \eqref{eq:SourceFunctionBounds}, \eqref{eq:MgeqPsi}, \eqref{eq:Sqrt1plusyMmInLq}, and \eqref{eq:OneplusyMmInL2}. Given $f \in L^2(\sO,\fw)$ obeying \eqref{eq:fBounds} and $\psi \in H^1(\sO,\fw)$ obeying \eqref{eq:ObstacleFunctionLessThanZero}, there exists a solution, $u\in H^1(\sO\cup\Gamma_0,\fw)$, to Problem \ref{prob:HomogeneousHestonVIProblem} and $u$ obeys
\begin{equation}
\label{eq:uBoundedObstacle}
\max\{m,\psi\} \leq u \leq M \quad\hbox{a.e. on }\sO.
\end{equation}
Moreover, if there is a $\varphi \in H^2(\sO,\fw)$ obeying Hypotheses \ref{hyp:AuxBoundUniquenessSolutionsNoncoerciveEquation} and \ref{hyp:AuxBoundUniquenessSolutionsNoncoerciveInequality}, then the solution, $u$, is unique.
\end{thm}

\begin{rmk}[Role of the hypotheses in Theorem \ref{thm:VIExistenceUniquenessEllipticHeston_Improved}]
\label{rmk:RoleHypothesesVIExistenceUniuqnessTheorem}
The pointwise growth and integral bounds involving $m, M$ are required in the proofs of existence and uniqueness because the
\begin{enumerate}
\item Sobolev embedding theorems may not hold for weighted Sobolev spaces;
\item Rellich-Kondrachov embedding theorems may not hold for weighted Sobolev spaces or unbounded domains;
\item Bilinear form \eqref{eq:HestonBilinearFormGarding} is non-coercive;
\item Domain $\sO$ is unbounded;
\item Coefficients of $A$ are unbounded on $\sO$; and
\item Functions $f$ and $\psi$ may be unbounded.
\end{enumerate}
See \S \ref{subsec:WeightedSobolevContinuousCompactEmbeddings} for additional comments on weighted Sobolev spaces and embedding theorems. The pointwise growth and integral bounds involving $\varphi$ are required in the proof of uniqueness.
\end{rmk}

\begin{lem}[A priori estimate for solutions to the non-coercive variational inequality]
\label{lem:VIExistenceUniquenessEllipticHeston}
Assume the hypotheses of Theorem \ref{thm:VIExistenceUniquenessEllipticHeston_Improved}. If $u\in V$ is a solution to Problem \ref{prob:HomogeneousHestonVIProblem} and $yu \in L^2(\sO,\fw)$, then
\begin{equation}
\label{eq:EllipticHestonAprioriEstimate}
\|u\|_V \leq C\left(\|f\|_H + \|(1+y)u\|_H + \|\psi^+\|_V\right),
\end{equation}
\end{lem}
where $C$ depends only on the constant coefficients of $A$.

\begin{proof}
Setting $f_\lambda := f+\lambda(1+y)u$ and using the definition \eqref{eq:BilinearFormCoerciveHeston} of $a_\lambda$ to write the variational inequality \eqref{eq:VIProblemHestonHomgeneous} as
$$
a_\lambda(u,v-u) \geq (f_\lambda,v-u)_H, \quad \forall v\in \KK,
$$
we see that \eqref{eq:EllipticHestonAprioriEstimate} follows from \eqref{eq:EllipticCoerciveHestonAPosterioriEstimate}.
\end{proof}

We first consider the question of existence in Theorem \ref{thm:VIExistenceUniquenessEllipticHeston_Improved}.

\begin{proof}[Proof of existence in Theorem \ref{thm:VIExistenceUniquenessEllipticHeston_Improved}] We adapt the proof of existence in \cite[Theorem 3.1.5]{Bensoussan_Lions}. By \eqref{eq:fBounds}, we have $AM \geq f$ a.e. on $\sO$ and so $(AM,v)_H \geq (f,v)_H$ for all $v \in V$, $v\geq 0$. Lemma \ref{lem:HestonIntegrationByParts} (which does not require $M=0$ on $\Gamma_1$) implies that $(AM,v)_H = a(M,v)$ and thus
\begin{equation}
\label{eq:StartSolutionSequenceConstruction}
a(M,v) \geq (f,v)_H, \quad\forall v\in V, v\geq 0.
\end{equation}
We shall use \eqref{eq:StartSolutionSequenceConstruction} to help establish the

\begin{claim}[Existence of a monotonically decreasing sequence solving a sequence of coercive variational inequalities]
\label{claim:MonotoneDecreasingSequenceSolvingCoerciveVarIneq}
If $u_0 = M$, then there exists a sequence $\{u_n\}_{n\geq 1} \subset V$ such that
\begin{gather}
 \label{eq:ungeqObstacle}
u_n \geq \psi \quad\hbox{a.e. on }\sO, \quad\forall n \geq 1,
\\
\label{eq:DecreasingSolutionSequenceConstruction}
a(u_n,v-u_n) + \lambda((1+y)u_n,v-u_n)_H \geq (f+\lambda (1+y)u_{n-1},v-u_n)_H,
\\
\notag
\forall v\in V, v\geq \psi \hbox{ a.e. on }\sO, n \geq 1,
\\
\label{eq:DecreasingSolutionSequence}
M \geq u_1 \geq \cdots \geq u_{n-1} \geq u_n \geq \cdots \geq m \quad\hbox{a.e on }\sO.
\end{gather}
\end{claim}

\begin{proof}[Proof of Claim \ref{claim:MonotoneDecreasingSequenceSolvingCoerciveVarIneq}]
Observe that if $u_{n-1}$ obeys $m\leq u_{n-1} \leq M$, as implied by \eqref{eq:DecreasingSolutionSequence}, then \eqref{eq:OneplusyMmInL2} will ensure that $(1+y)u_{n-1} \in L^2(\sO,\fw)$ and Theorem \ref{thm:VIExistenceUniquenessEllipticCoerciveHeston}, with source function
$$
f+\lambda (1+y)u_{n-1} \in L^2(\sO,\fw),
$$
will yield a (unique) solution, $u_n\in V, u_n\geq\psi$, to \eqref{eq:DecreasingSolutionSequenceConstruction}.

We now proceed by induction. To establish \eqref{eq:DecreasingSolutionSequence}, let $n=1$ and choose $v$ in \eqref{eq:DecreasingSolutionSequenceConstruction} such that
$$
v - u_1 = -(u_0-u_1)^-,
$$
that is,
$$
v = \min\{u_0,u_1\} \geq \psi \hbox{ a.e. on }\sO,
$$
using the facts that $u_0\geq \psi$, $u_1\geq \psi$  a.e on $\sO$, $\min\{u_0,u_1\} = 0$ on $\Gamma_1$ (trace sense), and $\min\{u_0,u_1\} \in V$ by Lemma \ref{lem:SobolevSpaceClosedUnderMaxPart}, so that $v\in \KK$. The choice $u_0=M$ is admissible in \eqref{eq:DecreasingSolutionSequenceConstruction} since it is only used to define the source function, $f+\lambda (1+y)u_0$, which does not require $u_0=0$ on $\Gamma_1$. Therefore, \eqref{eq:DecreasingSolutionSequenceConstruction} with $n=1$ gives
$$
-a(u_1,(u_0-u_1)^-) - \lambda((1+y)u_1,(u_0-u_1)^-)_H \geq -(f+\lambda (1+y)u_0,(u_0-u_1)^-)_H.
$$
Choosing $v=(u_0-u_1)^-$ in \eqref{eq:StartSolutionSequenceConstruction} and recalling that $u_0=M$ by hypothesis of Claim \ref{claim:MonotoneDecreasingSequenceSolvingCoerciveVarIneq} yields
$$
a(u_0,(u_0-u_1)^-) \geq (f,(u_0-u_1)^-)_H.
$$
By adding the preceding two inequalities we obtain
$$
a(u_0-u_1,(u_0-u_1)^-) + \lambda((1+y)(u_0-u_1),(u_0-u_1)^-)_H \geq 0,
$$
and thus
$$
a((u_0-u_1)^-,(u_0-u_1)^-) + \lambda\|(1+y)^{1/2}(u_0-u_1)^-\|_H^2 \leq 0.
$$
The preceding inequality and \eqref{eq:CoerciveHeston} gives
$$
\nu_1\|(u_0-u_1)^-\|_V^2 \leq a_\lambda((u_0-u_1)^-,(u_0-u_1)^-) \leq 0,
$$
so $(u_0-u_1)^-=0$ a.e. on $\sO$ and thus
$$
u_0\geq u_1 \quad\hbox{a.e. on }\sO.
$$
We now assume \eqref{eq:DecreasingSolutionSequence} is established for $\{u_1,\ldots,u_{n-1}\}$ and show that $u_{n-1}\geq u_n$ a.e. on $\sO$. We choose $v\in V$ to be defined in \eqref{eq:DecreasingSolutionSequenceConstruction} with $u_n$ and $u_{n-1}$ respectively by
\begin{align*}
v-u_n &= -(u_{n-1}-u_n)^- \quad\hbox{in \eqref{eq:DecreasingSolutionSequenceConstruction} for $u_n$},
\\
v-u_{n-1} &= (u_{n-1}-u_n)^- \quad\hbox{in \eqref{eq:DecreasingSolutionSequenceConstruction} for $u_{n-1}$},
\end{align*}
to give
\begin{align*}
{}&-a(u_n,(u_{n-1}-u_n)^-) - \lambda((1+y)u_n,(u_{n-1}-u_n)^-)_H
\\
&\quad \geq -(f+\lambda (1+y)u_{n-1},(u_{n-1}-u_n)^-)_H,
\\
{}&a(u_{n-1},(u_{n-1}-u_n)^-) + \lambda((1+y)u_{n-1},(u_{n-1}-u_n)^-)_H
\\
&\quad\geq (f+\lambda (1+y)u_{n-2},(u_{n-1}-u_n)^-)_H.
\end{align*}
Adding these inequalities yields
\begin{align*}
{}&a(u_{n-1}-u_n,(u_{n-1}-u_n)^-) + \lambda((1+y)(u_{n-1}-u_n),(u_{n-1}-u_n)^-)_H
\\
&\quad \geq \lambda((1+y)(u_{n-2}-u_{n-1},(u_{n-1}-u_n)^-)_H,
\end{align*}
and thus, by \eqref{eq:BilinearFormCoerciveHeston},
$$
a_\lambda((u_{n-1}-u_n)^-,(u_{n-1}-u_n)^-) \leq \lambda((1+y)(u_{n-1}-u_{n-2},(u_{n-1}-u_n)^-)_H \leq 0,
$$
where we use $u_{n-2}\geq u_{n-1}$ a.e. on $\sO$ to obtain the last inequality.  We conclude that
$$
u_{n-1}\geq u_n \quad\hbox{ a.e. on }\sO, \quad\forall n\geq 1,
$$
just as in the argument that $u_0\geq u_1$ a.e. on $\sO$. This establishes the monotonicity in \eqref{eq:DecreasingSolutionSequence}.

We may simultaneously establish the lower bound in \eqref{eq:DecreasingSolutionSequence}, that is
\begin{equation}
\label{eq:LowerBoundForSolutionSequence}
u_n \geq m \quad\hbox{a.e. on } \sO, \quad\forall n \geq 1.
\end{equation}
Again, we shall assume \eqref{eq:DecreasingSolutionSequence} is established for $\{u_1,\ldots,u_{n-1}\}$ and show that $u_n$ obeys \eqref{eq:LowerBoundForSolutionSequence}; we omit the initial step of assuming $n=1$ and showing that $u_1$ obeys \eqref{eq:LowerBoundForSolutionSequence} since the proof is virtually identical (one just replaces $u_n$ by $u_1$ and $u_{n-1}$ by $u_0=M$). We choose $v$ in \eqref{eq:DecreasingSolutionSequenceConstruction} by setting
$$
v-u_n = (u_n-m)^-,
$$
that is,
$$
v = \max\{u_n, m\}.
$$
We have $v\geq\psi$ a.e. on $\sO$ since $u_n\geq \psi$ a.e. on $\sO$, because $u_n\in \KK$ by its definition in \eqref{eq:DecreasingSolutionSequenceConstruction}. Moreover, since $u_n \in \KK$, we have $u_n\in V$ and $m\leq 0$ on $\Gamma_1$ by \eqref{eq:SourceFunctionTraceBounds}, so it follows from (the proof of) Lemma \ref{lem:SobolevSpaceClosedUnderMaxPart} that
$$
v = \max\{u_n, m\} \in V.
$$
Therefore, $v\in \KK$. From \eqref{eq:DecreasingSolutionSequenceConstruction}, we obtain
$$
a(u_n,(u_n-m)^-) + \lambda((1+y)u_n, (u_n-m)^-)_H \geq (f+\lambda (1+y)u_{n-1}, (u_n-m)^-)_H,
$$
which gives
\begin{align*}
{}& a(u_n-m,(u_n-m)^-) + \lambda((1+y)(u_n-m), (u_n-m)^-)_H
\\
{}&\qquad + a(m,(u_n-m)^-), (u_n-m)^-)_H
\\
{}&\quad = a(u_n-m,(u_n-m)^-) + \lambda((1+y)(u_n-m), (u_n-m)^-)_H
\\
{}&\qquad + (Am,(u_n-m)^-), (u_n-m)^-)_H \quad\hbox{(by Lemma \ref{lem:HestonIntegrationByParts})}
\\
&\quad\geq (f+\lambda (1+y)u_{n-1}, (u_n-m)^-)_H.
\end{align*}
Hence,
\begin{align*}
{}& a(u_n-m,(u_n-m)^-) + \lambda((1+y)(u_n-m), (u_n-m)^-)_H
\\
&\quad \geq (f-Am +\lambda (1+y)u_{n-1}, (u_n-m)^-)_H,
\end{align*}
and so, because $a(u,u^-) = -a(u^-,u^-)$ using $u = u^+ - u^-$ and the Definition \eqref{defn:HestonWithKillingBilinearForm} of $a(u,v)$,
\begin{align*}
{}&a((u_n-m)^-,(u_n-m)^-) + \lambda((1+y)(u_n-m)^-, (u_n-m)^-)_H
\\
&\quad\leq -(f-Am+\lambda (1+y)(u_{n-1}-m), (u_n-m)^-)_H.
\end{align*}
By virtue of \eqref{eq:CoerciveHeston} we then deduce
$$
\nu_1\|(u_n-m)^-\|_V^2 + (f-Am+\lambda (1+y)(u_{n-1}-m), (u_n-m)^-))_H \leq 0.
$$
By the induction hypothesis, $u_{n-1}-m\geq 0$ a.e. on $\sO$ and therefore, using $f-Am \geq 0$ a.e. on $\sO$ from \eqref{eq:fBounds},
$$
(f-Am+\lambda (1+y)(u_{n-1}-m), (u_n-m)^-)_H \geq 0.
$$
Hence, $\|(u_n-m)^-\|_V^2 \leq 0$ and so $(u_n-m)^-=0$ a.e. on $\sO$. Therefore, $u_n$ obeys \eqref{eq:LowerBoundForSolutionSequence}. By induction, the sequence $\{u_n\}_{n\geq 1}$ obeys \eqref{eq:DecreasingSolutionSequence} and this completes the proof of Claim \ref{claim:MonotoneDecreasingSequenceSolvingCoerciveVarIneq}.
\end{proof}

By taking $v=v_0\in\KK$ (see Assumption \ref{assump:NonemptyConvexSetHeston}) in \eqref{eq:DecreasingSolutionSequenceConstruction} and using \eqref{eq:BilinearFormCoerciveHeston} we obtain
\begin{align*}
{}\nu_1\|u_n\|_V^2 &\leq a(u_n,u_n) + \lambda((1+y)u_n,u_n)_H
\\
&\leq a(u_n,v_0) + \lambda((1+y)u_n,v_0)_H - (f+\lambda (1+y)u_{n-1},v_0-u_n)_H.
\end{align*}
Thus,
\begin{align*}
\nu_1\|u_n\|_V^2 &\leq C\|u_n\|_V\|v_0\|_V + \lambda\|(1+y)^{1/2}u_n\|_H\|(1+y)^{1/2}v_0\|_H
\\
&\quad + \|f\|_H\|u_n\|_H + \lambda\|(1+y)^{1/2}u_{n-1}\|_H\|(1+y)^{1/2}u_n\|_H
\\
&\quad + \|f\|_H\|v_0\|_H + \lambda\|(1+y)^{1/2}u_{n-1}\|_H\|(1+y)^{1/2}v_0\|_H, \quad \forall n\geq 1.
\end{align*}
But $(1+y)^{1/2}u_n$ is uniformly $L^2(\sO,\fw)$ bounded for all $n\geq 1$ by \eqref{eq:DecreasingSolutionSequence}, the fact that $M, m \in H^2(\sO,\fw)$, and the Definition \ref{defn:H2WeightedSobolevSpaces} of $H^2(\sO,\fw)$. Therefore, the preceding inequality gives
$$
\nu_1\|u_n\|_V^2 \leq C_1\|u_n\|_V + C_2, \quad \forall n\geq 1,
$$
for some $0<C_1,C_2<\infty$ and thus
\begin{equation}
\label{eq:BoundedDecreasingSequence}
\|u_n\|_V \leq C, \quad \forall n\geq 1,
\end{equation}
for some constant $C$ independent of $n\geq 1$. Also, \eqref{eq:DecreasingSolutionSequence} implies that
$$
(1+y)^{1/2}|u_n| \leq (1+y)^{1/2}(1+|m|+|M|) \quad\hbox{a.e. on } \sO, \quad\forall n\geq 1.
$$
Therefore, since $(1+y)^{1/2}(1+|m|+|M|) \in L^q(\sO,\fw)$ by \eqref{eq:Sqrt1plusyMmInLq}, Corollary \ref{cor:WeakUpperLowerBoundsIsStrongLp} (with $r=2$) implies that, after passing to a subsequence,
\begin{equation}
\label{eq:StrongL2ConvergenceOfDecreasingSequence}
(1+y)^{1/2}u_n \to (1+y)^{1/2}u \quad\hbox{strongly in $L^2(\sO,\fw)$ as $n\to\infty$.}
\end{equation}
We deduce from \eqref{eq:BoundedDecreasingSequence} that, after passing to a subsequence,
\begin{equation}
\label{eq:WeakConvergenceOfDecreasingSequence}
u_n \rightharpoonup u \quad\hbox{weakly in $V$ as $n\to\infty$.}
\end{equation}
Moreover, by passing to a subsequence, $u_n\to u$ pointwise a.e. on $\sO$ as $n\to\infty$ by \eqref{eq:StrongL2ConvergenceOfDecreasingSequence} and Corollary \ref{cor:Billingsley}. Thus, taking pointwise limits in \eqref{eq:ungeqObstacle} and \eqref{eq:DecreasingSolutionSequence}, we obtain
$$
\max\{m,\psi\} \leq u \leq M \quad \hbox{a.e. on }\sO,
$$
and therefore $u$ obeys the desired bounds \eqref{eq:uBoundedObstacle}. Moreover, for all $v\in \KK$,
\begin{align*}
{}&a(u_n,v) + \lambda((1+y)u_n,v)_H - (f,v-u_n)_H
\\
&= a(u_n,v-u_n) + \lambda((1+y)u_n,v-u_n)_H + \lambda((1+y)u_n,v)_H - (f,v-u_n)_H
\\
&\quad + a(u_n,u_n) - \lambda((1+y)u_n,v-u_n)_H
\\
&\geq (f+\lambda (1+y)u_{n-1},v-u_n)_H + \lambda((1+y)u_n,v)_H - (f,v-u_n)_H
\\
&\quad + a(u_n,u_n) - \lambda((1+y)u_n,v-u_n)_H
\quad\hbox{(by \eqref{eq:DecreasingSolutionSequenceConstruction})}
\\
&= \lambda((1+y)u_{n-1},v-u_n)_H + \lambda((1+y)u_n,v)_H + a(u_n,u_n)
\\
&\quad - \lambda((1+y)u_n,v-u_n)_H
\\
&= a(u_n,u_n) + \lambda((1+y)u_n,u_n)_H + \lambda((1+y)u_{n-1},v-u_n)_H.
\end{align*}
Therefore,
\begin{align*}
{}&a(u_n,v) + \lambda((1+y)u_n,v)_H - (f,v-u_n)_H
\\
&\quad \geq a_\lambda(u_n,u_n) + \lambda((1+y)u_{n-1},v-u_n)_H.
\end{align*}
Taking limits of both sides of the preceding inequality as $n\to\infty$,
\begin{align*}
{}&a(u,v) + \lambda((1+y)u,v)_H - (f,v-u)_H
\\
&\geq \liminf_{n\to\infty}a_\lambda(u_n,u_n) + \lim_{n\to\infty}\lambda((1+y)u_{n-1},v)_H  - \lim_{n\to\infty}\lambda((1+y)u_{n-1},u_n)_H
\\
&\geq a_\lambda(u,u) + \lambda((1+y)u,v)_H - \lambda((1+y)u,u)_H \quad\hbox{(by Lemma \ref{lem:BilinearFormWeakLimit})}
\\
&= a(u,u) + \lambda((1+y)u,v)_H,
\end{align*}
so that
$$
a(u,v-u) - (f,v-u)_H \geq 0, \quad \forall v\in \KK,
$$
and $u$ is the desired solution.
\end{proof}

\begin{rmk}[Alternative approach to proof of existence]
In the proof of existence in Theorem \ref{thm:VIExistenceUniquenessEllipticHeston_Improved} we could alternatively have chosen an increasing sequence, $u_n$, $n\geq 0$, starting from $u_0=0$, as suggested in \cite[p. 201]{Bensoussan_Lions}.
\end{rmk}


We next consider the question of uniqueness. We follow the broad outline of the proof of uniqueness in \cite[Theorem 3.1.5]{Bensoussan_Lions}, but adapted to take account of the complications described in Remark \ref{rmk:RoleHypothesesVIExistenceUniuqnessTheorem}. First, we shall need a preliminary reduction to the case of uniqueness when $f$ is positive and the solution $u$ is non-negative analogous to Lemmas \ref{lem:EllipticHestonReductionNonnegativeSourceFunction} and \ref{lem:EllipticHestonReductionNonnegativeSourceFunctionUniqueness}.

\begin{lem}[Reduction to the case of existence when the source function is positive and the solution non-negative]
\label{lem:VIEllipticHestonReductionNonnegativeSourceFunction}
Assume the hypotheses of Theorem \ref{thm:VIExistenceUniquenessEllipticHeston_Improved} for existence and uniqueness and let $u_\varphi \in H_0^1(\sO\cup\Gamma_0)\cap H^2(\sO,\fw)$ be as in Lemma \ref{lem:EllipticHestonUniquenessAuxiliaryFunction}. Define $\tilde M, \tilde m$ as in \eqref{eq:DefineTildeMm} and $\tilde f$ as in \eqref{eq:definetildef} and define
\begin{equation}
\label{eq:VIDefineTildepsi}
\tilde\psi := \psi + u_\varphi.
\end{equation}
Then, in addition to the conclusions of Lemma \ref{lem:EllipticHestonReductionNonnegativeSourceFunction}, we have $\tilde\psi \in H^1(\sO,\fw)$ and
\begin{equation}
\label{eq:VItildeObstacleFunctionTraceBound}
\tilde\psi \leq 0 \quad \hbox{on $\Gamma_1$ (trace sense)},
\end{equation}
while $\tilde M, \tilde m \in H^2(\sO,\fw)$ obey
\begin{equation}
\label{eq:L4TildeFunctionBounds}
(1+y)^{1/2}\tilde M, (1+y)^{1/2}\tilde m \in L^q(\sO,\fw),\quad\hbox{for some } q>2.
\end{equation}
Moreover, existence in Theorem \ref{thm:VIExistenceUniquenessEllipticHeston_Improved} of a solution, $u$, to Problem \ref{prob:HomogeneousHestonVIProblem} defined by $f$ and $\psi$ and obeying the bounds \eqref{eq:uBoundedObstacle} is equivalent to existence of a solution, $\tilde u$, to Problem \ref{prob:HomogeneousHestonVIProblem} defined by $\tilde f$ and $\tilde\psi$ and obeying
\begin{equation}
\label{eq:VItildeuBounds}
\max\{\tilde\psi,\tilde m\} \leq \tilde u \leq \tilde M \quad \hbox{ a.e. on }\sO.
\end{equation}
\end{lem}

\begin{proof} We first verify the conclusions in the preamble. Observe that \eqref{eq:L4TildeFunctionBounds} follows from \eqref{eq:Sqrt1plusyMmInLq} and \eqref{eq:Sqrt1plusyVarphiInLq}, and \eqref{eq:uvarphiBounds}. Since $u_\varphi=0$ on $\Gamma_1$ then clearly \eqref{eq:VItildeObstacleFunctionTraceBound} holds because of the condition $\psi\leq 0$ on $\Gamma_1$ (trace sense) in Problem \ref{prob:HomogeneousHestonVIProblem}.

\emph{Existence of $\tilde u$ implies existence of $u$.} By assumption, there exists a function $\tilde u \in V$ obeying \eqref{eq:VItildeuBounds} and
\begin{equation}
\label{eq:TildeVIProblemHestonHomgeneous}
a(\tilde u,\tilde v-\tilde u) \geq (\tilde f,\tilde v-\tilde u)_H, \quad \forall \tilde v\in V, \tilde v \geq \tilde\psi.
\end{equation}
By \eqref{eq:DefineTildeMm} and \eqref{eq:VItildeuBounds}, we have
\begin{equation}
\label{eq:VIExplicitTildeuBounds}
m+u_\varphi\leq \tilde u \leq M+u_\varphi \quad\hbox{a.e. on } \sO.
\end{equation}
Therefore, setting $u := \tilde u - u_\varphi \in V$ yields
$$
m\leq u \leq M \quad\hbox{a.e. on } \sO.
$$
Moreover,
$$
u  = \tilde u - u_\varphi \geq \tilde\psi - u_\varphi = \psi \quad\hbox{a.e. on }\sO.
$$
Consequently, $u$ obeys \eqref{eq:uBoundedObstacle} by combining the preceding two inequalities. For $v\in V$, write $v := \tilde v - u_\varphi \in V$ and note that $v \geq \psi$ a.e. on $\sO$ if and only if $\tilde v \geq \tilde\psi$ a.e. on $\sO$. Then,
\begin{align*}
a(u,v-u) &= a(\tilde u-u_\varphi, \tilde v-u_\varphi - (\tilde u-u_\varphi))
\\
&= a(\tilde u-u_\varphi, \tilde v-\tilde u)
\\
&= a(\tilde u, \tilde v-\tilde u) - a(u_\varphi, \tilde v-\tilde u)
\\
&\geq (\tilde f,\tilde v-\tilde u)_H - (A\varphi,\tilde v-\tilde u)_H \quad\hbox{(by \eqref{eq:Defnuvarphi} and \eqref{eq:TildeVIProblemHestonHomgeneous})}
\\
&= (f,\tilde v-\tilde u)_H  \quad\hbox{(by \eqref{eq:definetildef})}
\\
&= (f,v-u)_H, \quad \forall v\in V, v\geq \psi.
\end{align*}
Hence, $u$ is a solution to Problem \ref{prob:HomogeneousHestonVIProblem}  defined by the obstacle function, $\psi$, and source function, $f$.

\emph{Existence of $u$ implies existence of $\tilde u$.} By assumption, there is a solution $u\in V$ to Problem \ref{prob:HomogeneousHestonVIProblem} defined by the obstacle function, $\psi$, and source function, $f$, which obeys \eqref{eq:uBoundedObstacle}. Set $\tilde u := u + u_\varphi$, so \eqref{eq:uBoundedObstacle} implies that $\tilde u$ obeys \eqref{eq:VIExplicitTildeuBounds}, while
$$
\tilde u = u + u_\varphi \geq \psi + u_\varphi = \tilde\psi \quad\hbox{a.e. on }\sO,
$$
and thus $\tilde u$ obeys \eqref{eq:VItildeuBounds}. For $\tilde v\in V$, write $\tilde v := v + u_\varphi \in V$ and recall that $v \geq \psi$ a.e. on $\sO$ if and only if $\tilde v \geq \tilde\psi$ a.e. on $\sO$. Then,
\begin{align*}
a(\tilde u, \tilde v-\tilde u) &= a(u + u_\varphi, v + u_\varphi - (u + u_\varphi))
\\
&= a(u + u_\varphi, v - u)
\\
&= a(u, v - u) + a(u_\varphi, v - u)
\\
&\geq (f,v-u)_H + (A\varphi, v-u)_H \quad\hbox{(by \eqref{eq:Defnuvarphi} and \eqref{eq:VIProblemHestonHomgeneous})}
\\
&= (\tilde f, v-u)_H \quad\hbox{(by \eqref{eq:definetildef})}
\\
&= (\tilde f, \tilde v-\tilde u)_H, \quad \forall \tilde v \in V, \tilde v \geq\tilde\psi.
\end{align*}
Hence, $\tilde u$ obeys \eqref{eq:TildeVIProblemHestonHomgeneous} and thus is a solution to Problem \ref{prob:HomogeneousHestonVIProblem} defined by the obstacle function, $\tilde\psi$, and source function, $\tilde f$.
\end{proof}

\begin{lem}[Non-negative solutions]
\label{lem:VIuNonnegative}
Assume the hypotheses of Theorem \ref{thm:VIExistenceUniquenessEllipticHeston_Improved} for existence and uniqueness. Let $\tilde u \in H_0^1(\sO\cup\Gamma_0)$ be a solution to Problem \ref{prob:HomogeneousHestonVIProblem} defined by $\tilde f$ as in \eqref{eq:definetildef} and $\tilde \psi$ as in \eqref{eq:VIDefineTildepsi}. Then $\tilde u$ obeys
\begin{equation}
\label{eq:VIuNonnegative}
\tilde u \geq 0 \quad\hbox{a.e. on }\sO.
\end{equation}
\end{lem}

\begin{proof}
Observe that \eqref{eq:PositiveTildeSourceFunction} and \eqref{eq:tildefBounds} imply that $\tilde f$ obeys $0 < \tilde f \leq A\tilde M$ a.e. on $\sO$ and hence, replacing $\tilde m$ by zero in \eqref{eq:VItildeuBounds}, we obtain \eqref{eq:VIuNonnegative}.
\end{proof}

\begin{lem}[Reduction to the case of uniqueness when the source function is positive and the solution non-negative]
\label{lem:VIEllipticHestonReductionNonnegativeSourceFunctionUniqueness}
Assume the hypotheses of Theorem \ref{thm:VIExistenceUniquenessEllipticHeston_Improved} for existence and uniqueness and let $\tilde f$ be as in \eqref{eq:definetildef} and $\tilde \psi$ be as in \eqref{eq:VIDefineTildepsi}. Then uniqueness of a solution, $u$, to Problem \ref{prob:HomogeneousHestonVIProblem} defined by $f, \psi$ is equivalent to uniqueness of a solution, $\tilde u$, to Problem \ref{prob:HomogeneousHestonVIProblem} defined by $\tilde f, \tilde\psi$.
\end{lem}

\begin{proof}
Let $u_\varphi$ be as in Lemma \ref{lem:EllipticHestonUniquenessAuxiliaryFunction}. Lemma \ref{lem:VIEllipticHestonReductionNonnegativeSourceFunction} implies that $u_i\in H_0^1(\sO\cup\Gamma_0), i=1,2$ are two solutions to Problem \ref{prob:HomogeneousHestonVIProblem} defined by $f, \psi$ if and only if $\tilde u_i := u_i + u_\varphi\in H_0^1(\sO\cup\Gamma_0), i=1,2$ are two solutions to Problem \ref{prob:HomogeneousHestonVIProblem} defined by $\tilde f, \tilde\psi$. Therefore, $u_1 = u_2$ if and only if $\tilde u_1 = \tilde u_2$ and this yields the conclusion.
\end{proof}

\begin{proof}[Proof of uniqueness in Theorem \ref{thm:VIExistenceUniquenessEllipticHeston_Improved}] We assume the reduction embodied in Lemma \ref{lem:VIEllipticHestonReductionNonnegativeSourceFunction}. To simplify notation \emph{we shall omit the ``tildes'' and write $f,\psi,u$ for $\tilde f,\tilde\psi,\tilde u$.}

Suppose $u_1, u_2$ are two solutions to \eqref{eq:VIProblemHestonHomgeneous}, assumed non-negative by \eqref{eq:VIuNonnegative}. The proof of uniqueness is identical to the proof of uniqueness in Theorem \ref{thm:ExistenceUniquenessEllipticHeston_Improved} until we reach the point where we need to consider variational inequalities rather than variational equations. Therefore, keeping in mind that $u_1, u_2$ now solve \eqref{eq:VIProblemHestonHomgeneous} rather than \eqref{eq:IntroHestonWeakMixedProblemHomogeneous}, we note that $\beta_0 u_1$ satisfies the variational inequality
\begin{align*}
a_\lambda(\beta_0 u_1, \beta_0 v-\beta_0 u_1) &= a(\beta_0 u_1, \beta_0 v-\beta_0 u_1) + \lambda(\beta_0 (1+y)u_1, \beta_0 v-\beta_0 u_1)_H
\quad\hbox{(by \eqref{eq:BilinearFormCoerciveHeston})}
\\
&\geq (\beta_0 f, \beta_0 v-\beta_0 u_1)_H + \lambda(\beta_0 (1+y)u_1, \beta_0 v-\beta_0 u_1)_H \quad\hbox{(by \eqref{eq:VIProblemHestonHomgeneous})}
\\
&= (\beta_0 f + \lambda\beta_0 (1+y)u_1, \beta_0 v-\beta_0 u_1)_H
\\
&= (f_1,\beta_0 v-\beta_0 u_1)_H
\quad\hbox{(by definition \eqref{eq:VIBetaLambdafuInequality} of $f_1$)},
\end{align*}
for all $\beta_0 v \geq \beta_0\psi$, $\beta_0 v\in V$, and so
$$
a_\lambda(\beta_0 u_1, v-\beta_0 u_1) \geq (f_1, v-\beta_0 u_1)_H, \quad \forall v \in V, v\geq \beta_0\psi =: \psi_1.
$$
Moreover, $u_2$ satisfies the variational inequality
\begin{align*}
a_\lambda(u_2,v-u_2) &= a(u_2,v-u_2) + \lambda((1+y)u_2,v-u_2)_H
\\
&\geq (f,v-u_2)_H + \lambda((1+y)u_2,v-u_2)_H
\quad\hbox{(by \eqref{eq:VIProblemHestonHomgeneous})}
\\
&= (f+\lambda (1+y)u_2,v-u_2)_H
\\
&= (f_2,v-u_2)_H
\quad\hbox{(by definition \eqref{eq:VIBetaLambdafuInequality} of $f_2$)}, \quad \forall v \in V, v \geq \psi := \psi_2.
\end{align*}
We are now within the setting of Theorem \ref{thm:VIExistenceUniquenessEllipticCoerciveHestonIncreasingF} since \eqref{eq:CoerciveHeston} holds and
$$
f_1 \leq f_2 \hbox{ on $\sO$ by \eqref{eq:VIBetaLambdafuInequality} and } \psi_1 = \beta_0\psi < \psi = \psi_2\hbox{ on $\sO$}.
$$
Therefore, Theorem \ref{thm:VIExistenceUniquenessEllipticCoerciveHestonIncreasingF} implies that $\beta_0 u_1 \leq u_2$, analogous to \eqref{eq:betazerou1lequ2} in the proof of uniqueness in Theorem \ref{thm:ExistenceUniquenessEllipticHeston_Improved} and the conclusion follows exactly as in the proof of uniqueness in Theorem \ref{thm:ExistenceUniquenessEllipticHeston_Improved}.
\end{proof}

\begin{cor}[A posteriori estimate for solutions to the non-coercive variational inequality]
\label{cor:VIExistenceUniquenessEllipticHestonPowery}
Assume the hypotheses of Theorem \ref{thm:VIExistenceUniquenessEllipticHeston_Improved}. In addition, require that $M, m \in H^2(\sO,\fw)$ obey
\begin{equation}
\label{eq:1pluspoweryplus1mMLq}
(1+y^{s+1})M, (1+y^{s+1})m \in L^q(\sO,\fw) \quad\hbox{for some } q>2,
\end{equation}
and that $f \in L^2(\sO,\fw)$ obeys \eqref{eq:poweryfL2} and that $\psi \in H^1(\sO,\fw)$ obeys \eqref{eq:1pluspoweryfpsiH1}. If $u\in H^1_0(\sO\cup\Gamma_0,\fw)$ is the unique solution to Problem \ref{prob:HomogeneousHestonVIProblem}, then $y^su \in H^1(\sO,\fw)$, $y^{s+1}u \in L^2(\sO,\fw)$, and
\begin{equation}
\label{eq:EllipticHestonAPosterioriEstimatePoweryCorrected}
\|y^su\|_{H^1(\sO,\fw)} \leq C\left(\|y^sf\|_{L^2(\sO,\fw)} + \|(1+y^{s+1})u\|_{L^2(\sO,\fw)} + \|(1+y^{2s-1/2})\psi^+\|_{H^1(\sO,\fw)} \right),
\end{equation}
where $C$ depends only on $s$ and the constant coefficients of $A$.
\end{cor}

\begin{proof}
Setting $f_\lambda := f+\lambda(1+y)u$ and using the definition \eqref{eq:BilinearFormCoerciveHeston} of $a_\lambda$, we may view $u\in\KK$ as the unique solution to
$$
a_\lambda(u,v-u) \geq (f_\lambda,v-u)_H, \quad \forall v\in \KK.
$$
The source function $f_\lambda = f + \lambda(1+y)u$ obeys \eqref{eq:fBoundsCoercive} since
\begin{align*}
A_\lambda m &= Am + \lambda(1+y)m \quad\hbox{(by \eqref{eq:CoerciveHestonOperator})}
\\
&\leq f + \lambda(1+y)u \quad\hbox{(by \eqref{eq:fBounds} and \eqref{eq:uBoundedObstacle})}
\\
&\leq AM + \lambda(1+y)M \quad\hbox{(by \eqref{eq:fBounds} and \eqref{eq:uBoundedObstacle})}
\\
&= A_\lambda M \quad\hbox{(by \eqref{eq:uBoundedObstacle})}.
\end{align*}
Since $u$ obeys \eqref{eq:uBoundedObstacle} and $M, m$ obey \eqref{eq:1pluspoweryplus1mMLq}, then $y^s(1+y)u \in L^2(\sO,\fw)$ and so $y^sf_\lambda \in L^2(\sO,\fw)$. We can now apply Corollary \ref{cor:VIExistenceUniquenessEllipticCoerciveHestonPowery} to conclude that \eqref{eq:EllipticCoerciveHestonAPosterioriEstimatePoweryCorrected} holds with $f$ replaced by $f_\lambda$ and thus \eqref{eq:EllipticHestonAPosterioriEstimatePoweryCorrected} follows.
\end{proof}

For completeness, we can now state and prove a posteriori comparison estimates for solutions to the \emph{coercive} variational inequality:

\begin{cor}[A posteriori comparison principle for the coercive variational inequality]
\label{cor:ComparisionPrincipleCoerciveEllipticHestonVI}
Given $f, \psi$ as in Problem \ref{prob:HomogeneousHestonVIProblem}, let $u\in V, u\geq\psi$ be the unique solution to Problem \ref{prob:HomogeneousHestonVIProblemCoercive}. Then
$$
\max\{m,\psi\} \leq u \leq M \hbox{ a.e. on }\sO.
$$
\end{cor}

\begin{proof}
The conclusion follows from \eqref{eq:uBoundedObstacle}, noting that this inequality was established a fortiori by the proof of Theorem \ref{thm:VIExistenceUniquenessEllipticHeston_Improved} for the non-coercive bilinear form $a(u,v)$; the hypotheses \eqref{eq:Sqrt1plusyMmInLq} on $m, M$ are not required for existence of $u$, while the hypotheses on $\varphi$ were only required for uniqueness in the non-coercive case and may be omitted in the coercive case.
\end{proof}

\section{Regularity of solutions to the variational equation}
\label{sec:H2RegularityEquality}
We establish higher regularity results for solutions to the variational equation for the elliptic Heston operator, Problem \ref{prob:HestonWeakMixedBVPHomogeneous}. In \S \ref{subsec:IntermediateFirstOrderEstimatesVarEquality}, we prove an intermediate a priori estimate for first-order derivatives of these solutions (Proposition \ref{prop:AuxiliaryWeightedH1Estimate}) while in \S \ref{subsec:RefinedFirstOrderEstimatesVarEquality} we derive a refined a priori estimate for first-order derivatives (Proposition \ref{prop:AuxiliarySpecialWeightedH1EstimateSisZero}). In \S \ref{subsec:SecondOrderEstimatesVarEquality} we derive a priori estimates for second-order derivatives (Proposition \ref{prop:AuxiliaryWeightedH2EstimateSisZero}). In \S \ref{subsec:H2RegularityEquality} we show that solutions to the variational equation are in $H^2(\sO,\fw)$ and obtain an a priori $H^2(\sO,\fw)$ estimate for these solutions (Theorem \ref{thm:GlobalRegularityEllipticHestonSpecial}) together with an existence and uniqueness result for strong solutions (Theorem \ref{thm:ExistenceUniquenessH2RegularEllipticHeston}). We conclude in \S \ref{subsec:HolderRegularityEquality} by showing that solutions are H\"older continuous (Theorem \ref{thm:HolderContinuityHestonStatVarEquality}).

\subsection{Preliminary a priori estimate for first-order derivatives of a solution to the variational equation}
\label{subsec:IntermediateFirstOrderEstimatesVarEquality}
We obtain a preliminary a priori first-derivative estimate for a solution to Problem \ref{prob:HestonWeakMixedBVPHomogeneous}, weighted by a power of $y$.

\begin{prop}[A priori first-derivative estimate for a solution to the variational equation]
\label{prop:AuxiliaryWeightedH1Estimate}
There is a positive constant $C$ depending only on $\gamma$ and the constant coefficients of $A$ such that, if $u\in V$ is a solution to Problem \ref{prob:HestonWeakMixedBVPHomogeneous} and $y^{1/2}f, (1+y)u \in H$, then $y^{1/2} u \in V$ and
\begin{align}
\label{eq:AuxiliaryWeightedH1Estimate}
|yDu|_H \leq C\left(|y^{1/2}f|_H + |(1+y)u|_H\right),
\\
\label{eq:AuxiliaryWeightedH1NormEstimate}
\|y^{1/2}u\|_V \leq C\left(|y^{1/2}f|_H + |(1+y)u|_H\right).
\end{align}
\end{prop}

\begin{proof}
Let $\zeta_R$ be the cutoff function in Definition \ref{defn:RadialCutoffFunction}, set $\varphi := \zeta_R y^{1/2}$,
and observe that $y\zeta_Ru \in H^1_0(\sO\cup\Gamma_0,\fw)$ and
$$
|D\varphi| \leq 10(y^{-1/2} + y^{1/2}) \quad\hbox{on }\RR^2, \quad \forall R\geq 2.
$$
From \eqref{eq:BilinearFormSquaredFunction} and \eqref{eq:CommutatorInnerProductEstimate}, with the preceding choice of $\varphi$, we obtain
$$
|a(\zeta_Ry^{1/2} u, \zeta_Ry^{1/2} u) - a(u, y\zeta_R^2 u)| \leq C|y^{1/2}((y^{-1/2} + y^{1/2})^{1/2} + (y^{-1/2} + y^{1/2}))u|_H
$$
and thus
\begin{equation}
\label{eq:BilinearFormPoweryCommutator}
|a(\zeta_Ry^{1/2} u,\zeta_Ry^{1/2} u) - a(u, y\zeta_R^2 u)| \leq C|(1 + y)u|_H^2.
\end{equation}
From \eqref{eq:CoerciveHeston} we have
$$
\nu_1\|\zeta_Ry^{1/2} u\|_V^2 \leq a(\zeta_Ry^{1/2}u, \zeta_Ry^{1/2}u) + \lambda((1+y)\zeta_Ry^{1/2}u, \zeta_Ry^{1/2}u)_H,
$$
and as $a(u, \zeta_R^2yu) = (f, \zeta_R^2yu)_H$ by \eqref{eq:IntroHestonWeakMixedProblemHomogeneous}, then \eqref{eq:BilinearFormPoweryCommutator} yields
\begin{align*}
\|\zeta_Ry^{1/2} u\|_V^2 &\leq C\left(|(\zeta_Ry^{1/2}f, \zeta_Ry^{1/2}u)_H| + \lambda|((1+y)\zeta_Ry^{1/2}u, \zeta_Ry^{1/2}u)_H|\right) + C|(1 + y)u|_H^2
\\
&\leq C\left(|y^{1/2}f|_H|y^{1/2}u|_H + |(1+y)u|_H^2\right)
\\
&\leq C\left(|y^{1/2}f|_H^2 + |(1+y)u|_H^2\right),
\end{align*}
and thus
\begin{equation}
\label{eq:AuxiliaryWeightedH1NormEstimate_prefinal}
\|\zeta_Ry^{1/2} u\|_V \leq C\left(|y^{1/2}f|_H + |(1+y)u|_H\right),
\end{equation}
with constant $C$ depending only on the constant coefficients of $A$ and $\gamma$. Using
\begin{align*}
y^{1/2}D(\zeta_Ry^{1/2}u) &= y^{1/2}\left(\zeta_Ry^{1/2}Du + (D\zeta_R)y^{1/2}u + \frac{1}{2}(0, \zeta_Ry^{-1/2}u)\right)
\\
&= \zeta_RyDu + (D\zeta_R)yu + \frac{1}{2}(0, \zeta_Ru),
\end{align*}
and recalling the Definition \ref{defn:H1WeightedSobolevSpaces} of the norm for $H^1(\sO,\fw)$, we obtain
$$
|\zeta_RyDu|_H \leq \|\zeta_Ry^{1/2} u\|_V + |(1+y)u|_H, \quad\forall R\geq 2,
$$
with constant $C$ depending only on the constant coefficients of $A$ and $\gamma$. Combining the preceding inequality with \eqref{eq:AuxiliaryWeightedH1NormEstimate_prefinal} yields
$$
|\zeta_RyDu|_H \leq C\left(|y^{1/2}f|_H + |(1+y)u|_H\right).
$$
Taking limits as $R\to\infty$ and applying the dominated convergence theorem yields \eqref{eq:AuxiliaryWeightedH1Estimate}. The estimate \eqref{eq:AuxiliaryWeightedH1NormEstimate} follows from \eqref{eq:AuxiliaryWeightedH1Estimate}, Definition \ref{defn:H1WeightedSobolevSpaces}, and the identity $y^{1/2}D(y^{1/2}u) = yDu + \frac{1}{2}(0,u)$.
\end{proof}

\subsection{Refined a priori estimate for first-order derivatives for solutions to the variational equation}
\label{subsec:RefinedFirstOrderEstimatesVarEquality}
By considering an elliptic version, $Lu := y^{1-\beta}((y^\beta u_x)_x + (y^\beta u_y)_y)$, of the parabolic model operator, considered by Koch in \cite[Equation (4.43)]{Koch}, and the map $w \mapsto Tw := u$ defined by the solution to the equation $Lu = w$, one would anticipate estimates \eqref{eq:AuxiliarySpecialWeightedH1EstimateSisZero} and \eqref{eq:HEstimateyD2uSpecial} analogous to the first and second-order derivative estimates in \cite[Lemma 4.6.1]{Koch} and its formal proof; compare \cite[Theorems 1 and $1'$]{Kohn_Nirenberg_1967}. The first-order estimate \eqref{eq:AuxiliarySpecialWeightedH1EstimateSisZero} is sharper than \eqref{eq:VariationalEqualityHestonBoundH}.

Because the argument used by Koch in \cite[Proof of Lemma 4.6.1]{Koch} is formal (as Koch himself underlines \cite[p. 88]{Koch}), one of our goals in this subsection --- aside from extending his result to case of the Heston and similar degenerate, second-order elliptic operators with lower-order terms on unbounded domains --- is to provide a rigorous proof of \cite[Lemma 4.6.1]{Koch} and our extensions. In order to avoid technical difficulties which would arise if we used cutoff functions or finite differences (compare the proofs of \cite[Theorems 8.8 \& 8.12]{GT}), we shall instead appeal to

\begin{thm}[Existence and uniqueness of classical solutions when the source function is smooth with compact support]
\cite{Feehan_Pop_regularityweaksoln}, \cite{PopThesis}
\label{thm:SmoothnessUpToBdrySolution}
Suppose that $f\in C^\infty_0(\sO)$ and $\Gamma_1$ is $C^\infty$. Then there exists a solution $u\in C^\infty(\bar\sO)$ to Problem \ref{prob:HestonMixedBVPHomogeneousClassical}.
\end{thm}

\begin{rmk}[H\"older regularity analogue of Theorem \ref{thm:SmoothnessUpToBdrySolution}]
\label{rmk:SmoothnessUpToBdrySolution}
See Definition \ref{defn:ContinuousFunctions} for $C^\infty(\bar\sO)$. Theorem \ref{thm:SmoothnessUpToBdrySolution} is a special case of more general existence, uniqueness, regularity results, and Schauder estimates for classical solutions to the elliptic Heston equation in \cite{Feehan_Pop_regularityweaksoln}, \cite{PopThesis}, where we have a $C^{k+2,\alpha}$ boundary portion, $\Gamma_1$, and source function, $f\in C^{k,\alpha}_0(\sO)$ ($k\geq 0, \alpha\in(0,1)$), yielding a solution, $u\in C^{k+2,\alpha}(\bar\sO)$; these results are obtained by adapting the proofs of \cite[Theorems I.1.1 \& II.1.1]{DaskalHamilton1998} for a linearization of the (parabolic) porous medium equation.
\end{rmk}

\begin{rmk}[Existence and uniqueness of classical solutions to the coercive equation when the source function is smooth with compact support]
\label{rmk:SmoothnessUpToBdrySolutionCoercive}
Theorem \ref{thm:SmoothnessUpToBdrySolution} extends to the case where $A$ is replaced by $A_\lambda = A + \lambda(1+y)$ in Problem \ref{prob:HestonMixedBVPHomogeneousClassical}.
\end{rmk}

\begin{thm}[Regularity up to the boundary of a solution to the coercive variational equation when the source function is smooth with compact support]
\label{thm:SmoothnessUpToBdryWeakSolution}
Suppose that $u\in H^1_0(\sO\cup\Gamma_0,\fw)$ is a solution to \eqref{eq:VariationalEqualityCoercive} with source function $f\in L^2(\sO,\fw)$. If $f\in C^\infty_0(\sO)$ and $\Gamma_1$ is $C^\infty$, then $u\in C^\infty(\bar\sO)$ and is a solution to Problem \ref{prob:HestonMixedBVPHomogeneousClassical} with $A$ replaced by $A_\lambda$.
\end{thm}

\begin{proof}
Theorem \ref{thm:SmoothnessUpToBdrySolution} and Remark \ref{rmk:SmoothnessUpToBdrySolutionCoercive} provide a solution $\tilde u \in C^\infty(\bar\sO)$ to Problem \ref{prob:HestonMixedBVPHomogeneousClassical} with $A$ replaced by $A_\lambda$ and source function $f \in C^\infty_0(\sO)$, so $A_\lambda\tilde u = f$ on $\sO$ and $\tilde u=0$ on $\Gamma_1$. Since $C^\infty(\bar\sO) \subset H^2(\sO,\fw)$, then $\tilde u \in H^2(\sO,\fw)$ and Lemma \ref{lem:HestonWeightedNeumannBVPHomogeneous} implies that $\tilde u \in H^1_0(\sO\cup\Gamma_0,\fw)$ and that $\tilde u$ is a solution to \eqref{eq:VariationalEqualityCoercive}.

By hypothesis $u\in H^1_0(\sO\cup\Gamma_0,\fw)$ is a solution to \eqref{eq:VariationalEqualityCoercive} with source function $f \in L^2(\sO,\fw)$ and this solution must be unique by Theorem \ref{thm:ExistenceUniquenessEllipticCoerciveHeston}. Hence, $\tilde u = u$ a.e. on $\sO$, and thus $u\in C^\infty(\bar\sO)$, as desired.
\end{proof}

\begin{rmk}[H\"older regularity analogue of Theorem \ref{thm:SmoothnessUpToBdryWeakSolution}]
\label{rmk:SmoothnessUpToBdryWeakSolution}
Theorem \ref{thm:SmoothnessUpToBdryWeakSolution} extends to the more general case of $C^{k+2,\alpha}$ boundary portion, $\Gamma_1$, and source function, $f\in C^{k,\alpha}_0(\sO)$ ($k\geq 0, \alpha\in(0,1)$), yielding a solution, $u\in C^{k+2,\alpha}(\bar\sO)$, and is considered in \cite{Feehan_Pop_regularityweaksoln}, \cite{PopThesis}.
\end{rmk}

\begin{rmk}[Regularity up to the boundary of a solution to the non-coercive variational equation when the source function is smooth with compact support]
Although not used in this article, if we are given the additional hypotheses in Theorem \ref{thm:ExistenceUniquenessEllipticHeston_Improved} required to ensure uniqueness of solutions, Theorem \ref{thm:SmoothnessUpToBdryWeakSolution} extends to the case of a solution, $u$, to Problem \ref{prob:HestonWeakMixedBVPHomogeneous}.
\end{rmk}

We now prove an analogue of the first-order estimate obtained in the formal proof of \cite[Lemma 4.6.1]{Koch}.

\begin{prop}[A refined a priori first-order derivative estimate for solutions to the variational equation]
\label{prop:AuxiliarySpecialWeightedH1EstimateSisZero}
Assume the hypotheses of Theorem \ref{thm:ExistenceUniquenessEllipticHeston_Improved} for existence and require, in addition, that $\sO$ obey Hypothesis \ref{hyp:HestonDomainNearGammaZero} and that the boundary portion, $\Gamma_1$, is $C^{2,\alpha}$. If $u\in V$ is a solution to Problem \ref{prob:HestonWeakMixedBVPHomogeneous} and $yu \in L^2(\sO,\fw)$, then
\begin{equation}
\label{eq:AuxiliarySpecialWeightedH1EstimateSisZero}
|Du|_H \leq C\left(|f|_H + |(1+y)u|_H\right),
\end{equation}
where $C$ is a positive constant depending only on the constant coefficients of $A$.
\end{prop}

\begin{proof}
The non-coercive variational equation \eqref{eq:IntroHestonWeakMixedProblemHomogeneous} in Problem \ref{prob:HestonWeakMixedBVPHomogeneous} may be written as an equivalent coercive variational equation \eqref{eq:VariationalEqualityCoercive}, that is
\begin{equation}
\label{eq:VariationalEqualityCoerciveSourcef1plusyu}
a_\lambda(u,v)=(f_\lambda,v)_H, \quad\forall v\in H^1(\sO\cup\Gamma_0,\fw),
\end{equation}
where $a_\lambda(u,v)$ is defined by \eqref{eq:BilinearFormCoerciveHeston} and
\begin{equation}
\label{eq:Defnflambda}
f_\lambda := f + \lambda(1+y)u \in L^2(\sO,\fw).
\end{equation}

\begin{rmk}[Illustration of the proof when $\Gamma_1$ is $C^\infty$]
\label{rmk:ProofWhenGammaOneSmooth}
For the sake of clarity and notational simplicity in the proof of Proposition \ref{prop:AuxiliarySpecialWeightedH1EstimateSisZero}, we shall assume that $\Gamma_1$ is $C^\infty$ and apply the reduction to $u\in C^\infty(\bar\sO)$ when $f\in C^\infty_0(\sO)$ enabled by Theorem \ref{thm:SmoothnessUpToBdryWeakSolution}; however, at the cost of more cumbersome notation, one could just as easily assume $\Gamma_1$ is $C^{2,\alpha}$ ($\alpha\in(0,1)$) and apply the reduction in Remark \ref{rmk:SmoothnessUpToBdryWeakSolution} to the case $f\in C^\alpha_0(\sO)$, yielding a solution, $u\in C^{2,\alpha}(\bar\sO)$.
\end{rmk}

\noindent\textbf{Step 1.} \emph{Reduction to the case of a smooth source function with compact support and a smooth solution.} From Lemma \ref{lem:L2ApproximationLemma}, we may choose a sequence $\{f_{\lambda, n}\}_{n\geq 1} \subset C^\infty_0(\sO)$ such that
\begin{equation}
\label{eq:SmoothfnL2convergestoflambda}
f_{\lambda, n} \to f_\lambda \quad\hbox{strongly in $L^2(\sO,\fw)$ as $n\to\infty$}.
\end{equation}
Let $\{u_n\}_{n\geq 1} \subset C^\infty(\bar\sO)\cap H^1_0(\sO\cup\Gamma_0,\fw)$ be the corresponding sequence of solutions to \eqref{eq:VariationalEqualityCoerciveSourcef1plusyu} produced by Theorems \ref{thm:ExistenceUniquenessEllipticCoerciveHeston} and \ref{thm:SmoothnessUpToBdryWeakSolution}. Now $u_n-u_{n'}$ solves \eqref{eq:VariationalEqualityCoerciveSourcef1plusyu} with source function $f_{\lambda, n}-f_{\lambda,n'}$, for all $n,n'\geq 1$, and by \eqref{eq:VariationalEqualityCoerciveBoundfH} obeys,
$$
\|u_n-u_{n'}\|_{H^1(\sO,\fw)} \leq C\|f_{\lambda, n}-f_{\lambda,n'}\|_{L^2(\sO,\fw)}, \quad \forall n,n'\geq 1.
$$
Hence, the sequence $\{u_n\}_{n\geq 1}$ is Cauchy in $H^1_0(\sO\cup\Gamma_0,\fw)$ and
$$
u_n \to \tilde u \quad\hbox{strongly in $H^1(\sO,\fw)$ as $n\to\infty$},
$$
for some $\tilde u \in H^1_0(\sO\cup\Gamma_0,\fw)$ which necessarily solves \eqref{eq:VariationalEqualityCoerciveSourcef1plusyu}. Because the solution to \eqref{eq:VariationalEqualityCoerciveSourcef1plusyu} is unique by Theorem \ref{thm:ExistenceUniquenessEllipticCoerciveHeston}, we must have $\tilde u = u$ a.e. on $\sO$. Therefore,
\begin{equation}
\label{eq:SmoothunH1convergestou}
u_n \to u \quad\hbox{strongly in $H^1(\sO,\fw)$ as $n\to\infty$},
\end{equation}
and $u_n=0$ on $\Gamma_1$, for all $n\geq 1$, by Lemma \ref{lem:Traces}.

Suppose we have shown that
\begin{equation}
\label{eq:AuxiliarySpecialWeightedH1EstimateSisZero_nWithCutoff}
\|\zeta_RDu_n\|_{L^2(\sO,\fw)} \leq C\left(\|f_{\lambda, n}\|_{L^2(\sO,\fw)} + \|(1+y)\zeta_Ru_n\|_{L^2(\sO,\fw)}\right), \quad \forall n\geq 1,
\end{equation}
where $C$ depends at most on the constant coefficients of $A$ and $\zeta_R$ is the cutoff function in Definition \ref{defn:RadialCutoffFunction}. By \eqref{eq:AuxiliarySpecialWeightedH1EstimateSisZero_nWithCutoff} and the fact that $u_n-u_{n'}$ solves \eqref{eq:VariationalEqualityCoerciveSourcef1plusyu} with source function $f_{\lambda, n}-f_{\lambda,n'}$, for all $n,n'\geq 1$, we must also have
\begin{equation}
\label{eq:cutoffDuL2Cauchy}
\begin{aligned}
\|\zeta_RD(u_n-u_{n'})\|_{L^2(\sO,\fw)} &\leq C\left(\|f_{\lambda, n}-f_{\lambda,n'}\|_{L^2(\sO,\fw)} \right.
\\
&\quad + \left. \|(1+y)\zeta_R(u_n-u_{n'})\|_{L^2(\sO,\fw)}\right), \quad \forall n,n'\geq 1.
\end{aligned}
\end{equation}
Since $\{u_n\}_{n\geq 1}$ is Cauchy in $L^2(\sO,\fw)$ and, by Definition \ref{defn:RadialCutoffFunction} of the cutoff function, $\zeta_R$, obeys
$$
\|(1+y)\zeta_R(u_n-u_{n'})\|_{L^2(\sO,\fw)} \leq C(1+2R)\|u_n-u_{n'}\|_{L^2(\sO,\fw)},
$$
then $\{(1+y)\zeta_Ru_n\}_{n\geq 1}$ must be Cauchy in $L^2(\sO,\fw)$. Since $\{f_{\lambda, n}\}_{n\geq 1}$ is also Cauchy in $L^2(\sO,\fw)$ by \eqref{eq:SmoothfnL2convergestoflambda}, the sequence $\{\zeta_R Du_n\}_{n\geq 1}$ is Cauchy in $L^2(\sO,\fw)$ by \eqref{eq:cutoffDuL2Cauchy} and so
\begin{equation}
\label{eq:CutoffDunConvergenceToSomeLimit}
\zeta_R Du_n \to w \quad\hbox{strongly in $L^2(\sO,\fw)$ as $n\to\infty$},
\end{equation}
for some $w\in L^2(\sO,\fw)$, and thus,
\begin{equation}
\label{eq:SqRtyCutoffDunConvergenceToSomeLimit}
y^{1/2}\zeta_R Du_n \to y^{1/2}w \quad\hbox{in the sense of $L^2_{\textrm{loc}}(\sO)$ as $n\to\infty$}.
\end{equation}
Because $u_n \to u$ strongly in $H^1(\sO,\fw)$, we have
$$
y^{1/2}Du_n \to y^{1/2}Du \quad\hbox{strongly in $L^2(\sO,\fw)$ as $n\to\infty$},
$$
and therefore,
\begin{equation}
\label{eq:SqRtyCutoffDunConvergenceToTheLimit}
y^{1/2}\zeta_R Du_n \to y^{1/2}\zeta_RDu \quad\hbox{strongly in $L^2(\sO,\fw)$ as $n\to\infty$}.
\end{equation}
Consequently, $y^{1/2}w=y^{1/2}\zeta_RDu$ a.e. on $\sO$ by \eqref{eq:SqRtyCutoffDunConvergenceToSomeLimit} and \eqref{eq:SqRtyCutoffDunConvergenceToTheLimit}, and thus $w=\zeta_RDu$ a.e. on $\sO$ and \eqref{eq:CutoffDunConvergenceToSomeLimit} yields\footnote{In order to prove Proposition \ref{prop:AuxiliarySpecialWeightedH1EstimateSisZero} it would have been sufficient to use $\zeta_R Du_n \to \zeta_R Du$ weakly in $L^2(\sO,\fw)$, which follows immediately from the estimate \eqref{eq:AuxiliarySpecialWeightedH1EstimateSisZero_nWithCutoff} and the convergence results \eqref{eq:SmoothfnL2convergestoflambda} and \eqref{eq:SmoothunH1convergestou}; however, the strong convergence result \eqref{eq:CutoffDunConvergenceToTheLimit} is used in the proof of Proposition \ref{prop:AuxiliaryWeightedH2EstimateSisZero}.}
\begin{equation}
\label{eq:CutoffDunConvergenceToTheLimit}
\zeta_R Du_n \to \zeta_R Du \quad\hbox{strongly in } L^2(\sO,\fw).
\end{equation}
By taking limits in \eqref{eq:AuxiliarySpecialWeightedH1EstimateSisZero_nWithCutoff} as $n\to\infty$, we obtain
\begin{equation}
\label{eq:AuxiliarySpecialWeightedH1EstimateWithCutoff}
\|\zeta_R Du\|_{L^2(\sO,\fw)} \leq C\left(\|f_\lambda\|_{L^2(\sO,\fw)} + \|(1+y)\zeta_R u\|_{L^2(\sO,\fw)}\right), \quad\forall R\geq 2,
\end{equation}
where $C$ depends only on the constant coefficients of $A$. Finally, using $\|(1+y)\zeta_R u\|_{L^2(\sO,\fw)} \leq \|(1+y)u\|_{L^2(\sO,\fw)}$ and taking limits as $R\to\infty$ in \eqref{eq:AuxiliarySpecialWeightedH1EstimateWithCutoff} and applying the dominated convergence theorem yields
$$
\|Du\|_{L^2(\sO,\fw)} \leq C\left(\|f_\lambda\|_{L^2(\sO,\fw)} + \|(1+y)u\|_{L^2(\sO,\fw)}\right),
$$
and \eqref{eq:AuxiliarySpecialWeightedH1EstimateSisZero} follows from the preceding estimate and \eqref{eq:Defnflambda}.

\medskip
\noindent\textbf{Step 2.} \emph{Verification of the a priori estimate when the source function and solution are smooth.}
The preceding argument shows that it is enough to prove \eqref{eq:AuxiliarySpecialWeightedH1EstimateWithCutoff} when $f_\lambda$ is replaced by $\tilde f\in C^\infty(\bar\sO)$ and $u \in C^\infty(\bar\sO)$, with $u=0$ on $\Gamma_1$, solves\footnote{Since $\hbox{Vol}(\sO,\fw)<\infty$, we have $C^\infty(\bar\sO)\subset H^1(\sO,\fw)$, so $\tilde f\in L^2(\sO,\fw)$ and, because $u=0$ on $\Gamma_1$, then $u \in H^1_0(\sO\cup\Gamma_0,\fw)$.}
\begin{equation}
\label{eq:VariationalEqualityCoerciveSmoothSource}
a_\lambda(u,v)=(\tilde f,v)_H, \quad\forall v\in H^1(\sO\cup\Gamma_0,\fw).
\end{equation}
For $\delta_0>0$ as in Hypothesis \ref{hyp:HestonDomainNearGammaZero}, we have
$$
\|Du\|_{L^2(\sO,\fw)} \leq \|Du\|_{L^2(\sO\cap(\RR\times(0,\delta_0)),\fw)} + \|Du\|_{L^2(\sO\cap(\RR\times(\delta_0,\infty)),\fw)}.
$$
But
\begin{align*}
\|Du\|_{L^2(\sO\cap(\RR\times(\delta_0,\infty)),\fw)} &\leq \delta_0^{-1/2}\|y^{1/2}Du\|_{L^2(\sO\cap(\RR\times(\delta_0,\infty)),\fw)}
\\
&\leq C\|u\|_{H^1(\sO,\fw)}
\\
&\leq C\|\tilde f\|_{L^2(\sO,\fw)} \quad\hbox{(by \eqref{eq:VariationalEqualityCoerciveBoundfH})}.
\end{align*}
Consequently, to prove \eqref{eq:AuxiliarySpecialWeightedH1EstimateSisZero}, it is sufficient to consider the case\footnote{For this segment of the proof we only need to consider the case where $\Gamma_1$ is $C^\infty$, as assumed in Theorem \ref{thm:SmoothnessUpToBdrySolution}, though analogues of Theorem \ref{thm:SmoothnessUpToBdrySolution} hold for boundaries which are less regular by Remark \ref{rmk:SmoothnessUpToBdrySolution}.}
$$
\sO = \Gamma_0 \times \RR_+ \quad\hbox{and}\quad \Gamma_1 = \partial \Gamma_0\times \RR_+.
$$
Let $\zeta_R$ be the cutoff function in Definition \ref{defn:RadialCutoffFunction} and set
$$
v = \zeta^2 u_y \quad\hbox{on }\sO.
$$
Because $u = 0$ on $\Gamma_1$, we have $\zeta^2 u_y  = 0$ on $\Gamma_1$ and as $\zeta^2 u_y \in C^\infty_0(\bar\sO)$, since $u\in C^\infty(\bar\sO)$, we must have
$$
v \in H^1_0(\sO\cup\Gamma_0,\fw).
$$
Substituting $v=\zeta^2 u_y$ in the expression for $a_\lambda(u,v)$ given by \eqref{eq:HestonWithKillingBilinearForm} and \eqref{eq:BilinearFormCoerciveHeston} and recalling that by Assumption \ref{assump:HestonCoefficientb1} we may assume $b_1=0$, yields,
\begin{align*}
a_\lambda(u,v) &= \frac{1}{2}\int_\sO\left(u_xu_{xy} + \rho\sigma u_yu_{xy}
+ \rho\sigma u_xu_{yy} + \sigma^2u_yu_{yy}\right)\zeta^2 y\,\fw\,dxdy
\\
&\quad - \int_\sO\left(a_1yu_x - ru\right)u_y\zeta^2\,\fw\,dxdy
\\
&\quad - \frac{\gamma}{2}\int_\sO\left(u_x + \rho\sigma u_y\right)u_y\sign(x)\zeta^2 y\,\fw\,dxdy
\\
&\quad + \int_\sO\left(u_x + \rho\sigma u_y\right)u_y\zeta\zeta_x y\,\fw\,dxdy
\\
&\quad + \int_\sO\left(\rho\sigma u_x + \sigma^2u_y\right)u_y \zeta\zeta_y y\,\fw\,dxdy
\\
&\quad + \lambda\int_\sO(1+y)u\zeta^2\,\fw\,dxdy
\\
&=: K_1+K_2+K_3+K_4+K_5+K_6.
\end{align*}
Using $u_xu_{xy} = \frac{1}{2}(u_x^2)_y$, $u_yu_{yy} = \frac{1}{2}(u_y^2)_y$, and the expression for $\fw$ in \eqref{eq:HestonWeight}, we obtain
\begin{align*}
K_1 &= \frac{1}{2}\int_\sO y^\beta\left(\frac{1}{2}((u_x^2)_y + \rho\sigma u_yu_{xy}
+ \rho\sigma u_xu_{yy} + \frac{\sigma^2}{2}(u_y^2)_y\right) \zeta^2 e^{-\gamma|x|-\mu y}\,dxdy
\\
&= \frac{1}{4}\int_\sO y^\beta\left((u_x^2)_y + \sigma^2(u_y^2)_y\right) \zeta^2 e^{-\gamma|x|-\mu y}\,dxdy
\\
&\quad + \frac{1}{2}\int_\sO y^\beta\rho\sigma(u_xu_y)_y \zeta^2 e^{-\gamma|x|-\mu y}\,dxdy
\\
&=: K_{11} + K_{12}.
\end{align*}
Integration by parts with respect to $y$ in the term $K_{11}$, using the assumption in this step that $\sO = \Gamma_0\times\RR^+$, gives
\begin{align*}
K_{11} &= \frac{1}{4}\int_\sO y^\beta\left((u_x^2)_y + \sigma^2(u_y^2)_y\right) \zeta^2 e^{-\gamma|x|-\mu y}\,dxdy
\\
&= -\frac{\beta}{4}\int_\sO y^{\beta-1}\left(u_x^2 +\sigma^2 u_y^2\right) \zeta^2 e^{-\gamma|x|-\mu y}\,dxdy
\\
&\quad + \frac{\mu}{4}\int_\sO y^\beta\left(u_x^2 +\sigma^2 u_y^2\right) \zeta^2 e^{-\gamma|x|-\mu y}\,dxdy
\\
&\quad - \frac{1}{2}\int_\sO y^{\beta}\left(u_x^2 +\sigma^2 u_y^2\right)\zeta\zeta_y e^{-\gamma|x|-\mu y}\,dxdy
\\
&=: K_{111} + K_{112} + K_{113}.
\end{align*}
Integration by parts with respect to $y$ in the term $K_{12}$, using the assumption in this step that $\sO = \Gamma_0\times\RR^+$, gives
\begin{align*}
K_{12} &= \frac{1}{2}\int_\sO y^\beta\rho\sigma(u_xu_y)_y \zeta^2 e^{-\gamma|x|-\mu y}\,dxdy
\\
&= -\frac{\beta}{2}\int_\sO y^{\beta-1}\rho\sigma u_xu_y \zeta^2 e^{-\gamma|x|-\mu y}\,dxdy
\\
&\quad + \frac{\mu}{2}\int_\sO y^\beta\rho\sigma u_xu_y \zeta^2 e^{-\gamma|x|-\mu y}\,dxdy
\\
&\quad - \int_\sO y^{\beta}\rho\sigma u_xu_y \zeta\zeta_y e^{-\gamma|x|-\mu y}\,dxdy
\\
&=: K_{121} + K_{122} + K_{123}.
\end{align*}
Write
\begin{align*}
K_2 &= -\int_\sO a_1u_xu_y \zeta^2 y\,\fw\,dxdy + \int_\sO ruu_y \zeta^2 \,\fw\,dxdy
\\
&:= K_{21} + K_{22}.
\end{align*}
The identity \eqref{eq:VariationalEqualityCoerciveSmoothSource} with $v=\zeta^2 u_y$ gives
\begin{equation}
\label{eq:H1SpecialIntegralIdentity}
\begin{aligned}
{}&(K_{111} + K_{112}+ K_{113}) + (K_{121} + K_{122} + K_{123}) + K_{21} + K_{22}
\\
&\quad + K_3 + K_4 + K_5 + K_6 = (\tilde f,\zeta^2 u_y)_H,
\end{aligned}
\end{equation}
and it remains to bound $K_{111} + K_{121}$ using the other good terms. Observe that
\begin{align*}
-(K_{111} + K_{121}) &= \frac{\beta}{4}\int_\sO \left(u_x^2 +\sigma^2 u_y^2\right)\zeta^2\fw\,dxdy
\\
&\quad + \frac{\beta}{2}\int_\sO \rho\sigma u_xu_y\zeta^2\fw\,dxdy
\\
&\geq \frac{\beta(1-|\rho|)}{4}\int_\sO \left(u_x^2 +\sigma^2 u_y^2\right)\zeta^2\fw\,dxdy,
\end{align*}
using $2|\rho\sigma u_xu_y| \leq |\rho|(u_x^2+\sigma^2u_y^2)$. Hence, writing the identity \eqref{eq:H1SpecialIntegralIdentity} as
\begin{align*}
-(K_{111} + K_{121}) &= -(\tilde f,\zeta^2 u_y)_H + K_{112}+ K_{113} + K_{122} + K_{123} + K_{21} + K_{22}
\\
&\quad + K_3 + K_4 + K_5 + K_6,
\end{align*}
and using the universal estimate \eqref{eq:AuxiliarySpecialWeightedH1EstimateCutoffFunction} for $D\zeta = D\zeta_R$, yields
\begin{align*}
|\zeta Du|_H^2 \leq C\left(|\zeta \tilde f|_H|\zeta Du|_H + |y^{1/2}\zeta Du|_H^2 + |\zeta u|_H|\zeta Du|_H + |(1+y)\zeta u|_H^2\right).
\end{align*}
Rearrangement and taking square roots gives
\begin{align*}
|\zeta Du|_H &\leq C\left(|\zeta \tilde f|_H + |y^{1/2}\zeta Du|_H + |(1+y)\zeta u|_H\right)
\\
&\leq C\left(|\tilde f|_H + |y^{1/2}Du|_H + |(1+y)\zeta u|_H\right)
\\
&\leq C\left(|\tilde f|_H + |(1+y)\zeta u|_H\right) \quad\hbox{(by \eqref{eq:VariationalEqualityCoerciveBoundfH})},
\end{align*}
and thus, recalling that $\zeta = \zeta_R$, this gives \eqref{eq:AuxiliarySpecialWeightedH1EstimateWithCutoff}. This completes the proof.
\end{proof}

\subsection{A priori estimate for second-order derivatives of a solution to the variational equation}
\label{subsec:SecondOrderEstimatesVarEquality}
We have an analogue of \cite[Theorem 8.8]{GT}, \cite[Theorem 6.3.1]{Evans}.

\begin{thm}[Interior $H^2$ regularity]
\label{thm:LocalRegularityEllipticHeston}
If $u\in H^1_0(\sO\cup\Gamma_0,\fw)$ is a solution to Problem \ref{prob:HestonWeakMixedBVPHomogeneous}, then $u\in H^2_{\textrm{loc}}(\sO)$ and, for any pair of subdomains $\sO''\Subset\sO'\subset\sO$,
\begin{equation}
\label{eq:H2LocalHestonEstimate}
\|u\|_{H^2(\sO'')} \leq C\left(\|f\|_{L^2(\sO',\fw)} + \|(1+y)u\|_{L^2(\sO',\fw)}\right),
\end{equation}
where the constant $C$ depends only on $\sO'',\sO'$ and the constant coefficients of $A$.
\end{thm}

\begin{proof}
This follows from \cite[Theorem 8.8]{GT} or \cite[Theorem 6.3.1]{Evans}, and \eqref{eq:VariationalEqualityHestonBoundH}.
\end{proof}

We have an analogue of \cite[Theorem 8.12]{GT}, \cite[Theorem 6.3.4]{Evans}:

\begin{lem}[Local $H^2$ estimate near $\Gamma_1$]
\label{lem:GammaOneRegularityEllipticHeston}
Require that the domain, $\sO$, obey Hypotheses \ref{hyp:HestonDomainNearGammaZero} and \ref{hyp:HestonDomainNearGammaOne} with $k=2$. Let $\delta_0$ be as in Hypothesis \ref{hyp:HestonDomainNearGammaOne}. Suppose $f\in C^\infty_0(\sO)$ and that $u\in C^\infty(\bar\sO)$ is a solution to Problem \ref{prob:HestonStrongMixedBVPHomogeneous}. Then, for the \emph{bounded} subdomains $U'  = U_j' \subset U_j  = U \subset \RR^2$ defined in Hypothesis \ref{hyp:HestonDomainNearGammaOne}, we have
\begin{equation}
\label{eq:H2GammaOneLocalHestonEstimateWeights}
\|yD^2u\|_{L^2(U'\cap\sO,\fw)} \leq C\left(\|f\|_{L^2(U\cap\sO,\fw)} + \|(1+y)Du\|_{L^2(U\cap\sO,\fw)} + \|(1+y)u\|_{L^2(U\cap\sO,\fw)}\right),
\end{equation}
where the constant $C$ depends only on the constant coefficients of $A$ and the constants $\delta_0,\delta_1,M_1,R_1$ of Hypothesis \ref{hyp:HestonDomainNearGammaOne}.
\end{lem}

\begin{proof}
From \eqref{eq:IntroHestonMixedProblemHomogeneous} and \eqref{eq:IntroHestonMixedProblemHomogeneousBC}, we have
$$
\bar Au=y^{-1}f \quad\hbox{on }\sO\quad\hbox{and}\quad u=0 \quad\hbox{on }\Gamma_1,
$$
where, by  \eqref{eq:OperatorHestonIntro},
$$
\bar Au := -\frac{1}{2}\left(u_{xx} + 2\rho\sigma u_{xy} + \sigma^2 u_{yy}\right) - ((r-q)/y-1/2)u_x - \kappa(\theta/y-1)u_y + (r/y)u.
$$
Let $\eta\in C^\infty_0(\RR^2)$ be a cutoff function with $0\leq\eta\leq 1$, $\eta=1$ on $B(1/2)$, $\eta = 0$ on $\RR^2\less B(1)$, and $|D^\alpha\eta| \leq 100$ for multi-indices $|\alpha|\leq 2$. Let $\zeta = \eta\circ\Phi\in C^\infty_0(\RR^2)$, where $\Phi =\Phi_j$ as in Hypothesis \ref{hyp:HestonDomainNearGammaOne}, so $\zeta=1$ on $U'$ and $\zeta = 0$ on $\RR^2\less U$. Consequently,
$$
\bar A(\zeta u) = \bar f \quad\hbox{on }U\quad\hbox{and}\quad \zeta u=0 \quad\hbox{on }\partial(U\cap\sO),
$$
where
\begin{align*}
\bar f &:= \zeta y^{-1}f - [\bar A,\zeta]u
\\
&= \zeta y^{-1}f + \frac{1}{2}\left(\zeta_{xy}u + 2\zeta_x u_x + 2\rho\sigma(\zeta_{xx} u + \zeta_x u_y + \zeta_y u_x) + \sigma^2(\zeta_{yy} u + 2\zeta_y u_y)\right)
\\
&\quad +  ((r-q)/y-1/2)\zeta_x u + \kappa(\theta/y-1)\zeta_y u \quad\hbox{(by \eqref{eq:ACommutator})}.
\end{align*}
We have $\zeta u \in H^1_0(U\cap\sO)$ since $u\in C^\infty(\bar\sO)$ by hypothesis and $\zeta u=0$ on $\partial(U\cap\sO)$. Thus, $\zeta u \in H^1_0(U\cap\sO)$ and $\zeta u$ is a weak solution to $\bar A(\zeta u) = \bar f$ on $U\cap\sO$ in the sense of \cite[\S 6.1]{Evans}. From \cite[Equation (6.3.42)]{Evans} and examination of the proof of \cite[Theorems 6.3.1 \& 6.3.4]{Evans} to determine the dependencies of the constant $C$, we obtain
\begin{equation}
\label{eq:H2GammaOneLocalHestonEstimateWeights_prelim}
\begin{aligned}
\|\zeta u\|_{H^2(U\cap\sO)} &\leq C\left(\|\bar f\|_{L^2(U\cap\sO)} + \|\zeta u\|_{L^2(U\cap\sO)}\right)
\\
&\leq C\left(\|y^{-1}f\|_{L^2(U\cap\sO)} + \|Du\|_{L^2(U\cap\sO)} + \|u\|_{L^2(U\cap\sO)}\right),
\end{aligned}
\end{equation}
where $C$ depends only on the constant coefficients of $A$ and the constants $\delta_0,\delta_1,M_1,R_1$ of Hypotheses \ref{hyp:HestonDomainNearGammaZero} and \ref{hyp:HestonDomainNearGammaOne}. Let $\bar y = \max\{y:(x,y)\in U\}$ and $\underline{y} = \min\{y:(x,y)\in U\}$. Then
$$
\bar y \leq \underline{y} + \hbox{diam}(U) = \underline{y}(1 + \underline{y}^{-1}\hbox{diam}(U)),
$$
and so, because $\underline{y} \geq \delta_0/4$ since $U\subset \RR\times(\delta_0/4,\infty)$ by Hypothesis \ref{hyp:HestonDomainNearGammaOne}, we have
$$
\bar y \leq \underline{y}(1 + 4\delta_0^{-1}\hbox{diam}(U)).
$$
Hence,
\begin{align*}
\|yD^2u\|_{L^2(U'\cap\sO)} &\leq \bar y \|D^2u\|_{L^2(U'\cap\sO)} \leq C\underline{y}\|D^2u\|_{L^2(U'\cap\sO)}
\\
&\leq C\underline{y}\left(\|y^{-1}f\|_{L^2(U\cap\sO)} + \|Du\|_{L^2(U\cap\sO)} + \|u\|_{L^2(U\cap\sO)}\right)
\quad\hbox{(by \eqref{eq:H2GammaOneLocalHestonEstimateWeights_prelim})}
\\
&\leq C\left(\|f\|_{L^2(U\cap\sO)} + \|(1+y)Du\|_{L^2(U\cap\sO)} + \|(1+y)u\|_{L^2(U\cap\sO)}\right).
\end{align*}
Let $\bar x = \max\{|x|:(x,y)\in U\}$ and $\underline{x} = \min\{|x|:(x,y)\in U\}$. For $\beta\geq 1$,
\begin{align*}
\sup_{(x,y)\in U}y^{\beta-1}e^{-\gamma|x|-\mu y} &\leq \bar y^{\beta-1}e^{-\gamma|\underline{x}|-\mu \underline{y}}
\leq C_1\underline{y}^{\beta-1}e^{-\gamma|\bar{x}|-\mu \bar{y}}
\\
&\leq C_1\inf_{(x,y)\in U}y^{\beta-1}e^{-\gamma|x|-\mu y},
\end{align*}
where $C_1$ depends only on $\beta,\gamma,\mu,\delta_0,\hbox{diam}(U)$; when $0<\beta<1$, the same estimate holds using $\underline{y}^{\beta-1} \leq C_2\bar y^{\beta-1}$, where $C_2 = (1 + 4\delta_0^{-1}\hbox{diam}(U))^{1-\beta}$. Therefore, the estimate \eqref{eq:H2GammaOneLocalHestonEstimateWeights} follows, since $\fw = y^{\beta-1}e^{-\gamma|x|-\mu y}$ by \eqref{eq:HestonWeight}.
\end{proof}

Next, we have the following analogue of the second-derivative estimate in \cite[Lemma 4.6.1]{Koch}.

\begin{prop}[A priori second-derivative estimate for a solution to the variational equation]
\label{prop:AuxiliaryWeightedH2EstimateSisZero}
Require that the domain, $\sO$, obeys Hypotheses \ref{hyp:HestonDomainNearGammaZero} and \ref{hyp:HestonDomainNearGammaOne} with $k=2$ and $\alpha\in(0,1)$. Then there is a positive constant $C$, depending only on the constant coefficients of $A$ and the constants $\delta_0,\delta_1,M_1,R_1$ of Hypothesis \ref{hyp:HestonDomainNearGammaZero}, such that, if $f\in L^2(\sO,\fw)$ and $u\in H^1_0(\sO\cup\Gamma_0,\fw)$ is a solution to Problem \ref{prob:HestonWeakMixedBVPHomogeneous} and $f, yu, (1+y)Du\in L^2(\sO,\fw)$, then $u\in H^2_{\textrm{loc}}(\sO)$ and $yD^2u \in L^2(\sO,\fw)$ and
\begin{equation}
\label{eq:AuxiliaryWeightedH2EstimateSisZero}
\|yD^2u\|_{L^2(\sO,\fw)} \leq C\left(\|f\|_{L^2(\sO,\fw)} + \|(1+y)Du\|_{L^2(\sO,\fw)} + \|(1+y)u\|_{L^2(\sO,\fw)}\right).
\end{equation}
\end{prop}

\begin{rmk}[Finite differences and Proposition \ref{prop:AuxiliaryWeightedH2EstimateSisZero} when the domain is a half-plane]
When $\sO=\HH$, the a priori estimate and regularity result in Proposition \ref{prop:AuxiliaryWeightedH2EstimateSisZero} can be proved by adapting the finite difference arguments employed in the proofs of \cite[Theorems 6.3.1 \& 6.3.4]{Evans} or \cite[Theorems 8.8 \& 8.12]{GT}, without appealing to Theorem \ref{thm:SmoothnessUpToBdryWeakSolution}.
\end{rmk}

\begin{proof}[Proof of Proposition \ref{prop:AuxiliaryWeightedH2EstimateSisZero}]
Lemma \ref{lem:GammaOneRegularityEllipticHeston} takes care of the estimate for $yD^2u$ near $\Gamma_1$ so we now focus on the $L^2$ estimate for $yD^2u$ near $\Gamma_0$ and in the interior of $\sO$, and then combine these bounds to obtain the desired $L^2$ estimate for $yD^2u$ over $\sO$. In our proof of Proposition \ref{prop:AuxiliaryWeightedH2EstimateSisZero} we shall again appeal to Remark \ref{rmk:ProofWhenGammaOneSmooth} and assume $\Gamma_1$ is $C^\infty$ for the sake of clarity of exposition.

\medskip
\noindent\textbf{Step 1.} \emph{Reduction to the case of $u \in C^\infty(\bar\sO)$.} Let $\{f_{\lambda, n}\}_{n\geq 1} \subset C^\infty_0(\sO)$ and $\{u_n\}_{n\geq 1} \subset C^\infty(\bar\sO)$ be as in the proof of Proposition \ref{prop:AuxiliarySpecialWeightedH1EstimateSisZero}. Suppose we have shown that
\begin{equation}
\label{eq:AuxiliaryWeightedH2EstimateSisZero_n}
\begin{aligned}
\|y\zeta_RD^2u_n\|_{L^2(\sO,\fw)} &\leq C\left(\|f_{\lambda, n}\|_{L^2(\sO,\fw)} + \|(1+y)\zeta_{2R}Du_n\|_{L^2(\sO,\fw)} \right.
\\
&\quad + \left. \|(1+y)\zeta_{2R}u_n\|_{L^2(\sO,\fw)}\right), \quad \forall n\geq 1, R\geq 2,
\end{aligned}
\end{equation}
where $C$ is as in Proposition \ref{prop:AuxiliaryWeightedH2EstimateSisZero} and $\zeta_R$ is the cutoff function in Definition  \ref{defn:RadialCutoffFunction}. Since $u_n\to u$ and $\zeta_{2R}Du_n \to \zeta_{2R}Du$ strongly in $L^2(\sO,\fw)$ by \eqref{eq:SmoothunH1convergestou} and \eqref{eq:CutoffDunConvergenceToTheLimit} (with $R$ replaced by $2R$), respectively, we have
\begin{align*}
\zeta_{2R}(1+y)u_n &\to \zeta_{2R}(1+y)u \quad\hbox{strongly in $L^2(\sO,\fw)$ as $n\to\infty$},
\\
\zeta_{2R}(1+y)Du_n &\to \zeta_{2R}(1+y)Du \quad\hbox{strongly in $L^2(\sO,\fw)$ as $n\to\infty$}.
\end{align*}
Using \eqref{eq:SmoothfnL2convergestoflambda} and an argument similar to that used in the proof of Proposition \ref{prop:AuxiliaryWeightedH2EstimateSisZero}, we see that \eqref{eq:AuxiliaryWeightedH2EstimateSisZero_n} implies that the sequence $\{y\zeta_RD^2u_n\}_{n\geq 1}$ is Cauchy in $L^2(\sO,\fw)$ and so \footnote{It would suffice for the proof to use \eqref{eq:AuxiliaryWeightedH2EstimateSisZero_n} to show that $y\zeta_RD^2u_n \rightharpoonup w$ weakly in $L^2(\sO,\fw)$ as $n\to\infty$.}
\begin{equation}
\label{eq:yCutoffD2unConvergesSomeLimit}
y\zeta_RD^2u_n \to w \quad\hbox{strongly in $L^2(\sO,\fw)$ as $n\to\infty$},
\end{equation}
for some limit $w\in L^2(\sO,\fw)$. By Theorem \ref{thm:LocalRegularityEllipticHeston} we have $D^2u \in L^2_{\textrm{loc}}(\sO)$ and \eqref{eq:H2LocalHestonEstimate} implies that
\begin{align*}
\|D^2(u_n-u_{n'})\|_{L^2(\sO'')} &\leq C'\left(\|f_{\lambda, n}-f_{\lambda, n'}\|_{L^2(\sO',\fw)} + \|u_n-u_{n'}\|_{L^2(\sO',\fw)}\right),
\\
&\qquad \forall n,n'\geq 1, R\geq 2,
\end{align*}
for any $\sO''\Subset\sO'\Subset\sO$ and a constant $C'$ depending only on the constant coefficients of $A$, $\sO'',\sO'$, and $R$. Therefore,
\begin{equation}
\label{eq:yCutoffD2unConvergesLocallyTheLimit}
y\zeta_RD^2u_n \to y\zeta_RD^2u \quad\hbox{in the sense of $L^2_{\textrm{loc}}(\sO)$ as $n\to\infty$}.
\end{equation}
Consequently, $w = y\zeta_RD^2u$ a.e. on $\sO$ by \eqref{eq:yCutoffD2unConvergesSomeLimit} and \eqref{eq:yCutoffD2unConvergesLocallyTheLimit}, and so by \eqref{eq:yCutoffD2unConvergesSomeLimit}.
\begin{equation}
\label{eq:yCutoffD2unConvergesTheLimit}
y\zeta_RD^2u_n \to y\zeta_RD^2u \quad\hbox{strongly in $L^2(\sO,\fw)$ as $n\to\infty$}.
\end{equation}
Taking limits in \eqref{eq:AuxiliaryWeightedH2EstimateSisZero_n} as $n\to\infty$ gives
\begin{equation}
\label{eq:AuxiliaryWeightedH2EstimateSisZero_Cutoff}
\|y\zeta_RD^2u\|_{L^2(\sO,\fw)} \leq C\left(\|f_\lambda\|_{L^2(\sO,\fw)} + \|(1+y)\zeta_{2R}Du\|_{L^2(\sO,\fw)} + \|(1+y)\zeta_{2R}u\|_{L^2(\sO,\fw)}\right),
\end{equation}
for all $R\geq 2$, where $C$ is as in the hypotheses of Proposition \ref{prop:AuxiliaryWeightedH2EstimateSisZero}. Finally, using $|\zeta_{2R}|\leq 1$ in the right-hand side of \eqref{eq:AuxiliaryWeightedH2EstimateSisZero_Cutoff} and taking limits as $R\to\infty$ in \eqref{eq:AuxiliaryWeightedH2EstimateSisZero_Cutoff} and applying the dominated convergence theorem to the left-hand term yields
$$
\|yD^2u\|_{L^2(\sO,\fw)} \leq C\left(\|f_\lambda\|_{L^2(\sO,\fw)} + \|(1+y)Du\|_{L^2(\sO,\fw)} + \|(1+y)u\|_{L^2(\sO,\fw)}\right),
$$
and \eqref{eq:AuxiliaryWeightedH2EstimateSisZero} follows from the preceding estimate and \eqref{eq:Defnflambda}.

Assuming the reduction in Step 1 to $u\in C^\infty(\bar\sO)$ for the remainder of the proof of Proposition \ref{prop:AuxiliaryWeightedH2EstimateSisZero}, we shall derive the $L^2(\sO,\fw)$ estimate \eqref{eq:AuxiliaryWeightedH2EstimateSisZero} for $yD^2u$ using the following steps:
\begin{enumerate}
\item[(2)]\setcounter{enumi}{2} $L^2$ estimate for $yD^2u$ over $\sO^1_{\delta_1}\less\sO^0_{\delta_0/2}$;
\item $L^2$ estimate for $yD^2u$ over $\sO$ assuming $u=0$ on $\sO^1_{\delta_1/2}\less\sO^0_{\delta_0}$;
\item $L^2$ estimate for $yD^2u$ over $\sO$ without assuming $u=0$ on $\sO^1_{\delta_1/2}\less\sO^0_{\delta_0}$.
\end{enumerate}

\medskip
\noindent\textbf{Step 2.} \emph{$L^2$ estimate for $yD^2u$ over $\sO^1_{\delta_1}\less\sO^0_{\delta_0/2}$.} Let $\{U_j'\}$ and $\{U_j\}$ be as in Hypothesis \ref{hyp:HestonDomainNearGammaOne}. Then,
\begin{align*}
{}&\|yD^2u\|_{L^2(\sO^1_{\delta_1}\less\sO^0_{\delta_0/2},\fw)}^2
\\
&\leq \sum_j \int_{U_j'\cap\sO}y^2|D^2u|^2 \fw\,dxdy
\\
&\leq C\sum_j \int_{U_j\cap\sO}\left(|f|^2 + (1+y)^2|Du|^2 + (1+y)^2|u|^2\right)\fw\,dxdy
\quad\hbox{(by \eqref{eq:H2GammaOneLocalHestonEstimateWeights})}
\\
&\leq C(R_1+1)\int_\sO\left(|f|^2 + (1+y)^2|Du|^2 + (1+y)^2|u|^2\right)\fw\,dxdy
\quad\hbox{(by Hypothesis \ref{hyp:HestonDomainNearGammaOne}).}
\end{align*}
Thus,
\begin{equation}
\label{eq:AuxiliaryWeightedH2EstimateSisZero_GammaOne}
\begin{aligned}
\|yD^2u\|_{L^2(\sO^1_{\delta_1}\less\sO^0_{\delta_0/2},\fw)}
&\leq C\left(\|f\|_{L^2(\sO,\fw)} + \|(1+y)Du\|_{L^2(\sO,\fw)} \right.
\\
&\quad + \left. \|(1+y)u\|_{L^2(\sO,\fw)}\right),
\end{aligned}
\end{equation}
where the constant $C$ depends only on the constant coefficients of $A$ and $\gamma,\delta_0,\delta_1,M_1,R_1$.

\medskip
\noindent\textbf{Step 3.} \emph{$L^2$ estimate for $yD^2u$ over $\sO$ under the assumption that}
\begin{equation}
\label{eq:uZeroOnNeighborhoodGammaOne}
u = 0 \hbox{ on }\sO^1_{\delta_1/2}\less\sO^0_{\delta_0}.
\end{equation}
Because we shall need to integrate by parts with respect to $x$ or $y$ alone, we use \eqref{eq:uZeroOnNeighborhoodGammaOne} to extend $u$ by zero,
\begin{equation}
\label{eq:uExtendedByZero}
\bar u := \begin{cases} u &\hbox{on }\sO, \\ 0 &\hbox{on }(\RR\times[\delta_0,\infty))\less\sO, \end{cases}
\end{equation}
and observe that $\bar u \in C^\infty(\RR\times[\delta_0,\infty))$; we relabel $\bar u$ to $u$ for the remainder of this step. Moreover,
\begin{equation}
\label{eq:DomainNearGammaZero}
\sO\cap(\RR\times(0,\delta_0)) = \Gamma_0\times(0,\delta_0),
\end{equation}
by Hypothesis \ref{hyp:HestonDomainNearGammaZero}. Note that \eqref{eq:HestonModulusEllipticity} gives
\begin{equation}
\label{eq:SecondOrderEllipticInequality}
\begin{aligned}
\frac{\nu_0}{2}y\left(u_{xx}^2 + 2u_{xy}^2 + u_{yy}^2\right) &= \frac{\nu_0}{2}y\left(u_{xx}^2 + u_{xy}^2\right)
+ \frac{\nu_0}{2}y\left(u_{xy}^2 + u_{yy}^2\right)
\\
&\leq \frac{y}{2}\left(u_{xx}^2 + 2\rho\sigma u_{xx}u_{xy} + \sigma^2u_{xy}^2\right)
\\
&\quad + \frac{y}{2}\left(u_{xy}^2 + 2\rho\sigma u_{xy}u_{yy} + \sigma^2u_{yy}^2\right) \quad\hbox{on $\sO$}.
\end{aligned}
\end{equation}
Integrating by parts with respect to $y$, using \eqref{eq:uExtendedByZero} and \eqref{eq:DomainNearGammaZero}, gives
\begin{align*}
\int_\sO y^2u_{xy}^2\fw\,dxdy &= \int_\sO y^{\beta+1}u_{xy}u_{xy} e^{-\mu y-\gamma|x|}\,dxdy
\quad\hbox{(by \eqref{eq:HestonWeight})}
\\
&= -\int_\sO y^{\beta+1}u_xu_{xyy} e^{-\mu y-\gamma|x|}\,dxdy - \int_\sO y^{\beta}((\beta+1) - \mu y)u_xu_{xy} e^{-\mu y-\gamma|x|}\,dxdy
\\
&\quad + \int_{\Gamma_0} y^{\beta+1}u_xu_{xy} e^{-\gamma|x|}\,dx.
\end{align*}
But the integral over $\Gamma_0$ is zero since $u \in C^\infty(\bar\sO)$ and $\beta>0$. Then, integrating by parts with respect to $x$, using \eqref{eq:uExtendedByZero}, in the preceding equation gives
\begin{align*}
\int_\sO y^2u_{xy}^2\fw\,dxdy
&= \int_\sO y^{\beta+1}u_{xx}u_{yy} e^{-\mu y-\gamma|x|}\,dxdy - \int_\sO y^{\beta}((\beta+1) - \mu y)u_xu_{xy} e^{-\mu y-\gamma|x|}\,dxdy
\\
&\quad -\gamma\int_\sO y^{\beta+1}u_xu_{yy}\sign(x) e^{-\mu y-\gamma|x|}\,dxdy - \int_{\Gamma_1} y^{\beta+1}u_xu_{yy} e^{-\mu y-\gamma|x|}\,dxdy.
\end{align*}
But the integral over $\Gamma_1$ is zero since $u=0$ along $\Gamma_1$ and therefore $u_{yy}=0$ on $\Gamma_1\cap\sO^0_{\delta_0}$, while $u=0$ on $\sO^1_{\delta_1/2}\less\sO^0_{\delta_0}$ by the assumption \eqref{eq:uZeroOnNeighborhoodGammaOne}. Thus, by \eqref{eq:HestonModulusEllipticity} we obtain
\begin{equation}
\label{eq:MixedDerivativeIntegral}
\begin{aligned}
\int_\sO y^2u_{xy}^2\fw\,dxdy
&= \int_\sO y^2u_{xx}u_{yy} \fw\,dxdy - \int_\sO y((\beta+1) - \mu y)u_xu_{xy}\fw\,dxdy
\\
&\quad -\gamma\int_\sO y^2u_xu_{yy}\sign(x) \fw\,dxdy.
\end{aligned}
\end{equation}
Multiplying both sides of \eqref{eq:SecondOrderEllipticInequality} by $y\fw(x,y)$ and applying the preceding identity yields
\begin{align*}
\frac{\nu_0}{2}\int_\sO y^2|D^2u|^2\fw\,dxdy
&\leq
\frac{1}{2}\int_\sO y^2\left\{u_{xx}(u_{xx} + 2\rho\sigma u_{xy}) + \sigma^2u_{xy}^2\right\}\fw\,dxdy
\\
&\quad + \frac{1}{2}\int_\sO y^2\left\{u_{xy}^2 + u_{yy}(2\rho\sigma u_{xy} + \sigma^2u_{yy})\right\}\fw\,dxdy
\\
&= \frac{1}{2}\int_\sO y^2u_{xx}(u_{xx} + 2\rho\sigma u_{xy} + \sigma^2u_{yy})\fw\,dxdy \quad\hbox{(by \eqref{eq:MixedDerivativeIntegral})}
\\
&\quad + \frac{1}{2}\int_\sO y^2 u_{yy}(u_{xx} + 2\rho\sigma u_{xy} + \sigma^2u_{yy})\fw\,dxdy
\\
&\quad - \frac{1+\sigma^2}{2}\int_\sO y((\beta+1) - \mu y)u_xu_{xy}\fw\,dxdy
\\
&\quad - \frac{1+\sigma^2}{2}\gamma\int_\sO y^2u_xu_{yy}\sign(x) \fw\,dxdy,
\end{align*}
and thus
\begin{align*}
{}&\frac{\nu_0}{2}\int_\sO y^2|D^2u|^2\fw\,dxdy
\\
&\leq \frac{1}{2}\int_\sO y^2(u_{xx} + u_{yy})(u_{xx} + 2\rho\sigma u_{xy} + \sigma^2u_{yy})\fw\,dxdy
\\
&\quad - \frac{1+\sigma^2}{2}\int_\sO y((\beta+1) - \mu y)u_xu_{xy}\fw\,dxdy - \frac{1+\sigma^2}{2}\gamma\int_\sO y^2u_xu_{yy}\sign(x) \fw\,dxdy.
\end{align*}
We express our operator $A$ as $A=A_2+A_1+A_0$, where $A_i$ denotes the \emph{i}-th order part of $A$, and note that by \eqref{eq:OperatorHestonIntro} we have $-A_2u := \frac{1}{2}y(u_{xx} + 2\rho\sigma u_{xy} + \sigma^2u_{yy})$, $A_1u := -(r-q-\frac{y}{2})u_x - \kappa(\theta-y)u_y$, and $A_0u := ru$. Lemma \ref{lem:HestonWeightedNeumannBVPHomogeneous} implies that $u$ solves Problem \ref{prob:HestonMixedBVPHomogeneousClassical} since $C^\infty(\bar\sO) \subset H^2(\sO,\fw)$ and so $Au=f$ on $\sO$ by \eqref{eq:IntroHestonMixedProblemHomogeneous}. Therefore, because $-Au_2 = A_1u + A_0u - f$ on $\sO$, we obtain
\begin{align*}
\frac{\nu_0}{2}\int_\sO y^2|D^2u|^2\fw\,dxdy
&\leq \int_\sO y(u_{xx} + u_{yy})(A_1u+A_0u-f)\fw\,dxdy
\\
&\quad - \frac{1+\sigma^2}{2}\int_\sO ((\beta+1)yu_xu_{xy} - \mu y^2u_xu_{xy})\fw\,dxdy
\\
&\quad - \frac{1+\sigma^2}{2}\gamma\int_\sO y^2u_xu_{yy}\sign(x) \fw\,dxdy
\\
&\leq \frac{1}{2}|y(u_{xx} + u_{yy})|_H|A_1u+A_0u-f|_H
\\
&\quad + C\left(|(u_x, yu_{xy})_H| + |(yu_x, yu_{xy})_H| + |(yu_x, yu_{yy}\sign(x))_H|\right)
\\
&\leq C|yD^2u|_H\left(|(1+y)Du|_H  + |u|_H + |f|_H\right)
\\
&\quad + C\left(|u_x|_H|yu_{xy}|_H + |yu_x|_H|yu_{xy}|_H + |yu_x|_H|yu_{yy}|_H\right).
\end{align*}
Therefore,
$$
|yD^2u|_H^2 \leq C\left(|(1+y)Du|_H + |u|_H + |f|_H\right)|yD^2u|_H,
$$
and so
\begin{equation}
\label{eq:AuxiliaryWeightedH2EstimateSisZero_prefinal}
|yD^2u|_H \leq C\left(|(1+y)Du|_H + |u|_H + |f|_H\right).
\end{equation}
This is \eqref{eq:AuxiliaryWeightedH2EstimateSisZero}, except for the term $|u|_H \leq |(1+y)u|_H$ on the right-hand side, but obtained with the additional assumption \eqref{eq:uZeroOnNeighborhoodGammaOne}.

\medskip
\noindent\textbf{Step 4.} \emph{$L^2$ estimate for $yD^2u$ over $\sO$ without the assumption \eqref{eq:uZeroOnNeighborhoodGammaOne}.} We now remove the assumption \eqref{eq:uZeroOnNeighborhoodGammaOne} using a cutoff function argument. Let $\chi\in C^\infty(\bar\RR^2)$ be such that $0\leq\chi\leq 1$ on $\RR^2$ with
$$
\chi = \begin{cases}0 &\hbox{on }\sO^1_{\delta_1/2}\less\sO^0_{\delta_0}, \\ 1 &\hbox{ on }\sO\less(\sO^1_{\delta_1}\less\sO^0_{\delta_0/2}), \end{cases}
$$
and, for a positive constant $C$ depending only on the constants $\alpha,\delta_0,\delta_1,k,M_1$ in Hypothesis \ref{hyp:HestonDomainNearGammaOne},
\begin{equation}
\label{eq:FirstSecondDerivativeChiLinftyEstimates}
|D\chi| \leq C \quad\hbox{and}\quad |D^2\chi| \leq C \quad\hbox{on }\RR^2.
\end{equation}
Observe that the definition of $\chi$ implies
\begin{equation}
\label{eq:SuportOneMinusChi}
\supp(1-\chi) \subset \bar\sO^1_{\delta_1}\less\sO^0_{\delta_0/2}.
\end{equation}
From \eqref{eq:OperatorHestonIntro} and the fact that $Au = f$ on $\sO$ by \eqref{eq:IntroHestonMixedProblemHomogeneous}, we use \eqref{eq:ACommutator} to obtain
\begin{align*}
[A,\chi]u &:= A(\chi u) - \chi Au
\\
&= - \frac{y}{2}\left(\chi_{xx}u + 2\chi_xu_x + 2\rho\sigma(\chi_{xy} + \chi_yu_x + \chi_xu_y) + 2\sigma^2\chi_y u_y + \sigma^2\chi_{yy}u\right)
\\
&\quad - (r-q-y/2)\chi_xu - \kappa(\theta-y)\chi_yu,
\end{align*}
and hence $\chi u \in C^\infty(\bar\sO)$ solves
\begin{equation}
\label{eq:Achicutoffu}
A(\chi u) = f^\chi \quad\hbox{on }\sO, \quad \chi u = 0\quad\hbox{on }\Gamma_1,
\end{equation}
where
\begin{equation}
\label{eq:fchi}
f^\chi := \chi f + [A,\chi]u = \chi Au + [A,\chi]u.
\end{equation}
Moreover, $\chi u=0$ on $\sO^1_{\delta_1/2}\less\sO^0_{\delta_0}$, so we can apply \eqref{eq:AuxiliaryWeightedH2EstimateSisZero_prefinal} to the solution $\chi u \in C^\infty(\bar\sO)$ to \eqref{eq:Achicutoffu} to give
$$
|yD^2(\chi u)|_H \leq C\left(|(1+y)D(\chi u)|_H + |\chi u|_H + |f^\chi|_H\right),
$$
and thus, by \eqref{eq:Achicutoffu}, \eqref{eq:fchi}, and \eqref{eq:FirstSecondDerivativeChiLinftyEstimates}
\begin{equation}
\label{eq:AuxiliaryWeightedH2EstimateGammaZeroCutoff}
|yD^2(\chi u)|_H \leq C\left(|(1+y)Du|_H + |(1+y)u|_H + |f|_H\right),
\end{equation}
Hence, writing $u = \chi u + (1-\chi)u$ and using $\sO\cap\supp (1-\chi) \subset \sO^1_{\delta_1}\less\sO^0_{\delta_0/2}$ and
$D^2((1-\chi)u) = (1-\chi)D^2u - [D^2,\chi]u$, we obtain
\begin{align*}
|yD^2u|_H &\leq |yD^2(\chi u)|_H + |yD^2((1-\chi)u)|_H
\\
&\leq |yD^2(\chi u)|_H + |y(1-\chi)D^2u|_H + |y[D^2,\chi]u|_H
\\
&\leq |yD^2(\chi u)|_H + \|yD^2u\|_{L^2(\sO^1_{\delta_1}\less\sO^0_{\delta_0/2})} + C\left(|yDu|_H + |yu|_H\right)
\quad\hbox{(by \eqref{eq:FirstSecondDerivativeChiLinftyEstimates} and \eqref{eq:SuportOneMinusChi})}.
\end{align*}
Therefore, by applying \eqref{eq:AuxiliaryWeightedH2EstimateSisZero_GammaOne} and \eqref{eq:AuxiliaryWeightedH2EstimateGammaZeroCutoff} in the preceding inequality, we obtain the desired bound for a solution $u$ to Problem \ref{prob:HestonWeakMixedBVPHomogeneous},
\begin{equation}
\label{eq:AuxiliaryWeightedH2EstimateSmoothBoundedSolution}
|yD^2u|_H \leq C\left(|(1+y)Du|_H + |(1+y)u|_H + |f|_H\right),
\end{equation}
when $u\in C^\infty(\bar\sO)$. This completes the proof of Proposition \ref{prop:AuxiliaryWeightedH2EstimateSisZero}.
\end{proof}

\subsection[Global $H^2$ regularity of solutions to the variational equation]{Global $\mathbf{H^2}$ regularity of solutions to the variational equation}
\label{subsec:H2RegularityEquality}
We shall need

\begin{hyp}[Conditions on the source function]
\label{hyp:Sqrt1plusyfL2}
Require that $f\in L^2(\sO,\fw)$ obey
\begin{equation}
\label{eq:Sqrt1plusyfL2}
(1+y)^{1/2}f \in L^2(\sO,\fw).
\end{equation}
\end{hyp}

By combining the conclusions of Propositions \ref{prop:AuxiliarySpecialWeightedH1EstimateSisZero} and \ref{prop:AuxiliaryWeightedH2EstimateSisZero}, we obtain

\begin{cor}[A priori second-derivative estimate for a solution to the variational equation]
\label{cor:HEstimateyD2uSpecial}
Require that $\sO$ obey Hypotheses \ref{hyp:HestonDomainNearGammaZero} and \ref{hyp:HestonDomainNearGammaOne} with $k=2$ and $\alpha\in(0,1)$. Then there is a positive constant $C$, depending only on the constant coefficients of $A$ and the constants $\delta_0,\delta_1,M_1,R_1$ of Hypothesis \ref{hyp:HestonDomainNearGammaOne}, such that, if $f\in L^2(\sO,\fw)$ obeys \eqref{eq:Sqrt1plusyfL2} and $u\in H^1_0(\sO\cup\Gamma_0,\fw)$ is a solution to Problem \ref{prob:HestonWeakMixedBVPHomogeneous} and $yu\in L^2(\sO,\fw)$, then $yD^2u \in L^2(\sO,\fw)$ and
\begin{equation}
\label{eq:HEstimateyD2uSpecial}
\|yD^2u\|_{L^2(\sO,\fw)} \leq C\left(\|(1+y^{1/2})f\|_{L^2(\sO,\fw)}  + \|(1+y)u\|_{L^2(\sO,\fw)}\right).
\end{equation}
\end{cor}

\begin{proof}
Inequality \eqref{eq:AuxiliaryWeightedH1Estimate} gives
$$
|yDu|_H \leq C\left(|y^{1/2}f|_H + |(1+y)u|_H\right),
$$
while inequality \eqref{eq:AuxiliarySpecialWeightedH1EstimateSisZero} yields
$$
|Du|_H \leq C\left(|f|_H + |(1+y)u|_H\right).
$$
Thus, combining the preceding two estimates yields
\begin{equation}
\label{eq:1plusyDu}
|(1+y)Du|_H \leq C\left(|(1+y^{1/2})f|_H + |(1+y)u|_H\right).
\end{equation}
Now inequality \eqref{eq:AuxiliaryWeightedH2EstimateSisZero} yields
$$
|yD^2u|_H \leq C\left(|(1+y)Du|_H + |(1+y)u|_H + |f|_H\right),
$$
and combining the preceding bound with \eqref{eq:1plusyDu} yields the conclusion.
\end{proof}

\begin{hyp}[Combined conditions on the domain]
\label{hyp:DomainCombinedCondition}
Require that the domain, $\sO$, obeys Hypotheses \ref{hyp:HestonDomainNearGammaZero}, \ref{hyp:HestonDomainNearGammaOne} with $k=2$ and $\alpha\in(0,1)$, and \ref{hyp:GammeOneExtensionProperty} with $k=1$.
\end{hyp}

By combining the conclusions of Corollaries \ref{cor:HEstimateyD2uSpecial} and Proposition \ref{prop:AuxiliarySpecialWeightedH1EstimateSisZero} we obtain an analogue of \cite[Theorem 8.12]{GT}:

\begin{thm}[A priori global $H^2$ estimate for a solution to the variational equation]
\label{thm:GlobalRegularityEllipticHestonSpecial}
Require that the domain $\sO$ obeys Hypothesis \ref{hyp:DomainCombinedCondition}. If $f\in L^2(\sO,\fw)$ obeys \eqref{eq:Sqrt1plusyfL2} and $u\in H^1_0(\sO\cup\Gamma_0,\fw)$ is a solution to Problem \ref{prob:HestonWeakMixedBVPHomogeneous} and $yu \in L^2(\sO,\fw)$, then $u\in H^2(\sO,\fw)$ and $u$ solves Problem \ref{prob:HestonStrongMixedBVPHomogeneous}, and there is a positive constant $C$, depending only on the constant coefficients of $A$ and the constants $\delta_0,\delta_1,M_1,R_1$ of Hypothesis \ref{hyp:HestonDomainNearGammaOne}, such that
\begin{equation}
\label{eq:SecondDerivativeGlobalHestonEstimateSpecial}
\|u\|_{H^2(\sO,\fw)} \leq C\left(\|(1+y)^{1/2}f\|_{L^2(\sO,\fw)} + \|(1+y)u\|_{L^2(\sO,\fw)}\right).
\end{equation}
\end{thm}

\begin{proof}
Combining inequalities \eqref{eq:1plusyDu} and \eqref{eq:HEstimateyD2uSpecial} yields
$$
|yD^2u|_H + |(1+y)Du|_H \leq C\left(|(1+y^{1/2})f|_H + |(1+y)u|_H\right).
$$
The preceding estimate and the Definition \ref{defn:H2WeightedSobolevSpaces} of $H^2(\sO,\fw)$ yields \eqref{eq:SecondDerivativeGlobalHestonEstimateSpecial}. We see that $u$ solves Problem \ref{prob:HestonStrongMixedBVPHomogeneous} by applying Lemma \ref{lem:HestonWeightedNeumannBVPHomogeneous}.
\end{proof}

By combining Theorems \ref{thm:ExistenceUniquenessEllipticHeston_Improved} and \ref{thm:GlobalRegularityEllipticHestonSpecial} we obtain the following existence and uniqueness result for solutions to Problem \ref{prob:HestonStrongMixedBVPHomogeneous} with the aid of

\begin{hyp}[Conditions on envelope functions]
\label{hyp:UpperLowerBoundsSolutionsNoncoerciveEquation}
Require that there exist $M, m \in H^2(\sO,\fw)$ obeying \eqref{eq:SourceFunctionTraceBounds}, \eqref{eq:mMIneqOnDomain}, \eqref{eq:SourceFunctionBounds}, and \eqref{eq:OneplusyMmInL2}.
\end{hyp}

\begin{thm}[Existence and uniqueness for strong solutions to the non-coercive variational equation]
\label{thm:ExistenceUniquenessH2RegularEllipticHeston}
Assume the Hypothesis \ref{hyp:NoncoerciveHeston} on the coefficient, $r$, holds and that the Hypothesis \ref{hyp:DomainCombinedCondition} on the domain, $\sO$, holds. Suppose there are functions $M, m \in H^2(\sO,\fw)$ obeying \eqref{eq:SourceFunctionTraceBounds}, \eqref{eq:mMIneqOnDomain}, \eqref{eq:SourceFunctionBounds}, and \eqref{eq:OneplusyMmInL2}. Given a function $f \in L^2(\sO,\fw)$ obeying \eqref{eq:fBounds} and \eqref{eq:Sqrt1plusyfL2}, then there exists a solution, $u\in H^2(\sO,\fw)$, to Problem \ref{prob:HestonStrongMixedBVPHomogeneous} and $u$
\begin{enumerate}
\item Obeys the pointwise bounds \eqref{eq:uBounds},
\item Has the boundary property \eqref{eq:WeightedNeumannHomogeneousBCProb}, and
\item Obeys the estimate \eqref{eq:SecondDerivativeGlobalHestonEstimateSpecial}.
\end{enumerate}
Moreover, if there is a $\varphi \in H^2(\sO,\fw)$ obeying Hypothesis \ref{hyp:AuxBoundUniquenessSolutionsNoncoerciveEquation}, then the solution, $u$, is unique.
\end{thm}

\begin{proof}
Theorem \ref{thm:ExistenceUniquenessEllipticHeston_Improved} implies that there exists a $u \in H^1(\sO\cup\Gamma_0,\fw)$ which solves Problem \ref{prob:HestonWeakMixedBVPHomogeneous} and that $u$ obeys \eqref{eq:uBounds}. Because $M,m$ obey \eqref{eq:OneplusyMmInL2} and $u$ obeys \eqref{eq:uBounds}, we obtain $(1+y)u \in L^2(\sO,\fw)$. Consequently, Theorem \ref{thm:GlobalRegularityEllipticHestonSpecial} implies that $u\in H^2(\sO,\fw)$ and that $u$ solves Problem \ref{prob:HestonStrongMixedBVPHomogeneous} and obeys \eqref{eq:SecondDerivativeGlobalHestonEstimateSpecial}. Lemma \ref{lem:HestonWeightedNeumannBoundaryProperty} implies that $u$ has the boundary property \eqref{eq:WeightedNeumannHomogeneousBCProb}. Given $\varphi \in H^2(\sO,\fw)$ obeying Hypothesis \ref{hyp:AuxBoundUniquenessSolutionsNoncoerciveEquation}, the solution, $u \in H^1(\sO\cup\Gamma_0,\fw)$, to Problem \ref{prob:HestonWeakMixedBVPHomogeneous} is unique and therefore the solution, $u\in H^2(\sO,\fw)$, to Problem \ref{prob:HestonStrongMixedBVPHomogeneous} is unique.
\end{proof}

It is now straightforward to assemble the results we need to conclude the

\begin{proof}[Proof of Theorem \ref{thm:MainExistenceUniquenessBoundaryValueProblem}]
The hypotheses of Theorem \ref{thm:MainExistenceUniquenessBoundaryValueProblem} collect and summarize those of Theorem \ref{thm:ExistenceUniquenessH2RegularEllipticHeston} and so the result is a restatement of Theorem \ref{thm:ExistenceUniquenessH2RegularEllipticHeston} in the case $g=0$.

When $g\neq 0$, set $\tilde M := M-g$, $\tilde m := m-g$, and $\tilde f := f-Ag$. We obtain the inequalities \eqref{eq:SourceFunctionTraceBounds}, \eqref{eq:mMIneqOnDomain}, \eqref{eq:SourceFunctionBounds} for $\tilde M, \tilde m$ from the hypotheses on $M, m$ and the inequality \eqref{eq:fBounds} for $\tilde f$ from the hypotheses on $f$. The hypothesis $(1+y)^{1/2}g \in H^2(\sO,\fw)$ and the Definition \ref{defn:H2WeightedSobolevSpaces} imply that $(1+y)g \in L^2(\sO,\fw)$ and ensures that $\tilde M, \tilde m$ obey \eqref{eq:OneplusyMmInL2}. The hypothesis $(1+y)^{1/2}g \in H^2(\sO,\fw)$ and the Definition \ref{defn:H2WeightedSobolevSpaces} also imply that $(1+y)^{1/2}Ag \in L^2(\sO,\fw)$ and so the hypotheses on $f$ ensure that $\tilde f \in L^2(\sO,\fw)$ obeys \eqref{eq:Sqrt1plusyfL2}. Theorem \ref{thm:ExistenceUniquenessH2RegularEllipticHeston} now implies that there exists a solution, $\tilde u\in H^2(\sO,\fw)$, to Problem \ref{prob:HestonStrongMixedBVPHomogeneous} defined by the source function, $\tilde f$, the solution, $\tilde u$, obeys $\tilde m \leq \tilde u \leq \tilde M$, the solution, $\tilde u$, has the boundary property \eqref{eq:WeightedNeumannHomogeneousBCProb}, and $\tilde u$ obeys the estimate
$$
\|\tilde u\|_{H^2(\sO,\fw)} \leq C\left(\|(1+y)^{1/2}\tilde f\|_{L^2(\sO,\fw)} + \|(1+y)\tilde u\|_{L^2(\sO,\fw)}\right).
$$
Therefore, setting $u := \tilde u + g$, we see that $u$ is a solution to Problem \ref{prob:HestonStrongMixedBVPInhomogeneous} and obeys
\begin{align*}
\|u\|_{H^2(\sO,\fw)} &\leq C\left(\|(1+y)^{1/2}f\|_{L^2(\sO,\fw)} + \|(1+y)^{1/2}Ag\|_{L^2(\sO,\fw)} + \|(1+y)g\|_{L^2(\sO,\fw)} \right.
\\
&\quad + \left. \|(1+y)u\|_{L^2(\sO,\fw)}\right)
\\
&\leq C\left(\|(1+y)^{1/2}f\|_{L^2(\sO,\fw)} + \|(1+y)^{1/2}g\|_{H^2(\sO,\fw)} + \|(1+y)u\|_{L^2(\sO,\fw)}\right),
\end{align*}
as required. Finally, setting $\tilde \varphi := \varphi - g$, the hypotheses on $\varphi$ ensure that $\tilde\varphi$ obeys Hypothesis \ref{hyp:AuxBoundUniquenessSolutionsNoncoerciveEquation} (with $m$ replaced by $\tilde m$), and so the solution, $\tilde u$, to Problem \ref{prob:HestonStrongMixedBVPHomogeneous} is unique and hence the solution, $u$, to Problem \ref{prob:HestonStrongMixedBVPInhomogeneous} is unique.
\end{proof}

\subsection{Global H\"older regularity of solutions to the variational equation}
\label{subsec:HolderRegularityEquality}

\begin{thm}[H\"older continuity of solutions to the variational equation]
\label{thm:HolderContinuityHestonStatVarEquality}
Assume the hypotheses of Theorem \ref{thm:GlobalRegularityEllipticHestonSpecial}. If in addition, $f$ obeys
\begin{equation}
\label{eq:fLq}
f\in L^{q'}_{\textrm{loc}}(\sO\cup\Gamma_0), \quad\hbox{for some }q'>2+\beta,
\end{equation}
then $u\in C^\alpha_{\textrm{loc}}(\sO\cup\Gamma_1)\cap C_{\textrm{loc}}(\bar\sO)$, for all $\alpha\in [0,1)$. If in addition, $f$ obeys
\begin{equation}
\label{eq:fCnu}
f\in C^\alpha(\sO), \quad\hbox{for some } 0<\alpha<1,
\end{equation}
then $u\in C^{2,\alpha}(\sO)\cap C^\alpha_{\textrm{loc}}(\sO\cup\Gamma_1)\cap C_{\textrm{loc}}(\bar\sO)$.
\end{thm}

\begin{proof}
We have $u\in H^2(\sO,\fw)$ by Theorem \ref{thm:GlobalRegularityEllipticHestonSpecial} and so Lemma \ref{lem:H2SobolevEmbedding} implies that $u \in C^\alpha_{\textrm{loc}}(\sO\cup\Gamma_1)$, for $0\leq \alpha < 1$. Because $u\in H^1_0(\sO\cup\Gamma_0,\fw)$ is a solution to Problem \ref{prob:HestonWeakMixedBVPHomogeneous} and $f\in L^q_{\textrm{loc}}(\bar\sO)$, the fact that $u \in C_{\textrm{loc}}(\sO\cup\bar\Gamma_0)$ follows from \cite{Feehan_Pop_regularityweaksoln}. Hence, $u\in C^\alpha_{\textrm{loc}}(\sO\cup\Gamma_1)\cap C_{\textrm{loc}}(\bar\sO)$ since $\partial\sO = \bar\Gamma_0\cup\Gamma_1$. When $f\in C^\alpha(\sO)$, we obtain $u\in C^{2,\alpha}(\sO)$ by applying \cite[Theorem 6.13]{GT} to balls $B\Subset\sO$.
\end{proof}

\begin{rmk}[H\"older continuity up to $\Gamma_0$]
When $f\in L^{q'}_{\textrm{loc}}(\bar\sO)$, we expect that the result $u \in C_{\textrm{loc}}(\sO\cup\bar\Gamma_0)$ in \cite{Feehan_Pop_regularityweaksoln} can be extended to $u\in C^{\alpha_0}_{\textrm{loc}}(\sO\cup\bar\Gamma_0)$ for some $\alpha_0\in(0,1)$ depending only on the constant coefficients of $A$ and $f$, and so it would follow that $u\in C^{\alpha_0}_{\textrm{loc}}(\bar\sO)$.
\end{rmk}

\begin{rmk}[Comments on Theorems \ref{thm:GlobalRegularityEllipticHestonSpecial} and \ref{thm:HolderContinuityHestonStatVarEquality} and Problem \ref{prob:HestonMixedBVPHomogeneousClassical}]
The solution, $u$, provided by Theorems \ref{thm:GlobalRegularityEllipticHestonSpecial} and \ref{thm:HolderContinuityHestonStatVarEquality} almost matches our definition of a \emph{classical} solution in Problem \ref{prob:HestonMixedBVPHomogeneousClassical}, except that we have not shown that \eqref{eq:WeightedNeumannHomogeneousBCProb} holds everywhere pointwise along $\Gamma_0$ rather than in the trace sense. By adapting the methods of Daskalopoulos and Hamilton for the linearization of the porous medium equation \cite{DaskalHamilton1998}, we would expect that $u\in C^{2,\alpha}_s(\sO\cup\Gamma_0)$; see \cite[Theorems I.1.1, I.12.2, \& II.1.1]{DaskalHamilton1998} and \cite{PopThesis}.
\end{rmk}

\begin{proof}[Proof of Theorem \ref{thm:MainRegularityBoundaryValueProblem}]
When $g=0$, Theorem \ref{thm:MainRegularityBoundaryValueProblem} follows from Theorems \ref{thm:GlobalRegularityEllipticHestonSpecial} and \ref{thm:HolderContinuityHestonStatVarEquality}, with the remaining conclusion that $u \in C^{k+2,\alpha}_{\textrm{loc}}(\sO\cup\Gamma_1)$ obtained by applying \cite[Theorem 6.19]{GT} to $B\cap\sO$, where $B\subset\HH$ are balls centered at points in $\sO\cup \Gamma_1$, and using a cutoff function argument to localize the assertion of \cite[Theorem 6.19]{GT}; see the last paragraph of \cite[\S 6.4]{GT}.

When $g\neq 0$, we set $\tilde u := u-g$ and $\tilde f := f - Ag$, as in the proof of Theorem \ref{thm:MainExistenceUniquenessBoundaryValueProblem}. We have $Ag \in L^q_{\textrm{loc}}(\sO\cup\Gamma_0)$ by the hypothesis that $g \in W^{2,q}_{\textrm{loc}}(\sO\cup\Gamma_0)$ and so $\tilde f$ obeys \eqref{eq:fLq} (with $q$ in place of $q'$). Theorem \ref{thm:HolderContinuityHestonStatVarEquality} implies that $\tilde u\in C^\alpha_{\textrm{loc}}(\sO\cup\Gamma_1)\cap C_{\textrm{loc}}(\bar\sO)$. Because $g \in W^{2,2}_{\textrm{loc}}(\sO)$, we obtain $g \in C_{\textrm{loc}}^\alpha(\bar\sO)$ via the Sobolev embedding $W^{2,2}(\sO') \to C^\alpha(\bar\sO')$ for $\sO'\Subset\bar\sO$ \cite[Theorem 5.4, Part II ($\textrm{C}''$)]{Adams} and thus $u = \tilde u + g \in C^\alpha_{\textrm{loc}}(\sO\cup\Gamma_1)\cap C_{\textrm{loc}}(\bar\sO)$.

When $k\geq 0$, the hypothesis $g \in C^{k+2,\alpha}_{\textrm{loc}}(\sO\cup\Gamma_1)$ ensures that $Ag \in C^{k+\alpha}_{\textrm{loc}}(\sO\cup\Gamma_1)$ and hence $\tilde f \in C^{k+\alpha}_{\textrm{loc}}(\sO\cup\Gamma_1)$. We obtain $\tilde u \in C^{k+2,\alpha}_{\textrm{loc}}(\sO\cup\Gamma_1)$ from the case $g=0$ and thus $u = \tilde u + g \in C^{k+2,\alpha}_{\textrm{loc}}(\sO\cup\Gamma_1)$, as desired.
\end{proof}

\section{Regularity of solutions to the variational inequality}
\label{sec:RegularityInequality}
In this section we establish higher regularity results for solutions to the variational inequality for the elliptic Heston operator, Problem \ref{prob:HomogeneousHestonVIProblem}. In \S \ref{subsec:H2RegularityCoerciveInequality} we show that solutions to the coercive variational inequality (Theorem \ref{thm:VIRegularityEllipticHeston}) are in $H^2(\sO,\fw)$, while in \S \ref{subsec:H2RegularityInequality} we extend that regularity result to the case of the non-coercive variational inequality (Theorem \ref{thm:VIRegularityEllipticHestonNoncoercive}) and establish an existence and uniqueness result for strong solutions (\ref{thm:VIExistenceUniquenessH2RegularEllipticHeston}). Because the obstacle function is often not in $H^2(\sO,\fw)$, in \S \ref{subsec:LocalH2RegularityInequality} we extend Theorem \ref{thm:VIRegularityEllipticHestonNoncoercive} to the case where the obstacle function is only in $H^2(\sU,\fw)$ for some open subset $\sU\subseteqq\sO$ (Theorem \ref{thm:H2BoundSolutionHestonVarIneqLipschitz}). With the aid of additional hypotheses on the source and obstacle functions, we obtain local $W^{2,p}$, $C^{1,\alpha}$, and $C^{1,1}$ regularity results in \S \ref{subsec:LocalW2pC1AlphaRegularityInequality} (Theorem \ref{thm:LocalW2pRegularityHestonVI} and Corollaries \ref{cor:LocalC1alphaRegularity} and \ref{cor:LocalC11Regularity}).

\subsection[Global $H^2$ regularity of solutions to the coercive variational inequality]{Global $\mathbf{H^2}$ regularity of solutions to the coercive variational inequality}
\label{subsec:H2RegularityCoerciveInequality}
Before proceeding to the question of regularity proper, we shall need the following analogue of \cite[Theorem 3.1.3]{Bensoussan_Lions}.

\begin{hyp}[Conditions on the obstacle function]
\label{hyp:OstacleFunctionAprioriEstimates}
Require that the obstacle function $\psi \in H^1(\sO,\fw)$ obey \eqref{eq:ObstacleFunctionLessThanZero} and
\begin{equation}
\label{eq:1plusyObstacleFunctionL2}
(1+y)\psi \in L^2(\sO,\fw).
\end{equation}
\end{hyp}

\begin{thm}[A priori estimates for solutions to the penalized equation and coercive variational inequality]
\label{thm:VIPenalizationEstimateHeston}
Require that the domain $\sO$ obeys Hypotheses \ref{hyp:HestonDomainNearGammaZero} and \ref{hyp:GammeOneExtensionProperty} with $k=1$. Require that the obstacle function $\psi$ obeys \eqref{eq:ObstacleFunctionLessThanZero}, \eqref{eq:1plusyObstacleFunctionL2}, and $\psi \in H^2(\sO,\fw)$. If $u_\eps\in V$ is a solution to Problem \ref{prob:HestonPenalizedEquation}, then
\begin{align}
\label{eq:L2PenaltyFunctionBound}
|(\psi-u_\eps)^+|_H &\leq \eps\left(|f|_H + |A\psi|_H + \lambda|(1+y)\psi|_H\right),
\\
\label{eq:VIPenalizationEstimateHeston_prefinal}
\|(\psi - u_\eps)^+\|_V &\leq \sqrt{\eps}\sqrt{1/\nu_1}\left(|f|_H + |A\psi|_H + \lambda|(1+y)\psi|_H\right),
\end{align}
where $\nu_1$ is the constant in \eqref{eq:CoerciveHeston}, and, in addition, if respectively, $u\in V$ is a solution to Problem \ref{prob:HomogeneousHestonVIProblemCoercive}, then
\begin{equation}
\label{eq:VIPenalizationEstimateHeston}
\|u_\eps - u\|_V \leq C\sqrt{\eps}\left(|f|_H + |A\psi|_H + \lambda|(1+y)\psi|_H\right),
\end{equation}
and $C$ depends only on the constant coefficients of $A$.
\end{thm}

\begin{proof}
We adapt the argument of \cite[Theorem 3.1.3]{Bensoussan_Lions}. Since $\psi\leq 0$ on $\Gamma_1$ and $u_\eps = 0$ on $\Gamma_1$, then $\psi -u_\eps\leq 0$ on $\Gamma_1$ and $(\psi - u_\eps)^+ = 0$ on $\Gamma_1$, all in the trace sense, and thus $(\psi - u_\eps)^+ \in H^1_0(\sO\cup\Gamma_0$ by Lemmas \ref{lem:EvansGamma1TraceZero} and Lemma \ref{lem:SobolevSpaceClosedUnderMaxPart}. Therefore, we may choose $v = -(\psi - u_\eps)^+ \in V=H^1_0(\sO\cup\Gamma_0$ in \eqref{eq:PenalizedProblem} to give
\begin{equation}
\label{eq:MinusAlambdaUeps}
-a_\lambda(u_\eps,(\psi-u_\eps)^+) + \frac{1}{\eps}|(\psi-u_\eps)^+|_H^2 = -(f, (\psi-u_\eps)^+)_H,
\end{equation}
where we recall from \eqref{eq:PenalizationOperator} that $\beta_\eps(w) = -\frac{1}{\eps}(\psi - w)^+, \forall w\in V$. Since
\begin{align*}
a_\lambda(\psi,(\psi - u_\eps)^+) &= a(\psi,(\psi - u_\eps)^+) + \lambda((1+y)\psi,(\psi - u_\eps)^+)_H
\\
&= (A\psi, (\psi - u_\eps)^+)_H + \lambda((1+y)\psi,(\psi - u_\eps)^+)_H \quad\hbox{(by Lemma \ref{lem:HestonIntegrationByParts})}
\\
&= (A\psi + \lambda(1+y)\psi, (\psi - u_\eps)^+)_H,
\end{align*}
and $a_\lambda(v,v^+) = a_\lambda(v^+,v^+), \forall v \in V,$ we have
\begin{align*}
{}&a_\lambda((\psi - u_\eps)^+, (\psi - u_\eps)^+)  + \frac{1}{\eps}|(\psi-u_\eps)^+|_H^2
\\
&= -a_\lambda(u_\eps, (\psi - u_\eps)^+)  + \frac{1}{\eps}|(\psi-u_\eps)^+|_H^2 + a_\lambda(\psi, (\psi - u_\eps)^+)
\\
&= (-f, (\psi-u_\eps)^+)_H + a_\lambda(\psi, (\psi - u_\eps)^+)
\quad\hbox{(by \eqref{eq:MinusAlambdaUeps})}
\\
&= (-f+A\psi+\lambda(1+y)\psi, (\psi-u_\eps)^+)_H.
\end{align*}
Hence, \eqref{eq:CoerciveHeston} and the preceding equation yields
\begin{align}
\notag
\nu_1\|(\psi - u_\eps)^+\|_V^2 &\leq a_\lambda((\psi - u_\eps)^+, (\psi - u_\eps)^+)
\\
\label{eq:VIPenalizationEstimateHeston_intermed}
&\leq \left(|f|_H + |A\psi|_H + \lambda|(1+y)\psi|_H\right)|(\psi-u_\eps)^+|_H,
\\
\label{eq:L2PenaltyFunctionBound_prefinal}
\frac{1}{\eps}|(\psi-u_\eps)^+|_H^2 &\leq \left(|f|_H + |A\psi|_H + \lambda|(1+y)\psi|_H\right)|(\psi-u_\eps)^+|_H,
\end{align}
where $\nu_1$ depends only on the constant coefficients of $A$. The estimate \eqref{eq:L2PenaltyFunctionBound_prefinal}, after dividing by $|(\psi-u_\eps)^+|_H$, yields
$$
|(\psi-u_\eps)^+|_H \leq \eps\left(|f|_H + |A\psi|_H + \lambda|(1+y)\psi|_H\right),
$$
which is \eqref{eq:L2PenaltyFunctionBound}. Combining the estimate \eqref{eq:VIPenalizationEstimateHeston_intermed} with \eqref{eq:L2PenaltyFunctionBound} to bound the factor $|(\psi-u_\eps)^+|_H$ on the right-hand side of \eqref{eq:VIPenalizationEstimateHeston_intermed} yields
$$
\|(\psi - u_\eps)^+\|_V \leq \sqrt{\eps}\nu_1^{-1/2}\left(|f|_H + |A\psi|_H + \lambda|(1+y)\psi|_H\right),
$$
which is \eqref{eq:VIPenalizationEstimateHeston_prefinal}.

It remains to prove \eqref{eq:VIPenalizationEstimateHeston}. Writing
\begin{align*}
u - u_\eps &= u - \psi + (\psi - u_\eps)
\\
&= u - \psi + (\psi - u_\eps)^+ - (\psi - u_\eps)^-
\\
&= r_\eps + (\psi - u_\eps)^+,
\end{align*}
where we define
\begin{equation}
\label{eq:reps}
r_\eps := u - \psi - (\psi - u_\eps)^-,
\end{equation}
we see that proof of \eqref{eq:VIPenalizationEstimateHeston} reduces to finding a suitable bound for $\|r_\eps\|_V$, since
\begin{equation}
\label{eq:DiffUandUepsBasic}
\|u - u_\eps\|_V \leq \|r_\eps\|_V + \|(\psi - u_\eps)^+\|_V.
\end{equation}
Because $u - u_\eps \in V$ and $(\psi - u_\eps)^+ \in V$, we see that $r_\eps = u - u_\eps - (\psi - u_\eps)^+ \in V$. We now choose $v=r_\eps$ in \eqref{eq:PenalizedProblem} to give
\begin{equation}
\label{eq:repsVarEqn}
a_\lambda(u_\eps,r_\eps) - \frac{1}{\eps}((\psi-u_\eps)^+, r_\eps)_H = (f, r_\eps)_H,
\end{equation}
where we recall from \eqref{eq:PenalizationOperator} that $\beta_\eps(w) = -\frac{1}{\eps}(\psi - w)^+, \forall w\in V$. Next, taking $v-u = -r_\eps$ in \eqref{eq:VIProblemHestonHomgeneous}, that is, $v = u-r_\eps \in V$ and, by the expression \eqref{eq:reps} for $r_\eps$,
$$
v = \psi + (\psi - u_\eps)^- \geq \psi,
$$
and so $v\in \KK$, we obtain
$$
a_\lambda(u,-r_\eps) \geq (f, -r_\eps)_H.
$$
Adding the preceding inequality and \eqref{eq:repsVarEqn}, we obtain
$$
a_\lambda(u_\eps-u,r_\eps) - \frac{1}{\eps}((\psi-u_\eps)^+, r_\eps)_H \geq 0.
$$
Substituting the expression \eqref{eq:reps} for $r_\eps$ in the second term in the preceding inequality and noting that $u\geq\psi$ a.e. on $\sO$, we obtain
$$
a_\lambda(u_\eps-u,r_\eps) \geq a_\lambda(u_\eps-u,r_\eps) - \frac{1}{\eps}((\psi-u_\eps)^+, u - \psi)_H \geq 0.
$$
Substituting $u_\eps-u = -r_\eps - (\psi - u_\eps)^+$ gives
$$
a_\lambda(r_\eps + (\psi - u_\eps)^+,r_\eps) \leq 0,
$$
and thus
\begin{align*}
\nu_1\|r_\eps\|_V^2 &\leq a_\lambda(r_\eps,r_\eps) \quad\hbox{(by \eqref{eq:CoerciveHeston})}
\\
&\leq a_\lambda((\psi - u_\eps)^+,-r_\eps) \quad\hbox{(by preceding inequality)}
\\
&\leq C\|(\psi - u_\eps)^+\|_V\|r_\eps\|_V, \quad\hbox{(by \eqref{eq:ContinuousCoerciveHeston})},
\end{align*}
where $C$ depends only on the constant coefficients of $A$, to give
\begin{equation}
\label{eq:reps_Vbound}
\|r_\eps\|_V \leq C\|(\psi - u_\eps)^+\|_V.
\end{equation}
Combining the estimates \eqref{eq:VIPenalizationEstimateHeston_prefinal}, \eqref{eq:DiffUandUepsBasic}, and \eqref{eq:reps_Vbound} gives the desired bound \eqref{eq:VIPenalizationEstimateHeston}.
\end{proof}

We will need an extension of the estimates \eqref{eq:L2PenaltyFunctionBound} and \eqref{eq:VIPenalizationEstimateHeston_prefinal} in the statement of Theorem \ref{thm:VIPenalizationEstimateHeston}.

\begin{hyp}[Conditions on the obstacle function]
\label{hyp:sqrtyApsiL2andy32psiL2}
Require that the obstacle function $\psi \in H^1(\sO,\fw)$ obeys
\begin{equation}
\label{eq:sqrtyApsiL2}
(1+y)^{1/2}A\psi \in L^2(\sO,\fw).
\end{equation}
\end{hyp}

\begin{lem}[A priori estimates for a solution to the penalized equation]
\label{lem:VIPenalizationEstimateHestonPowery}
Require that the domain $\sO$ obeys Hypotheses \ref{hyp:HestonDomainNearGammaZero} and \ref{hyp:GammeOneExtensionProperty} with $k=1$. Require that $f\in L^2(\sO,\fw)$ obeys \eqref{eq:Sqrt1plusyfL2}, that $\psi \in H^2(\sO,\fw)$, and that $\psi$ obeys \eqref{eq:sqrtyApsiL2} and
\begin{equation}
\label{eq:y32psiL2}
(1+y)^{3/2}\psi \in L^2(\sO,\fw).
\end{equation}
If $u_\eps\in V$ is a solution to Problem \ref{prob:HestonPenalizedEquation}, then $(1+y^{1/2})(\psi-u_\eps)^+ \in H^1(\sO\cup\Gamma_0,\fw)$ and there are positive constants $C$ and $\eps_0$, depending only the constant coefficients of $A$ and $\gamma$, such that for all $0<\eps\leq\eps_0$,
\begin{align}
\label{eq:yWeightedL2PenaltyFunctionBound}
|(1+y^{1/2})(\psi-u_\eps)^+|_H &\leq \eps C\left(|(1+y)^{1/2}f|_H + |(1+y)^{1/2}A\psi|_H + |(1+y)^{3/2}\psi|_H\right),
\\
\label{eq:yWeightedH1PenaltyFunctionBound}
\|(1+y^{1/2})(\psi-u_\eps)^+\|_V &\leq \sqrt{\eps}C\left(|(1+y)^{1/2}f|_H + |(1+y)^{1/2}A\psi|_H + |(1+y)^{3/2}\psi|_H\right).
\end{align}
\end{lem}

\begin{rmk}[A priori $H^1$ estimate for a solution to the penalized equation]
The estimate \eqref{eq:yWeightedH1PenaltyFunctionBound} is not used elsewhere in this article, but is included for completeness as it is an easy consequence of the proof of  \eqref{eq:yWeightedL2PenaltyFunctionBound}.
\end{rmk}

\begin{proof}[Proof of Lemma \ref{lem:VIPenalizationEstimateHestonPowery}]
We adapt the derivations of \eqref{eq:L2PenaltyFunctionBound} and \eqref{eq:VIPenalizationEstimateHeston_prefinal} in the statement of Theorem \ref{thm:VIPenalizationEstimateHeston}. From the proof of Theorem \ref{thm:VIPenalizationEstimateHeston} we have $(\psi - u_\eps)^+ \in V$. Let $\varphi \in C^\infty_0(\RR^2)$. Clearly $\varphi(\psi - u_\eps)^+ = 0$ and $\varphi^2(\psi - u_\eps)^+ = 0$ on $\Gamma_1$ in the trace sense, as this is true for $(\psi - u_\eps)^+$. We also see that $\varphi(\psi - u_\eps)^+$ and $\varphi^2(\psi - u_\eps)^+$ are in $V$. Substituting $v = -\varphi^2(\psi - u_\eps)^+$ in \eqref{eq:PenalizedProblem} gives
$$
-a_\lambda(u_\eps,\varphi^2(\psi-u_\eps)^+) + \frac{1}{\eps}|\varphi(\psi-u_\eps)^+|_H^2 = (-\varphi f, \varphi(\psi-u_\eps)^+)_H.
$$
Since $\psi \in H^2(\sO,\fw)$, we also have
\begin{align*}
{}&a_\lambda(\psi,\varphi^2(\psi - u_\eps)^+)
\\
&\quad = a(\psi,\varphi^2(\psi - u_\eps)^+) + \lambda((1+y)\psi,\varphi^2(\psi - u_\eps)^+)_H
\\
&\quad = (A\psi, \varphi^2(\psi - u_\eps)^+)_H + \lambda((1+y)\psi,\varphi^2(\psi - u_\eps)^+)_H
\quad\hbox{(by Lemma \ref{lem:HestonIntegrationByParts})}
\\
&\quad = (\varphi A\psi + \lambda(1+y)\varphi\psi, \varphi(\psi - u_\eps)^+)_H
\\
&\quad = (\varphi A_\lambda\psi, \varphi(\psi - u_\eps)^+)_H \quad\hbox{(by \eqref{eq:CoerciveHestonOperator})}.
\end{align*}
Because $a_\lambda(v,v^+) = a_\lambda(v^+,v^+), \forall v \in V$, adding the preceding two identities yields
\begin{equation}
\label{eq:PenalizedOperatorWith1PlusYSqFactor}
\begin{aligned}
{}&a_\lambda((\psi - u_\eps)^+, \varphi^2(\psi - u_\eps)^+) + \frac{1}{\eps}|\varphi(\psi-u_\eps)^+|_H^2
\\
&\quad = (-\varphi f+\varphi A_\lambda\psi, \varphi(\psi-u_\eps)^+)_H.
\end{aligned}
\end{equation}
From \eqref{eq:BilinearFormSquaredFunction}, we have
$$
a_\lambda(\varphi (\psi - u_\eps)^+, \varphi (\psi - u_\eps)^+)
= a((\psi - u_\eps)^+,\varphi^2 (\psi - u_\eps)^+) + ([A,\varphi](\psi - u_\eps)^+, \varphi (\psi - u_\eps)^+)_H
$$
and thus, adding the two preceding identities yields
\begin{align*}
{}&a_\lambda(\varphi (\psi - u_\eps)^+, \varphi (\psi - u_\eps)^+) + \frac{1}{\eps}|\varphi(\psi-u_\eps)^+|_H^2
\\
&\quad = (-\varphi f+\varphi A_\lambda\psi, \varphi(\psi-u_\eps)^+)_H + ([A,\varphi](\psi - u_\eps)^+, \varphi (\psi - u_\eps)^+)_H.
\end{align*}
Applying the estimates \eqref{eq:CoerciveHeston} and \eqref{eq:CommutatorInnerProductEstimate} yields
\begin{equation}
\label{eq:yWeightedPenaltyFunctionBound_prefinal2}
\begin{aligned}
{}&\nu_1\|\varphi(\psi - u_\eps)^+\|_V^2  + \frac{1}{\eps}|\varphi(\psi-u_\eps)^+|_H^2
\\
&\quad\leq \left(|\varphi f|_H + |\varphi A_\lambda\psi|_H\right)|\varphi(\psi-u_\eps)^+|_H + C|y^{1/2}(|D\varphi| + |D\varphi|^{1/2})(\psi-u_\eps)^+|_H^2,
\end{aligned}
\end{equation}
where $C$ is a positive constant depending only on $\gamma$ and the constant coefficients of $A$. We now choose $\varphi = \zeta_R(1+y^{1/2})$, where $\zeta_R$ is given by Definition \ref{defn:RadialCutoffFunction}, so $D\varphi = (1+y^{1/2})D\zeta_R + (0, y^{-1/2}\zeta_R)$, and use \eqref{eq:yWeightedPenaltyFunctionBound_prefinal2} to estimate
\begin{align*}
{}&\frac{1}{\eps}|\zeta_R(1+y^{1/2})(\psi-u_\eps)^+|_H^2
\\
&\leq \left(|\zeta_R(1+y^{1/2})f|_H + |\zeta_R(1+y^{1/2})A_\lambda\psi|_H\right)|\zeta_R(1+y^{1/2})(\psi-u_\eps)^+|_H
\\
&\quad + C\left|y^{1/2}\left(\left((y^{-1/2} + (1+y^{1/2})|D\zeta_R|\right) + \left(y^{-1/2} + (1+y^{1/2})|D\zeta_R|\right)^{1/2}\right)(\psi-u_\eps)^+\right|_H^2
\\
&\leq \left(|(1+y^{1/2})f|_H + |(1+y^{1/2})A_\lambda\psi|_H\right)|(1+y^{1/2})(\psi-u_\eps)^+|_H
\\
&\quad + C|(1+y^{1/4})(\psi-u_\eps)^+|_H^2 \quad\hbox{(by \eqref{eq:RadialCutoffFunctionFirstDerivative}).}
\end{align*}
By taking limits as $R\to\infty$ in the preceding inequality and applying the dominated convergence theorem, noting that $1 + y^{1/4} \leq 2 + y^{1/2}, \forall y>0$, we obtain
$$
\frac{1}{\eps}|(1+y^{1/2})(\psi-u_\eps)^+|_H \leq |(1+y^{1/2})f|_H + |(1+y^{1/2})A_\lambda\psi|_H + C|(1+y^{1/2})(\psi-u_\eps)^+|_H.
$$
If $1/\eps \geq 2C$, that is, $0 <\eps \leq 1/2C$, the preceding inequality gives \eqref{eq:yWeightedL2PenaltyFunctionBound}, noting that $A_\lambda = A + \lambda(1+y)$ by \eqref{eq:CoerciveHestonOperator}.

We now complete the estimate of $\|\varphi(\psi - u_\eps)^+\|_V$. By Definition \ref{defn:H1WeightedSobolevSpaces} and recalling that $\varphi = \zeta_R(1 + y)^{1/2}$, we have
\begin{align*}
{}&\|\zeta_R(1 + y)^{1/2}(\psi - u_\eps)^+\|_V^2
\\
&= \left|y^{1/2}\left(\zeta_RD\left((1 + y^{1/2})(\psi - u_\eps)^+\right) + (D\zeta_R)(1+y^{1/2})(\psi - u_\eps)^+\right)\right|_H^2
\\
&\quad + \left|(1 + y)^{1/2}\zeta_R(1 + y^{1/2})(\psi - u_\eps)^+\right|_H^2,
\end{align*}
and thus
\begin{align*}
{}&|y^{1/2}\zeta_RD((1 + y^{1/2})(\psi - u_\eps)^+)|_H + |(1 + y)^{1/2}\zeta_R(1 + y^{1/2})(\psi - u_\eps)^+|_H
\\
&\leq \|(1 + y^{1/2})\zeta_R(\psi - u_\eps)^+\|_V + |y^{1/2}((1 + y^{1/2})(D\zeta_R)(\psi - u_\eps)^+)|_H
\\
&\leq \|(1 + y^{1/2})\zeta_R(\psi - u_\eps)^+\|_V + |(1 + y^{1/2})(\psi - u_\eps)^+)|_H
\quad\hbox{(by \eqref{eq:RadialCutoffFunctionFirstDerivative}).}
\end{align*}
Combining the preceding inequality with \eqref{eq:yWeightedPenaltyFunctionBound_prefinal2} and taking the limit as $R\to\infty$ and applying the dominated convergence theorem,as in the derivation of \eqref{eq:yWeightedL2PenaltyFunctionBound}, gives
\begin{align*}
\nu_1\|(1 + y^{1/2})(\psi - u_\eps)^+\|_V^2
&\leq \left(|(1 + y^{1/2})f|_H + |(1 + y^{1/2})A_\lambda\psi|_H\right)|(1 + y^{1/2})(\psi-u_\eps)^+|_H
\\
&\quad + C|(1+y^{1/2})(\psi-u_\eps)^+|_H^2.
\end{align*}
We now use \eqref{eq:yWeightedL2PenaltyFunctionBound} to bound the factor $|(1+y^{1/2})(\psi-u_\eps)^+|_H$ on the right-hand side of the preceding identity and take square roots to obtain \eqref{eq:yWeightedH1PenaltyFunctionBound}.
\end{proof}

Turning to the question of regularity, from \cite[\S 3.1.9 \& 3.1.15]{Bensoussan_Lions}, we have the following analogue of \cite[Theorem 3.1.8 \& Corollary 3.1.1]{Bensoussan_Lions} and $H^2$ analogue of the $W^{2,p}$ regularity results \cite[Lemma 1.3.1 \& Theorem 1.3.2]{Friedman_1982} and \cite[Theorem 3.1.21 \& Corollary 3.1.6]{Bensoussan_Lions}.

\begin{hyp}[Conditions on the obstacle function]
\label{hyp:1plusyPower3over2H1psi}
Require that the obstacle function $\psi \in H^1(\sO,\fw)$ obeys
\begin{equation}
\label{eq:1plusyPower3over2H1psi}
(1+y)^{3/2}\psi \in H^1(\sO,\fw).
\end{equation}
\end{hyp}

\begin{hyp}[Conditions on the source function]
\label{hyp:1plusyfL2}
Require that the source function $f \in L^2(\sO,\fw)$ obeys
\begin{equation}
\label{eq:1plusyfL2}
(1+y)f \in L^2(\sO,\fw).
\end{equation}
\end{hyp}

\begin{thm}[Global $H^2$ regularity and a posteriori estimate for the solution to the coercive variational inequality]
\label{thm:VIRegularityEllipticHeston}
Require that the domain $\sO$ obeys Hypothesis \ref{hyp:DomainCombinedCondition} and that there are $M, m \in H^2(\sO,\fw)$ obeying \eqref{eq:SourceFunctionTraceBounds}, \eqref{eq:mMIneqOnDomain}, \eqref{eq:SourceFunctionBounds}, \eqref{eq:MgeqPsi}, and \eqref{eq:OneplusyMmInL2}. Require that $\psi \in H^2(\sO,\fw)$ and that $\psi$ obeys \eqref{eq:sqrtyApsiL2} and \eqref{eq:1plusyPower3over2H1psi}. Require that $f\in L^2(\sO,\fw)$ obeys \eqref{eq:fBoundsCoercive} and \eqref{eq:1plusyfL2}. If $u$ is the unique solution to Problem \ref{prob:HomogeneousHestonVIProblemCoercive}, then $u \in H^2(\sO,\fw)$, $yu \in L^2(\sO,\fw)$, and
\begin{equation}
\label{eq:H2BoundSolutionCoerciveHestonSatVarIneqCorrected}
\begin{aligned}
\|u\|_{H^2(\sO,\fw)} &\leq C\left(\|(1+y)f\|_{L^2(\sO,\fw)} + \|(1+y)^{1/2}A\psi\|_{L^2(\sO,\fw)}\right.
\\
&\quad + \left. \|(1+y^{3/2})\psi\|_{H^1(\sO,\fw)} + \|(1+y)u\|_{L^2(\sO,\fw)}\right).
\end{aligned}
\end{equation}
where $C$ depends only on the constant coefficients of $A$ and the constants $\delta_0,\delta_1,M_1,R_1$ of Hypothesis \ref{hyp:HestonDomainNearGammaOne}.
\end{thm}

\begin{proof}
Let $\{u_\eps\}_{\eps\in(0,1]}$ be the consequence constructed in the proof of existence in Theorem \ref{thm:VIExistenceUniquenessEllipticCoerciveHeston}. Lemma \ref{lem:ComparisionPrinciplePenalizedHeston} implies that the sequence $\{u_\eps\}_{\eps\in(0,1]}$ obeys \eqref{eq:uepsBounds}, so
\begin{equation}
\label{eq:1plusyuBoundbyMm}
|(1+y)u_\eps| \leq (1+y)(|M|+|m|) \quad\hbox{a.e. on }\sO,
\end{equation}
and thus \eqref{eq:OneplusyMmInL2} ensures that $yu_\eps \in L^2(\sO,\fw), \forall \eps>0$. Applying \eqref{eq:PenalizedEquationAPosterioriEstimatePoweryCorrected} with $s=1$ gives
$$
\|yu_\eps\|_V \leq C\left(|yf|_H + |(1+y)u_\eps|_H + \|(1 + y^{3/2})\psi\|_V\right),
$$
and thus $y^{3/2}u_\eps \in L^2(\sO,\fw), \forall \eps>0$, and
\begin{equation}
\label{eq:y3over2uepsL2}
|y^{3/2}u_\eps|_H \leq C\left(|yf|_H + |(1+y)u_\eps|_H + \|(1 + y^{3/2})\psi\|_V \right).
\end{equation}
Lemma \ref{lem:VIPenalizationEstimateHestonPowery} implies that $(1+y)^{1/2}\beta_\eps(u_\eps)\in L^2(\sO,\fw), \forall \eps \in (0,\eps_0]$. Consequently, setting
$$
f_\eps := f - \beta_\eps(u_\eps) - \lambda(1+y)u_\eps,
$$
we obtain $(1+y)^{1/2}f_\eps\in L^2(\sO,\fw), \forall \eps \in (0,\eps_0]$. As $u_\eps\in V$ is a solution to \eqref{eq:PenalizedProblem}, we may view $u_\eps$ as the solution to the equivalent non-coercive variational equation,
$$
a(u_\eps,v) = (f_\eps,v)_H, \quad \forall v\in V.
$$
Theorem \ref{thm:GlobalRegularityEllipticHestonSpecial} now implies that $u_\eps \in H^2(\sO,\fw), \forall \eps \in (0,\eps_0]$ and by \eqref{eq:SecondDerivativeGlobalHestonEstimateSpecial} obeys
$$
\|u_\eps\|_{H^2(\sO,\fw)} \leq C\left(\|(1+y)^{1/2}f_\eps\|_{L^2(\sO,\fw)} + \|(1+y)u_\eps\|_{L^2(\sO,\fw)}\right), \quad \forall \eps \in (0,\eps_0].
$$
Substituting the expression for $f_\eps$ into the preceding inequality gives, for all $\eps \in (0,\eps_0]$,
\begin{align*}
\|u_\eps\|_{H^2(\sO,\fw)} &\leq C\left(\|(1+y)^{1/2}f\|_{L^2(\sO,\fw)} + \|(1+y)^{1/2} \beta_\eps(u_\eps)\|_{L^2(\sO,\fw)} \right.
\\
&\quad + \left. \|(1+y)^{3/2}u_\eps\|_{L^2(\sO,\fw)}\right)
\\
&\leq C\left(\|(1+y)^{1/2}f\|_{L^2(\sO,\fw)} + |(1+y)^{1/2}A\psi|_H + |(1+y)^{3/2}\psi|_H \right.
\quad\hbox{(by \eqref{eq:yWeightedL2PenaltyFunctionBound})}
\\
&\quad + \left. \|(1+y)^{3/2}u_\eps\|_{L^2(\sO,\fw)}\right).
\end{align*}
Substituting the estimate \eqref{eq:y3over2uepsL2} for $|y^{3/2}u_\eps|_H$ into the preceding estimate for $\|u_\eps\|_{H^2(\sO,\fw)}$ and combining terms yields
\begin{equation}
\label{eq:H2BoundSolutionPenalizedEquation}
\begin{aligned}
\|u_\eps\|_{H^2(\sO,\fw)} &\leq C\left(\|(1+y)f\|_{L^2(\sO,\fw)} + |(1+y)^{1/2}A\psi|_H + \|(1 + y^{3/2})\psi\|_V \right.
\\
&\quad + \left. \|(1+y)u_\eps\|_{L^2(\sO,\fw)}\right), \forall \eps \in (0,\eps_0],
\end{aligned}
\end{equation}
where $C$ depends only on $\gamma$ and the constant coefficients of $A$. Therefore $\{u_\eps\}_{\eps\in(0,\eps_0]}$ has a subsequence which converges weakly in $H^2(\sO,\fw)$ to a limit in $H^2(\sO,\fw)$. This limit must be $u$, since Theorem \ref{thm:VIPenalizationEstimateHeston} implies that $\{u_\eps\}_{\eps\in(0,\eps_0]}$ converges strongly in $V$ to $u\in V$ as $\eps\to 0$ and again, after passing to a subsequence, $u_\eps\to u$ pointwise a.e. on $\sO$ by Corollary \ref{cor:Billingsley}. Because of \eqref{eq:1plusyuBoundbyMm}, we can apply the dominated convergence theorem to conclude that
$$
\lim_{\eps\downarrow 0}\|(1+y)u_\eps\|_{L^2(\sO,\fw)} = \|(1+y)u\|_{L^2(\sO,\fw)}.
$$
Since
$$
\|u\|_{H^2(\sO,\fw)} \leq \liminf_{\eps\to 0}\|u_\eps\|_{H^2(\sO,\fw)},
$$
by \cite[Appendix D]{Evans}, the estimate \eqref{eq:H2BoundSolutionCoerciveHestonSatVarIneqCorrected} follows from \eqref{eq:H2BoundSolutionPenalizedEquation} and the preceding application of the dominated convergence theorem.
\end{proof}

\begin{rmk}[A posteriori estimate in Theorem \ref{thm:VIRegularityEllipticHeston}]
\label{rmk:VIRegularityEllipticHeston}
The hypothesis in Theorem \ref{thm:VIRegularityEllipticHeston} that $u$ is unique is used to conclude that the desired bound applies to the given solution.
\end{rmk}

\subsection[Global $H^2$ regularity of solutions to the non-coercive variational inequality]{Global $\mathbf{H^2}$ regularity of solutions to the non-coercive variational inequality}
\label{subsec:H2RegularityInequality}
We extend the regularity results of \S \ref{subsec:H2RegularityCoerciveInequality} to the non-coercive operator, $A$, and its bilinear form, $a(\cdot,\cdot)$.

\begin{hyp}[Conditions on envelope functions]
\label{hyp:GlobalH2UpperLowerBoundsSolutionsNoncoerciveInequality}
There are $M, m \in H^2(\sO,\fw)$ obeying
\begin{equation}
\label{eq:OneplusysquaredMmInL2}
(1+y)^2M, (1+y)^2m \in L^2(\sO,\fw).
\end{equation}
\end{hyp}

\begin{thm}[Global $H^2$ regularity and a posteriori estimate for the solution to the non-coercive variational inequality]
\label{thm:VIRegularityEllipticHestonNoncoercive}
Assume the hypotheses of Theorem \ref{thm:VIExistenceUniquenessEllipticHeston_Improved} and, in addition, require that the domain $\sO$ obeys Hypothesis \ref{hyp:DomainCombinedCondition}; that $M, m \in H^2(\sO,\fw)$ obey \eqref{eq:OneplusysquaredMmInL2}; that $\psi \in H^2(\sO,\fw)$ and that $\psi$ obeys \eqref{eq:sqrtyApsiL2} and \eqref{eq:1plusyPower3over2H1psi}; and that $f\in L^2(\sO,\fw)$ obeys \eqref{eq:1plusyfL2}. If $u$ is the unique solution to Problem \ref{prob:HomogeneousHestonVIProblem}, then $u \in H^2(\sO,\fw)$, $y^2u \in L^2(\sO,\fw)$,  and
\begin{equation}
\label{eq:H2BoundSolutionHestonStatVarIneqCorrected}
\begin{aligned}
\|u\|_{H^2(\sO,\fw)} &\leq C\left(\|(1+y)f\|_{L^2(\sO,\fw)} + \|(1+y)^{1/2}A\psi\|_{L^2(\sO,\fw)}\right.
\\
&\quad + \left. \|(1+y^{3/2})\psi\|_{H^1(\sO,\fw)} + \|(1+y)^2u\|_{L^2(\sO,\fw)}\right),
\end{aligned}
\end{equation}
where $C$ depends only on the constant coefficients of $A$ and the constants $\delta_0,\delta_1,M_1,R_1$ of Hypothesis \ref{hyp:HestonDomainNearGammaOne}.
\end{thm}

\begin{proof}
Because $u$ solves Problem \ref{prob:HomogeneousHestonVIProblem}, we have
$$
a(u,v-u) \geq (f,v-u)_{L^2(\sO,\fw)}, \quad\forall v\in \KK,
$$
and thus
$$
a_\lambda(u,v-u) \geq (f_\lambda,v-u)_{L^2(\sO,\fw)}, \quad\forall v\in \KK,
$$
where
$$
f_\lambda := f + \lambda(1+y)u.
$$
By \eqref{eq:1plusyfL2} we have $(1+y)f\in L^2(\sO,\fw)$, while \eqref{eq:uBoundedObstacle} and \eqref{eq:OneplusysquaredMmInL2} imply that $(1+y)^2u \in L^2(\sO,\fw)$, and therefore $(1+y)f_\lambda\in L^2(\sO,\fw)$. Theorem \ref{thm:VIRegularityEllipticHeston}, with $f$ replaced by $f_\lambda$, then ensures that $u\in H^2(\sO,\fw)$ and the estimate \eqref{eq:H2BoundSolutionHestonStatVarIneqCorrected} follows from \eqref{eq:H2BoundSolutionCoerciveHestonSatVarIneqCorrected}.
\end{proof}

\begin{rmk}
\label{rmk:VIRegularityEllipticHestonNoncoercive}
The hypothesis in Theorem \ref{thm:VIRegularityEllipticHestonNoncoercive} that $u$ is unique is used to conclude that the desired bound applies to the given solution.
\end{rmk}

By combining Theorems \ref{thm:VIExistenceUniquenessEllipticHeston_Improved} and \ref{thm:VIRegularityEllipticHestonNoncoercive} we obtain the following existence and uniqueness result.

\begin{thm}[Existence and uniqueness of a strong solution to the non-coercive variational inequality]
\label{thm:VIExistenceUniquenessH2RegularEllipticHeston}
Assume the hypotheses of Theorem \ref{thm:VIExistenceUniquenessEllipticHeston_Improved} and \ref{thm:VIRegularityEllipticHestonNoncoercive}. Then there exists a unique solution $u\in H^2(\sO,\fw)$ to Problem \ref{prob:HestonObstacleInhomogeneousStrong} (with $g = 0$), the solution $u$ obeys \eqref{eq:uBoundedObstacle}, has the boundary property \eqref{eq:WeightedNeumannHomogeneousBCProb}, and obeys the estimate \eqref{eq:H2BoundSolutionHestonStatVarIneqCorrected}.
\end{thm}

\begin{proof}
Theorem \ref{thm:VIExistenceUniquenessEllipticHeston_Improved} implies that there exists a unique $u \in H^1(\sO\cup\Gamma_0,\fw)$ to Problem \ref{prob:HomogeneousHestonVIProblem}. Consequently, Theorem \ref{thm:VIRegularityEllipticHestonNoncoercive} implies that $u\in H^2(\sO,\fw)$ and obeys the estimate \eqref{eq:H2BoundSolutionHestonStatVarIneqCorrected}, while Lemma \ref{lem:VIStrongFormHeston} ensures that $u$ solves Problem \ref{prob:HestonObstacleInhomogeneousStrong} (with $g = 0$). Lemma \ref{lem:HestonWeightedNeumannBoundaryProperty} implies that $u$ obeys \eqref{eq:WeightedNeumannHomogeneousBCProb}.
\end{proof}

Again, it is straightforward to assemble the results we need to conclude the

\begin{proof}[Proof of Theorem \ref{thm:MainExistenceUniquenessObstacleProblem}]
When $g=0$, the hypotheses of Theorem \ref{thm:MainExistenceUniquenessObstacleProblem} collect and summarize those of Theorem \ref{thm:VIExistenceUniquenessH2RegularEllipticHeston} and so the result is a restatement of Theorem \ref{thm:VIExistenceUniquenessH2RegularEllipticHeston}. When $g\neq 0$, the result follows just as in the proof of Theorem \ref{thm:MainExistenceUniquenessBoundaryValueProblem}, with the additional choice of $\tilde\psi := \psi-g$; see also Remark \ref{rmk:ReductionInhomogeneousVariationalInequality}.
\end{proof}

\subsection[Local $H^2$ regularity of solutions to the variational inequality]{Local $\mathbf{H^2}$ regularity of solutions to the variational inequality}
\label{subsec:LocalH2RegularityInequality}
We have the following analogue of \cite[Theorem IV.8.6]{Kinderlehrer_Stampacchia_1980} (where $\psi$ is assumed Lipschitz) and which will be useful in situations where we do not know that $\psi$ is in $H^2(\sO,\fw)$.

\begin{thm}[Local $H^2$ regularity and estimate for solutions to the non-coercive variational inequality]
\label{thm:H2BoundSolutionHestonVarIneqLipschitz}
Assume the hypotheses of Theorem \ref{thm:VIExistenceUniquenessEllipticHeston_Improved} and, in addition, that
\begin{gather}
\label{eq:1plusyPower3over2f}
(1+y)^{3/2}f \in L^2(\sO,\fw),
\\
\label{eq:1plusyPower5over2Mm}
(1+y)^{5/2}m, (1+y)^{5/2}M \in L^2(\sO,\fw).
\end{gather}
Let $\sU \subseteqq \sO$ be an open subset. Require that $\sO$ obeys Hypothesis \ref{hyp:DomainCombinedCondition} and that $\sU$ obeys Hypothesis \ref{hyp:DomainCombinedCondition} with the role of $\Gamma_1$ replaced by $\HH\cap\partial\sU$. Require that $\psi \in H^1(\sO,\fw)$ obeys \eqref{eq:1plusyPower3over2H1psi} and
\begin{equation}
\label{eq:sqrtyApsiL2Interior}
(1+y)^{1/2}A\psi \in L^2(\sU,\fw).
\end{equation}
If $u \in H^1_0(\sO\cup\Gamma_0)$ is the unique solution to Problem \ref{prob:HomogeneousHestonVIProblem}, then $u \in H^2(\sU',\fw)$ for every open subset $\sU' \subset \sU$ with $\bar\sU'\less\partial\sO \subset \sU$,
\begin{equation}
\label{eq:DistanceUPrimeUPositive}
\hbox{dist}(\sO\cap\partial\sU',\sO\cap\partial\sU) > 0.
\end{equation}
Moreover, $y^{5/2}u \in L^2(\sO,\fw)$, and
\begin{equation}
\label{eq:H2BoundSolutionHestonVarIneqLipschitz}
\begin{aligned}
\|u\|_{H^2(\sU',\fw)} &\leq C\left(\|(1+y)^{3/2}f\|_{L^2(\sO,\fw)} + \|(1+y)^{5/2}u\|_{L^2(\sO,\fw)} \right.
\\
&\quad + \left. \|(1+y)^{1/2}A\psi\|_{L^2(\sU,\fw)} + \|(1+y)^{3/2}\psi\|_{H^1(\sO,\fw)}\right),
\end{aligned}
\end{equation}
where $C$ depends only on the constant coefficients of $A$ and the constants of Hypothesis \ref{hyp:HestonDomainNearGammaOne} prescribing the geometry of $\HH\cap\partial\sU$ and $\Gamma_1$, and $\hbox{dist}(\sO\cap\partial\sU',\sO\cap\partial\sU)$.
\end{thm}

\begin{rmk}
\label{rmk:H2BoundSolutionHestonVarIneqLipschitz}
We note that
\begin{enumerate}
\item Neither $\sU'$ nor $\sU$ are required to be bounded in the hypotheses of Theorem \ref{thm:H2BoundSolutionHestonVarIneqLipschitz}.
\item The difference between the powers of $1+y$ appearing on the right-hand sides of \eqref{eq:H2BoundSolutionHestonStatVarIneqCorrected} and \eqref{eq:H2BoundSolutionHestonVarIneqLipschitz} is an artifact of the method of proof of Theorem \ref{thm:H2BoundSolutionHestonVarIneqLipschitz}.
\end{enumerate}
\end{rmk}

\begin{proof}
Because $M, m \in H^2(\sO,\fw)$ obey \eqref{eq:1plusyPower5over2Mm} and $u$ obeys \eqref{eq:uBounds}, then $(1+y)^{5/2}u \in L^2(\sO,\fw)$. That observation and the condition \eqref{eq:1plusyPower3over2f} on $f$ and conditions \eqref{eq:1plusyPower3over2H1psi} and \eqref{eq:sqrtyApsiL2Interior} on $\psi$ ensure that the right-hand side of \eqref{eq:H2BoundSolutionHestonVarIneqLipschitz} is finite.

Let $\zeta \in C^\infty(\bar\HH)$ be a cutoff function such that $0\leq \zeta\leq 1$ on $\HH$, $\zeta = 1$ on $\sU'$, $\zeta>0$ on $\sU$, and $\zeta  = 0$ on $\sO\less \sU$. We shall use $\zeta$ to localize the variational inequality in Problem \ref{prob:HomogeneousHestonVIProblem}. By \eqref{eq:DistanceUPrimeUPositive} and construction of $\zeta$, there is a positive constant, $C_0$, depending only on $\hbox{dist}(\sO\cap\partial\sU',\sO\cap\partial\sU)$ such that
\begin{equation}
\label{eq:CutoffFunctionSecondDerivativeBounds}
\|\zeta\|_{C^2(\HH)} \leq C_0.
\end{equation}
We obtain $\zeta\psi \in H^1(\sU,\fw)$ by \eqref{eq:CutoffFunctionSecondDerivativeBounds} and the fact that $\psi \in H^1(\sO,\fw)$. Because $\zeta = 0$ on $\partial\sU\less\partial\sO$ and $\psi\leq 0$ on $\Gamma_1 = \partial\sO\less\bar\Gamma_0$ (trace sense), then $\zeta\psi\leq 0$ on $\partial\sU\less\bar\Gamma_0$ (trace sense). Similarly, as $\zeta=0$ on $\partial\sU\less\partial\sO$ and $u=0$ on $\partial\sO\less\bar\Gamma_0$ (trace sense), then $\zeta u = 0$ on $\partial\sU\less\bar\Gamma_0$ (trace sense) and so
\begin{equation}
\label{eq:CutoffuH1Zero}
\zeta u \in H^1_0(\sU\cup\Gamma_0,\fw)
\end{equation}
by Lemma \ref{lem:EvansGamma1TraceZero}.

\begin{claim}
\label{claim:ObstacleFunctionLocalization}
If $u$ is a solution to Problem \ref{prob:HomogeneousHestonVIProblem} on $\sO$ with obstacle function, $\psi \in H^1(\sO,\fw)$, and source function, $f\in L^2(\sO,\fw)$, then $\zeta u \in H^1_0(\sU\cup\Gamma_0,\fw)$ is a solution to Problem \ref{prob:HomogeneousHestonVIProblem} on $\sU$ with obstacle function, $\zeta\psi\in H^1(\sU,\fw)$, and source function,
\begin{equation}
\label{eq:fCutoff}
f_\zeta := \zeta f+[A,\zeta]u \in L^2(\sU,\fw).
\end{equation}
\end{claim}

\begin{proof}[Proof of Claim \ref{claim:ObstacleFunctionLocalization}]
Recall that $C^\infty_0(\sO\cup\Gamma_0)$ is dense in $H^1_0(\sO\cup\Gamma_0,\fw)$ by Definition \ref{defn:H1WeightedSobolevSpaces} and so there is a sequence $\{u_n\}_{n\geq 1} \subset C^\infty_0(\sO\cup\Gamma_0)$ such that $u_n \to u$ strongly in $H^1_0(\sO\cup\Gamma_0,\fw)$. Then, for all $v\in C^\infty_0(\sO\cup\Gamma_0)$,
\begin{align*}
a(\zeta u_n,\zeta v - \zeta u_n) &= (A(\zeta u_n), \zeta v-\zeta u_n)_H \quad\hbox{(by Lemma \ref{lem:HestonIntegrationByParts})}
\\
&= (\zeta Au_n + [A,\zeta]u_n, \zeta v-\zeta u_n)_H
\\
&= (Au_n, \zeta^2 v-\zeta^2 u_n)_H + ([A,\zeta]u_n, \zeta v-\zeta u_n)_H
\\
&= a(u_n,\zeta^2(v-u_n))_H  + ([A,\zeta]u_n, \zeta v-\zeta u_n)_H  \quad\hbox{(by Lemma \ref{lem:HestonIntegrationByParts})}.
\end{align*}
Taking limits as $n\to\infty$,
$$
a(\zeta u,\zeta v - \zeta u) = a(u,\zeta^2(v-u))_H  + ([A,\zeta]u, \zeta v-\zeta u)_H,
$$
and because $C^\infty_0(\sO\cup\Gamma_0)$ is dense in $H^1_0(\sO\cup\Gamma_0,\fw)$, the identity continues to hold for all $v \in H^1_0(\sO\cup\Gamma_0,\fw)$. Now suppose $v\geq \psi$ and recall that $u\geq \psi$. But $\zeta^2(v-u) = \zeta^2v + (1-\zeta^2)u - u$ and $\zeta^2v + (1-\zeta^2)u \geq \psi$ (see Remark \ref{rmk:ConvexityObstacleTestFunctionSpace}) since $v, u \geq \psi$. Therefore,
\begin{align*}
a(u,\zeta^2(v-u))_H &= a(u, \zeta^2v + (1-\zeta^2)u - u)
\\
&\geq (f,\zeta^2v + (1-\zeta^2)u - u)_H
\\
&= (f,\zeta^2(v-u))_H, \quad\forall v\geq \psi,
\end{align*}
and so we have
\begin{align*}
a(\zeta u,\zeta v - \zeta u) &\geq (\zeta f,\zeta(v-u))_H  + ([A,\zeta]u, \zeta v-\zeta u)_H
\\
&= (f_\zeta,\zeta v-\zeta u)_H, \quad\forall v\geq \psi, v\in H^1_0(\sO\cup\Gamma_0,\fw).
\end{align*}
The image of the bounded linear map $H^1_0(\sO\cup\Gamma_0,\fw) \to H^1_0(\sU\cup\Gamma_0,\fw)$, $v\mapsto \zeta v$, is dense in $H^1_0(\sU\cup\Gamma_0,\fw)$ and thus
\begin{equation}
\label{eq:CutoffVI}
a(\zeta u, v - \zeta u) \geq (f_\zeta, v-\zeta u)_H, \quad\forall v\geq \zeta\psi, v\in H^1_0(\sU\cup\Gamma_0,\fw),
\end{equation}
as required. Since $[A,\zeta]u := A(\zeta u) - \zeta Au$, we have
\begin{equation}
\label{eq:CommutatorL2Estimate}
\|[A,\zeta]u\|_{L^2(\sO,\fw)} \leq C\left(\|yDu\|_{L^2(\sU,\fw)} + \|(1+y)u\|_{L^2(\sU,\fw)}\right).
\end{equation}
But \eqref{eq:uBoundedObstacle} implies that $(1+y)u \in L^2(\sO,\fw)$, since $(1+y)m, (1+y)M \in L^2(\sO,\fw)$ by \eqref{eq:OneplusyMmInL2}, while
\begin{align*}
\|yDu\|_{L^2(\sU,\fw)} &\leq \|y^{1/2}D(y^{1/2}u)\|_{L^2(\sO,\fw)} + \frac{1}{2}\|u\|_{L^2(\sO,\fw)}
\\
&\leq 2\|y^{1/2}u\|_{H^1(\sO,\fw)},
\end{align*}
and thus, applying \eqref{eq:EllipticHestonAPosterioriEstimatePoweryCorrected} with $s=1/2$,
\begin{equation}
\label{eq:yDuInteriorNonCoerciveL2Bound}
\|yDu\|_{L^2(\sU,\fw)} \leq C\left(\|y^{1/2}f\|_{L^2(\sO,\fw)} + \|(1+y^{3/2})u\|_{L^2(\sO,\fw)} + \|(1+y^{1/2})\psi^+\|_{H^1(\sO,\fw)} \right).
\end{equation}
Indeed, \eqref{eq:uBoundedObstacle} implies that $(1+y^{3/2})u \in L^2(\sO,\fw)$, since, a fortiori, $(1+y)^{5/2}M, (1+y)^{5/2}m \in L^2(\sO,\fw)$ by \eqref{eq:1plusyPower5over2Mm}; $y^{1/2}\psi^+ \in H^1(\sO,\fw)$ since, a fortiori, $(1+y^{3/2})\psi \in H^1(\sO,\fw)$ by \eqref{eq:1plusyPower3over2H1psi} and thus $y^{1/2}\psi^+ \in H^1(\sO,\fw)$ by Lemma \ref{lem:SobolevSpaceClosedUnderMaxPart}; and $y^{1/2}f\in L^2(\sO,\fw)$ since, a fortiori, $(1+y)^{3/2}f\in L^2(\sO,\fw)$ by \eqref{eq:1plusyPower3over2f}. Therefore, $f_\zeta \in L^2(\sO,\fw)$ by \eqref{eq:fCutoff}, \eqref{eq:CommutatorL2Estimate}, \eqref{eq:yDuInteriorNonCoerciveL2Bound}, and the fact that $(1+y)u \in L^2(\sO,\fw)$, and hence $f_\zeta \in L^2(\sU,\fw)$, as required. This competes the proof of Claim \ref{claim:ObstacleFunctionLocalization}.
\end{proof}

By Claim \ref{claim:ObstacleFunctionLocalization} and applying the estimate \eqref{eq:H2BoundSolutionHestonStatVarIneqCorrected} to $\zeta u$, we obtain
\begin{align*}
\|u\|_{H^2(\sU',\fw)} &\leq C\|\zeta u\|_{H^2(\sU,\fw)}
\\
&\leq C\left(\|(1+y)f_\zeta\|_{L^2(\sU,\fw)}  + \|(1+y)^{1/2}A(\zeta\psi)\|_{L^2(\sU,\fw)} \right.
\\
&\quad + \left. \|(1+y^{3/2})\zeta\psi\|_{H^1(\sU,\fw)} + \|(1+y)^2\zeta u\|_{L^2(\sU,\fw)}\right)
\\
&\leq C\left(\|(1+y)[A,\zeta]u\|_{L^2(\sU,\fw)}  + \|(1+y)^2\zeta u\|_{L^2(\sU,\fw)} \right.
\\
&\quad + \|(1+y)^{1/2}A(\zeta\psi)\|_{L^2(\sU,\fw)} + \|(1+y^{3/2})\zeta\psi\|_{H^1(\sU,\fw)}
\\
&\quad + \left. \|(1+y)\zeta f\|_{L^2(\sU,\fw)}\right).
\end{align*}
Thus, using the pointwise identity \eqref{eq:ACommutator} to bound the expression $[A,\zeta]u$ and writing $A(\zeta\psi) = \zeta A\psi + [A,\zeta]\psi$, we obtain
\begin{align*}
\|u\|_{H^2(\sU',\fw)} &\leq C\left(\|(1+y)yDu\|_{L^2(\sU,\fw)} + \|(1+y)^2u\|_{L^2(\sU,\fw)} \right.
\\
&\quad + \|(1+y)^{1/2}\zeta A\psi\|_{L^2(\sU,\fw)} + \|(1+y)^{1/2}[A,\zeta]\psi\|_{L^2(\sU,\fw)}
\\
&\quad + \left. \|(1+y^{3/2})\zeta\psi\|_{H^1(\sU,\fw)} + \|(1+y)f\|_{L^2(\sU,\fw)} \right).
\end{align*}
But $\|(1+y)yDu\|_{L^2(\sU,\fw)} \leq \|(1+y)yDu\|_{L^2(\sO,\fw)}$ and we can apply Corollary \ref{cor:VIExistenceUniquenessEllipticHestonPowery} with $s=3/2$, noting that $(1+y)^{5/2}M, (1+y)^{5/2}m \in L^2(\sO,\fw)$ by \eqref{eq:1plusyPower5over2Mm}; $y^{3/2}f\in L^2(\sO,\fw)$ by \eqref{eq:1plusyPower3over2f}; and $(1+y)\psi^+ \in H^1(\sO,\fw$ since, a fortiori, $(1+y)^{3/2}\psi\in H^1(\sO,\fw$ by \eqref{eq:1plusyPower3over2H1psi}. Therefore, \eqref{eq:EllipticHestonAPosterioriEstimatePoweryCorrected} gives
\begin{equation}
\label{eq:y32uNonCoerciveH1Bound}
\|y^{3/2}u\|_{H^1(\sO,\fw)} \leq C\left(\|y^{3/2}f\|_{L^2(\sO,\fw)} + \|(1+y^{5/2})u\|_{L^2(\sO,\fw)} + \|(1+y)\psi^+\|_{H^1(\sO,\fw)} \right),
\end{equation}
and thus, using $\|y^2Du\|_{L^2(\sU,\fw)} \leq \|y^2Du\|_{L^2(\sO,\fw)}$,
\begin{align*}
\|(1+y)yDu\|_{L^2(\sU,\fw)} &\leq \|yDu\|_{L^2(\sU,\fw)} + \|y^2Du\|_{L^2(\sO,\fw)}
\\
&\leq \|yDu\|_{L^2(\sU,\fw)} + \|y^{1/2}D(y^{3/2}u)\|_{L^2(\sO,\fw)} + (3/2)\|yu\|_{L^2(\sO,\fw)}
\\
&\leq \|yDu\|_{L^2(\sU,\fw)} + \|y^{3/2}u\|_{H^1(\sO,\fw)} + (3/2)\|yu\|_{L^2(\sO,\fw)}
\\
&\leq C\left(\|y^{1/2}f\|_{L^2(\sO,\fw)} + \|(1+y)^{3/2}u\|_{L^2(\sO,\fw)} + \|(1+y)^{1/2}\psi^+\|_{H^1(\sO,\fw)} \right.
\\
&\quad + \left. \|y^{3/2}f\|_{L^2(\sO,\fw)} + \|(1+y)^{5/2}u\|_{L^2(\sO,\fw)} + \|(1+y)\psi^+\|_{H^1(\sO,\fw)}\right)
\\
&\qquad\hbox{(by \eqref{eq:yDuInteriorNonCoerciveL2Bound} and \eqref{eq:y32uNonCoerciveH1Bound})}
\\
&\leq C\left(\|(1+y)^{3/2}f\|_{L^2(\sO,\fw)} + \|(1+y)^{5/2}u\|_{L^2(\sO,\fw)} + \|(1+y)\psi^+\|_{H^1(\sO,\fw)}\right).
\end{align*}
Substituting the preceding estimate into the preceding bound for $\|u\|_{H^2(\sU',\fw)}$ and using the commutator identity \eqref{eq:ACommutator} and \eqref{eq:CutoffFunctionSecondDerivativeBounds} to bound $\|(1+y)^{1/2}[A,\zeta]\psi\|_{L^2(\sU,\fw)}$ yields
\begin{equation}
\begin{aligned}
\label{eq:H2BoundSolutionHestonVarIneqLipschitz_prefinal}
\|u\|_{H^2(\sU',\fw)} &\leq C\left(\|(1+y)^{5/2}u\|_{L^2(\sO,\fw)} + \|(1+y)^{3/2}f\|_{L^2(\sO,\fw)} \right.
\\
&\quad + \|(1+y)^{1/2}A\psi\|_{L^2(\sU,\fw)} + \|(1+y)^{1/2}yD\psi\|_{L^2(\sO,\fw)}
\\
&\quad + \left. \|(1+y)^{3/2}\psi\|_{H^1(\sO,\fw)} + \|(1+y)\psi^+\|_{H^1(\sO,\fw)} \right),
\end{aligned}
\end{equation}
and thus \eqref{eq:H2BoundSolutionHestonVarIneqLipschitz} follows from \eqref{eq:H2BoundSolutionHestonVarIneqLipschitz_prefinal} by Definition \ref{defn:H1WeightedSobolevSpaces} of $H^1(\sO,\fw)$ and Lemma \ref{lem:SobolevSpaceClosedUnderMaxPart}. This completes the proof.
\end{proof}

The proof of Theorem \ref{thm:H2BoundSolutionHestonVarIneqLipschitz} yields the following useful

\begin{cor}[Existence and uniqueness of strong solutions to the localized obstacle problem]
\label{cor:H2BoundSolutionHestonVarIneqLipschitz}
Assume the hypotheses of Theorem \ref{thm:H2BoundSolutionHestonVarIneqLipschitz} and let $u\in H^1_0(\sO\cup\Gamma_0,\fw)$ and $\sU'\subset\sU\subset\sO$ be as in the statement of Theorem \ref{thm:H2BoundSolutionHestonVarIneqLipschitz}. Let $\zeta \in C^\infty(\bar\HH)$ be a cutoff function such that $0\leq \zeta\leq 1$ on $\HH$, $\zeta = 1$ on $\sU'$, $\zeta>0$ on $\sU$, and $\zeta  = 0$ on $\sO\less \sU$. Then $\zeta u \in H^2(\sU,\fw)$ is the unique solution to Problem \ref{prob:HestonObstacleInhomogeneousStrong} (with $g = 0$) with source function $f_\zeta \in L^2(\sU,\fw)$ given by \eqref{eq:fCutoff}, and obstacle function, $\zeta\psi \in H^1(\sU,\fw)$.
\end{cor}

\begin{proof}
One obtains the conclusion either by directly applying Theorem \ref{thm:VIExistenceUniquenessH2RegularEllipticHeston} or observing that
Claim \ref{claim:ObstacleFunctionLocalization} implies that $\zeta u \in H^1_0(\sO\cup\Gamma_0,\fw)$ is a solution to Problem \ref{prob:HomogeneousHestonVIProblem} and repeating the proof of Theorem \ref{thm:VIExistenceUniquenessH2RegularEllipticHeston}.
\end{proof}

\subsection[Local $W^{2,p}$ and $C^{1,\alpha}$ regularity of solutions to the obstacle problem]{Local $\mathbf{W^{2,p}}$ and $\mathbf{C^{1,\alpha}}$ regularity of solutions to the obstacle problem}
\label{subsec:LocalW2pC1AlphaRegularityInequality}
We have the following local analogue of \cite[Corollary 2.1.2]{Bensoussan_Lions}. The unweighted Sobolev spaces in the statement of Theorem \ref{thm:LocalW2pRegularityHestonVI}, namely $W^{2,p}_{\textrm{loc}}(\sO\cup\Gamma_1)$, are defined in the standard way \cite{Adams}, \cite{GT}.

\begin{thm}[Local $W^{2,p}$ regularity up to $\Gamma_1$ of solutions to the variational inequality]
\label{thm:LocalW2pRegularityHestonVI}
Assume the hypotheses of Theorem \ref{thm:VIExistenceUniquenessH2RegularEllipticHeston} and let $u \in H^2(\sO,\fw)\cap H^1_0(\sO\cup\Gamma_0,\fw)$ be the unique solution to Problem \ref{prob:HomogeneousHestonVIProblem}. Suppose in addition that, for $2<p<\infty$,
\begin{equation}
\label{eq:SourceObstacleLp}
f \in L^p_{\textrm{loc}}(\sO\cup\Gamma_1) \quad\hbox{and}\quad \psi \in W^{2,p}_{\textrm{loc}}(\sO\cup\Gamma_1).
\end{equation}
Then $u \in W^{2,p}_{\textrm{loc}}(\sO\cup\Gamma_1)$.
\end{thm}

\begin{proof}
Let $\sU'\subset \sU \subset \sO$ and $\zeta \in C^\infty(\bar\HH)$ be as in the hypotheses of Theorem \ref{thm:H2BoundSolutionHestonVarIneqLipschitz} and preamble to Claim \ref{claim:ObstacleFunctionLocalization} and require in addition that $\sU$ is bounded and $\bar\sU \subset \sO\cup\Gamma_1$, so $\hbox{dist}(\sU,\Gamma_0)>0$. Because $\sU$ is bounded and $\bar\sU \subset \sO\cup\Gamma_1$ and $u \in H^2(\sO,\fw)$, we obtain $u \in W^{2,2}(\sU)$. Therefore, $[A,\zeta]u \in W^{1,2}(\sU)$ and thus $[A,\zeta]u \in L^p(\sU)$ by \cite[Theorem 5.4, Part I (A)]{Adams} and hence $f_\zeta \in L^p(\sU)$ by identity \eqref{eq:fCutoff} and integrability property \eqref{eq:SourceObstacleLp}. Again, because $\sU$ is bounded and $\bar\sU \subset \sO\cup\Gamma_1$ we obtain $\zeta\psi \in W^{1,p}(\sU)$ by \eqref{eq:SourceObstacleLp}. Now $\zeta u \in H^1_0(\sU\cup\Gamma_0,\fw)$ by \eqref{eq:CutoffuH1Zero} and $H^1_0(\sU\cup\Gamma_0,\fw)= W^{1,2}_0(\sU)$ since $\sU$ is bounded and $\bar\sU \subset \sO\cup\Gamma_1$. Moreover, since $\zeta u \in W^{1,2}_0(\sU)$ obeys \eqref{eq:CutoffVI}, it also obeys the coercive variational inequality (compare the proof of Lemma \ref{lem:VIExistenceUniquenessEllipticHeston}),
\begin{equation}
\label{eq:CoerciveLocalizedVI}
a_\lambda(\zeta u, v - \zeta u) \geq (f_{\zeta,\lambda}, v-\zeta u)_{L^2(\sU)}, \quad\forall v\geq \zeta\psi, v \in W^{1,2}_0(\sU),
\end{equation}
with source function $f_{\zeta,\lambda} := f_\zeta + \lambda(1+y)\zeta u \in L^2(\sU)$, noting that $L^2(\sU,\fw)=L^2(\sU)$. But $f_{\zeta,\lambda} \in L^p(\sU)$, since $u\in W^{2,2}(\sU)$ and $W^{2,2}(\sU) \to L^p(\sU)$ by Lemma  \ref{lem:H2SobolevEmbedding}. Because $A$ is uniformly elliptic on $\sU$, then \cite[Exercise 1.3.1]{Friedman_1982} implies that the solution, $\zeta u \in W^{1,2}_0(\sU)$, to the variational inequality \eqref{eq:CoerciveLocalizedVI} is in $W^{2,p}(\sU)$ and hence $u\in W^{2,p}(\sU')$ since $\zeta = 1$ on $\sU'$. Because $\sU'\subset \sO$ was otherwise arbitrary, the conclusion follows.
\end{proof}

\begin{rmk}[Alternative regularity sources]
Although \cite[Theorem 1.3.2]{Friedman_1982} yields a conclusion which is the same as \cite[Exercise 3.1]{Friedman_1982}, the hypotheses on $f,g,\psi$ are stronger than those of \cite[Exercise 3.1]{Friedman_1982} (namely, $f \in C^\alpha(\bar \sU)$, $g \in C^{2,\alpha}(\bar \sU)$, and $\psi \in C^2(\bar \sU)$). Similarly, \cite[Corollary 2.1.3]{Bensoussan_Lions} yields a conclusion which is the same as \cite[Exercise 3.1]{Friedman_1982} but the hypotheses of \cite[Theorem 2.1.9 \& Corollary 2.1.3]{Bensoussan_Lions} are not obeyed by the Heston bilinear form \eqref{eq:HestonWithKillingBilinearForm}.
\end{rmk}

We obtain the following analogue of \cite[Corollary 2.1.3]{Bensoussan_Lions}.

\begin{cor}[$C^{1,\alpha}$ regularity in the interior and up to $\Gamma_1$]
\label{cor:LocalC1alphaRegularity}
Assume the hypotheses of Theorem \ref{thm:LocalW2pRegularityHestonVI}. Then $u \in C^{1,\alpha}_{\textrm{loc}}(\sO\cup\Gamma_1)$, for $0<\alpha\leq 1-2/p$.
\end{cor}

\begin{proof}
Let $\sU\subset \sO$ be as in the proof of Theorem \ref{thm:LocalW2pRegularityHestonVI}. According to \cite[Theorem 5.6, Part II ($\textrm{C}'$)]{Adams}, the embedding $W^{2,p}(\sU)\to C^{1,\alpha}(\bar\sU)$ is continuous for $0<\alpha\leq 1-2/p$, and so the conclusion follows from Theorem \ref{thm:LocalW2pRegularityHestonVI}.
\end{proof}

We obtain the optimal interior regularity where the obstacle function is sufficiently smooth:

\begin{cor}[Interior $C^{1,1}$ regularity where the obstacle function is $C^2$]
\label{cor:LocalC11Regularity}
Assume the hypotheses of Theorem \ref{thm:VIExistenceUniquenessH2RegularEllipticHeston} and let $u \in H^2(\sO,\fw)\cap H^1_0(\sO\cup\Gamma_0,\fw)$ be the unique solution to Problem \ref{prob:HomogeneousHestonVIProblem}. Suppose in addition that, for some $0<\alpha<1$ and $\sU''\subseteqq\sO$ an open (but possibly unbounded) subset with $C^{2,\alpha}$ boundary portion, $\partial\sU''\cap\HH$, and
\begin{equation}
\label{eq:SourceCalphaObstacleC2}
f \in C^\alpha_{\textrm{loc}}(\bar\sU''\cap\HH) \quad\hbox{and}\quad \psi \in C^2_{\textrm{loc}}(\bar\sU''\cap\HH).
\end{equation}
Then $u \in C^{1,1}(\sU'')$ and, if $\sU'' = \sO$, then $u \in C^{1,1}(\sO)$.
\end{cor}

\begin{proof}
Let $\sU'\subset \sU$ and $\zeta \in C^\infty(\bar\HH)$ be as in the proof of Theorem \ref{thm:LocalW2pRegularityHestonVI} and require in addition that $\sU\subset\sU''$ and $\partial\sU$ is $C^{2,\alpha}$. As in the proof of Theorem \ref{thm:LocalW2pRegularityHestonVI}, the function $\zeta u \in H^2(\sU)\cap H^1_0(\sU)$ is the unique solution to \eqref{eq:CoerciveLocalizedVI} with obstacle function, $\zeta\psi \in H^1(\sU)$, and source function, $f_\zeta\in L^2(\sU)$, as defined in \eqref{eq:fCutoff}.

Because $[A,\zeta]$ is a first-order differential operator and $u \in C^{1,\alpha}_{\textrm{loc}}(\sO\cup\Gamma_1)$ by Corollary \ref{cor:LocalC1alphaRegularity}, then $[A,\zeta]u \in C^\alpha_{\textrm{loc}}(\sO\cup\Gamma_1)$. Therefore, $f_\zeta \in C^\alpha(\bar\sU)$ by \eqref{eq:fCutoff} and \eqref{eq:SourceCalphaObstacleC2} and the fact that $\sU\Subset\sO\cup\Gamma_1$, while \eqref{eq:SourceCalphaObstacleC2} implies that $\zeta\psi \in C^2(\bar\sU)$ and Theorem \ref{thm:LocalW2pRegularityHestonVI} implies that $\zeta u\in W^{2,p}(\sU)$, for any $1<p<\infty$. Corollary \ref{cor:H2BoundSolutionHestonVarIneqLipschitz} implies that $\zeta u\in W^{2,p}(\sU)$ is the unique solution to Problem \ref{prob:HestonObstacleInhomogeneousStrong} (with $g = 0$ on $\Gamma_1$) with domain $\sU$, obstacle function, $\zeta\psi \in C^2(\bar\sU)$, and source function, $f_\zeta\in C^\alpha(\bar\sU)$.

Since $A$ is uniformly elliptic on $\sU$ by the fact that $\hbox{dist}(\sU,\Gamma_0)>0$ by our choice of $\sU$, then  \cite[Theorems 1.4.1 \& 1.4.3]{Friedman_1982} imply that $\zeta u\in W^{2,\infty}_{\textrm{loc}}(\sU)$, noting that the condition \cite[Equation (1.3.9)]{Friedman_1982} is stronger than \cite[Equation (1.3.19)]{Friedman_1982}. But $W^{2,\infty}_{\textrm{loc}}(\sU) = C^{1,1}(\sU)$ by \cite[p. 23]{Friedman_1982} and so $\zeta u\in C^{1,1}(\sU)$ and as $\zeta = 1$ on $\sU'$, then $u\in C^{1,1}(\sU')$. Because $\sU'\subset \sO$ was otherwise arbitrary, the conclusion follows.
\end{proof}

As before, it is straightforward to assemble the results we need to conclude the

\begin{proof}[Proofs of Theorems \ref{thm:MainRegularityObstacleProblem} and \ref{thm:MainOptimalRegularityObstacleProblem}]
When $g=0$, Theorem \ref{thm:MainRegularityObstacleProblem} restates Theorem \ref{thm:LocalW2pRegularityHestonVI} and Corollary \ref{cor:LocalC1alphaRegularity}. Theorem \ref{thm:MainOptimalRegularityObstacleProblem} follows from Theorem \ref{thm:VIExistenceUniquenessH2RegularEllipticHeston} and Corollary \ref{cor:LocalC11Regularity} with $\sU'' = \sO$.

When $g\neq 0$, the hypothesis $g \in W^{2,p}_{\textrm{loc}}(\sO\cup\Gamma_1)$ ensures that $\tilde f := f-Ag$ and $\tilde\psi := \psi-g$ obey \eqref{eq:SourceObstacleLp} and so we obtain Theorem \ref{thm:MainRegularityObstacleProblem} for $u$ from the result for $\tilde u = u-g$. Similarly, the hypothesis $g \in C^2_{\textrm{loc}}(\sO\cup\Gamma_1)$ ensures that $\tilde f := f-Ag$ and $\tilde\psi := \psi-g$ obey \eqref{eq:SourceCalphaObstacleC2} with $\sU'' = \sO$. Therefore, we obtain Theorem \ref{thm:MainOptimalRegularityObstacleProblem} for $u$ from the result for $\tilde u = u-g$.
\end{proof}

\appendix

\section{Weighted Sobolev spaces}
\label{sec:WeightedSobolevSpaces}
While there are many excellent references for weighted Sobolev spaces (see, for example, \cite{Kufner, KufnerOpic, LevendorskiDegenElliptic, Stredulinsky, Turesson_2000} and references contained therein), we shall need extensions or refinements of those results for the specific weighted Sobolev spaces we employ. This section serves to develop a toolkit of results for our weighted Sobolev spaces, specifically the spaces $H^1(\sO,\fw), H^1_0(\sO\cup\Gamma_0,\fw), H^1_0(\sO,\fw)$ (Definition \ref{defn:H1WeightedSobolevSpaces}) and $H^2(\sO,\fw)$ (Definition \ref{defn:H2WeightedSobolevSpaces}) required by this article. Except for a few cases specific to $d=2$, the results in this section apply to domains $\sO\subset\HH$ as in Definition \ref{defn:HestonDomain}, where $\HH = \RR^{d-1}\times(0,\infty)$ with $d\geq 2$ rather than $d=2$, as assumed in the body of this article. When $d\geq 2$, we denote points in $\HH$ by $(x,y)$, where $x=(x_1,\ldots,x_{d-1})$ and $y=x_d$, and denote Lebesgue measure on $\RR^{d-1}$ by $dx$.

\subsection{Approximation by smooth functions}
We have the following useful result and proof due to C. Pop \cite{PopThesis}.

\begin{lem}[Equivalence of weighted $H^1$ Sobolev spaces when $\beta > 1$]
\label{lem:CameliasH1ApproximationLemma}
When $\beta > 1$, one has $H^1_0(\sO\cup\Gamma_0,\fw) = H^1_0(\sO,\fw)$.
\end{lem}

\begin{rmk}
\label{rmk:CameliasH1ApproximationLemmaKufner}
Lemma \ref{lem:CameliasH1ApproximationLemma} can be viewed as a special case of \cite[Theorem 9.10]{Kufner}.
\end{rmk}

\begin{proof}
Let $u \in H^1_0(\sO\cup\Gamma_0,\fw)$. Then there is a sequence $\{u_m\}_{m\geq 0} \subset C^\infty_0(\sO\cup\Gamma_0)$ such that $u_m \to u$ in $H^1(\sO,\fw)$ as $m\to\infty$. Hence, we may assume without loss of generality that $u \in C^\infty_0(\sO\cup\Gamma_0)$. Let $\{\varphi_m\}_{m\geq 1}\subset C^\infty(\RR^2)$ be a sequence of cutoff functions with the properties
$$
\varphi_m (x,y)
=
\begin{cases}
1, & y\geq 2/m,
\\
0, & y\leq 1/m,
\end{cases}
$$
and
$$
0\leq \varphi_{m,y} \leq 2m \hbox{ on $\sO$ and } \varphi_{m,x} = 0 \hbox{ on }\sO.
$$
Thus, $u_m := \varphi_m u\in H^1_0(\sO,\fw)$ and
\begin{align*}
\|u - u_m\|_{H^1(\sO,\fw)}^2
&= \int_\sO \left\{y\left((1-\varphi_m)^2u_x^2 + \varphi_{m,y}^2u^2 - 2(1-\varphi_m)\varphi_{m,y} uu_y + (1-\varphi_m)^2u_y^2\right) \right.
\\
&\qquad + \left. (1+y)(1-\varphi_m)^2u^2 \right\}\fw(x,y)\,dxdy
\\
&\leq
\int_\sO 1_{\{0 < y < 2/m\}} (1+y)u^2\fw(x,y)\,dxdy
\\
&\quad + 2\int_\sO 1_{\{0 < y < 2/m\}} y\left(u_x^2 + u_y^2\right)\fw(x,y)\,dxdy
\\
&\quad + \int_\sO yu^2\varphi_{m,y}^2\fw(x,y)\,dxdy.
\end{align*}
The first two terms converge to zero as $m\to\infty$, since
$$
\|u\|_{H^1(\sO,\fw)}^2
=
\int_\sO \left\{y(u_x^2 + u_y^2) + (1+y)u^2\right\}\fw(x,y)\,dxdy < \infty.
$$
For the third term, recall that $u\in C^\infty_0(\sO\cup\Gamma_0)$. Hence, $\|u\|_{L^\infty(\sO)} < \infty$ and therefore
\begin{align*}
\int_\sO yu^2\varphi_{m,y}^2\fw(x,y)\,dxdy
& =
\int_\sO 1_{\{1/m < y < 2/m\}} u^2\varphi_{m,y}^2 y^\beta e^{-\mu y - \gamma|x|}\,dxdy
\\
& \leq C\|u\|_{L^\infty(\sO)}^2 m^2\int_{1/m}^{2/m} y^\beta\,dy
\\
&= C\|u\|_{L^\infty(\sO)}^2 \frac{2^{\beta+1}-1}{\beta+1} m^{1 - \beta}.
\end{align*}
Consequently, the third term also tends to zero as $m\to\infty$, provided $\beta>1$.
\end{proof}

\begin{rmk}
\label{rmk:CameliasH1ApproximationLemma}
Due to the final estimate in the proof of Lemma \ref{lem:CameliasH1ApproximationLemma} involving cutoff function derivative, the case $\beta = 1$ cannot be included.
\end{rmk}

\begin{lem}[Density of smooth functions with compact support]
\label{lem:L2ApproximationLemma}
For any $\beta > 0$, one has that $C^\infty_0(\sO)$ is a dense subset of $L^2(\sO,\fw)$.
\end{lem}

\begin{proof}
Let $\eps>0$ and let $\{\zeta_m\}_{m\geq 1} \subset C^\infty_0(\HH)$ be a sequence of cutoff functions such that $0\leq\zeta_m\leq 1$, $\supp\zeta_m \subset[-2m,2m]\times[1/2m,2m]$, $\zeta_m = 1$ on $[-m,m]\times[1/m,m]$, and set $\sO_m := \sO\cap((-m,m)\times(1/m,m))$. By \cite[Corollary 2.19]{Adams}, $C^\infty_0(\sO_m)$ is a dense subset of $L^2(\sO_m)$ and so we may choose a sequence $\{f_m\}_{m\geq 1} \in C^\infty_0(\sO_{2m})$ such that
$$
\int_{\sO_{2m}}|\zeta_mf-f_m|^2\,dxdy < \frac{\eps^2}{4M_m}, \quad m\geq 1,
$$
where
$$
M_m := \max_{(x,y)\in\bar\sO_{2m}}y^{\beta-1}e^{-\gamma|x|-\mu y} > 0.
$$
Therefore, recalling that $\fw(x,y) = y^{\beta-1}e^{\gamma|x|-\mu y}$ by \eqref{eq:HestonWeight}, we obtain
\begin{equation}
\label{eq:L2ApproxFromAdams}
\int_\sO|\zeta_mf-f_m|^2 \fw\,dxdy < \frac{\eps^2}{4}, \quad m\geq 1.
\end{equation}
But $\hbox{Vol}(\sO,\fw)  = \int_\sO 1\,\fw\,dxdy < \infty$, while
$$
\int_\sO|f-\zeta_mf|^2\fw\,dxdy = \int_{\sO\less\sO_m}|(1-\zeta_m)f|^2\fw\,dxdy \leq \int_{\sO\less\sO_m}|f|^2\fw\,dxdy,
$$
and so, because $\lim_{m\to\infty}\hbox{Vol}(\sO\less\sO_m,\fw) = 0$ and $\|f\|_{L^2(\sO,\fw)} < \infty$, we have
$$
\lim_{m\to\infty}\|f-\zeta_m f\|_{L^2(\sO,\fw)} = 0,
$$
by \cite[Exercise 1.12]{RudinRealComplex}. In particular, we may choose $m_0\geq 1$ such that
\begin{equation}
\label{eq:L2ApproxCutoff}
\int_\sO|f-\zeta_mf|^2\fw\,dxdy < \frac{\eps^2}{4}, \quad m \geq m_0.
\end{equation}
Hence,
\begin{align*}
\|f-f_m\|_{L^2(\sO,\fw)} &\leq \|f-\zeta_mf\|_{L^2(\sO,\fw)} + \|\zeta_mf-f_m\|_{L^2(\sO,\fw)}
\\
&<\eps, \quad\forall m\geq m_0, \quad\hbox{(by \eqref{eq:L2ApproxFromAdams} and \eqref{eq:L2ApproxCutoff})}.
\end{align*}
This completes the proof.
\end{proof}

Using a more careful choice of cutoff function, based on \cite[Lemma 7.2.10]{DK}, we can extend the conclusion of Lemma \ref{lem:CameliasH1ApproximationLemma} to include the case $\beta = 1$:

\begin{lem}[Cut-off functions with integral norm decay]
\label{lem:DonaldsonKronheimerCutoffFunction}
There is a positive constant $c$ such that the following holds. For any $N > 4$ and $\eps > 0$, there is a $C^\infty$ cutoff function $\psi = \psi_{N,\eps}$ on $\RR$
such that
$$
\psi(y)
=
\begin{cases}
1 &\hbox{if }|y| \geq \eps/2,
\\
0 &\hbox{if }|y| \leq \eps/N,
\end{cases}
$$
and satisfying the following estimates:
$$
\int_\RR |\nabla\psi(y)|^2 |y|^{\beta-1}\,dy
\leq
\begin{cases}
c^2 2^{2-\beta}(\beta-1)^{-1}(\log N)^{-2}\eps^{\beta-1}, &\hbox{if }\beta > 1,
\\
2c^2(\log N)^{-1}, &\hbox{if }\beta = 1.
\end{cases}
$$
\end{lem}

\begin{proof}
Fix a $C^\infty$ cutoff function, $\phi:\RR\to[0,1]$, such that $\phi(t)=1$ if $t\geq 1$ and $\phi(t)=0$ if $t\leq 0$. Now define a $C^\infty$ cutoff function $\chi = \chi_N : \RR \to [0, 1]$, depending on the parameter $N$, by setting
$$
\chi(t) := \phi\left(\frac{\log N + t}{\log N - \log 2}\right).
$$
Therefore,
$$
\chi(t)
=
\begin{cases}
1 &\hbox{if }t \geq -\log 2,
\\
0 &\hbox{if }t \leq -\log N,
\end{cases}
$$
and there is a positive constant $c$ independent of $N$ such that
\begin{align*}
\left|\frac{d\chi}{dt}\right|
&\leq \left|\phi'\left(\frac{\log N + t}{\log N - \log 2}\right)\right|\frac{1}{\log N - \log 2}
\\
&\leq c\frac{1}{\log N}.
\end{align*}
Now define the cutoff function $\psi = \psi_{N,\eps} : \RR \to [0, 1]$, depending on the parameters $N$ and $\eps$, by setting
$$
\psi(y) := \chi(\log |y| - \log\eps), \quad y \in \RR,
$$
and observe that $\psi$ satisﬁes the following pointwise bounds,
$$
|\nabla\psi(y)| \leq c\frac{1}{|y|\log N}, \quad y \in \RR\less\{0\},
$$
and
$$
\psi(y)
=
\begin{cases}
1 &\hbox{if }|y| \geq \eps/2,
\\
0 &\hbox{if }|y| \leq \eps/N.
\end{cases}
$$
Consequently,
$$
\int_\RR |\nabla\psi(y)|^2 |y|^\beta \,dy \leq \frac{2c^2}{(\log N)^2}\int^{\eps/2}_{\eps/N} y^{\beta-2}\,dy.
$$
If $\beta>1$,
\begin{align*}
\int_\RR |\nabla\psi(y)|^2 |y|^\beta\,dy
&\leq
\frac{2c^2}{(\log N)^2}\left((1/2)^{\beta-1} - (1/N)^{\beta-1}\right)\frac{\eps^{\beta-1}}{\beta-1}
\\
&\leq \frac{c^2 2^{2-\beta}}{(\beta-1)(\log N)^2}\eps^{\beta-1}.
\end{align*}
If $\beta=1$,
$$
\int_\RR |\nabla\psi(y)|^2 |y|^\beta\,dy \leq \frac{2c^2}{(\log N)^2}\int^{\eps/2}_{\eps/N} y^{-1}\,dy
\leq
\frac{2c^2}{(\log N)^2}(\log N - \log 2),
$$
and the conclusion follows.
\end{proof}

We now extend Lemma \ref{lem:CameliasH1ApproximationLemma} to include the case $\beta = 1$:

\begin{lem}[Equivalence of weighted $H^1$ Sobolev spaces when $\beta\geq 1$]
\label{lem:ImprovedH1ApproximationLemma}
When $\beta \geq 1$, one has $H^1_0(\sO\cup\Gamma_0,\fw) = H^1_0(\sO,\fw)$.
\end{lem}

\begin{rmk}
\label{rmk:ImprovedH1ApproximationLemma}
Since only the boundary $\Gamma_0$ is material in its proof, Lemma \ref{lem:ImprovedH1ApproximationLemma} can be strengthened to assert that $H^1(\sO,\fw) = H^1_0(\sO\cup\Gamma_1,\fw)$ when $\beta\geq 1$.
\end{rmk}

\begin{proof}
Lemma \ref{lem:CameliasH1ApproximationLemma} covers the case $\beta > 1$, so it suffices to consider $\beta=1$. From the proof of Lemma \ref{lem:CameliasH1ApproximationLemma}, it is enough to estimate the term
\begin{align*}
\int_\sO yu^2|\nabla\varphi_m(y)|^2\fw(x,y)\,dxdy
&=
\int_\sO u^2|\nabla\varphi_m(y)|^2 y e^{-\mu y - \gamma|x|}\,dxdy
\\
&\leq
C\|u\|_{L^\infty(\sO)}^2\int_{\RR^+} |\nabla\varphi_m(y)|^2 y e^{-\mu y}\,dy.
\end{align*}
In Lemma \ref{lem:DonaldsonKronheimerCutoffFunction}, choose $\eps = 2/m, m\geq 5$ and $N = m$ and $\varphi_m = \psi_{\eps,N}$, so that
$$
\int_{\RR^+} |\nabla\varphi_m(y)|^2 y e^{-\mu y}\,dy \leq c^2(\log m)^{-1}, \quad m\geq 5.
$$
Therefore,
$$
\int_\sO yu^2|\nabla\varphi_m(y)|^2\fw(x,y)\,dxdy \to 0, \hbox{ as } m\to \infty,
$$
and the result follows.
\end{proof}

We recall the classical

\begin{thm}[Hardy inequality] \cite[\S 1.5]{DrabekKufnerNicolosi}, \cite[Theorem 5.2]{Kufner}, \cite[Lemma 1.3]{KufnerOpic}, \cite[Chapter 1]{KufnerPersson}, \cite[Equation (0.32)]{KufnerSandig}, \cite[Theorem A.3]{Stein}
\label{thm:HardyInequality}
Let $1<p<\infty$, $\beta\neq p-1$. Let $v:(0,\infty)\to\RR$ be differentiable a.e. on $(0,\infty)$ such that
$$
\int_0^\infty |v'(y)|^p y^\beta\,dy < \infty.
$$
Further, let $v$ satisfy the conditions
\begin{align*}
v(0) &= \lim_{y\to 0} v(y) = 0 \quad\hbox{for}\quad \beta<p-1,
\\
v(\infty) &= \lim_{y\to \infty} v(y) = 0 \quad\hbox{for}\quad \beta>p-1.
\end{align*}
Then the following inequality holds:
\begin{equation}
\label{eq:ClassicHardyInequality}
\int_0^\infty |v(y)|^p y^{\beta-p}\,dy \leq \left(\frac{p}{\beta-p+1}\right)^p\int_0^\infty |v'(y)|^p y^\beta\,dy.
\end{equation}
\end{thm}

See \cite{Feehan_generalizedhardy} for a survey of extensions of Hardy's inequality. Next we recall a special case of a result of Kufner \cite{Kufner}. If $\sO\subset\RR^d$ is a \emph{bounded} domain and $T\subset\partial\sO$, denote $W^{k,p}(\sO;d_T,\eps)$ by \cite[\S 3.3]{Kufner}
\begin{equation}
\label{eq:KufnerSobolevNorm}
W^{k,p}(\sO;d_T,\eps) := \left\{u \in L(\sO): \|u\|_{W^{k,p}(\sO;d_T,\eps)} < \infty, \quad\forall \alpha, |\alpha|\leq k\right\},
\end{equation}
where $1\leq p < \infty$, $k\geq 0$ is an integer, $\eps\in\RR$, $d_T(x) := \hbox{dist}(x,T), x\in\sO$, and
$$
\|u\|_{W^{k,p}(\sO;d_T,\eps)}^p := \sum_{|\alpha|\leq k}\int_\sO|D^\alpha u(x)|^p d_T^\eps(x)\,dx,
$$
and $D^\alpha u$ is defined in the sense of distributions \cite[\S 1]{Kufner}, and denote $L^p(\sO;d_T,\eps) = W^{0,p}(\sO;d_T,\eps)$. Require that $\bar T$ obey a uniform exterior cone condition in the sense of \cite[Definition 4.10 and Remark 7.5]{Kufner}. We have the following analogue of the density result \cite[\S 3.17 \& Theorem 3.18]{Adams}, \cite[Theorem 5.3.3]{Evans} for standard Sobolev spaces.

\begin{thm}[Density of smooth functions for power weights]
\label{thm:KufnerPowerWeight}
\cite[Theorems 7.2, 7.4, Remark 7.5, Proposition 7.6 and Remark 11.12 (iii)]{Kufner}
Let $\sO$ be a \emph{bounded}, $\sC^0$ domain in the sense of \cite[Definition 4.2]{Kufner} and require that $\bar T$ obey a uniform exterior cone condition. Then $C^\infty(\bar\sO)$ is a dense subset of $W^{k,p}(\sO;d_T,\eps)$ for $\eps > -1$.
\end{thm}

\begin{rmk}[$\sC^0$ domains]
\label{rmk:C0domain}
If $\sO$ has the \emph{strong local Lipschitz property} in the sense of \cite[\S 4.5]{Adams} (compare the concept of a $\sC^{0,1}$ domain in \cite[Definition 4.3]{Kufner}), then $\sO$ will be a $\sC^0$ domain; moreover, if $\sO$ has the strong local Lipschitz property, then by \cite[\S 4.7]{Adams} it has the \emph{segment property} in the sense of \cite[\S 4.2]{Adams}. When $\sO$ is bounded, then the condition that $\sO$ have the strong local Lipschitz property reduces to the condition that $\partial\sO$ is locally Lipschitz (that is, each point in $\partial\sO$ has a neighborhood $U\subset\bar\sO$ such that $U\cap\partial\sO$ is the graph of a Lipschitz function). In particular, when $\sO$ in Definition \ref{defn:HestonDomain} obeys Hypotheses \ref{hyp:HestonDomainNearGammaZero} and \ref{hyp:HestonDomainNearGammaOne} when $k=1$, then $\sO$ necessarily has the strong local Lipschitz property (see \cite[\S 4.7]{Adams}) and thus is a $\sC^0$ domain.
\end{rmk}

As a consequence of Theorem \ref{thm:KufnerPowerWeight}, we obtain

\begin{lem}[Equivalence of definitions of weighted Sobolev spaces]
\label{lem:KufnerPowerWeightedSobolevSpaceEquivalents}
Let $\sO$ be a domain as in Definition \ref{defn:HestonDomain}. If $\sO$ is \emph{bounded}, $\sO$ obeys Hypothesis \ref{hyp:HestonDomainNearGammaZero}, and $\partial\sO$ is locally Lipschitz, and  then, for any $\beta > 0$,
$$
H^k(\sO,\fw) = W^{k,2}(\sO;d_{\Gamma_0},\beta+k-1), \quad k=0,1,2.
$$
\end{lem}

\begin{proof}
The derivatives appearing in the definitions of $H^k(\sO,\fw)$ and $W^{k,2}(\sO;d_{\Gamma_0},\beta+k-1)$ are defined in the sense of distributions and so it is enough to show that the norms associated with these weighted Sobolev spaces are equivalent. This is clear when $k=0$, so it is immediate that $L^2(\sO,\fw) = W^{0,2}(\sO;d_{\Gamma_0},\beta-1)$.

By Hypothesis \ref{hyp:HestonDomainNearGammaZero}, there is a $\delta_0>0$ such that $\sO\cap\{y<y_0\} = \Gamma_0\times(0,\delta_0)$. Let $\chi\in C^\infty(\bar\HH)$ be a cutoff function such that $\chi=1$ on $\RR\times[0,\delta_0/2)$ and $\chi=0$ on $\RR\times(\delta_0,\infty)$.

Choose $k=1$ and $\eps=\beta$ and suppose $u\in H^1(\sO,\fw)$. Then,
\begin{align*}
\|u\|_{W^{1,2}(\sO;d_{\Gamma_0},\beta)}^2 &= \int_\sO\left(y|Du|^2 + yu^2\right)y^{\beta-1}\,dxdy \quad\hbox{(by \eqref{eq:KufnerSobolevNorm})}
\\
&\leq \int_\sO\left(y|Du|^2 + (1+y)u^2\right)y^{\beta-1}\,dxdy
\\
&\leq \left(\inf_{(x,y)\in\sO} e^{-\gamma|x|-\mu y}\right)\int_\sO\left(y|Du|^2 + (1+y)u^2\right)y^{\beta-1}e^{-\gamma|x|-\mu y}\,dxdy
\\
&=  C\|u\|_{H^1(\sO,\fw)}^2 \quad\hbox{(by Definition \ref{defn:H1WeightedSobolevSpaces})},
\end{align*}
where $C := \inf_{(x,y)\in\sO} e^{-\gamma|x|-\mu y} > 0$ since $\sO$ is bounded, and thus $u\in W^{1,2}(\sO;d_{\Gamma_0},\beta)$. Hence, $H^1(\sO,\fw) \subset W^{1,2}(\sO;d_{\Gamma_0},\beta)$. On the other hand, if $u\in W^{1,2}(\sO;d_{\Gamma_0},\beta)$, then
\begin{align*}
\int_\sO u^2 y^{\beta-1}\,dxdy &\leq \int_{\sO\cap\{y< \delta_0\}} (\chi u)^2 y^{\beta-1}\,dxdy + \int_{\sO\cap\{y\geq \delta_0/2\}} u^2 y^{\beta-1}\,dxdy
\\
&\leq C\int_{\sO\cap\{y< \delta_0\}} (\chi u)_y^2 y^{\beta+1}\,dxdy + \int_{\sO\cap\{y\geq \delta_0/2\}} u^2 y^{\beta-1}\,dxdy \quad\hbox{(Theorem \ref{thm:HardyInequality})}
\\
&\leq C\int_{\sO\cap\{y< \delta_0\}} u_y^2 y^{\beta+1}\,dxdy + C\int_{\sO\cap\{y\geq \delta_0/2\}} u^2 (y^{\beta+1} + y^{\beta-1})\,dxdy,
\\
&\leq C\int_{\sO\cap\{y< \delta_0\}} u_y^2 y^\beta\,dxdy + C\int_{\sO\cap\{y\geq \delta_0/2\}} u^2 y^\beta\,dxdy \quad\hbox{(since $\sO$ bounded)},
\end{align*}
noting that Theorem \ref{thm:HardyInequality} applies since $\beta+1>1$ for $\beta>0$ and $\chi u(\cdot,\delta_0)=0$. Therefore,
\begin{equation}
\label{eq:HardyInequalityWH1Equivalence}
\int_\sO u^2 y^{\beta-1}\,dxdy \leq C\int_\sO (u_y^2 + u^2) y^\beta\,dxdy.
\end{equation}
Consequently, using $\sup_{(x,y)\in\sO} e^{-\gamma|x|-\mu y} < \infty$ as $\sO$ is bounded,
\begin{align*}
\|u\|_{H^1(\sO,\fw)}^2 &= \int_\sO\left(y|Du|^2 + (1+y)u^2\right)y^{\beta-1}e^{-\gamma|x|-\mu y}\,dxdy
 \quad\hbox{(by Definition \ref{defn:H1WeightedSobolevSpaces})}
\\
&\leq \left(\sup_{(x,y)\in\sO} e^{-\gamma|x|-\mu y}\right)\int_\sO\left(y|Du|^2 + (1+y)u^2\right)y^{\beta-1}\,dxdy
\\
&\leq C\int_\sO\left(y|Du|^2 + yu^2\right)y^{\beta-1}\,dxdy \quad\hbox{(by \eqref{eq:HardyInequalityWH1Equivalence})}
\\
&= C\|u\|_{W^{1,2}(\sO;d_{\Gamma_0},\beta)}^2  \quad\hbox{(by \eqref{eq:KufnerSobolevNorm})}.
\end{align*}
Hence, $u\in H^1(\sO,\fw)$ and so $W^{1,2}(\sO;d_{\Gamma_0},\beta) \subset H^1(\sO,\fw)$. Thus, $H^1(\sO,\fw) = W^{1,2}(\sO;d_{\Gamma_0},\beta)$.

Choose $k=2$ and $\eps=\beta+1$ and suppose $u\in H^2(\sO,\fw)$. Then,
\begin{align*}
\|u\|_{W^{2,2}(\sO;d_{\Gamma_0},\beta+1)}^2 &= \int_\sO\left(y^2|D^2u|^2 + y^2|Du|^2 + y^2u^2\right)y^{\beta-1}\,dxdy  \quad\hbox{(by \eqref{eq:KufnerSobolevNorm})}
\\
&\leq C\int_\sO\left(y^2|D^2u|^2 + (1+y^2)|Du|^2 + (1+y)u^2\right)y^{\beta-1}\,dxdy \quad\hbox{(since $\sO$ bounded)}
\\
&\leq C\left(\inf_{(x,y)\in\sO} e^{-\gamma|x|-\mu y}\right)
\\
&\qquad \times\int_\sO\left(y^2|D^2u|^2 + (1+y^2)|Du|^2 + (1+y)u^2\right)y^{\beta-1}e^{-\gamma|x|-\mu y}\,dxdy
\\
&\leq C\|u\|_{H^2(\sO,\fw)}^2 \quad\hbox{(by Definition \ref{defn:H2WeightedSobolevSpaces} and as $\sO$ bounded)},
\end{align*}
and so $u\in W^{2,2}(\sO;d_{\Gamma_0},\beta+1)$. Hence, $H^2(\sO,\fw) \subset W^{2,2}(\sO;d_{\Gamma_0},\beta+1)$. On the other hand, if $u\in W^{2,2}(\sO;d_{\Gamma_0},\beta+1)$, then the derivation of \eqref{eq:HardyInequalityWH1Equivalence} shows that
\begin{equation}
\label{eq:HardyInequalityWH2Equivalenceu}
\int_\sO u^2 y^{\beta-1}\,dxdy \leq C\int_\sO (u_y^2 + u^2) y^{\beta+1}\,dxdy,
\end{equation}
while
\begin{align*}
\int_\sO |Du|^2 y^{\beta-1}\,dxdy &\leq \int_{\sO\cap\{y<\delta_0\}}\left((\chi u)_x^2 + (\chi u)_y^2\right) y^{\beta-1}\,dxdy + \int_{\sO\cap\{y\geq \delta_0/2\}}|Du|^2 y^{\beta-1}\,dxdy
\\
&\leq \int_{\sO\cap\{y<\delta_0\}}\left((\chi u)_{xy}^2 + (\chi u)_{yy}^2\right) y^{\beta+1}\,dxdy
\\
&\quad + \int_{\sO\cap\{y\geq \delta_0/2\}}|Du|^2 y^{\beta-1}\,dxdy \quad\hbox{(Theorem \ref{thm:HardyInequality})}
\\
&\leq \int_{\sO\cap\{y<\delta_0\}}\left(u_{xy}^2 + u_{yy}^2\right) y^{\beta+1}\,dxdy
\\
&\quad + \int_{\sO\cap\{y\geq \delta_0/2\}}\left(|Du|^2y^{\beta-1} + |Du|^2y^{\beta+1} + u^2y^{\beta+1}\right)\,dxdy
\\
&\leq \int_{\sO\cap\{y<\delta_0\}}\left(u_{xy}^2 + u_{yy}^2\right) y^{\beta+1}\,dxdy
\\
&\quad + \int_{\sO\cap\{y\geq \delta_0/2\}}\left(|Du|^2 + u^2\right)y^{\beta+1}\,dxdy \quad\hbox{(since $\sO$ bounded)}
\end{align*}
noting that Theorem \ref{thm:HardyInequality} applies since $\beta+1>1$ for $\beta>0$ and $\chi u(\cdot,\delta_0)=0$. Therefore,
\begin{equation}
\label{eq:HardyInequalityWH2EquivalenceDu}
\int_\sO |Du|^2 y^{\beta-1}\,dxdy \leq C\int_\sO \left(|D^2u|^2 + |Du|^2 + u^2\right)y^{\beta+1}\,dxdy.
\end{equation}
Consequently, by Definition \ref{defn:H2WeightedSobolevSpaces},
\begin{align*}
\|u\|_{H^2(\sO,\fw)}^2 &= \int_\sO\left(y^2|D^2u|^2 + (1+y^2)|Du|^2 + (1+y)u^2\right)y^{\beta-1}e^{-\gamma|x|-\mu y}\,dxdy
\\
&\leq \left(\sup_{(x,y)\in\sO} e^{-\gamma|x|-\mu y}\right)\int_\sO\left(y^2|D^2u|^2 + (1+y^2)|Du|^2 + (1+y)u^2\right)y^{\beta-1}\,dxdy
\\
&\leq C\int_\sO\left(y^2|D^2u|^2 + (1+y^2)|Du|^2 + (1+y^2)u^2\right)y^{\beta-1}\,dxdy \quad\hbox{(as $\sO$ bounded)}
\\
&\leq C\int_\sO\left(y^2|D^2u|^2 + y^2|Du|^2 + y^2u^2\right)y^{\beta-1}\,dxdy
\quad\hbox{(by \eqref{eq:HardyInequalityWH2Equivalenceu} and \eqref{eq:HardyInequalityWH2EquivalenceDu})}
\\
&= C\|u\|_{W^{2,2}(\sO;d_{\Gamma_0},\beta+1)}^2  \quad\hbox{(by \eqref{eq:KufnerSobolevNorm})}.
\end{align*}
Hence, $u\in H^2(\sO,\fw)$ and so $W^{2,2}(\sO;d_{\Gamma_0},\beta+1) \subset H^2(\sO,\fw)$. Thus, we obtain $H^2(\sO,\fw) = W^{2,2}(\sO;d_{\Gamma_0},\beta+1)$.
\end{proof}

We have the following analogue, for our domains (Definition \ref{defn:HestonDomain}) and weighted Sobolev spaces (Definitions \ref{defn:H1WeightedSobolevSpaces} and \ref{defn:H2WeightedSobolevSpaces}), of the density result \cite[Theorem 3.18]{Adams}, \cite[Theorem 5.3.3]{Evans} for standard Sobolev spaces.

\begin{cor}[Density of smooth functions in weighted Sobolev spaces]
\label{cor:KufnerPowerWeight}
Let $\sO$ be a domain as in Definition \ref{defn:HestonDomain} such that $\sO$ obeys Hypothesis \ref{hyp:HestonDomainNearGammaZero} and $\partial\sO$ has the strong local Lipschitz property. Then $C^\infty_0(\bar\sO)$ is a dense subset of $H^k(\sO,\fw)$ for $\beta > 0$ and $k=0,1,2$.
\end{cor}

\begin{proof}
Fix $\eps>0$ and let $\{\sO_m\}_{m\geq 0}$ be a sequence of bounded, Lipschitz subdomains of $\sO$ such that $\cup_{m\geq 0}\sO_m = \sO$ and $\sO_m\subset\sO_{m+1}, \forall m\geq 0$, together with a sequence of cutoff functions $\{\chi_m\}_{m\geq 1} \subset C^\infty_0(\bar\HH)$ with $\supp\chi_m\subset\sO_m$, $\chi_m=1$ on $\sO_{m-1}$, and $\|\chi_m\|_{C^2(\HH)} \leq K, \forall m\geq 1$, for some universal constant $K\geq 1$ independent of $m$. (For example, we may take $\sO_m := \sO\cap B(m+1)$, $m\geq 0$, where $B(R)$ is the open ball in $\RR^2$ with center at the origin and radius $R>0$.) Because $\|\chi_m\|_{C^2(\HH)} \leq K, \forall m\geq 1$, for $k=1,2$, we have
\begin{equation}
\label{eq:SobolevNormEffectOfCutoff}
\|(1-\chi_m)v\|_{H^k(\sO_m,\fw)} \leq c_0\|v\|_{H^k(\sO_m\less\sO_{m-1},\fw)}, \quad \forall v \in H^k(\sO_m),\fw),
\end{equation}
where $c_0=c_0(K)\geq 1$ is a universal constant independent of $m\geq 1$ or $v \in H^k(\sO_m),\fw)$. Theorem \ref{thm:KufnerPowerWeight} and Lemma \ref{lem:KufnerPowerWeightedSobolevSpaceEquivalents} imply that there exists a sequence $\{u_m\}_{m\geq 1}$ such that, for each $m\geq 1$, we have $u_m \in C^\infty(\bar\sO_m)$ and
\begin{equation}
\label{eq:ApproxuOnOmplus1}
\|u-u_m\|_{H^k(\sO_m,\fw)} < \frac{\eps}{4c_0}, \quad m \geq 1.
\end{equation}
Since $\|u\|_{H^k(\sO,\fw)} < \infty$, we may choose $m_0\geq 1$ large enough that
\begin{equation}
\label{eq:HkSmallTail}
\|u\|_{H^k(\sO\less\sO_{m-1},\fw)} < \frac{\eps}{4c_0}, \quad \forall m \geq m_0.
\end{equation}
Therefore, noting that $\supp\chi_mu_m\subset\sO_m$,
\begin{align*}
\|u-\chi_mu_m\|_{H^k(\sO,\fw)} &\leq \|u\|_{H^k(\sO\less\sO_m),\fw)} + \|u-\chi_mu_m\|_{H^k(\sO_m,\fw)}
\\
&< \frac{\eps}{4c_0} + \|u-u_m\|_{H^k(\sO_m,\fw)} + \|(1-\chi_m)u_m\|_{H^k(\sO_m,\fw)}
\\
&\qquad\hbox{(by \eqref{eq:HkSmallTail} for $m\geq m_0$)}
\\
&< \frac{\eps}{2c_0} + c_0\|u_m\|_{H^k(\sO_m\less\sO_{m-1},\fw)} \quad\hbox{(by \eqref{eq:ApproxuOnOmplus1} and \eqref{eq:SobolevNormEffectOfCutoff})}
\\
&\leq \frac{\eps}{2} + c_0\|u\|_{H^k(\sO_m\less\sO_{m-1},\fw)} + c_0\|u-u_m\|_{H^k(\sO_m\less\sO_{m-1},\fw)}
\\
&\leq \frac{\eps}{2} + c_0\|u\|_{H^k(\sO\less\sO_{m-1},\fw)} + c_0\|u-u_m\|_{H^k(\sO_m,\fw)}
\\
&< \eps, \quad \forall m \geq m_0  \quad\hbox{(by \eqref{eq:HkSmallTail} and \eqref{eq:ApproxuOnOmplus1})}.
\end{align*}
Consequently, the sequence $\{\chi_mu_m\}_{m\geq 1} \subset C^\infty_0(\bar\sO)$ converges strongly in $H^k(\sO,\fw)$ to $u \in H^k(\sO,\fw)$ and this completes the proof.
\end{proof}

We then have the following analogue of \cite[Theorem 3.16]{Adams}:

\begin{cor}[Meyers-Serrin theorem for weighted Sobolev spaces]
\label{cor:MeyersSerrinWeightedSobolev}
Let $\sO$ be a domain as in Definition \ref{defn:HestonDomain} such that $\sO$ obeys Hypothesis \ref{hyp:HestonDomainNearGammaZero} and $\partial\sO$ has the strong local Lipschitz property. Then $H^k(\sO,\fw)$ is the completion of $C^\infty(\sO)\cap H^k(\sO,\fw)$ for $\beta > 0$ and $k=0,1,2$.
\end{cor}

\begin{proof}
This follows from Corollary \ref{cor:KufnerPowerWeight}, since $C^\infty_0(\bar\sO) \subset C^\infty(\sO)\cap H^k(\sO,\fw)$.
\end{proof}

Recall from \cite[\S 1.26]{Adams} that for each integer $\ell\geq 0$, $C^\ell(\bar\sO)$ denotes the Banach space of functions $u\in C^\ell(\sO)$ for which $D^\alpha u$ is bounded and uniformly continuous on $\sO$ for $0\leq |\alpha| \leq \ell$, with norm
$$
\|u\|_{C^\ell(\bar\sO)} := \max_{0\leq|\alpha|\leq\ell}\sup_{(x,y)\in\sO}|D^\alpha u(x,y)|.
$$
By Definitions \ref{defn:H1WeightedSobolevSpaces} and \ref{defn:H2WeightedSobolevSpaces}, one sees that $C^0(\bar\sO)\subset H^0(\sO,\fw)$, $\ell\geq 0$, and $C^\ell(\bar\sO)\subset H^1(\sO,\fw)$, $\ell\geq 1$, and $C^\ell(\bar\sO)\subset H^1(\sO,\fw)$, $\ell\geq 2$. We have the following analogue, for our weighted Sobolev spaces, of \cite[\S 3.17 \& \S 3.18]{Adams}, \cite[Theorem 5.3.3]{Evans}:

\begin{cor}[Density of functions with bounded derivatives]
\label{cor:KufnerPowerWeightBoundedDerivatives}
Let $\sO$ be a domain as in Definition \ref{defn:HestonDomain} such that $\sO$ obeys Hypothesis \ref{hyp:HestonDomainNearGammaZero} and $\partial\sO$ has the strong local Lipschitz property. Then the Banach space $C^\ell(\bar\sO)$ is a dense subset of $H^k(\sO,\fw)$ for $\beta > 0$, $k=0,1,2$, and all integers $\ell\geq k$.
\end{cor}

\begin{proof}
The space $C^\ell(\bar\sO)$ is a dense subset of $H^k(\sO,\fw)$ for $\beta > 0$, $k=0,1,2$, and all integers $\ell\geq k$ because, a fortiori, $C^\infty_0(\bar\sO)$ is a dense subset of $H^k(\sO,\fw)$ by Corollary \ref{cor:KufnerPowerWeight}.
\end{proof}

\begin{rmk}
\label{rmk:SmoothFunctionsDenseInL2}
When $k=0$, Corollaries \ref{cor:KufnerPowerWeight}, \ref{cor:MeyersSerrinWeightedSobolev}, and \ref{cor:KufnerPowerWeightBoundedDerivatives} also follow from Lemma \ref{lem:L2ApproximationLemma} because, a fortiori, $C^\infty_0(\sO)$ is a dense subset of $L^2(\sO,\fw)$.
\end{rmk}

\subsection{Continuous and compact embeddings}
\label{subsec:WeightedSobolevContinuousCompactEmbeddings}
The continuous Sobolev embedding theorem \cite{Adams} does not hold in general for weighted Sobolev spaces on unbounded domains and partial results such as \cite[Theorem 18.11]{KufnerOpic} are not applicable to our choice of weights in Definition \ref{defn:H1WeightedSobolevSpaces}. However, we have the following partial analogue of the standard Sobolev embedding theorem \cite{Adams}.

\begin{lem}[Sobolev embedding]
\label{lem:H2SobolevEmbedding}
Let $\sO$ be a domain as in Definition \ref{defn:HestonDomain} such that $\partial\sO$ obeys the uniform interior cone condition. If $0\leq\alpha<1$, then $H^2(\sO,\fw)\subset C^\alpha_{\textrm{loc}}(\sO\cup\Gamma_1)$.
\end{lem}

\begin{proof}
Observe that $H^2(\sO,\fw)\subset H^2_{\textrm{loc}}(\sO)$. Let $\sO'\subset\sO$ be a bounded domain such that $\bar\sO'\subset\sO\cup\Gamma_1$ and $\sO'$ obeys the uniform interior cone condition. If $u \in H^2(\sO,\fw)$, then $u \in H^2(\sO')$ and thus $u \in C^\alpha(\bar\sO'), 0\leq\alpha<1$ by \cite[Theorem 5.4, Part II ($\textrm{C}''$)]{Adams}. Therefore, $u \in C^\alpha_{\textrm{loc}}(\sO\cup\Gamma_1)$.
\end{proof}

The Rellich-Kondrachov compact embedding theorem does hold in general for weighted Sobolev spaces or unbounded domains \cite{Adams}, \cite{KufnerOpic} and so we avoid its use in this article. When restrictions are added one has partial results, for example:

\begin{thm}[Compact embedding for weighted Sobelev spaces]
\label{thm:RellichBoundedDomainWeightedSobolevSpace}
Let $\sO$ be a domain as in Definition \ref{defn:HestonDomain} and require, in addition, that $\sO$ is \emph{bounded} and has the uniform cone property.
Then the embeddings
$$
H^1(\sO,\fw) \to L^2(\sO,\fw) \quad\hbox{and}\quad H^1_0(\sO,\fw) \to L^2(\sO,\fw),
$$
are compact if and only if $\beta>1$ and, respectively, $\beta\neq 1$.
\end{thm}

\begin{proof}
Because of Lemma \ref{lem:KufnerPowerWeightedSobolevSpaceEquivalents} and the fact that $ L^2(\sO,\fw) =  L^2(\sO,y^{\beta-1}dy)$ for bounded $\sO$, the result is a special case of \cite[Theorem 19.17]{KufnerOpic} (see \cite[\S 15.6 \& 16.10]{KufnerOpic} for their definitions of weighted Sobolev spaces).
\end{proof}

See Antoci \cite{Antoci_2003} and Goldshtein and Ukhlov \cite{Goldshtein_Ukhlov_2009} for more recent results and surveys. Rather than the usual Rellich-Kondrachov compact embedding theorem, we instead make use of the following folk theorem in compactness arguments.

\begin{thm}[Weak $L^\infty$ convergence implies strong $L^p$ convergence]
\label{thm:WeakLInfinityIsStrongLp}
\cite{Khurana_1976, Moreira_Teixeira_2004, Zolezzi_1975}
Let $(\Omega,\Sigma,\nu)$ be a positive, totally finite measure space and let $\{u_n\}_{n=1}^\infty$ be a sequence which converges weakly to $u$ in $L^\infty(\Omega,\Sigma,\nu)$, that is $u_n\rightharpoonup u$ in $L^\infty(\Omega,\Sigma,\nu)$. Then $\{u_n\}_{n=1}^\infty$ converges strongly to $u$ in $L^p(\Omega,\Sigma,\nu)$, that is $u_n\to u$ in $L^p(\Omega,\Sigma,\mu)$, for $1\leq p<\infty$.
\end{thm}

\begin{cor}[Weak $L^\infty$ convergence implies strong $L^p$ convergence]
\label{cor:WeakLInfinityIsStrongLp}
Let $\{u_n\}_{n=1}^\infty$ be a sequence which converges weakly to $u$ in $L^\infty(\sO,\fw)$, that is $u_n\rightharpoonup u$ in $L^\infty(\sO,\fw)$. Then $\{u_n\}_{n=1}^\infty$ converges strongly to $u$ in $L^p(\sO,\fw)$, that is $u_n\to u$ in $L^p(\sO,\fw)$, for $1\leq p<\infty$.
\end{cor}

\begin{proof}
Set $(\Omega,\Sigma,\nu) = (\sO,\sB(\sO),\fw\,dxdy)$ and recall from \eqref{eq:HestonWeight} that
$$
\fw(x,y) = y^{\beta-1} e^{-\gamma|x|-\mu y}, \quad (x,y) \in \HH.
$$
Hence,
$$
\nu(\sO) = \int_\sO \fw(x,y)\,dxdy < \infty,
$$
and the result follows from Theorem \ref{thm:WeakLInfinityIsStrongLp}.
\end{proof}

We can strengthen the preceding corollary to obtain another useful compactness result which we employ in proofs of several key results.

\begin{cor}[Existence of convergent subsequences]
\label{cor:WeakUpperLowerBoundsIsStrongLp}
Let $1\leq r < q < \infty$ and suppose $M \in L^q(\sO,\fw)$ and $M > 0$ a.e. on $\sO$. Let $\{u_n\}_{n=1}^\infty \subset L^r(\sO,\fw)$ be a sequence such that
\begin{equation}
\label{eq:SequencePointwiseBound}
|u_n| \leq M \quad\hbox{a.e. on } \sO, n \geq 1.
\end{equation}
Then there is a subsequence, relabeled as $\{u_n\}_{n=1}^\infty$, which converges strongly in $L^r(\sO,\fw)$ to a limit $u \in L^r(\sO,\fw)$.
\end{cor}

\begin{proof}
By \eqref{eq:SequencePointwiseBound}, we have
$$
|u_n|/M \leq 1 \quad\hbox{a.e. on } \sO, \forall n\geq 1,
$$
Define $r<p<\infty$ by $1/r = 1/p + 1/q$. Therefore, Corollary \ref{cor:WeakLInfinityIsStrongLp} implies that, after passing to a subsequence,
$$
u_n/M \to \tilde u \quad\hbox{strongly in } L^p(\sO,\fw),
$$
for some $\tilde u  \in L^p(\sO,\fw)$. Set $u := M\tilde u$ and observe that $u \in L^r(\sO,\fw)$, since $\|M\tilde u\|_{L^r(\sO,\fw)} \leq \|M\|_{L^q(\sO,\fw)}\|\tilde u\|_{L^p(\sO,\fw)}$. But
$$
\|u_n-u\|_{L^r(\fw,\sO)} \leq \|(u_n-u)/M\|_{L^p(\fw,\sO)}\|M\|_{L^q(\fw,\sO)} \quad\forall n\geq 1,
$$
and consequently, because $\|M\|_{L^q(\fw,\sO)} < \infty$,
$$
u_n \to u \quad\hbox{strongly in }L^r(\sO,\fw),
$$
as desired.
\end{proof}

We shall also need the following well-known convergence results (see, for example, \cite{Billingsley_1995}, or \cite{Labutin}):

\begin{thm}[Convergence in measure implies convergence pointwise a.e.]
\label{thm:Billingsley}
Let $(\Omega,\Sigma,\nu)$ be a complete measurable space with $\nu(\Omega)<\infty$. Let $\{f_n\}_{n\geq 1}$ be a sequence of measurable functions such that $f_n\to f$ in measure as $n\to\infty$. Then there exists a subsequence, relabeled as $\{f_n\}_{n\geq 1}$, such that $f_n\to f$ pointwise $\nu$-a.e. on $\Omega$.
\end{thm}

\begin{cor}[Convergence in $L^p$ implies convergence pointwise a.e.]
\label{cor:Billingsley}
Let $(\Omega,\Sigma,\nu)$ be a complete measurable space and let $1\leq p<\infty$. If $\{f_n\}_{n\geq 1}$ is a sequence of measurable functions such that $f_n\to f$ in $L^p(\Omega,\nu)$ as $n\to\infty$, then $f_n\to f$ in measure as $n\to\infty$; if in addition $\nu(\Omega)<\infty$, then after passing to a subsequence, $f_n\to f$ point wise $\nu$-a.e. on $\Omega$.
\end{cor}

\begin{proof}
For any $\eps>0$, then
\begin{align*}
\nu\left(\left\{P\in\Omega:|f(P)-f_n(P)|>\eps\right\}\right) &= \int_\Omega 1_{\{|f-f_n|>\eps\}}\,d\nu
\\
&\leq \frac{1}{\eps^p}\int_\Omega |f-f_n|^p\,d\nu,
\end{align*}
and since $\lim_{n\to\infty}\|f-f_n\|_{L^p(\Omega,\nu)}\to 0$, then
$$
\lim_{n\to\infty} \nu\left(\left\{P\in\Omega:|f(P)-f_n(P)|>\eps\right\}\right) = 0,
$$
as desired. Theorem \ref{thm:Billingsley} yields the remaining conclusion.
\end{proof}

\subsection{Boundary trace operators}
We recall the following special case ($p=2$) of the definition of extension operator in \cite[\S 4.24]{Adams}.

\begin{defn}[Extension operator]
\label{defn:ExtensionOperator}
For a domain $\sU\subset\RR^d$ (with $d\geq 2$) and an integer $k\geq 1$, we call a bounded linear map $E:H^k(\sU)\to H^k(\RR^d)$ a \emph{simple $k$-extension operator for $\sU$} if $Eu = u$ a.e. on $\sU$ and $\|Eu\|_{H^k(\RR^d)} \leq K\|u\|_{H^k(\sU)}$ for some constant $K>0$ depending only on $\sU$ and $k$. The operator $E$ is called a \emph{strong $k$-extension operator for $\sU$} if for each $0\leq \ell\leq k$, the restriction of $E$ to $H^\ell(\sU)$ is a simple $\ell$-extension operator for $\sU$.
\end{defn}

According to \cite[Theorem 4.26]{Adams} if $\sU$ is a half-space or $\sU$ has the uniform $C^k$-regularity property with bounded boundary $\partial\sU$, then $\sU$ has a strong $k$-extension operator; furthermore, from \cite[\S 4.29]{Adams} the domain $\sU$ can have a strong $k$-extension operator even when $\partial\sU$ is not bounded provided $\sU$ is regular enough that \cite[Equation (4.28)]{Adams} holds, a relatively mild requirement. In our applications, we shall often use a variant of Definition \ref{defn:ExtensionOperator} where $\RR^d$ is replaced by the half-space $\HH\subset\RR^d$.

The following result is a special case ($p=2$) of \cite[Theorem 5.22]{Adams} (and extension of \cite[Theorem 5.5.1]{Evans} in the case $k=1$ and $\sU$ is bounded):

\begin{thm}[Boundary trace theorem]
\label{thm:BoundaryTrace}
Let $\sU\subset\RR^d$ (with $d\geq 2$) be a domain having the uniform $C^k$-regularity property $(k\geq 1$) and suppose there is a simple $k$-extension operator $E$ for $\sU$. Then there exists a bounded linear operator
$$
T_{\partial\sU}:H^k(\sU) \to L^2(\partial\sU),
$$
such that $T_{\partial\sU}u = u|_{\partial\sU}$ when $u \in H^k(U)\cap C_{\textrm{loc}}(\bar\sU)$ and
$$
\|Tu\|_{L^2(\partial\sU)} \leq C\|u\|_{H^k(\sU)},
$$
for some constant $C$ depending only on $\sU$.
\end{thm}

We have the following analogue of Theorem \ref{thm:BoundaryTrace} for our weighted Sobelev spaces.

\begin{lem}[$\Gamma_1$-boundary traces of functions in weighted Sobolev spaces]
\label{lem:Traces}
Let $\sO$ be as in Definition \ref{defn:HestonDomain} and suppose that, for an integer $k\geq 1$, \emph{(i)} $\Gamma_1$ has the uniform $C^k$-regularity property, \emph{(ii)} there is a simple $k$-extension operator from $\sO$ to $\HH$, and \emph{(iii)} the intersection of $\bar\Gamma_1$ with $\Gamma_0$ is $C^k$-transverse \footnote{Hypothesis \ref{hyp:HestonDomainNearGammaZero} implies that the intersection of $\bar\Gamma_1$ with $\Gamma_0$ is $C^\infty$-transverse.}, that is, $\bar\Gamma_1\pitchfork\Gamma_0$. Then there are bounded linear operators,
\begin{align}
T_{\Gamma_1\cap U}^1 {}&: H^1(\sO,\fw)\to L^2(\Gamma_1\cap U,\fw) \quad\hbox{when $k=1$} ,
\\
T_{\Gamma_1}^1 {}&: H^1(\sO,\fw)\to L^2(\Gamma_1,\fw) \quad\hbox{$k=1$} ,
\\
T_{\Gamma_1}^2 {}&: H^2(\sO,\fw)\to H^1(\Gamma_1,\fw) \quad\hbox{$k=2$},
\end{align}
for $U := \RR\times(\delta,\infty)$ and any $\delta>0$, such that
\begin{align*}
T_{\Gamma_1\cap U}^1u &= u|_{\Gamma_1\cap U} \quad\hbox{when } u\in C_{\textrm{loc}}(\sO\cup\Gamma_1),
\\
T_{\Gamma_1}^1u &= y^{1/2}u|_{\Gamma_1} \quad\hbox{when } u\in C_{\textrm{loc}}(\sO\cup\Gamma_1),
\\
T_{\Gamma_1}^2u &= u|_{\Gamma_1} \quad\hbox{when } u\in C^1_{\textrm{loc}}(\sO\cup\Gamma_1).
\end{align*}
\end{lem}

\begin{proof}
We shall provide a detailed proof in the case where the domain $\sO$ is a quadrant, that is $\sO = (x_0,\infty)\times (0,\infty)$, so $\Gamma_0 = (x_0,\infty)\times\{0\}$ and $\Gamma_1 = \{x_0\}\times(0,\infty)$. The general case for a domain $\sO$ permitted by Definition \ref{defn:HestonDomain} follows using a partition of unity and boundary-straightening via $C^k$-diffeomorphisms of $\bar\HH$ as in the proofs of the extension and trace theorems for standard  Sobolev spaces \cite[Theorems 4.26 \& 5.22]{Adams} or \cite[Theorems 5.4.1 \& 5.5.1]{Evans}. The additional regularity conditions on $\Gamma_1$ are required for these boundary straightening arguments.

\emph{The trace operator\/} $T_{\Gamma_1\cap U}^1: H^1(\sO,\fw)\to L^2(\Gamma_1\cap U,\fw)$. We first assume $u \in C^1(\bar\sO)$. Let $\zeta \in C^\infty(\bar\HH)$ with $0\leq \zeta\leq 1$ on $\HH$, $\zeta=1$ on $\RR\times[\delta,\infty)$, $\zeta = 0$ on $\RR\times [0,\delta/2]$, and $\zeta(x,y)=\zeta(y), \forall(x,y)\in\HH$ (that is, $\zeta$ is constant with respect to $x\in\RR$). Therefore, noting that $dS = dy$ along $\Gamma_1$,
\begin{align*}
\|T_{\Gamma_1\cap U}^1u\|_{L^2(\Gamma_1\cap U,\fw)}^2 &= \int_{\Gamma_1\cap U} u^2\fw\,dS
\leq \int_{\Gamma_1} \zeta u^2\fw\,dy
\\
&= -\int_\sO \left(\zeta u^2 y^{\beta-1}e^{-\gamma|x|-\mu y}\right)_x\,dxdy \quad\hbox{(integration by parts)}
\\
&= -\int_\sO \left(\zeta_xu^2 + 2\zeta uu_x -\gamma\sign(x)\zeta u^2\right)y^{\beta-1}e^{-\gamma|x|-\mu y}\,dxdy
\\
&\leq C\int_\sO (\zeta u_x^2 + (\zeta + |\zeta_x|)u^2)\fw\,dxdy
\\
&\leq C\int_\sO (yu_x^2 + (1+y)u^2)\fw\,dxdy
\\
&\leq C\|u\|_{H^1(\sO,\fw)}^2 \quad\hbox{(by Definition \ref{defn:H1WeightedSobolevSpaces})},
\end{align*}
where $C=C(\delta,\gamma)$. Hence,
$$
\|T_{\Gamma_1\cap U}^1u\|_{L^2(\Gamma_1\cap U,\fw)} \leq C\|u\|_{H^1(\sO,\fw)}, \quad\forall u \in C^1(\bar\sO).
$$
Thus, recalling that $C^1(\bar\sO)$ is a dense subset of $H^1(\sO,\fw)$ by Corollary \ref{cor:KufnerPowerWeightBoundedDerivatives}, if $\{u_n\}_{n\geq 1} \subset C^1(\bar\sO)$ is any sequence converging in $H^1(\sO,\fw)$ to $u\in H^1(\sO,\fw)$, the sequence $\{u_n|_{\Gamma_1\cap U}\}_{n\geq 1}$ is Cauchy in $L^2(\Gamma_1\cap U,\fw)$ and converges to a function $T_{\Gamma_1\cap U}^1u \in L^2(\Gamma_1\cap U,\fw)$.

Finally, if $u\in H^1(\sO,\fw) \cap C_{\textrm{loc}}(\sO\cup\Gamma_1)$, then (as in the proof of \cite[Theorem 5.5.1]{Evans}), we may appeal to the fact that the sequence $\{u_n\}_{n\geq 1} \subset C^1(\bar\sO)$ constructed in the proof of Corollary \ref{cor:KufnerPowerWeightBoundedDerivatives} (specifically, Theorem \ref{thm:KufnerPowerWeight}) converges uniformly on compact subsets of $\sO\cup\Gamma_1$ to $u$ on $\sO\cup\Gamma_1$. Hence, $T_{\Gamma_1\cap U}^1u = u|_{\Gamma_1\cap U}$ if $u\in H^1(\sO,\fw) \cap C_{\textrm{loc}}(\sO\cup\Gamma_1)$.

\emph{The trace operator\/} $T_{\Gamma_1}^1: H^1(\sO,\fw)\to L^2(\Gamma_1,\fw)$. As in the case of $T_{\Gamma_1\cap U}^1$, for $u \in C^1(\bar\sO)$ we have
\begin{align*}
\|T_{\Gamma_1}^1u\|_{L^2(\Gamma_1,\fw)}^2 &= \int_{\Gamma_1} yu^2\fw\,dS = \int_{\Gamma_1} yu^2\fw\,dy
\\
&= -\int_\sO \left(u^2 y^\beta e^{-\gamma|x|-\mu y}\right)_x\,dxdy \quad\hbox{(integration by parts)}
\\
&= -\int_\sO \left(2uu_x -\gamma\sign(x)u^2\right)y^\beta e^{-\gamma|x|-\mu y}\,dxdy
\\
&\leq C\int_\sO (u_x^2 + u^2)y\fw\,dxdy
\\
&\leq C\int_\sO (y|Du|^2 + (1+y)u^2)\fw\,dxdy
\\
&= C\|u\|_{H^1(\sO,\fw)}^2 \quad\hbox{(by Definition \ref{defn:H1WeightedSobolevSpaces})},
\end{align*}
where $C=C(\gamma)$. Hence,
$$
\|T_{\Gamma_1}^1u\|_{L^2(\Gamma_1,\fw)} \leq C\|u\|_{H^1(\sO,\fw)}, \quad\forall u \in C^1(\bar\sO).
$$
The conclusion follows just as in the case of $T_{\Gamma_1\cap U}^1$.

\emph{The trace operator\/} $T_{\Gamma_1}^2: H^2(\sO,\fw)\to H^1(\Gamma_1,\fw)$. We first assume $u \in C^2(\bar\sO)$. Integration by parts gives
\begin{align*}
\int_{\Gamma_1} y^2(u_x^2 + u_y^2)\fw\,dS &= -\int_\sO \left((u_x^2 + u_y^2)y^{\beta+1} e^{-\gamma|x|-\mu y}\right)_x\,dxdy
\\
&= -\int_\sO \left(2u_xu_{xx} + 2u_yu_{yx} -\gamma\sign(x)(u_x^2 + u_y^2)\right)y^{\beta+1} e^{-\gamma|x|-\mu y}\,dxdy
\\
&\leq C\int_\sO y^2\left(u_{xx}^2 + u_{yx}^2 + u_x^2 + u_y^2\right)\fw\,dxdy
\\
&\leq C\|u\|_{H^2(\sO,\fw)}^2 \quad\hbox{(by Definition \ref{defn:H2WeightedSobolevSpaces})},
\end{align*}
where $C=C(\gamma)$. Similarly,
$$
\int_{\Gamma_1} yu^2\fw\,dy \leq C\int_\sO y(u_x^2 + u^2)\fw\,dxdy,
$$
for a universal constant $C$. Therefore, by Definition \ref{defn:H2WeightedSobolevSpaces},
$$
\int_{\Gamma_1} \left(y^2(u_x^2 + u_y^2) + (1+y)u^2\right)\fw\,dy \leq C\|u\|_{H^2(\sO,\fw)}^2,
$$
where $C=C(\gamma)$, and we obtain
$$
\|T_{\Gamma_1}^2u\|_{H^1(\Gamma_1,\fw)} \leq C\|u\|_{H^2(\sO,\fw)}, \quad\forall u \in C^2(\bar\sO).
$$
Thus, recalling that $C^2(\bar\sO)$ is a dense subset of $H^2(\sO,\fw)$ by Corollary \ref{cor:KufnerPowerWeightBoundedDerivatives}, if $\{u_n\}_{n\geq 1} \subset C^2(\bar\sO)$ is any sequence converging in $H^2(\sO,\fw)$ to $u\in H^2(\sO,\fw)$, the sequence $\{u_n|_{\Gamma_1}\}_{n\geq 1}$ is Cauchy in $H^1(\Gamma_1,\fw)$ and converges to a function $T_{\Gamma_1}^2u \in H^1(\Gamma_1,\fw)$.

Finally, if $u\in H^2(\sO,\fw) \cap C^1_{\textrm{loc}}(\sO\cup\Gamma_1)$, then we may appeal to the fact that the sequence $\{u_n\}_{n\geq 1} \subset C^2(\bar\sO)$ constructed in the proof of Corollary \ref{cor:KufnerPowerWeightBoundedDerivatives} (specifically, Theorem \ref{thm:KufnerPowerWeight}), together with its sequences of first-order derivatives, converge uniformly on compact subsets of $\sO\cup\Gamma_1$, to $u$ and its first-order derivatives on $\sO\cup\Gamma_1$. Hence, $T_{\Gamma_1}^2u = u|_{\Gamma_1}$ if $u\in H^2(\sO,\fw) \cap C^1_{\textrm{loc}}(\sO\cup\Gamma_1)$. This completes the proof.
\end{proof}

Defining trace operators for the boundary portion $\Gamma_0$ requires more care:

\begin{lem}[$\Gamma_0$-boundary traces of functions in weighted Sobolev spaces]
\label{lem:TracesGammaZero}
Let $\sO$ be as in Definition \ref{defn:HestonDomain} and require that $\sO$ obeys\footnote{It is enough that $\Gamma_1$ is $C^k$ in a neighborhood of $\Gamma_0$ and the intersection of $\bar\Gamma_1$ with $\Gamma_0$ is $C^k$-transverse.} Hypothesis \ref{hyp:HestonDomainNearGammaZero}. Then there are bounded linear operators,
\begin{align}
\label{eq:TracesGammaZerou}
{}&T_{\Gamma_0}^{2,0}: H^2(\sO,\fw)\to L^2(\Gamma_0,e^{-\gamma|x|}dx), \quad 0<\beta<1,
\\
\label{eq:TracesGammaZeroybetaDu}
{}&T_{\Gamma_0}^{2,1}: H^2(\sO,\fw)\to L^2(\Gamma_0,e^{-\gamma|x|}dx;\RR^2), \quad \beta>0,
\end{align}
such that
\begin{align*}
T_{\Gamma_0}^{2,0}u &= u|_{\Gamma_0} \hbox{ when } u\in C_{\textrm{loc}}(\sO\cup\Gamma_0),
\\
T_{\Gamma_0}^{2,1}u &= y^\beta Du|_{\Gamma_0} \hbox{ when } y^\beta Du\in C_{\textrm{loc}}(\sO\cup\Gamma_0;\RR^2).
\end{align*}
\end{lem}

\begin{rmk}[Existence of $\Gamma_0$ boundary trace operators]
\label{rmk:WelldefinedTraceZeroOperator}
When $0<\beta<1$, \cite[Theorem 9.15 \& Equation (9.42)]{Kufner} suggests that a trace operator $T_{\Gamma_0}^{1,0}:H^1(\sO,\fw) \to L^2(\sO,e^{-\gamma|x|}dx)$ is well-defined, but we shall not need this refinement. However, when $\beta\geq 1$, $\Gamma_0$-traces of functions in $H^1(\sO,\fw)$ generally will not exist \cite[Examples 9.16 \& 9.17]{Kufner}.
\end{rmk}

\begin{rmk}[Domain of the boundary trace operator $T_{\Gamma_0}^{2,0}$]
\label{rmk:DomainT20TraceZeroOperator}
Using Hardy's inequality (Theorem \ref{thm:HardyInequality}), one can show when $0<\beta<1$ that $T_{\Gamma_0}^{2,0}$ extends to a bounded linear map from $H^1_0(\sO,\fw)$ or even $H^1_0(\sO\cup\Gamma_1,\fw)$ to $L^2(\Gamma_0,e^{-\gamma|x|}dx)$, but not from $H^1_0(\sO,\fw)$ or $H^1_0(\sO\cup\Gamma_0,\fw)$ to $L^2(\Gamma_0,e^{-\gamma|x|}dx)$, since the proof of Lemma \ref{lem:TracesGammaZero} shows that one would need $u=0$ on $\Gamma_0$ (trace sense) when applying Hardy's inequality in this situation.
\end{rmk}

\begin{proof}[Proof of Lemma \ref{lem:TracesGammaZero}]
As in the proof of Lemma \ref{lem:Traces}, we provide a detailed argument when $\sO = (x_0,\infty)\times (0,\infty)$, so $\Gamma_0 = (x_0,\infty)\times\{0\}$, with the general case for $\sO$ following with the aid of a partition of unity and boundary-straightening.

\emph{The trace operator\/} $T_{\Gamma_0}^{2,0}: H^2(\sO,\fw)\to L^2(\Gamma_0,e^{-\gamma|x|}dx)$ ($0<\beta<1$). We first suppose $u \in C^2(\bar\sO)$. Let $\zeta\in C^\infty(\bar\HH)$ with $0\leq \zeta \leq 1$ on $\HH$, $\zeta = 1$ on $\RR\times[0,1/2]$ and $\zeta = 0$ on $\RR\times[1,\infty)$. Then, noting that $dS = dx$ along $\Gamma_0$,
\begin{align*}
\int_{\Gamma_0} u^2 e^{-\gamma|x|}\,dS &= \int_{\Gamma_0} \zeta u^2 e^{-\gamma|x|}\,dx
\\
&= -\int_\sO (\zeta u^2)_y e^{-\gamma|x|}\,dxdy \quad\hbox{(integration by parts)}
\\
&= -\int_\sO (\zeta_yu^2 + 2\zeta uu_y) e^{-\gamma|x|}\,dxdy
\\
&\leq \int_\sO (|\zeta_y|u^2 + 2\zeta|u||u_y|) e^{-\gamma|x|}\,dxdy
\\
&\leq e^\mu\int_\sO (|\zeta_y|u^2 + 2\zeta|u||u_y|) y^{\beta-1}e^{-\gamma|x|-\mu y}\,dxdy \quad\hbox{(because $0<\beta<1$)}
\\
&\leq C\int_\sO (u_y^2 + u^2)y^{\beta-1}e^{-\gamma|x|-\mu y}\,dxdy
\\
&\leq C\int_\sO (|Du|^2 + u^2)\fw\,dxdy,
\end{align*}
for $C=C(\mu)$, and thus, by Definition \ref{defn:H2WeightedSobolevSpaces},
\begin{equation}
\label{eq:TracesGammaZeroU}
\|u\|_{L^2(\Gamma_0,e^{-\gamma|x|}dx)} \leq C\|u\|_{H^2(\sO,\fw)},  \quad\forall u \in C^2(\bar\sO).
\end{equation}
Thus, recalling that $C^2(\bar\sO)$ is a dense subset of $H^2(\sO,\fw)$ by Corollary \ref{cor:KufnerPowerWeightBoundedDerivatives}, if $\{u_n\}_{n\geq 1} \subset C^2(\bar\sO)$ is any sequence converging in $H^2(\sO,\fw)$ to $u\in H^2(\sO,\fw)$, the sequence $\{u_n|_{\Gamma_0}\}_{n\geq 1}$ is Cauchy in $L^2(\Gamma_0,e^{-\gamma|x|}dx)$ and converges to a function $T_{\Gamma_0}^{2,0}u \in L^2(\Gamma_0,e^{-\gamma|x|}dx)$ when $0<\beta<1$.

Finally, if $u\in H^2(\sO,\fw) \cap C_{\textrm{loc}}(\sO\cup\Gamma_0)$, then, we may appeal to the fact that the sequence $\{u_n\}_{n\geq 1} \subset C^2(\bar\sO)$ constructed in the proof of Corollary \ref{cor:KufnerPowerWeightBoundedDerivatives} (specifically, Theorem \ref{thm:KufnerPowerWeight}) converges uniformly on compact subsets of $\sO\cup\Gamma_0$ to $u$ on $\sO\cup\Gamma_0$. Hence, $T_{\Gamma_0}^{2,0}u = u|_{\Gamma_0}$ if $u\in H^2(\sO,\fw) \cap C_{\textrm{loc}}(\sO\cup\Gamma_0)$.

\emph{The trace operator\/} $T_{\Gamma_0}^{2,1}: H^2(\sO,\fw)\to L^2(\Gamma_0,e^{-\gamma|x|}dx)$ ($\beta>0$). Similarly, if $u\in C^2(\bar\sO)$ but allowing any $\beta>0$, we obtain
\begin{align*}
\int_{\Gamma_0} y^{2\beta}|Du|^2 e^{-\gamma|x|}\,dS &= \int_{\Gamma_0} \zeta y^{2\beta}|Du|^2 e^{-\gamma|x|}\,dx
\\
&= -\int_\sO (\zeta y^{2\beta}|Du|^2)_y e^{-\gamma|x|}\,dxdy  \quad\hbox{(integration by parts)}
\\
&= -\int_\sO \left(y\zeta_y|Du|^2 + 2y\zeta(u_xu_{xy} + u_yu_{yy}) + 2\zeta\beta|Du|^2\right) y^{2\beta-1}e^{-\gamma|x|}\,dxdy
\\
&\leq \int_\sO \left((y|\zeta_y| + 2\zeta\beta)|Du|^2 + 2y\zeta|Du||D^2u|\right) y^{\beta-1}e^{-\gamma|x|}\,dxdy
\\
&\quad\hbox{(since $\supp\zeta\subset[0,1]$ and $\beta>0$)}
\\
&\leq e^\mu\int_\sO \left(y|\zeta_y| + 2\zeta\beta + \zeta^2)|Du|^2 + y^2|D^2u|^2\right) y^{\beta-1}e^{-\gamma|x|-\mu y}\,dxdy
\\
&\leq C\int_\sO (y^2|D^2u|^2 + (1+y)|Du|^2)y^{\beta-1}e^{-\gamma|x|-\mu y}\,dxdy
\\
&\leq C\int_\sO (y^2|D^2u|^2 + (1+y)^2|Du|^2 + (1+y)u^2)\fw\,dxdy,
\end{align*}
for $C=C(\mu)$, and thus, by Definition \ref{defn:H2WeightedSobolevSpaces},
$$
\|y^\beta Du\|_{L^2(\Gamma_0,e^{-\gamma|x|}dx)} \leq C\|u\|_{H^2(\sO,\fw)},   \quad\forall u \in C^2(\bar\sO).
$$
Thus, recalling that $C^2(\bar\sO)$ is a dense subset of $H^2(\sO,\fw)$ by Corollary \ref{cor:KufnerPowerWeightBoundedDerivatives}, if $\{u_n\}_{n\geq 1} \subset C^2(\bar\sO)$ is any sequence converging in $H^2(\sO,\fw)$ to $u\in H^2(\sO,\fw)$, the sequence of weighted gradients $\{y^\beta Du_n|_{\Gamma_0}\}_{n\geq 1}$ is Cauchy in $L^2(\Gamma_0,e^{-\gamma|x|}dx,\RR^2)$ and converges to a limit $T_{\Gamma_0}^{2,1}u \in L^2(\Gamma_0,e^{-\gamma|x|}dx,\RR^2$ when $\beta>0$.

Finally, if $u\in H^2(\sO,\fw)$ and $y^\beta Du \in C_{\textrm{loc}}(\sO\cup\Gamma_0;\RR^2)$, then, we may appeal to the fact that the sequence $\{u_n\}_{n\geq 1} \subset C^\infty(\bar\sO)$ constructed in the proof of Corollary \ref{cor:KufnerPowerWeightBoundedDerivatives} (specifically, Theorem \ref{thm:KufnerPowerWeight}) has the property that $y^\beta Du_n$ converges uniformly on compact subsets of $\sO\cup\Gamma_0$ to $y^\beta Du \in C_{\textrm{loc}}(\sO\cup\Gamma_0;\RR^2)$. Hence, $T_{\Gamma_0}^{2,1}u = y^\beta Du|_{\Gamma_0}$ if $u\in H^2(\sO,\fw)$ and $y^\beta Du \in C_{\textrm{loc}}(\sO\cup\Gamma_0;\RR^2)$.
\end{proof}

We have the following analogue of \cite[Theorem 5.5.2]{Evans} for the $\Gamma_0$ boundary portion of $\partial\sO$:

\begin{lem}[$\Gamma_0$-trace zero functions in $H^1(\sO,\fw)$]
\label{lem:PartialEvansTraceZero}
Let $\sO$ be as in Definition \ref{defn:HestonDomain} and require that $\sO$ obeys Hypothesis \ref{hyp:HestonDomainNearGammaZero}. When $0<\beta<1$ and $u \in H^2(\sO,\fw)$,
$$
u \in H^1_0(\sO\cup\Gamma_1,\fw) \quad\hbox{if and only if}\quad T_{\Gamma_0}^{2,0}u = 0 \hbox{ a.e. on $\Gamma_0$},
$$
\end{lem}

\begin{rmk}
\label{rmk:PartialEvansTraceZero}
The result is false when $\beta\geq 1$. For example, when $\sO=\HH$, the constant function, $u = 1$, is in $H^1_0(\HH,\fw)$ by Lemma \ref{lem:ImprovedH1ApproximationLemma}.
\end{rmk}

\begin{proof}[Proof of Lemma \ref{lem:PartialEvansTraceZero}]
Suppose first that $u \in H^1_0(\sO\cup\Gamma_1,\fw)$. By the Definition \ref{defn:H1WeightedSobolevSpaces} of $H^1_0(\sO\cup\Gamma_1,\fw)$ and the Definition \ref{defn:H2WeightedSobolevSpaces} of $H^2(\sO,\fw)$, there exists $\{u_m\}_{m\geq 0}\subset C^2_0(\sO\cup\Gamma_1)$ such that $u_m\to u$ in $H^2(\sO,\fw)$. As the trace operator, $T_{\Gamma_0}^{2,0}:H^2(\sO,\fw) \to L^2(\Gamma_0,e^{-\gamma|x|}dx)$, is bounded by Lemma \ref{lem:TracesGammaZero} and $T_{\Gamma_0}^{2,0}u_m=0$ on $\Gamma_0$ for all $m\geq 0$, we deduce that $T_{\Gamma_0}^{2,0}u=0$ a.e. on $\Gamma_0$, as desired.

Conversely, suppose that $T_{\Gamma_0}^{2,0}u=0$ a.e. on $\Gamma_0$. Hence, by the proof of Lemma \ref{lem:TracesGammaZero}, there exist $\{u_m\}_{m\geq 0}\subset C^2(\bar\sO)$ such that, as $m\to\infty$,
\begin{gather}
\label{eq:umH2sequence}
u_m\to u\quad\hbox{in }H^2(\sO,\fw),
\\
\label{eq:traceumH2sequence}
T_{\Gamma_0}^{2,0}u_m = u_m|_{\Gamma_0} \to 0\quad\hbox{in }L^2(\Gamma_0,e^{-\gamma|x|}dx).
\end{gather}
It suffices to consider the case $\sO = \HH$, so $\Gamma_0=\RR\times\{0\}$ and $\Gamma_1 = \emptyset$, since the general case for $\sO$ as in Definition \ref{defn:HestonDomain} follows by standard methods \cite[Proof of Theorem 5.5.2]{Evans}. Then,
$$
u_m(x,y)- u_m(x,0) = \int_0^yu_{m,y}(x,z)\,dz
$$
and
$$
|u_m(x,y)| \leq |u_m(x,0)| + \int_0^y|u_{m,y}(x,z)|\,dz.
$$
Thus,
\begin{align*}
|u_m(x,y)|^2 &\leq 2|u_m(x,0)|^2 + 2\left(\int_0^y|u_{m,y}(x,z)|\,dz\right)^2
\\
&\leq 2|u_m(x,0)|^2 + 2y\int_0^y|u_{m,y}(x,z)|^2\,dz,
\end{align*}
and
\begin{align*}
\int_\RR|u_m(x,y)|^2e^{-\gamma|x|}\,dx &\leq C\left(\int_\RR|u_m(x,0)|^2e^{-\gamma|x|}\,dx\right.
\\
&\quad + \left. y\int_0^y\int_\RR |Du_m(x,z)|^2 e^{-\gamma|x|}\,dxdz\right)
\\
&\leq C\left(\int_\RR|u_m(x,0)|^2e^{-\gamma|x|}\,dx\right.
\\
&\quad + \left. y\int_0^y\int_\RR |Du_m(x,z)|^2 z^{\beta-1}e^{-\gamma|x|-\mu z}\,dxdz\right),
\\
&\quad\hbox{for $0<\beta<1$ and $0\leq y\leq 1$.}
\end{align*}
Letting $m\to\infty$ and applying \eqref{eq:umH2sequence} and \eqref{eq:traceumH2sequence}, we deduce that
\begin{equation}
\label{eq:TraceZeroEstimate}
\begin{aligned}
\int_\RR|u(x,y)|^2e^{-\gamma|x|}\,dx &\leq Cy\int_0^y\int_\RR |Du(x,z)|^2 \fw(x,z)\,dxdz,
\\
&\quad\hbox{for $0<\beta<1$ and $0\leq y\leq 1$.}
\end{aligned}
\end{equation}
Let $\zeta\in C^\infty(\bar\RR)$ satisfy $\zeta = 1$ on $(\infty,1]$, $\zeta = 0$ on $[2,\infty)$, and $0\leq\zeta\leq 1$, and write $\zeta_m(x,y) := \zeta(my), (x,y)\in\HH$, and set
$$
w_m := u(1-\zeta_m), \quad m\geq 1.
$$
Then
$$
w_{m,x} = u_x(1-\zeta_m), \quad w_{m,y} = u_y(1-\zeta_m) - mu\zeta_y(m\cdot).
$$
Consequently,
\begin{align*}
\|u - w_m\|_{H^1(\sO,\fw)}^2
&= \int_\sO \left\{y\left(\zeta_m^2|Du|^2 + 2muu_y\zeta_m\zeta_y(m\cdot) + m^2\zeta_y^2(m\cdot)u^2\right) \right.
\\
&\qquad + \left. (1+y)\zeta_m^2u^2 \right\}\fw(x,y)\,dxdy
\\
&\leq C\int_\sO 1_{\{0 < y < 2/m\}}\left(y|Du|^2 + (1+y)u^2\right)\fw(x,y)\,dxdy
\\
&\quad + Cm^2\int_\sO 1_{\{0 < y < 2/m\}} yu^2\fw(x,y)\,dxdy =: I_m + J_m.
\end{align*}
The integral $I_m$ converges to zero as $m\to\infty$, since
$$
\|u\|_{H^1(\sO,\fw)}^2
=
\int_\sO \left\{y|Du|^2 + (1+y)u^2\right\}\fw(x,y)\,dxdy < \infty.
$$
To estimate the term $J_m$, we apply \eqref{eq:TraceZeroEstimate},
\begin{align*}
J_m &= Cm^2\int_0^{2/m}\int_\RR yu^2\fw(x,y)\,dxdy
\\
&= Cm^2\int_0^{2/m}y^\beta \left(\int_\RR u^2 e^{-\gamma|x|}\,dx\right)e^{-\mu y}dy
\\
&\leq Cm^2\left(\int_0^{2/m}y^{1+\beta}\,dy\right)\left(\int_0^{2/m}\int_\RR|Du(x,z)|^2 z^{\beta-1}e^{-\gamma|x|-\mu z}\,dxdz\right)
\\
&\leq Cm^{-\beta}\int_0^{2/m}\int_\RR|Du|^2\fw(x,z)\,dxdz \to 0\quad\hbox{as }m\to \infty,
\end{align*}
since
$$
\|u\|_{H^2(\sO,\fw)}^2
=
\int_\sO \left\{y^2|Du|^2 + (1+y^2)|Du|^2 + (1+y)u^2\right\}\fw(x,y)\,dxdy < \infty.
$$
Consequently, the term $J_m$ also tends to zero as $m\to\infty$, and we conclude that $w_m\to u$ in $H^1(\sO,\fw)$ as $m\to\infty$. But $w_m = 0$ on $\RR\times[0,1/m)$ and we can therefore mollify the $w_m$ to produce functions $u_m\in C^2_0(\HH)$ such that $u_m\to u$ in $H^1(\sO,\fw)$ as $m\to\infty$ using the method of proofs of \cite[Theorems 5.3.1 \& 5.3.2]{Evans}. Hence, $u\in H^1_0(\sO\cup\Gamma_1,\fw)$.
\end{proof}

We also have the following analogue of \cite[Theorem 5.5.2]{Evans} for the $\Gamma_1$ boundary portion of $\partial\sO$:

\begin{lem}[$\Gamma_1$-trace zero functions in $H^1(\sO,\fw)$]
\label{lem:EvansGamma1TraceZero}
Assume the hypotheses of Lemma \ref{lem:Traces} with $k=1$. For all $\beta>0$ and $u\in H^1(\sO,\fw)$,
$$
u \in H^1_0(\sO\cup\Gamma_0,\fw) \quad\hbox{if and only if}\quad T_{\Gamma_1}^1u = 0 \hbox{ a.e. on $\Gamma_1$}.
$$
\end{lem}

\begin{proof}
If $u \in H^1_0(\sO\cup\Gamma_0,\fw)$ then, just as in the proof of Lemma \ref{lem:PartialEvansTraceZero}, we have $T_{\Gamma_1}^1u = 0$ a.e. on $\Gamma_1$ by now appealing to Lemma \ref{lem:Traces}. Conversely, if $T_{\Gamma_1}^1u = 0$ a.e. on $\Gamma_1$, it is straightforward to adapt the proof of \cite[Theorem 5.5.2]{Evans} (and simpler than for Lemma \ref{lem:PartialEvansTraceZero}) to show that $u \in H^1_0(\sO\cup\Gamma_0,\fw)$.
\end{proof}

Although the boundary properties under discussion in Lemma \ref{lem:ybetaDuZeroTrace} will be superseded in the sequel to this article by verifications that solutions to the problems considered in this article are at least in $C^1_{\textrm{loc}}(\sO\cup\Gamma_0)$, nevertheless they provide insight when coupled with Example \ref{exmp:CIR}. Lemma \ref{lem:TracesGammaZero} shows that for all $z\in [0,\delta_0)$, where $\delta_0$ is as in Hypothesis \ref{hyp:HestonDomainNearGammaZero}, there is a well-defined trace operator
\begin{equation}
\label{eq:ybetaDuTrace}
T_y: H^2(\sO,\fw) \to L^2(\Gamma_0,e^{-\gamma|x|}\,dx),
\end{equation}
such that $T_yu = y^\beta(\rho u_x + \sigma u_y)|_{\Gamma_0\times\{y\}}$ when $y^\beta Du\in C_{\textrm{loc}}(\sO\cup\Gamma_0;\RR^2)$.

\begin{lem}[Equivalence of boundary conditions]
\label{lem:ybetaDuZeroTrace}
Let $\sO$ be as in Definition \ref{defn:HestonDomain} and require that $\sO$ obeys Hypothesis \ref{hyp:HestonDomainNearGammaZero}. Suppose $u \in H^2(\sO,\fw)$. Then there is a positive constant $C = C(\beta,\gamma,\mu,\rho,\sigma)$ such that
\begin{equation}
\label{eq:ybetaDuZeroTrace}
\|T_yu - T_0u\|_{L^1(\Gamma_0,e^{-\gamma|x|}\,dx)} \leq Cy^\beta\|u\|_{H^2(\sO,\fw)}, \quad 0< y\leq \delta_0.
\end{equation}
Moreover, $y^\beta(\rho u_x + \sigma u_y) = 0$ on $\Gamma_0$ (trace sense) if and only if $y^\beta(\rho u_x + \sigma u_y) \to 0$ in $L^1(\Gamma_0,e^{-\gamma|x|}\,dx)$ as $y\downarrow 0$.
\end{lem}

\begin{proof}
Denoting $w := y^\beta(\rho u_x + \sigma u_y)$, then
$$
w(\cdot,z) - w(\cdot,0) = \int_0^z w_y(\cdot,y)\,dy, \quad 0 < z \leq \delta_0.
$$
But $w_y = y^{\beta-1}(\rho u_x + \sigma u_y) + y^\beta(\rho u_{xy} + \sigma u_{yy})$, and thus
\begin{align*}
{}&\int_{\Gamma_0}|w(\cdot,z) - w(\cdot,0)|\,e^{-\gamma|x|}dx
\\
&\quad \leq \int_{\Gamma_0}\int_0^z |w_y(\cdot,y)| e^{-\gamma|x|}\,dydx
\\
&\quad = \int\limits_{\Gamma_0\times(0,z)} |y^{\beta-1}(\rho u_x + \sigma u_y) + y^\beta(\rho u_{xy} + \sigma u_{yy})|e^{-\gamma|x|}\,dxdy
\\
&\quad \leq e^{\mu}\int\limits_{\Gamma_0\times(0,z)} \left(|\rho u_x + \sigma u_y| + y|\rho u_{xy} + \sigma u_{yy}|\right)y^{\beta-1} e^{-\gamma|x|-\mu y}\,dxdy
\quad\hbox{(for $0<z\leq 1$)}
\\
&\quad = e^{\mu}\int_{\sO_z} \left(|\rho u_x + \sigma u_y| + y|\rho u_{xy} + \sigma u_{yy}|\right)\fw\,dxdy
\\
&\quad \leq e^{\mu}\hbox{Vol}^{1/2}(\sO_z,\fw)\left(\int_{\sO_z} \left(|\rho u_x + \sigma u_y| + y|\rho u_{xy} + \sigma u_{yy}|\right)^2\fw\,dxdy\right)^{1/2}
\\
&\quad \leq C\hbox{Vol}^{1/2}(\sO_z,\fw)\left(\|Du\|_{L^2(\sO_z,\fw)} + \|yD^2u\|_{L^2(\sO_z,\fw)}\right)
\\
&\quad \leq C\hbox{Vol}^{1/2}(\sO_z,\fw)\|u\|_{H^2(\sO,\fw)},
\end{align*}
where $\sO_z := \Gamma_0 \times (0,z)$ and $C = C(\mu,\rho,\sigma)$. Moreover,
\begin{align*}
\hbox{Vol}(\sO_z,\fw) &= \int\limits_{\Gamma_0\times(0,z)} y^{\beta-1} e^{-\gamma|x|-\mu y}\,dxdy
\\
&\leq 2\int_0^\infty e^{-\gamma x}\,dx \int_0^z y^{\beta-1}\,dy
\\
&= \frac{2}{\beta\gamma}z^\beta.
\end{align*}
Combining these observations gives \eqref{eq:ybetaDuZeroTrace}.

Suppose $y^\beta(\rho u_x + \sigma u_y) = 0$ on $\Gamma_0$ (trace sense). Then $T_0u=0$ a.e. on $\Gamma_0$ and \eqref{eq:ybetaDuZeroTrace} implies that
$$
T_yu  = y^\beta(\rho u_x(\cdot,y) + \sigma u_y(\cdot,y)) \to 0 \quad\hbox{in $L^1(\Gamma_0, e^{-\gamma|x|})$ as $y\to 0$}.
$$
Conversely, if $y^\beta(\rho u_x + \sigma u_y)(\cdot,y) \to 0$ in $L^1(\Gamma_0, e^{-\gamma|x|})$ as $y\to 0$, then \eqref{eq:ybetaDuZeroTrace} implies that $T_0u=0$ a.e. on $\Gamma_0$. Hence, $y^\beta(\rho u_x + \sigma u_y) = 0$ on $\Gamma_0$ (trace sense).
\end{proof}

\subsection{Weighted Sobolev spaces and the chain rule}
We have the following version, for our weighted Sobolev spaces, of the analogous results for the standard Sobolev space $H^1(\sO)$ given by \cite[Equations (2.5.44) \& (2.5.45)]{Bensoussan_Lions} (who also include $H^1_0(\sO)$), \cite[Lemma 7.6]{GT}, and \cite[Corollary 2.1.6]{Turesson_2000}, for certain types of weighted Sobolev spaces.

\begin{lem}[Weighted Sobolev spaces and the chain rule]
\label{lem:SobolevSpaceClosedUnderMaxPart}
Let $u,v\in H^1_0(\sO\cup T,\fw)$, where $T\subset\partial\sO$ is relatively open. Then $u^\pm, |u|, \max\{u,v\}, \min\{u,v\} \in H^1_0(\sO\cup T,\fw)$ and
\begin{align*}
D_iu^+ &= \begin{cases}D_iu, & u > 0, \\ 0, & u \leq 0,\end{cases}
\\
D_iu^- &= \begin{cases}0, & u \geq 0, \\ -D_iu & u < 0,\end{cases}
\\
D_i|u| &= \begin{cases}D_iu, & u > 0, \\ 0, & u = 0, \\ -D_iu & u < 0,\end{cases}
\\
D_i\max\{u,v\} &= \begin{cases}D_iu, & u > v, \\ 0, & u = v, \\ D_iv & u < v,\end{cases}
\\
D_i\min\{u,v\} &= \begin{cases}D_iv, & u > v, \\ 0, & u = v, \\ D_iu & u < v.\end{cases}
\end{align*}
\end{lem}

\begin{proof}
For the case of $u^\pm$ and $|u|= u^+ + u^-$, together with $T=\partial\sO$ so $H^1_0(\sO\cup T,\fw) = H^1(\sO,\fw)$, the result follows from the proof of \cite[Lemma 7.6]{GT}, which translates from $H^1(\sO)$ to $H^1(\sO,\fw)$ without change. (Note that our convention, $x^- = \max\{-x,0\} = - \min\{x,0\}$, is opposite in sign to that of \cite[p. 152]{GT}.) Because $\max\{u,v\} = u + (v-u)^+$ and $u,v \in H^1(\sO,\fw)$, it follows that $\max\{u,v\} \in H^1(\sO,\fw)$. Similarly, as $\min\{u,v\} = u - (v-u)^-$ and $u,v \in H^1(\sO,\fw)$, it follows that $\min\{u,v\} \in H^1(\sO,\fw)$.

When $T\subsetneqq\partial\sO$, it suffices to consider the case of $u^+$, as the remaining cases follow as above. The proof of \cite[Lemma 7.6]{GT} for $u^+$ uses the approximation to $u^+$ given by
$$
f_\eps(u) = \begin{cases}(u^2+\eps^2)^{1/2} - \eps, & u > 0, \\ 0, & u \leq 0,\end{cases}
$$
for $\eps>0$. But if $u \in C^\infty_0(\sO\cup T)$, then $f_\eps(u) \in C^\infty_0(\sO\cup T)$. Thus, if $u=0$ on $T$ (trace sense), then $f_\eps(u) = 0$ on $T$ (trace sense). The remainder of the proof of \cite[Lemma 7.6]{GT} now shows that $u^+ \in H^1_0(\sO\cup T,\fw)$.
\end{proof}

\subsection{An application of Hardy's inequality}
Theorem \ref{thm:HardyInequality}, with $p=2$, yields the following version of Proposition \ref{prop:FullEnergyHeston} \emph{without} assuming $b_1=0$ when $0<\beta<1$:

\begin{prop}[Continuity estimate via Hardy's inequality]
\label{prop:FullEnergyHestonHardy}
Assume the coefficients of $A$ and constant $\gamma_0$ satisfy the hypotheses of Proposition \ref{prop:EnergyGardingHeston}. Then
\begin{equation}
\label{eq:StrongerHestonBilinearFormContinuityHardy}
|a(u,v)| \leq C\|u\|_V\|v\|_V, \quad\forall u\in H^1(\sO,\fw), v\in H_0^1(\sO\cup\Gamma_1,\fw),
\end{equation}
where $C$ is a positive constant depending at most on the coefficients $r,q,\kappa,\theta,\rho,\sigma$.
\end{prop}

\begin{proof}
Recall that $H_0^1(\sO\cup\Gamma_1,\fw) = H^1(\sO,\fw)$ when $\beta\geq 1$ by Lemma \ref{lem:ImprovedH1ApproximationLemma} and Remark \ref{rmk:ImprovedH1ApproximationLemma}. When $\beta<1$, we require $v\in H_0^1(\sO\cup\Gamma_1,\fw)$ to ensure $v=0$ on $\Gamma_0$ (trace sense). By Definition \ref{defn:H1WeightedSobolevSpaces} we may assume without loss of generality that $v\in C^\infty_0(\sO\cup\Gamma_1)$ when $\beta<1$ and $v\in C^\infty_0(\sO\cup\Gamma_1)$ when $\beta\geq 1$ by Corollary \ref{cor:KufnerPowerWeight} and Remark \ref{rmk:ImprovedH1ApproximationLemma}. Let $\varphi\in C^\infty(\bar\RR)$ be such that $0\leq\varphi\leq 1$ on $\RR$, $\varphi(y)=1$ for $y\leq 1$, $\varphi(y)=0$ for $y\geq 2$, and $|\varphi'(y)|\leq 2$ for all $y\in\RR$.

Suppose $\supp v \subset (x_0,x_1)\times (0,\infty)$. Theorem \ref{thm:HardyInequality} yields, for all $x\in (x_0,x_1)$,
\begin{align*}
\int_0^1 v^2(x,y) y^{\beta-2}\,dy &\leq \int_0^\infty |\varphi(y)v(x,y)|^2 y^{\beta-2}\,dy
\\
&\leq \frac{4}{(\beta-1)^2}\int_0^\infty |\varphi(y)v_y(x,y) + \varphi'(y)v(x,y)|^2 y^\beta\,dy \quad\hbox{(by \eqref{eq:ClassicHardyInequality})}
\\
&\leq C\int_0^2 \left(v_y^2 + v^2\right)y^\beta\,dy,
\end{align*}
where the constant $C$ depends at most on $\beta$. Observe that
\begin{align*}
\|y^{-1/2}v\|_{L^2(\sO,\fw)}^2 &= \int_{x_0}^{x_1}\int_0^\infty y^{-1}v^2\,\fw(x,y)\,dydx
\\
&= \int_{x_0}^{x_1}\int_0^1 y^{\beta-2}v^2 e^{-\gamma|x|-\mu y}\,dydx + \int_{x_0}^{x_1}\int_1^\infty y^{\beta-2}v^2 e^{-\gamma|x|-\mu y}\,dydx.
\end{align*}
The first integral obeys
\begin{align*}
\int_{x_0}^{x_1}\int_0^1 y^{\beta-2}v^2 e^{-\gamma|x|-\mu y}\,dydx
&\leq
\int_{x_0}^{x_1}\left(\int_0^1 y^{\beta-2}v^2\,dy\right) e^{-\gamma|x|}\,dx
\\
&\leq C\int_{x_0}^{x_1}\left(\int_0^2 y^\beta (v_y^2 + v^2)\,dy\right)e^{-\gamma|x|}dx
\\
&\leq C\int_{x_0}^{x_1}\int_0^2 y(v_y^2 + v^2) y^{\beta-1} e^{-\gamma|x|-\mu y}dydx
\\
&\leq C\int_{x_0}^{x_1}\int_0^\infty y\left(v_y^2 + v^2\right) y^{\beta-1}e^{-\gamma|x|-\mu y}dydx
\\
&= C\|v\|_V^2.
\end{align*}
The second integral obeys
\begin{align*}
\int_{x_0}^{x_1}\int_1^\infty y^{\beta-2}v^2 e^{-\gamma|x|-\mu y}\,dydx
&\leq
\int_{x_0}^{x_1}\int_1^\infty v^2 y^{\beta-1}e^{-\gamma|x|-\mu y}\,dydx
\\
&\leq \int_{x_0}^{x_1}\int_0^\infty v^2 y^{\beta-1}e^{-\gamma|x|-\mu y}\,dydx
\\
&= \frac{\sigma^2}{2}\|v\|_V^2.
\end{align*}
Combining these two integral estimates yields
$$
\|y^{-1/2}v\|_{L^2(\sO,\fw)} \leq C\|v\|_V.
$$
The result now follows by combining the preceding estimate with \eqref{eq:StrongHestonBilinearFormContinuity}.
\end{proof}

\section{The Lax-Milgram theorem and a priori estimates}
\label{sec:LaxMilgram}
We summarize a few consequences of the Lax-Milgram theorem which we use throughout our article. We first recall the classical

\begin{thm}[Lax-Milgram theorem] \cite[Theorem 6.2.1]{Evans}, \cite[Theorem 5.8]{GT}
\label{thm:LaxMilgram}
Let $V$ be a Hilbert space. Suppose $b:V\times V\to\RR$ is a \emph{continuous} bilinear function, that is, there is a positive constant $c_1$ such that
\begin{equation}
\label{eq:LaxMilgramContinuous}
|b(u,v)| \leq c_1\|u\|_V\|v\|_V, \quad\forall u,v \in V,
\end{equation}
which is \emph{coercive}, that is, there is a positive constant $c_2$ such that
\begin{equation}
\label{eq:LaxMilgramCoercive}
b(u,u) \geq c_2\|u\|_V^2, \quad\forall u\in V.
\end{equation}
Then for each $f\in V'$, there exists a unique $u\in V$ such that
\begin{equation}
\label{eq:LaxMilgramEquation}
b(u,v) = f(v), \quad \forall v\in V.
\end{equation}
\end{thm}

\begin{cor}[A priori estimate for Lax-Milgram solutions]
\label{cor:LaxMilgram}
Let $b:V\times V\to\RR$ and $c_2>0$ be as in Theorem \ref{thm:LaxMilgram} and let $f\in V'$. If $u\in V$ is a solution to \eqref{eq:LaxMilgramEquation}, then
\begin{equation}
\label{eq:LaxMilgramAPrioriEstimate}
\|u\|_V \leq (1/c_2)\|f\|_{V'},
\end{equation}
where
$$
\|f\|_{V'} = \sup_{v\in V\setminus\{0\}}\frac{|f(v)|}{\|v\|_V}.
$$
Suppose $H$ is a Hilbert space such that $V\hookrightarrow H\hookrightarrow V'$, via inclusion and $h\mapsto (h,\cdot)_H$ respectively, and that $|v|_H \leq \|v\|_V$. If $f\in H$, then
\begin{equation}
\label{eq:LaxMilgramAPrioriEstimatefH}
\|u\|_V \leq (1/c_2)|f|_H.
\end{equation}
\end{cor}

\begin{proof}
We may assume without loss that $u\neq 0$. Then \eqref{eq:LaxMilgramContinuous} and \eqref{eq:LaxMilgramCoercive} give
$$
\|u\|_V^2 \leq (1/c_2)b(u,u) = (1/c_2)f(u) \leq (1/c_2)\|f\|_{V'}\|u\|_V,
$$
and \eqref{eq:LaxMilgramAPrioriEstimate} follows. When $f\in H$, observe that
$$
\|f\|_{V'} = \sup_{v\in V\less\{0\}}\frac{|f(v)|}{\|v\|_V}
\\
\leq \sup_{v\in V\less\{0\}}\frac{|f(v)|}{|v|_H}
\\
\leq \sup_{v\in H\less\{0\}}\frac{|f(v)|}{|v|_H} = |f|_H,
$$
and this yields \eqref{eq:LaxMilgramAPrioriEstimatefH}.
\end{proof}

The following observations will be useful when constructing solutions to variational inequalities.

\begin{lem}[Bilinear forms and weak limits]
\label{lem:BilinearFormWeakLimit}
Suppose $b:V\times V\to\RR$ is a continuous bilinear form on a Hilbert space $V$. Let $\{u_n\}_{n\geq 1}, \{v_n\}_{n\geq 1} \subset V$ be sequences such that $u_n \rightharpoonup u \in V$ weakly and $v_n \to v\in V$ strongly. Then
\begin{enumerate}
\item $\lim_{n\to\infty} b(u_n,v) = b(u,v)$;
\item $\lim_{n\to\infty} b(u_n,v_n) = b(u,v)$;
\item If $b:V\times V\to\RR$ is coercive, then $b(u,u) \leq \liminf_{n\to\infty} b(u_n,u_n)$.
\end{enumerate}
\end{lem}

\begin{proof}
By the proof of \cite[Theorem 6.2.1]{Evans} (via the Riesz Representation Theorem \cite[\S D.3]{Evans}), there is a bounded, linear operator $B:V\to V$ such that
$$
b(u,v) = (Bu, v)_V, \quad\forall u, v \in V.
$$
Moreover, if $B^*:V\to V$ is the adjoint operator defined by $(Bu, v)_V = (u, B^*v)_V$, for all $u,v \in V$, then
$$
\lim_{n\to\infty} b(u_n,v) = \lim_{n\to\infty} (Bu_n,v)_V = \lim_{n\to\infty} (u_n,B^*v)_V = (u,B^*v)_V = (Bu,v)_V = b(u,v).
$$
This proves (1). Since weakly convergent sequences are bounded \cite[\S D.4]{Evans}, there is a positive constant $K$ such that $\|u_n\|_V \leq K, \forall n\geq 1$. Then \eqref{eq:LaxMilgramContinuous} yields
\begin{align*}
|b(u_n,v_n) - b(u,v)| &= |b(u_n,v_n) - b(u_n,v) + b(u_n,v) - b(u,v)|
\\
&\leq |b(u_n,v_n-v)| + |b(u_n-u,v)|
\\
&\leq c_1K\|v_n-v\|_V + |b(u_n,v) - b(u,v)|,
\end{align*}
and so (2) follows from (1). When $b:V\times V\to \RR$ is coercive, then $u\mapsto \sqrt{b(u,u)}$ defines a norm on $V$ which is equivalent to $\|v\|_V$ and thus (3) follows from \cite[\S D.4]{Evans}.
\end{proof}

\section{Explicit solution to the elliptic Cox-Ingersoll-Ross equation}
\label{sec:CIR}
The following example illustrates some of the subtleties surrounding the boundary behavior of solutions to the elliptic Cox-Ingersoll-Ross equation, and thus the Heston equation, near $\Gamma_0$.

\begin{exmp}[Elliptic Cox-Ingersoll-Ross equation and confluent hypergeometric functions]
\label{exmp:CIR}
When the source function, $f$, in Problem \ref{prob:HestonWeakMixedBVPHomogeneous} is independent of $x$, it is natural to consider $f\in L^2(\RR_+,\fm)$ with $\fm(y):=y^{\beta-1}e^{-\mu y}$ and examine the problem of existence, uniqueness, and regularity of weak solutions $u \in H^1(\RR_+,\fm)$ to the elliptic \emph{Cox-Ingersoll-Ross} equation,
\begin{equation}
\label{eq:CIR}
Bu = f \hbox{ a.e. on }\RR_+,
\end{equation}
that is
\begin{equation}
\label{eq:CIRWeak}
b(u,v) = (f,v)_{L^2(\RR_+,\fm)}, \forall v\in H^1(\RR_+,\fm),
\end{equation}
where
$$
Bu :=  -\frac{\sigma^2}{2}yu_{yy} - \kappa(\theta-y)u_y + ru,
$$
and
$$
b(u,v) := \int_{\RR_+}\left(\frac{\sigma^2}{2}yu_yv_y + ruv\right)\fm\,dy,
$$
with (noting that the definition here differs slightly from that of Definition \ref{defn:H1WeightedSobolevSpaces})
$$
\|v\|_{H^1(\RR_+,\fm)}^2 := \int_{\RR_+}\left(zv_z^2 + v^2\right)\fm\,dz.
$$
Theorem \ref{thm:ExistenceUniquenessEllipticCoerciveHeston} shows that there exists a unique solution $u \in H^1(\RR_+,\fm)$ to \eqref{eq:CIRWeak}, while Theorem \ref{thm:GlobalRegularityEllipticHestonSpecial} shows that $u \in H^2(\RR_+,\fm)$, where  (noting that the definition here differs slightly from that of Definition \ref{defn:H2WeightedSobolevSpaces})
$$
\|v\|_{H^2(\RR_+,\fm)}^2 := \int_{\RR_+}\left(z^2v_{zz}^2 + (1+z^2)v_z^2 + \right)\fm\,dz.
$$
Lemma \ref{lem:HestonWeightedNeumannBVPHomogeneous} implies that this $u$ solves \eqref{eq:CIR}, and also obeys the weighted-Neumann boundary condition,
\begin{equation}
\label{eq:CIRWeightedNeumann}
y^\beta u'(0) = 0.
\end{equation}
It is instructive to examine the question of uniqueness of solutions to \eqref{eq:CIR} with the aid of explicit formulas from \cite[\S 13]{AbramStegun}. Suppose $f=0$ in \eqref{eq:CIR}. Writing $z := \mu y = 2\kappa y/\sigma^2$ and $v(z) := u(y)$, the equation $Bu=0$ on $\RR_+$ becomes
\begin{equation}
\label{eq:KummerEquation}
zv_{zz} + (\beta-z)v_z - av = 0 \hbox{ on }\RR_+,
\end{equation}
where $a := r/\kappa\in\RR$. This is the \emph{Kummer equation} with solution \cite[\S 13.1.11]{AbramStegun}
\begin{equation}
\label{eq:KummerSolution}
v(z) = c_1 U(a,\beta,z) + c_2 M(a,\beta,z), \quad z\in \RR_+,
\end{equation}
where $c_1,c_2\in\RR$ and $M(a,\beta,z), U(a,\beta,z)$ are the \emph{Kummer} or \emph{confluent hypergeometric functions} \cite[\S 13]{AbramStegun}. For any $\beta>0$, one knows that $M(a,\beta,z) \sim z^{-a} + z^{a-\beta}e^z$ as $z\to\infty$ \cite[\S 13.5.1]{AbramStegun} and thus $M\notin H^1(\RR_+,\fm)$ and we must have $c_2=0$, where we now write the weight as $\fm(z) = z^{\beta-1}e^{-z}$.

For any $\beta>0$, one knows that $U(a,\beta,z) \sim z^{-a}$ as $z\to\infty$ \cite[\S 13.5.2]{AbramStegun}; for $\beta>1$, one has $U(a,\beta,z) \sim z^{1-\beta}$ as $z\to 0$ by \cite[\S 13.5.6--8]{AbramStegun}; for $\beta=1$, one has $U(a,\beta,z) \sim \log z$ as $z\to 0$ by \cite[\S 13.5.9]{AbramStegun}; and, for $0<\beta<1$, one has $U(a,\beta,z) \sim z^{1-\beta}$ as $z\to 0$ by \cite[\S 13.5.10]{AbramStegun} and $U_z(a,\beta,z) = -aU(a+1,\beta+1,z) \sim z^{-\beta}$ by \cite[\S 13.4.21 \& \S 13.5.8]{AbramStegun}. Therefore, $U\notin H^1(\RR_+,\fm)$ when $\beta\geq 1$ while $U\in H^1(\RR_+,\fm)$ when $0<\beta<1$. But $U_{zz}(a,\beta,z) = a(a+1)U(a+2,\beta+2,z) \sim z^{-\beta-1}$ by \cite[\S 13.4.22 \& \S 13.5.6]{AbramStegun}. Therefore, $zU_{zz}  \sim z^{-\beta}$ as $z\to 0$ and so $U\notin H^2(\RR^+,\fm)$ by Definition \ref{defn:H2WeightedSobolevSpaces}. Consequently, we must have $c_1=0$ in \eqref{eq:KummerSolution}. (Equivalently, $U_z \sim z^{-\beta}$ as $z\to 0$ and so $U_z\notin L^2(\RR^+,\fm)$ and therefore $U\notin H^2(\RR^+,\fm)$ by Definition \ref{defn:H2WeightedSobolevSpaces}.)  Hence, $u=0$ is the unique solution in $H^1(\RR_+,\fm)$ to \eqref{eq:CIR} when $f=0$.

Observe that as $z\to 0$ we have $U(a,\beta,z) \sim \Gamma(1-\beta)/\Gamma(1+a-\beta)$ \cite[\S 13.5.10]{AbramStegun} and $z^\beta U_z(a,\beta,z) \sim
-a\Gamma(\beta)/\Gamma(a+1)$ by \cite[\S 13.4.21 \& \S 13.5.8]{AbramStegun}, and so the boundary conditions $v(0)=0$ or $z^\beta v_z(0)=0$ also imply that $c_1=0$.

Lastly, suppose that $f\in C^\alpha(\bar\RR_+)$ and $u_s\in C^{2,\alpha}(\bar\RR_+)$ is a solution to $Bu=f$ on $\RR_+$. If $u_s(0)\neq 0$ and, when $0<\beta<1$, $c\in\RR\less\{0\}$ is chosen such $u := u_s+cU$ obeys $u(0)=0$ (that is $c=-u_s(0)/U(0)$), then $u\in C(\bar\RR_+)\cap C^2(\RR_+)$ solves $Bu=f$ on $\RR_+$ and $u(0)=0$ but $u\notin C^1(\bar\RR_+)$.
\qed
\end{exmp}

%
%

\bibliography{finance,financejpm,jpm,master,mfpde,nsf,nsf08,nsf09,projects}
\bibliographystyle{amsplain}

\end{document}